\documentclass[11pt,a4paper]{amsart}
\usepackage[T1]{fontenc}
\usepackage[nomath]{lmodern}
\usepackage{amsmath,amssymb,amsthm,mathtools,mathrsfs,esint}
\usepackage[a4paper,left=24mm,right=24mm,top=26mm,bottom=26mm]{geometry}
\usepackage{microtype,enumitem,booktabs,xcolor}
\usepackage[colorlinks=true,linkcolor=blue!45!black,citecolor=blue!45!black,urlcolor=blue!45!black]{hyperref}
\hypersetup{pdftitle={Local sheeting and rigidity for general double-well potentials},pdfauthor={}}

\makeatletter
\def\@setaddresses{\par\nobreak\begingroup\footnotesize
  \def\author##1{\nobreak\addvspace\bigskipamount}
  \def\\{\unskip, \ignorespaces}
  \interlinepenalty\@M
  \def\address##1##2{\begingroup
    \par\addvspace\bigskipamount\indent
    \@ifnotempty{##1}{(\ignorespaces##1\unskip) }
    {\scshape\ignorespaces##2}\par\endgroup}
  \def\email##1##2{\begingroup
    \@ifnotempty{##2}{\nobreak\indent\emailaddrname
      \@ifnotempty{##1}{, \ignorespaces##1\unskip}\/\space
      \ttfamily##2\par}\endgroup}
  \addresses\endgroup}
\makeatother
\newtheorem{theorem}{Theorem}[section]
\newtheorem{proposition}[theorem]{Proposition}
\newtheorem{lemma}[theorem]{Lemma}
\newtheorem{corollary}[theorem]{Corollary}
\theoremstyle{definition}

\theoremstyle{remark}
\newtheorem{remark}[theorem]{Remark}
\numberwithin{equation}{section}
\newcommand{\C}{\mathcal C}
\DeclareMathOperator{\Div}{div}

\allowdisplaybreaks[1]
\setlist{topsep=4pt,itemsep=2pt} 
\title[Sheeting and rigidity for Allen--Cahn]{Sheeting property and rigidity for stable solutions to the Allen--Cahn equation}
\author[Y. Liu]{Yong Liu}
\address{School of Mathematics and Statistics, Beijing Technology and Business University, Beijing, China}
\email{yliumath@btbu.edu.cn}

\author[T. Luo]{Tianci Luo}
\address{School of Mathematical Sciences, University of Science and Technology of China, Hefei 230026, P.R. China}
\email{Luo\_tianci@mail.ustc.edu.cn}

\author[K. Wang]{Kelei Wang}
\address{School of Mathematics and Statistics, Wuhan University, Wuhan, China}
\email{wangkelei@whu.edu.cn}

\author[J.C. Wei]{Juncheng Wei}
\address{Juncheng Wei, Department of Mathematics, Chinese University of Hong Kong, Shatin, New Territories, Hong Kong}
\email{wei@math.cuhk.edu.hk}

\author[Y. Wei]{Yong Wei}
\address{School of Mathematical Sciences, University of Science and Technology of China, Hefei 230026, P.R. China}
\email{yongwei@ustc.edu.cn}

 \author[K. Wu]{Ke Wu}
\address{School of Mathematics, Yunnan Normal University, Kunming, China}
\email{kewu@ynnu.edu.cn}

\date{}

\begin{document}

\begin{abstract}
We prove a sheeting theorem for stable solutions to the Allen--Cahn
equation in ambient dimensions $2$ through $10$. We show that weak convergence of
the measures $\varepsilon|Du_\varepsilon|^2\,dX$ to a hyperplane with
integer multiplicity implies that the zero sets are disjoint smooth graphs.
As a consequence, stable entire solutions in $\mathbb R^N$, $N\le7$,
with $O(R^{N-1})$ energy growth are one-dimensional.
When $W''(-1)=W''(1)$, every blowdown of an entire solution with finite
Morse index and this energy growth has multiplicity one for
$4\le N\le10$. For even potentials, we also obtain multiplicity one
for limits with bounded energy and Morse index on closed manifolds
of dimensions $3$ through $7$, under bumpiness or positive Ricci curvature.
\end{abstract}
\maketitle

\section{Introduction and main results}\label{intro}

This paper is devoted to the sheeting property of stable solutions of
the Allen--Cahn equation and its application to the classification of
stable entire solutions. Our objectives are to recover the individual
transition layers from a weak planar limit and to prove
one-dimensionality under the natural energy growth bound.
We also use this local analysis to study multiplicity for finite-index
entire solutions and for bounded-index sequences on closed manifolds.

Throughout the paper, $W$
is a smooth nonnegative function on $\mathbb R$ with precisely three
critical points $-1,0,1$, and
\begin{equation}\label{wells}
 W(-1)=W(1)=0,\qquad W>0\ \text{in }(-1,1),\qquad
 W''(-1)>0,\quad W''(1)>0,\quad W''(0)<0.
\end{equation}
The two wells have equal depth, and their second derivatives may differ.
These assumptions allow asymmetric potentials and nonconcave $\sqrt{2W}$.
We denote by $g$ the increasing heteroclinic solution of
\begin{equation*}
 g''=W'(g),\qquad g(0)=0,\qquad g(\pm\infty)=\pm1.
\end{equation*}
This solution connects the two wells. Its energy is
\begin{equation*}
 \sigma_0=\int_{\mathbb R}g'(t)^2\,dt
          =\int_{-1}^1\sqrt{2W(s)}\,ds.
\end{equation*}
We use $N=n+1$ for the ambient dimension, with
$X=(x,z)\in\mathbb R^n\times\mathbb R$, and denote the horizontal
gradient by $D_x$.

We consider the singularly perturbed Allen--Cahn equation
\begin{equation}\label{acn}
 \varepsilon\Delta u_\varepsilon=\frac{1}{\varepsilon}W'(u_\varepsilon),
 \qquad -1\le u_\varepsilon\le1 \quad \text{in} \ \ \Omega\subset\mathbb R^N.
\end{equation}
The parameter $\varepsilon>0$ measures the thickness of a
transition layer between the two wells of $W$. The equation is the
Euler--Lagrange equation of the energy
\begin{equation*}
 \mathcal E_\varepsilon(u;\Omega)
 =\int_\Omega\left(\frac{\varepsilon}{2}|Du|^2
                  +\frac{W(u)}{\varepsilon}\right)\,dX.
\end{equation*}
We call a solution of \eqref{acn} stable in $\Omega$ if
\begin{equation}\label{stab}
 \int_\Omega\left(\varepsilon|D\varphi|^2+
       \varepsilon^{-1}W''(u_\varepsilon)\varphi^2\right)\,dX\ge0
 \qquad\text{for every }\varphi\in C_c^1(\Omega).
\end{equation}
The local sheeting question is whether weak convergence of the measures
$\varepsilon|Du_\varepsilon|^2\,dX$ to a hyperplane with integer
multiplicity forces the zero sets to decompose into disjoint smooth
graphs. Note that several transition layers may approach the same hyperplane.

The corresponding regularity problem for minimal hypersurfaces has
a long history, which goes back to the work of De Giorgi
\cite{DeGiorgi1960} on partial regularity of area-minimizing boundaries.
In that work, he introduced the harmonic approximation method:
a sufficiently flat minimal hypersurface can be approximated by a
harmonic function. Allard \cite{Allard} subsequently established a
regularity theorem for stationary integral varifolds near a plane of
multiplicity one. These results made tilt-excess decay, which controls
the variation of tangent hyperplanes, a fundamental method in geometric
measure theory.

Higher multiplicity presents an additional difficulty for the analysis. Several sheets
may approach the same plane, and weak convergence alone leaves open
the possibility of small necks or singularities. For stable minimal
hypersurfaces, Schoen--Simon \cite{SS81} established a sheeting theorem
under a suitable size assumption on the singular set. Their theorem
shows that stability, together with weak closeness to a hyperplane,
possibly with higher multiplicity, forces a local decomposition into
disjoint smooth minimal graphs. Wickramasekera \cite{Wic14} extended
this theory to stable integral varifolds under a structural hypothesis
excluding classical junction singularities. His theorem replaces the
earlier assumption on the size of the singular set and preserves the
optimal bound of codimension $7$ for the singular set in the
hypersurface. More recently, Bellettini \cite{Bel25} obtained curvature
and tilt estimates by De Giorgi iteration on stable minimal immersions.
His tilt estimate controls the direction of the tangent hyperplanes
and leads to a graphical description near a hyperplane. For further
recent developments, we refer to Minter--Xiao \cite{MX26} and
Wang--Zhang \cite{WZ26}.

The variational relation between the Allen--Cahn equation and minimal
hypersurfaces suggests an analogous sheeting theorem for transition
layers. The convergence of $\mathcal E_\varepsilon$ to $\sigma_0$ times
the perimeter functional in the sense of $\Gamma$-convergence makes
this relation precise for energy minimizers \cite{Modica87}.
For general critical points with
uniform energy, Hutchinson--Tonegawa \cite{HT} established
subsequential convergence to stationary integral varifolds. The
limiting multiplicity counts the transition layers that concentrate
on the same interface. Under the further assumption of stability,
Tonegawa \cite{Ton05} derived a stability inequality for the limit.
Tonegawa--Wickramasekera \cite{TW12} then established optimal regularity
of its support, whose singular set has Hausdorff dimension at most
$N-8$.

To recover the geometry of the solutions from the limiting varifolds,
a further regularity argument is needed. Even when the limiting
interface is a hyperplane, a graphical description of the individual
transition layers still requires proof. In the multiplicity one case,
this description and the associated uniform $C^{1,\theta}$ estimate
were essentially established by Savin \cite{Savin09} and later
explicitly stated and reproved by Wang \cite{Wang17}. Savin used the
Harnack inequality and improvement of flatness in the $L^\infty$
setting, while Wang used a tilt-excess decay estimate in the spirit
of Allard. In the higher-multiplicity case, the layers may approach
the hyperplane at different rates, and small necks connecting them
may disappear in the weak limit. Their size must therefore be
controlled before the individual graphs can be constructed. Recall that Wang--Wei \cite[Theorem~1.1 and Corollary~1.2]{WW} obtained uniform
second-order estimates in ambient dimensions at most ten. Their
corollary starts from ordered transition layers whose graph functions
converge in $C^1$.  

The sheeting question also arises in the classification of entire
solutions. De Giorgi \cite{DG78} conjectured that a bounded entire
solution of the classical Allen--Cahn equation
\begin{equation*}
 \Delta u=u^3-u,
\end{equation*}
which is strictly monotone in one direction is one-dimensional when
$N\le8$. This conjecture was proved by Ghoussoub--Gui \cite{GG98} in
dimension two, by Ambrosio--Cabr\'e \cite{AC00} in dimension three,
and recently in our work \cite{Liu26} in dimension four.
Alberti--Ambrosio--Cabr\'e \cite{AAC01} studied the relation between
monotonicity, stability, and local minimality for general
nonlinearities. Savin \cite{Savin09} proved regularity of flat level
sets for minimizers. His result also proves De Giorgi's conjecture
through dimension eight when the solution tends to $\pm1$ along the
monotone direction. The limit assumption ensures the minimizing
property needed by this argument. In dimension nine,
del Pino--Kowalczyk--Wei \cite{DPKW11} constructed bounded monotone
solutions with nonplanar level sets. These results parallel those
of the Bernstein problem for minimal graphs.

Stability is a weaker hypothesis than monotonicity, because a positive
directional derivative solves the linearized equation and implies
the stability inequality. The stable De Giorgi conjecture asks whether
every bounded stable entire solution of the classical Allen--Cahn
equation is one-dimensional in ambient dimensions at most seven.
The two-dimensional case follows from \cite{GG98}, and we recently
proved the three-dimensional case in \cite{Liu26}. The conjecture
remains open in dimensions $4\le N\le7$ without an energy growth assumption. Pacard--Wei \cite{PW13}
constructed nonplanar stable solutions in dimension eight from minimal
cones, showing that this dimension range cannot be enlarged.

Towards a proof of this conjecture, it is natural to first consider solutions with the following energy
growth bound:
\begin{equation}\label{natgrowth}
 \int_{B_R}\left[\frac{1}{2}|\nabla u|^2+W(u)\right]\,dX
 \le CR^{N-1}\qquad\text{for }R\ge1.
\end{equation}
The stable De Giorgi conjecture with this additional assumption is
usually called the weak stable De Giorgi conjecture. In dimensions
$N\le3$, the conclusion follows from the quotient Liouville theorem
\cite[Theorem~3.1 and Proposition~4.2]{AAC01}. Florit-Simon--Serra
\cite[Theorem~1.3]{FSS25} recently proved the $N=4$ case.
Note that the energy bound assumption permits the rescaling
\begin{equation*}
 u_\varepsilon(X)=u(\varepsilon^{-1}X),\qquad \varepsilon\to0.
\end{equation*}
These functions solve \eqref{acn}. The compactness and stability
results \cite{HT,Ton05,TW12} yield subsequential stationary integral
varifold limits, called blowdown limits. The monotonicity formula
implies that each such limit is a cone. When $N\le7$, regularity
and Simons's theorem \cite{Simons} imply that the cone is a
hyperplane, possibly with integer multiplicity greater than one.

There is a more general class of solutions called Finite Morse index solutions. The Morse index is the maximal dimension of a space on which the
second variation is negative definite. Such solutions are stable
outside a compact set, so the energy bound still yields stationary
cone limits, with stability away from the origin. Several layers
can contribute to the same limiting cone in dimension three.
The entire solutions constructed near a catenoid by
del Pino--Kowalczyk--Wei \cite{DPKW13} have finite Morse index and
quadratic energy growth. Their two ends have logarithmic heights
and converge under rescaling to the same plane, each contributing
one transition layer. Thus their blowdown limit is a plane with
multiplicity two. For the classical potential, Florit-Simon
\cite[Theorems~1.1--1.2]{FS26} proved one-dimensionality in dimension
four under finite Morse index and bounded energy density.  

 Min-max
constructions produce minimal hypersurfaces with positive integer
multiplicities. Marques--Neves \cite[Conjecture~1.2]{MN21} formulated
the expectation that, for generic metrics in ambient dimensions three
through seven, the components obtained by min-max are two-sided and
have multiplicity one. Zhou \cite[Theorem~A]{Zhou20} proved this
conclusion for the volume-spectrum min-max hypersurfaces of bumpy
metrics in these dimensions. Allen--Cahn min-max constructions were
developed by Guaraco \cite{Gua18} and Gaspar--Guaraco \cite{GG18}.

For Allen--Cahn solutions, one can ask about every sequence with
uniformly bounded energy and Morse index, independently of its
construction. Chodosh--Mantoulidis \cite[Theorem~1.7 and Corollary~6.1]{CM20}
proved multiplicity one for such sequences in dimension three, for
bumpy metrics and for metrics with positive Ricci curvature, under
their assumptions on the even potential. Their argument controls
the separation between layers near the finitely many points where
stability may fail.

Our first main result concerns the local sheeting question in ambient
dimensions two through ten.
\begin{theorem}\label{sheet}
Let $1\le n\le9$. Suppose that $\varepsilon_j\downarrow0$ and that
$u_j$ are smooth stable solutions of \eqref{acn} in $B_1^{n+1}$,
with $-1\le u_j\le1$. Assume
\begin{equation}\label{planar}
 \varepsilon_j|Du_j|^2\,dX
 \rightharpoonup Q\sigma_0\mathcal H^n\lfloor\{z=0\}
 \quad\text{locally in }B_1^{n+1},\qquad Q\in\mathbb N.
\end{equation}
Then for every product $B_r^n\times(-h,h)\Subset B_1^{n+1}$ and all
sufficiently large $j$, there are smooth functions
$f_{1,j}<\cdots<f_{Q,j}$ on $B_r^n$ such that
\begin{equation*}
 \{u_j=0\}\cap(B_r^n\times(-h,h))
 =\bigcup_{\alpha=1}^Q\{(x,f_{\alpha,j}(x))\mid x\in B_r^n\}.
\end{equation*}
For any $0<\theta<1$, there exists a constant $C(r,\theta)$ such that
\begin{equation*}
 \max_\alpha\|f_{\alpha,j}\|_{C^{2,\theta}(B_r)}\leq C(r,\theta).
\end{equation*}
\end{theorem}

Our second main result is the one-dimensional classification under
the natural energy growth bound.
\begin{theorem}\label{global}
Let $N\le7$, and let $U\in C^2(\mathbb R^N;(-1,1))$ be a stable
entire solution of $\Delta U=W'(U)$, satisfying the energy growth
condition \eqref{natgrowth} with $u=U$.
Then we have
\begin{equation*}
 U(X)=g(e\cdot X-b)
 \quad\text{for some }e\in\mathbb S^{N-1},\ b\in\mathbb R.
\end{equation*}
\end{theorem}

This theorem resolves the weak stable De Giorgi conjecture in its
full dimension range. 
The next theorem treats general potentials with
$W''(-1)=W''(1)$ and allows nonflat limiting cones in dimensions
eight through ten. 
\begin{theorem}\label{finindex}
Let $4\le N\le10$, and assume that $W$ satisfies \eqref{wells}
with $W''(-1)=W''(1)$.
Let $U\in C^2(\mathbb R^N;(-1,1))$ be an entire solution of
$\Delta U=W'(U)$ with finite Morse index. Suppose that, for some
$K>0$,
\begin{equation}\label{fingrowth}
 \int_{B_R}\left(\frac12|DU|^2+W(U)\right)\,dX
 \le KR^{N-1}\qquad\text{for }R\ge1.
\end{equation}
Then every varifold blowdown of $U$ is a stationary stable
minimal hypercone with multiplicity one.
\end{theorem}

We emphasize that in this result, $N$ can not be $3$, as can be seen from solutions with catenoidal ends.

Our final result concerns bounded-index sequences on a closed
manifold. A metric is bumpy if every normal Jacobi field on every
closed immersed minimal hypersurface vanishes identically. Such a
field solves the linearized minimal hypersurface equation. 
\begin{theorem}\label{multone}
Let $(M^N,\mathbf g)$ be a smooth closed Riemannian manifold,
$3\le N\le7$, and suppose that $\mathbf g$ is bumpy or that
$\operatorname{Ric}_{\mathbf g}>0$.
Assume that $W$ satisfies \eqref{wells} and is even.
Let $\varepsilon_j\downarrow0$ and let
$u_j\in C^\infty(M;[-1,1])$ satisfy
\begin{equation}\label{closedac}
 \varepsilon_j^2\Delta_{\mathbf g}u_j=W'(u_j).
\end{equation}
Suppose that, for fixed $E_0,I_0<\infty$,
\begin{equation}\label{closedbd}
 \mathcal E_{\varepsilon_j}(u_j;M)
 :=\int_M\left(\frac{\varepsilon_j}{2}|\nabla u_j|_{\mathbf g}^2
              +\frac{W(u_j)}{\varepsilon_j}\right)d\mu_{\mathbf g}
 \le E_0,
 \qquad \operatorname{ind}(u_j)\le I_0.
\end{equation}
Then every nonzero subsequential limit of the associated varifolds,
normalized by $\sigma_0^{-1}$, is of the form
\begin{equation*}
 V=\sum_{\ell=1}^J|\Sigma_\ell|,
\end{equation*}
where the $\Sigma_\ell$ are pairwise disjoint, smooth, closed,
embedded, two-sided minimal hypersurfaces.
In particular, every component has multiplicity one, and along
the same subsequence
\begin{equation*}
 \left(\frac{\varepsilon_j}{2}|\nabla u_j|_{\mathbf g}^2
              +\frac{W(u_j)}{\varepsilon_j}\right)d\mu_{\mathbf g}
 \rightharpoonup
 \sigma_0\sum_{\ell=1}^J
             \mathcal H_{\mathbf g}^{N-1}\lfloor\Sigma_\ell.
\end{equation*}
\end{theorem}

The theorem applies to arbitrary sequences satisfying
\eqref{closedbd}, independently of a min-max characterization.
Here evenness of the potential allows us to change the sign of a solution while
keeping the same equation. It also makes the centered heteroclinic
odd, as required in the barrier construction of
Chodosh--Mantoulidis \cite[Theorem~7.4 and Lemma~7.10]{CM20}.

The proof of our local sheeting theorem begins with an estimate for the horizontal
gradient energy, before the layers have been expressed as graphs.
Tonegawa's tangential stability inequality \cite{Ton05} controls
variations in the direction of the horizontal gradient. We use Serre's
higher integrability estimate for nonnegative symmetric tensors
with divergence constraint \cite{Serre} to control the concentration of
the projected horizontal gradient energy. Together with estimates
along level curves, these inequalities yield a superlinear bound
for the energy near points with small gradient or large tilt.
Their contribution is therefore negligible at the scale of the
tilt-excess. This estimate allows harmonic approximation of the
individual layers starting from weak planar convergence, including
when several layers approach the same plane.

Strong compactness of the vertical energy integrals then recovers
the full tilt-excess from the graphical regions. Following the
harmonic approximation in Florit-Simon--Serra \cite{FSS25}, we use
one affine function for all layers.
Possible decreasing of multiplicity is handled by an induction argument. As in Wang
\cite{Wang17}, the iteration must account for the transition
thickness: excess decay applies while the excess is much larger
than $\varepsilon^2$. At the last scale, the rescaled solutions
approach a stable entire solution with the full energy bound and
a stronger bound for its horizontal gradient energy, uniform in the
center of the ball. Proposition~\ref{crit} classifies these limits
and completes the iteration.

This directional classification uses tangential stability to control
the number of regions where the solution is not close to a 
vertical transition, following Liu--Wang--Wei--Wu \cite{LWWW24}.
On the remaining set, we construct ordered layers and project the
equation and stability onto heteroclinic derivatives, as in
Wang--Wei \cite{WWends,WW}. The resulting Toda equations describe
the exponentially small interactions between adjacent layers.
We estimate the errors in weighted averages, including the terms
from the excluded regions. Spherical averaging and radial Hardy
tests \cite{WangToda19} then exclude several layers.
With $n=N-1$, the strict inequality $(n-2)/4<2$, equivalently
$N\le10$, leaves room for the averaged errors and explains
the dimension range of local sheeting. In the global problem,
the restriction $N\le7$ enters earlier, when we prove that the blowdown
cone is a hyperplane. Sheeting on expanding balls then
shows that the zero set consists of finitely many parallel
hyperplanes. Half-space rigidity \cite{BCN93,BCN97} and unique
continuation imply one-dimensionality.

  For finite Morse index solutions, we exclude several
layers over the regular parts of planar and nonflat blowdown cones.
For bounded-index sequences on a closed manifold, a component with
higher multiplicity carries a positive Jacobi field, on its normal
double cover when it is one-sided. Bounds for the separation between
layers near the finitely many excluded points allow this field to
extend across them. The metric assumptions then exclude higher
multiplicity. These arguments connect local control of each transition
layer with the multiplicity of the limiting hypersurface.
On closed manifolds, we then obtain convergence of Morse indices
for bumpy metrics and another proof of the corresponding existence
results for infinitely many minimal hypersurfaces.

This paper is organized as follows.
Sections~\ref{locenergy}--\ref{vertcomp} develop the local estimates
and excess decay, Section~\ref{rigidpart} proves directional rigidity,
and Section~\ref{sheetproof} treats sheeting, global rigidity, and
higher-dimensional consequences.
Sections~\ref{finproof}--\ref{multproof} prove the two multiplicity
results, and Appendix~\ref{aux} contains the one-dimensional estimates.

\section{Energy estimates and exceptional points}\label{locenergy}

In this section, we estimate the horizontal gradient energy before constructing the
individual layers. The planar limit first yields local bounds for the
full energy and confines the transition region to a thin slab.
We then estimate the difference between the gradient and potential
energy densities. Finally, we cover the points where the gradient is
small or has large horizontal tilt. The energy of this covering is
bounded by a power of the tilt-excess greater than one.

This last bound is needed in Section~\ref{vertcomp}, where we divide
the estimates by the tilt-excess. A bound proportional to that excess
would leave an error in the limit. The tangential stability inequality
and the higher integrability estimate for the projected energy yield
the additional power needed to make the error vanish.

We work with $n\ge2$. The planar case will follow by adding an
independent coordinate in Section~\ref{sheetproof}. We use
$\C_r=B_r^n\times(-r,r)$ and define the tilt-excess relative to the
horizontal hyperplane by
\begin{equation*}
 E_\varepsilon=\int_{\C_{64}}\varepsilon|D_xu_\varepsilon|^2\,dX.
\end{equation*}
It is the squared horizontal component of the level-set normal,
integrated against $\varepsilon|Du_\varepsilon|^2\,dX$, and it
vanishes when the solution depends only on $z$.
Unless a different domain is specified, the local estimates concern
smooth stable solutions of \eqref{acn} on $\C_{64}$.
Here $B_r^n(x)$ is a ball in the horizontal space, and
$\omega_n=\mathcal L^n(B_1^n)$ is the volume of its unit ball. Constants may
depend on $W$, the dimension, the fixed domains, the chosen compact
intervals of level values, and the local energy bound. They are
independent of $\varepsilon$ and the particular solution.

\subsection{Full energy bounds and geometric stability}
We first clarify the potential assumptions in later applications of
Wang \cite{Wang17}. By \eqref{wells}, $W''$ has a positive lower bound
in fixed neighborhoods of $-1$ and $1$. We may extend its restriction
to $[-1,1]$ smoothly so that $W''\ge\kappa>0$ for $|s|\ge\gamma$,
with some $\gamma\in(0,1)$, as required in Section~2 of that paper.
Beyond $[-1,1]$, we smoothly continue $W''$ to a positive
constant and integrate twice from the corresponding well.
This changes $W$ only outside the range of our solutions.
Their equation, second variation, energy, and heteroclinic remain
unchanged. We use this extension whenever we invoke that paper.

We begin with the relation between the gradient energy and the full energy.

\begin{lemma}\label{gtfene}
Let $\varepsilon\downarrow0$ and let $u_\varepsilon$ be stable
solutions of \eqref{acn} in an open set, with a uniform local bound
for $\int\varepsilon|Du_\varepsilon|^2\,dX$. Their full energy is locally
uniformly bounded. If the measures $\varepsilon|Du_\varepsilon|^2\,dX$ converge to
$Q\sigma_0\mathcal H^n\lfloor\{z=0\}$, with $Q\in\mathbb N$, then the full energy
measures have the same local limit, the horizontal gradient energy
tends to zero on compact interior regions, and, for every fixed $b\in(0,1)$, the sets $\{|u_\varepsilon|\le b\}$ converge locally in Hausdorff distance to the hyperplane.
\end{lemma}
\begin{proof}
We choose $s_0>0$ so that $W''\le-c_0<0$ on $[-2s_0,2s_0]$,
and $\zeta\in C_c^\infty((-2s_0,2s_0))$ equal to one on
$[-s_0,s_0]$. For a fixed compactly supported spatial cutoff $\eta$,
we obtain from stability with $\eta\zeta(u_\varepsilon)$
\begin{equation*}
 \frac1\varepsilon\int_{\{|u_\varepsilon|\le s_0\}}\eta^2\,dX
 \le C\int\varepsilon|Du_\varepsilon|^2\eta^2\,dX
                              +C\varepsilon\int|D\eta|^2\,dX.
\end{equation*}
By \eqref{wells}, for $s\in[-1,1]$ we have
\begin{equation*}
 W(s)\le C\bigl(-sW'(s)+\mathbf1_{[-s_0,s_0]}(s)\bigr),
 \qquad -sW'(s)\ge0.
\end{equation*}
We test the equation with $u_\varepsilon\eta^2$ and apply Young's
inequality to obtain
\begin{equation*}
 \int\frac{-u_\varepsilon W'(u_\varepsilon)}\varepsilon\eta^2\,dX
 \le C\int\varepsilon|Du_\varepsilon|^2\eta^2\,dX
                               +C\varepsilon\int|D\eta|^2\,dX.
\end{equation*}
Combining these inequalities, we control the potential energy on the
region where $\eta=1$ by the gradient energy on its support plus
$C\varepsilon$. This proves the first assertion.

For the assertions about the planar limit, we use the compactness,
equipartition, and local Hausdorff convergence of the transition sets
proved by Hutchinson--Tonegawa \cite[Theorem~1]{HT}. Since the
limiting varifold is supported on the horizontal hyperplane, its tangent
hyperplanes agree with that hyperplane almost everywhere. Varifold
convergence therefore implies vanishing horizontal gradient energy on
compact subsets.
\end{proof}

For the remaining estimates we assume
\begin{equation}\label{bound}
 \mathcal E_\varepsilon(u_\varepsilon;\C_{64})\le\Lambda.
\end{equation}
This full-energy bound provides uniform control at smaller scales.
On fixed nested domains, Hutchinson--Tonegawa
\cite[Propositions~3.3--3.4]{HT} bound the discrepancy (the difference between the gradient
and potential energy densities) from above and prove an energy
monotonicity formula. Together, these results yield
\begin{equation}\label{lmearl}
 \frac\varepsilon2|Du_\varepsilon|^2-
       \frac{W(u_\varepsilon)}\varepsilon\le C,\qquad
 r^{-n}\mathcal E_\varepsilon(u_\varepsilon;B_r(X))
 \le s^{-n}\mathcal E_\varepsilon(u_\varepsilon;B_s(X))+Cs
 \quad\text{for }0<r<s.
\end{equation}
A bound on a fixed larger domain consequently implies
$\mathcal E_\varepsilon(u_\varepsilon;B_r(X))\le Cr^n$ uniformly
in interior centers. At radii $r<\varepsilon$, the same bound also
follows from the microscopic gradient estimate and boundedness of $W$.

We also need pointwise bounds at the scale of a transition.
Applying the standard interior elliptic estimates (see Gilbarg--Trudinger
\cite[Chapters~6 and~9]{GT}) to $u_\varepsilon(X+\varepsilon Y)$, we obtain
\begin{equation}\label{ellipt}
 \varepsilon|Du_\varepsilon|+\varepsilon^2|D^2u_\varepsilon|
 \le C\quad\text{on }\C_{56}.
\end{equation}
The strong maximum principle implies either $u_\varepsilon\equiv\pm1$
or $-1<u_\varepsilon<1$. For a constant solution, expressions
divided by $\sqrt{2W(u_\varepsilon)}$ are defined to be zero and all
the estimates below are immediate. Otherwise, after rescaling by
$\varepsilon$, the positive functions $1\pm u_\varepsilon$ solve
linear equations with uniformly H\"older-continuous coefficients,
by \eqref{ellipt}. Harnack's inequality and interior Schauder
estimates \cite[Theorems~6.2 and~8.20]{GT}, together with the
nondegeneracy of the wells, yield
\begin{equation}\label{relip}
 \varepsilon|Du_\varepsilon|+\varepsilon^2|D^2u_\varepsilon|
 \le C\min\{1-u_\varepsilon,1+u_\varepsilon\}
 \le C\sqrt{2W(u_\varepsilon)}
 \quad\text{on }\C_{48}.
\end{equation}

We now write stability in a form that controls the variation of the
level-set normals. We use the geometric stability inequality of
Sternberg--Zumbrun \cite{SZ98} and its horizontal version, both used by
Tonegawa \cite[Proposition~1, (4.1), and Proposition~4, (4.7)]{Ton05}.
By approximation, these inequalities hold for every compactly supported
Lipschitz function $\eta$:
\begin{align}\label{szum}
 \int\varepsilon\left(|D^2u_\varepsilon|^2-|D|Du_\varepsilon||^2\right)\eta^2\,dX
 &\le\int\varepsilon|Du_\varepsilon|^2|D\eta|^2\,dX,\notag\\
 \int\varepsilon\left(|D(D_xu_\varepsilon)|^2-|D|D_xu_\varepsilon||^2\right)\eta^2\,dX
 &\le\int\varepsilon|D_xu_\varepsilon|^2|D\eta|^2\,dX.
\end{align}
With fixed cutoffs supported in $\C_{56}$ and equal to one on
$\C_{48}$, the energy and excess bounds yield
\begin{equation}\label{gloca}
 \int_{\C_{48}}\varepsilon\left(|D^2u_\varepsilon|^2-|D|Du_\varepsilon||^2\right)\,dX\le C\Lambda,
 \qquad
 \int_{\C_{48}}\varepsilon\left(|D(D_xu_\varepsilon)|^2-|D|D_xu_\varepsilon||^2\right)\,dX\le CE_\varepsilon.
\end{equation}

The next lemma converts small geometric stability energy into
approximation by a one-dimensional transition on a fixed ball.
Equivalently, failure of this approximation forces a positive lower
bound for the integral in \eqref{psmal}. In particular, points with
a small gradient in the fixed transition range have a definite cost
in this integral. We will
use this cost to bound the number of disjoint neighborhoods in
Proposition~\ref{cover}.
\begin{lemma}\label{detect}
Fix $0\le b<1$, $L>0$, and $\delta>0$. There are $M>\max\{1,L\}$ and
$c>0$ such that the following holds. Let $v\in C^2(B_M^{n+1};[-1,1])$
be a stable solution of $\Delta v=W'(v)$, with $|v(0)|\le b$.
If
\begin{equation}\label{psmal}
 \int_{B_1}\bigl(|D(D_xv)|^2-|D|D_xv||^2\bigr)\,dX<c,
\end{equation}
then, for some $e\in\mathbb S^n$,
\begin{equation*}
 \|v-g(e\cdot X+g^{-1}(v(0)))\|_{C^2(B_L)}<\delta.
\end{equation*}
In particular, for suitable $M,c$,
\begin{equation}\label{sdete}
 |v(0)|\le b,\quad |Dv(0)|\le\tfrac12\sqrt{2W(v(0))}
 \quad\Longrightarrow\quad
 \int_{B_1}\bigl(|D(D_xv)|^2-|D|D_xv||^2\bigr)\,dX\ge c.
\end{equation}
Both conclusions hold with $|D^2v|^2-|D|Dv||^2$ in place of the
horizontal integrand.
\end{lemma} 
\begin{proof}
We argue by contradiction. If the assertion fails, there are
$M_j\to\infty$ and stable solutions $v_j$ on $B_{M_j}$ for
which the integral in \eqref{psmal} tends to zero while the
asserted approximation fails. Interior estimates yield, after
passage to a subsequence,
$v_j\to v_\infty$ in $C^2_{\mathrm{loc}}(\mathbb R^{n+1})$,
where $v_\infty$ is stable. The vanishing integral and the
horizontal identity in Tonegawa
\cite[Proposition~4 and Remark~2, pp.~450--451]{Ton05} imply
\[
D\left(\frac{D_xv_\infty}{|D_xv_\infty|}\right)=0
\quad\text{on }B_1\cap\{|D_xv_\infty|>0\}.
\]
On any nonempty connected component of this set, the
horizontal gradient therefore has a fixed direction.
The horizontal derivatives perpendicular to this direction
solve the Jacobi equation and vanish on an open set.
By unique continuation \cite{Aronszajn}, they vanish
throughout $\mathbb R^{n+1}$. If $D_xv_\infty=0$ on $B_1$,
the same argument applies to every horizontal derivative.
Thus $v_\infty$ depends on at most two variables.

We now determine this limit. Stability passes to the remaining
variables by testing with product cutoffs in the invariant directions.
The bounded gradient estimate and
Farina--Sciunzi--Valdinoci \cite[Theorem~1.1]{FSV08}
then imply one-dimensionality.
Under \eqref{wells}, the first integral and Sturm's theorem
classify bounded stable solutions of the one-dimensional
equation as the constants $\pm1$ or translated heteroclinics.
Since $|v_\infty(0)|\le b<1$, we obtain
\[
v_\infty(X)
=g\bigl(e\cdot X+g^{-1}(v_\infty(0))\bigr)
\]
for some $e\in\mathbb S^n$.
The local $C^2$ convergence and
$g^{-1}(v_j(0))\to g^{-1}(v_\infty(0))$
contradict the failure of the asserted approximation.

For \eqref{sdete}, we take $L=1$ and
$\delta<\frac12\min_{[-b,b]}\sqrt{2W}$,
and use $g'=\sqrt{2W(g)}$.
Finally, the almost-everywhere inequality
\[
|D^2v|^2-|D|Dv||^2
\ge
|D(D_xv)|^2-|D|D_xv||^2
\]
implies both conclusions for the full-gradient integrand.
\end{proof}

For the entire limits considered later, one-dimensionality still
leaves the direction of the transition to be determined. The next
lemma uses the tangential growth bound to exclude every nonvertical
direction. Thus the preceding detection argument also determines the
form of these limits.
\begin{lemma}\label{kzero}
Let $n\ge2$ and let $U\in C^2(\mathbb R^{n+1};[-1,1])$ be a
stable entire solution of
$\Delta U=W'(U)$, with $U(0)=0$. Suppose that, for all sufficiently
large $R$,
\begin{equation*}
 \int_{B_R}|D_xU|^2\,dX\le AR^{n-2}.
\end{equation*}
If $|D(D_xU)|^2-|D|D_xU||^2=0$ wherever $|D_xU|>0$,
then $U=g(z)$ or $g(-z)$.
\end{lemma}
\begin{proof}
On $\{D_xU=0\}$, both $D(D_xU)$ and $D|D_xU|$ vanish almost
everywhere. Thus the integral in \eqref{psmal} is zero. From
Lemma~\ref{detect} with $L\to\infty$ and $\delta\to0$, and
compactness of $\mathbb S^n$, we obtain $U(X)=g(e\cdot X)$ for some
$e\in\mathbb S^n$. Here $U(0)=0$ fixes the translation.
If $e$ has a nonzero horizontal component, we integrate over the slab
$\{|e\cdot X|\le1\}\cap B_R$ and obtain
\begin{equation*}
 \int_{B_R}|D_xU|^2\,dX\ge cR^n
\end{equation*}
for all large $R$, contrary to the assumed bound. Hence $e$ is one of
the two vertical unit vectors, which proves the conclusion.
\end{proof}

\subsection{Quantitative bounds for the difference of energy densities}\label{disc}
Our objective is to control the difference between the gradient and
potential energy densities at the accuracy required for normalized
compactness. The weighted estimates proved here are used to recover
the full tilt-excess from the layer heights, including the contribution
near the wells.

We set $S(s)=\sqrt{2W(s)}$ on $[-1,1]$. Hutchinson--Tonegawa
\cite[Proposition~4.3]{HT} proved that
$\varepsilon|Du_\varepsilon|^2/2-W(u_\varepsilon)/\varepsilon$
tends to zero in $L^1$ on compact interior domains.
Tonegawa \cite[Proposition~3, (4.6)]{Ton05} also obtained a weighted
$L^2$ bound for the gradient of this difference using stability.
For the normalized compactness argument, we need the weighted
$O(\varepsilon^2)$ estimates in Proposition~\ref{discr}, including
control of the derivative of
$S(u_\varepsilon)-\varepsilon|Du_\varepsilon|$.

The proof combines a local gradient bound with stability and the
equation for the energy difference. The main difficulty is that the
coefficient from the usual stability test may change sign.
An auxiliary one-dimensional equation will produce a positive
coefficient, and the detection lemma will control the remaining region
where the gradient is too small. We finish with an estimate for the
variation of the normalized vertical derivative.

We begin with a local version of Modica's gradient bound
\cite{Modica85}, with an exponentially small error.

\begin{lemma}\label{ndisc}
For sufficiently small $\varepsilon$,
\begin{equation}\label{ndest}
 \sup_{\C_{32}}
 \bigl(\varepsilon|Du_\varepsilon|-S(u_\varepsilon)\bigr)_+
 \le Ce^{-c/\varepsilon}.
\end{equation}
\end{lemma} \begin{proof}
By Hutchinson--Tonegawa \cite[Proposition~3.3]{HT}, on a fixed
interior domain containing $\overline{\C_{32}}$ we have
\[
 \varepsilon^2|Du_\varepsilon|^2
 -S(u_\varepsilon)^2\le C\varepsilon.
\]
In particular,
\[
 \bigl(\varepsilon|Du_\varepsilon|
       -S(u_\varepsilon)\bigr)_+
 \le C\sqrt{\varepsilon}.
\]
To improve this preliminary bound, we divide the function by a
positive weight chosen so that the resulting operator has a strictly
positive zeroth-order coefficient. From the maximum principle, we will then
deduce exponential decay away from the boundary.

We first derive the differential inequality for this construction.
By differentiating the equation for $u_\varepsilon$,
regularizing $|Du_\varepsilon|$, and passing to the limit,
we obtain
\begin{equation}\label{fdequa}
 \begin{split}
 &\bigl[-\varepsilon^2\Delta+S'(u_\varepsilon)^2
       -S''(u_\varepsilon)\varepsilon|Du_\varepsilon|\bigr]
       \bigl(S(u_\varepsilon)-\varepsilon|Du_\varepsilon|\bigr)\\
 &\quad=
 \varepsilon^2\Delta(\varepsilon|Du_\varepsilon|)
 -W''(u_\varepsilon)\varepsilon|Du_\varepsilon|
 \ge0
 \end{split}
\end{equation}
in distributions. Here we have used
$W''=S'^2+SS''$.
To obtain a strictly positive coefficient for the function
itself after division, we will construct a smooth positive
function $h$ on $[-1,1]$ such that
\begin{equation}\label{gbphas}
 -2Wh''-W'h'+(S'^2-SS'')h\ge ch
\end{equation}
for some $c>0$.

For this construction, we consider the one-dimensional operator
\[
 -\partial_{tt}+S'(g(t))^2-S(g(t))S''(g(t)).
\]
Since $g'=S(g)$, its quadratic form is
\[
 \int_{\mathbb R}|v'+S'(g)v|^2\,dt.
\]
The solutions of $v'+S'(g)v=0$ are multiples of $1/g'$.
Thus its only square-integrable solution is zero.
Moreover, the potential of the second-order operator tends
to the positive constants $W''(-1)$ and $W''(1)$.

These facts imply that the quadratic form bounds a positive
multiple of $\|v\|_2^2$. Indeed, otherwise we could choose
functions of $L^2$ norm one whose quadratic forms tend to zero.
This sequence is bounded in $H^1$. After we pass to a
subsequence, its local limit solves the first-order equation
and is therefore zero. The positive lower bound for the
potential outside a fixed interval then forces the
$L^2$ mass outside that interval to tend to zero as well.
This contradicts the normalization.

We next construct a bounded solution of
\[
 -v''+\bigl(S'(g)^2-S(g)S''(g)\bigr)v=1
\]
with limits $1/W''(\pm1)$ as $t\to\pm\infty$.
We start with a smooth function equal to these constants
near the two ends. Its error in the equation decays
exponentially and belongs to $L^2$.
Using the positive lower bound for the quadratic form,
we solve for an $L^2$ correction.
The resulting solution has the stated limits, and
$v',v''\to0$ at both ends.
Testing the equation with its negative part and using
the positivity of the quadratic form, we obtain $v\ge0$.
The equation then implies $v>0$.

We smoothly replace $v$ by its limiting constants for
sufficiently large $|t|$. The function remains positive,
and the second-order operator applied to it remains
at least $1/2$. We now set $h=v\circ g^{-1}$ on $(-1,1)$.
This function is constant near each endpoint, so it
extends smoothly and positively to $[-1,1]$.
Using $g'=S(g)$ and $g''=W'(g)$, we obtain
\eqref{gbphas}, with a suitable $c>0$.

We now use this weight in the operator in \eqref{fdequa}.
On the set where
$\varepsilon|Du_\varepsilon|>S(u_\varepsilon)$,
its value on $h(u_\varepsilon)$ differs from the left-hand
side of \eqref{gbphas}, evaluated at $u_\varepsilon$, by
\[
 -h''(u_\varepsilon)
 \bigl[\varepsilon^2|Du_\varepsilon|^2
       -S(u_\varepsilon)^2\bigr]
 -
 S''(u_\varepsilon)h(u_\varepsilon)
 \bigl[\varepsilon|Du_\varepsilon|-S(u_\varepsilon)\bigr].
\]
The initial gradient estimate bounds the first bracket
by $C\varepsilon$ and the second by $C\sqrt{\varepsilon}$.
Since $h$ has a positive minimum on $[-1,1]$, for small
$\varepsilon$ we therefore have
\[
 \bigl[-\varepsilon^2\Delta+S'(u_\varepsilon)^2
       -S''(u_\varepsilon)\varepsilon|Du_\varepsilon|\bigr]
 h(u_\varepsilon)
 \ge \frac c2h(u_\varepsilon)
\]
on this set.

By \eqref{fdequa} and Kato's inequality, the quotient
\[
 \frac{\bigl(\varepsilon|Du_\varepsilon|
             -S(u_\varepsilon)\bigr)_+}
      {h(u_\varepsilon)}
\]
is a nonnegative subsolution of
\[
 -\varepsilon^2\Delta
 -2\varepsilon^2D\log h(u_\varepsilon)\cdot D
 +\frac c2.
\]
The preceding lower bound is needed only where the
quotient is positive. The quotient vanishes elsewhere,
so this differential inequality holds throughout the
interior domain in the weak sense.

The quotient is uniformly bounded. The initial gradient estimate
also implies
\begin{equation*}
 \varepsilon|D\log h(u_\varepsilon)|\le C.
\end{equation*}
On a fixed cube, with its center translated to the origin,
we use sums of $\cosh(\lambda X_j/\varepsilon)$,
where $\lambda>0$ is fixed and sufficiently small.
The bound for the first-order coefficient ensures that
these sums are supersolutions.
We normalize them to dominate the quotient on the
boundary. On a smaller concentric cube, their values
are bounded by $Ce^{-c/\varepsilon}$.
The maximum principle therefore yields the same bound
for the quotient there.

A finite covering of $\overline{\C_{32}}$ by such smaller
cubes, together with the upper bound for $h$, proves
\eqref{ndest}.
\end{proof}

The stability test $S(u_\varepsilon)\eta$ in
\cite[Lemma~2.3]{Vil22} controls the weighted difference of the two
energy densities when $-S''$ has a positive lower bound.
Our assumptions, however, allow $S''$ to change sign.
To handle this difficulty, we construct the following auxiliary
functions. They produce a positive coefficient when we combine the
stability inequality with the equation for the energy difference.
\begin{lemma}\label{phase}
There are smooth functions $h,c$ on $[-1,1]$, with $\min c>0$, such that
\begin{equation}\label{ptode}
 (Sh'+2S'h)'=S''+c,\qquad h(-1)=h(1)=0.
\end{equation}
Moreover, $|h|\le CS$, and the derivatives of $h$ are bounded.
\end{lemma}
\begin{proof}
We choose a strictly positive smooth function $c$ satisfying
\begin{align*}
 \int_{-1}^1c(s)\,ds
 &=\sqrt{W''(-1)}+\sqrt{W''(1)},\\
 \int_{-1}^1\left(\int_{-1}^s S(t)\,dt\right)c(s)\,ds
 &=\sqrt{W''(1)}\,\sigma_0.
\end{align*}
One may choose $c$ proportional to
$\exp[\tau\int_{-1}^sS(t)\,dt]$. Its normalized first moment is strictly
increasing in $\tau$, with derivative equal to the variance of the
primitive, and ranges from $0$ to $\sigma_0$. After choosing $\tau$,
we fix the proportionality constant by the first identity and set
\begin{equation*}
 h(s)=\frac1{S(s)^2}\int_{-1}^sS(t)
\left[S'(t)-\sqrt{W''(-1)}+\int_{-1}^t c(r)\,dr\right]\,dt.
\end{equation*}
The bracket vanishes at both endpoints, and the two moment identities
make the numerator vanish at $s=1$. Its zero has order at least three
at each end of $[-1,1]$, whereas $S^2$ has a zero of order two. Thus $h$ is
smooth, vanishes at the endpoints, and is $O(S)$.
By differentiation, we obtain \eqref{ptode}.
\end{proof}

We can now estimate the energy difference and the derivative of
$S(u_\varepsilon)-\varepsilon|Du_\varepsilon|$ together.
The resulting $O(\varepsilon^2)$ bound is negligible relative to
$E_\varepsilon$ in the compactness argument, where we assume
$\varepsilon^2=o(E_\varepsilon)$. This is the quantitative form of
equipartition needed after normalization by the tilt-excess.
\begin{proposition}\label{discr}
On $\C_{16}$,
\begin{equation}\label{adest}
 \begin{split}
 \int_{\C_{16}}\biggl\{&
 \frac{S(u_\varepsilon)}\varepsilon
       \left|W(u_\varepsilon)-\frac{\varepsilon^2}2|Du_\varepsilon|^2\right|
 +\frac{(S(u_\varepsilon)-\varepsilon|Du_\varepsilon|)^2}\varepsilon\\
 &+\varepsilon
 \left|D\bigl(S(u_\varepsilon)-\varepsilon|Du_\varepsilon|\bigr)\right|^2
 \biggr\}\,dX\le C\varepsilon^2.
 \end{split}
\end{equation}
In particular, for every fixed $0\le b<1$,
\begin{equation}\label{cdisc}
 \int_{\C_{16}\cap\{|u_\varepsilon|\le b\}}
 \frac1\varepsilon
 \left|W(u_\varepsilon)-\frac{\varepsilon^2}2|Du_\varepsilon|^2\right|\,dX
 \le C_b\varepsilon^2.
\end{equation}
\end{proposition}
\begin{proof}
We first use stability to bound the derivative of the gradient defect
by the weighted energy difference. We then estimate this difference
itself, treating first the case with a favorable sign and then the
general potential.

To simplify the notation, we set
\begin{equation*}
 p=\varepsilon|Du_\varepsilon|,\qquad d=S(u_\varepsilon)-p,
 \qquad \xi=W(u_\varepsilon)-\frac{p^2}2.
\end{equation*}
We take $\eta\in C_c^\infty(\C_{32})$, with $0\le\eta\le1$. By
Lemma~\ref{ndisc}, terms bounded by $Ce^{-c/\varepsilon}$ may be added
to the right sides below. The bounds \eqref{relip} yield $p\le CS(u_\varepsilon)$.
We first apply stability to $S(u_\varepsilon)\eta$, as in
\cite[Lemma~2.3]{Vil22}. After rescaling, the identity used there yields
\begin{equation*}
 -2\int S(u_\varepsilon)S''(u_\varepsilon)\xi\eta^2\,dX
       \le\varepsilon^2\int S(u_\varepsilon)^2|D\eta|^2\,dX.
\end{equation*}
The measure
$\mu=\varepsilon^2\Delta p-W''(u_\varepsilon)p$ is nonnegative by
\eqref{fdequa}. Since
\begin{equation*}
 \varepsilon^2\Delta S(u_\varepsilon)-W''(u_\varepsilon)S(u_\varepsilon)
 =-2S''(u_\varepsilon)\xi,
\end{equation*}
we obtain by testing stability with $[S(u_\varepsilon)+p]\eta$
\begin{equation}\label{mwcont}
 \langle\mu,[S(u_\varepsilon)+p]\eta^2\rangle
 \le C\int S(u_\varepsilon)\xi_+\eta^2\,dX
       +C\varepsilon^2\int S(u_\varepsilon)^2|D\eta|^2\,dX
       +Ce^{-c/\varepsilon}.
\end{equation}
The coefficient $S'(u_\varepsilon)^2-S''(u_\varepsilon)p$ in
\eqref{fdequa} is bounded and has a fixed positive lower bound in
two well neighborhoods. On the complementary interval,
$d^2\le CS(u_\varepsilon)\xi_++Ce^{-c/\varepsilon}$.
We test \eqref{fdequa} with $d\eta^2$ and use \eqref{mwcont} to obtain
\begin{equation}\label{gsprel}
 \int\bigl(\varepsilon^2|D(d\eta)|^2+d^2\eta^2\bigr)\,dX
 \le C_0\int S(u_\varepsilon)\xi_+\eta^2\,dX
       +C_0\varepsilon^2\int S(u_\varepsilon)^2|D\eta|^2\,dX
       +Ce^{-c/\varepsilon}.
\end{equation}
We fix $C_0$ at this stage. This estimate will allow us to absorb
the gradient defect once the weighted energy difference is controlled.

First consider the case in which $-S''$ has a positive lower
bound on $[-1,1]$. Lemma~\ref{ndisc} bounds the contribution of
$\xi<0$ exponentially, and \eqref{bound} yields
$\int_{\C_{32}}S(u_\varepsilon)^2\,dX\le C\varepsilon$.
The stability inequality obtained using $S(u_\varepsilon)\eta$
therefore implies
$\int S(u_\varepsilon)\xi_+\eta^2\,dX\le C\varepsilon^3$.
We first choose $\eta=1$ on $\C_{24}$ and then apply
\eqref{gsprel} with a cutoff supported in $\C_{24}$ and equal to one
on $\C_{16}$. Division by $\varepsilon$ proves \eqref{adest}
in this case. This includes $W(s)=(1-s^2)^2/4$, for which
$S''=-\sqrt2$. We now prove the same bound for the general
potentials in \eqref{wells}.

Define
\begin{equation*}
 T=D^2u_\varepsilon
       -\frac{S'(u_\varepsilon)}{S(u_\varepsilon)}
                         Du_\varepsilon\otimes Du_\varepsilon.
\end{equation*}
By direct differentiation of the equation, we obtain
\begin{equation}\label{gtran}
 \varepsilon^2\Delta\xi
 -2\varepsilon^2\frac{S'(u_\varepsilon)}{S(u_\varepsilon)}
                   Du_\varepsilon\cdot D\xi
 =4\frac{S'(u_\varepsilon)^2}{S(u_\varepsilon)^2}\xi^2
                                  -\varepsilon^4|T|^2.
\end{equation}
The corresponding identity for the quartic potential, written in
normalized form, appears in \cite[Lemma~2.1, (2.1)]{Liu26}.
At a regular point, $T$ differs from $D^2u_\varepsilon$ only in the
normal--normal entry, and
\begin{equation*}
 \operatorname{tr}T
 =2\varepsilon^{-2}\frac{S'(u_\varepsilon)}{S(u_\varepsilon)}\xi.
\end{equation*}
Since $\xi=(S(u_\varepsilon)+p)d/2$, we deduce
\begin{equation}\label{ghboun}
 \varepsilon^4|T|^2
 \le C\varepsilon^4\bigl(|D^2u_\varepsilon|^2-|D|Du_\varepsilon||^2\bigr)
          +Cd^2.
\end{equation}
On the zero set of $Du_\varepsilon$, its derivative vanishes almost
everywhere, so the estimate holds almost everywhere there as well.

We take $h,c$ from Lemma~\ref{phase} and multiply \eqref{gtran}
by $h(u_\varepsilon)\eta^2$. Two integrations by parts yield
\begin{align*}
 &\int S(u_\varepsilon)(S''+c)(u_\varepsilon)\xi\eta^2\,dX\\
 =&\int\left(2h''+4\frac{S'}S h'+4\frac{S''}S h\right)(u_\varepsilon)
                 \xi^2\eta^2\,dX-\varepsilon^4\int h(u_\varepsilon)|T|^2\eta^2\,dX\\
 &-\varepsilon^2\int\xi\left[
       h(u_\varepsilon)\Delta(\eta^2)
       +2\left(h'+\frac{S'}S h\right)(u_\varepsilon)
                          Du_\varepsilon\cdot D(\eta^2)\right]\,dX.
\end{align*}
The coefficient of $\xi^2$ is bounded by $C/S(u_\varepsilon)$, and
$\xi^2/S(u_\varepsilon)\le CS(u_\varepsilon)d^2$. The tensor term is
controlled by \eqref{ghboun} and $|h|\le CS$.
Young's inequality bounds the cutoff terms by an arbitrarily small
multiple of $\int S(u_\varepsilon)\xi_+\eta^2\,dX$ and
\begin{equation*}
 C\varepsilon^2\int S(u_\varepsilon)^2
                       (|D\eta|^2+|\eta\Delta\eta|)\,dX
       +Ce^{-c/\varepsilon}.
\end{equation*}
We combine this identity with the stability inequality obtained
above using $S(u_\varepsilon)\eta$. Since $\min c>0$, we obtain
\begin{equation}\label{gmast}
 \begin{split}
 \int S(u_\varepsilon)\xi_+\eta^2\,dX
 &\le C_1\varepsilon^4\int
       \bigl(|D^2u_\varepsilon|^2-|D|Du_\varepsilon||^2\bigr)\eta^2\,dX
       +C_1\int S(u_\varepsilon)d^2\eta^2\,dX\\
 &\quad+C_1\varepsilon^2\int S(u_\varepsilon)^2
                    (|D\eta|^2+|\eta\Delta\eta|)\,dX+Ce^{-c/\varepsilon}.
 \end{split}
\end{equation}

Let us choose $\delta>0$ so that $C_1C_0\delta<1/4$. On two sufficiently
small well neighborhoods, $S\le\delta$, and \eqref{gsprel} bounds
their contribution to $\int S(u_\varepsilon)d^2\eta^2\,dX$ by
\begin{equation*}
 C_0\delta\int S(u_\varepsilon)\xi_+\eta^2\,dX
       +C\varepsilon^2\int S(u_\varepsilon)^2|D\eta|^2\,dX
       +Ce^{-c/\varepsilon}.
\end{equation*}
Outside those neighborhoods and where $0\le d<tS(u_\varepsilon)$,
we have $S(u_\varepsilon)d^2\le2tS(u_\varepsilon)\xi_+$.
We then choose $t>0$ so that $2C_1t<1/4$.

At a point in $\operatorname{supp}\eta$ outside those neighborhoods
where $d\ge tS(u_\varepsilon)$, we have $|u_\varepsilon|\le b<1$
and $p\le(1-t)S(u_\varepsilon)$. We choose the allowed error in
Lemma~\ref{detect} smaller than $t\min_{[-b,b]}S/2$.
From its full-gradient form, we then obtain on the ball of radius $\varepsilon$
about this point
\begin{equation*}
 \int_{B_\varepsilon}\varepsilon
       \bigl(|D^2u_\varepsilon|^2-|D|Du_\varepsilon||^2\bigr)\,dX
       \ge c\varepsilon^{n-2}.
\end{equation*}
Here $b,t$ are fixed, and the required larger microscopic ball lies
in the domain for small $\varepsilon$. A disjoint-ball covering and
\eqref{gloca} bound the number of such balls by
$C\varepsilon^{2-n}$. Their fivefold enlargements cover this
set, whose volume is therefore at most $C\varepsilon^3$.
Since $S$ and $d$ are bounded, its contribution is $C\varepsilon^3$.
The set $d<0$ contributes only an exponential error.

Combining the estimates on the well neighborhoods, the region with
small relative gradient defect, and the remaining covered region,
we have proved
\begin{equation*}
 \int S(u_\varepsilon)d^2\eta^2\,dX
 \le(C_0\delta+2t)\int S(u_\varepsilon)\xi_+\eta^2\,dX
       +C\varepsilon^3,
\end{equation*}
where \eqref{bound} bounds the cutoff terms.
The geometric term in \eqref{gmast} is $O(\varepsilon^3)$ by
\eqref{gloca}. The choices of $\delta,t$ thus absorb its second term
and yield $\int S(u_\varepsilon)\xi_+\eta^2\,dX\le C\varepsilon^3$.
From Lemma~\ref{ndisc}, we obtain the same bound with $|\xi|$.
We choose the cutoff equal to one on $\C_{24}$ and supported in $\C_{32}$.
We apply \eqref{gsprel} again, with a cutoff equal to one
on $\C_{16}$ and supported in $\C_{24}$, and obtain
\begin{equation*}
 \int_{\C_{16}}\bigl(S(u_\varepsilon)|\xi|+d^2+
                                      \varepsilon^2|Dd|^2\bigr)\,dX
       \le C\varepsilon^3.
\end{equation*}
Division by $\varepsilon$ proves \eqref{adest}. From the lower bound of
$S$ on $[-b,b]$, we deduce \eqref{cdisc}. Tests against $\mu$ are justified
by uniform and $H^1$ approximation on compact supports.
\end{proof}

The final estimate controls the variation of
$\varepsilon u_{\varepsilon,z}/S(u_\varepsilon)$.
We obtain it from the equation for this quotient and the quantitative
energy-difference bound just proved.
 
\begin{lemma}\label{rlemm}
On $\C_{12}$,
\begin{equation}\label{resti}
 \int_{\C_{12}}\frac{S(u_\varepsilon)^2}\varepsilon
 \left|D\left(\frac{\varepsilon u_{\varepsilon,z}}
                         {S(u_\varepsilon)}\right)\right|^2\,dX\le C.
\end{equation}
\end{lemma}
\begin{proof}
The equation for $u_\varepsilon$ and its derivative in $z$ yield
\begin{equation*}
 \varepsilon^2\operatorname{div}\left[
  S(u_\varepsilon)^2D\left(\frac{\varepsilon u_{\varepsilon,z}}
                              {S(u_\varepsilon)}\right)\right]
 =2S''(u_\varepsilon)\varepsilon u_{\varepsilon,z}
       \left(W(u_\varepsilon)-\frac{\varepsilon^2}2|Du_\varepsilon|^2\right).
\end{equation*}
We choose a smooth cutoff $\eta$ supported in $\C_{16}$ and equal
to one on $\C_{12}$. By \eqref{relip}, the quotient is bounded.
We test its equation with the quotient times $\eta^2$ and apply
Young's inequality. The left side of \eqref{resti} is then bounded by
\begin{equation*}
 C\varepsilon^{-2}\int_{\C_{16}}
 \frac{S(u_\varepsilon)}\varepsilon
       \left|W(u_\varepsilon)-\frac{\varepsilon^2}2|Du_\varepsilon|^2\right|\,dX
  +C\int_{\C_{16}}\frac{S(u_\varepsilon)^2}\varepsilon|D\eta|^2\,dX,
\end{equation*}
which is at most $C$ by \eqref{adest} and \eqref{bound}.
\end{proof}

\subsection{Tilt-excess near points with a small gradient or large tilt}\label{badtilt}
In this subsection, we isolate the points where the gradient is small or
strongly tilted before constructing the layer heights. The objective
is a covering whose energy is bounded by a power of $E_\varepsilon$
greater than one. Its contribution therefore vanishes after division
by the tilt-excess, which is the normalization used in the subsequent
compactness and harmonic approximation arguments.

The proof treats the two kinds of points separately. Small-gradient
points have a definite cost in the horizontal geometric stability
inequality. For points with large tilt, we first control the
concentration of the tangential energy in selected horizontal
directions. We then use the thin transition slab to estimate their
projection. Together, these bounds provide the relative error control
needed for harmonic approximation.

Fix $\Lambda<\infty$, $0<b<\bar b<1$, and $0<\tau<1$.
The points to be covered form the set
\begin{equation}\label{bset}
 \left\{X\in\C_{16}\mid \ |u_\varepsilon(X)|\le b,\quad
 \varepsilon|Du_\varepsilon(X)|\le\frac12\sqrt{2W(u_\varepsilon)}\
 \ \text{or}\
 |D_xu_\varepsilon(X)|\ge\tau|Du_\varepsilon(X)|\right\}.
\end{equation}
We use $\pi(x,z)=x$ for the horizontal projection.

\begin{proposition}\label{cover}
There are positive constants $\varepsilon_0,E_0,h_0,c,C$, depending
only on $\Lambda,b,\bar b,\tau,n,W$, with the following property.
Assume
\begin{equation}\label{nassu}
 \int_{\C_{64}}\varepsilon|Du_\varepsilon|^2\,dX\le\Lambda,
 \quad 0<\varepsilon\le\varepsilon_0,\quad E_\varepsilon\le E_0,
 \quad
 \{|u_\varepsilon|\le\bar b\}\cap\C_{32}\subset\{|z|\le h_0\}.
\end{equation}
Then the set in \eqref{bset} is covered by finitely many balls
$B_{r_\ell}(X_\ell)\subset\C_{20}$ satisfying
\begin{equation}\label{ncove}
 c\varepsilon\le r_\ell\le C\varepsilon,
 \qquad \sum_\ell r_\ell^n\le
 C E_\varepsilon^{\,1+1/[2(n-1)]}.
\end{equation}
For their union $\mathcal U_\varepsilon$,
\begin{equation}\label{nener}
 \int_{\mathcal U_\varepsilon}\left(\frac{\varepsilon}{2}|Du_\varepsilon|^2+
                   \frac{W(u_\varepsilon)}{\varepsilon}\right)\,dX
 +\frac{|\mathcal U_\varepsilon|}{\varepsilon}
 +|\pi(\mathcal U_\varepsilon)|
 \le C E_\varepsilon^{\,1+1/[2(n-1)]}.
\end{equation}
\end{proposition}

The first step is a directional concentration estimate for the
tangential energy. We integrate in two variables and use the
divergence of a nonnegative matrix to control the resulting density
on the remaining variables.

\begin{lemma}\label{plemm}
Let $f\in C^\infty(\mathbb S^{n-1})$ be nonnegative, even, and
supported in sufficiently small neighborhoods of $e_n$ and $-e_n$.
We take $\chi\in C_c^\infty(\C_{32})$, with $0\le\chi\le1$ and
$\chi=1$ on $\C_{28}$. We write $x=(w,t)$, $w\in\mathbb R^{n-1}$.
For every measurable $A\subset\mathbb R^{n-1}$,
\begin{equation}\label{pui}
 \int_{\{w\in A\}}\chi
 f\left(\frac{D_xu_\varepsilon}{|D_xu_\varepsilon|}\right)
 \varepsilon|D_xu_\varepsilon|^2\,dX
 \le CE_\varepsilon |A|^{1/(n-1)}.
\end{equation}
The integrand is zero where $D_xu_\varepsilon=0$. The constant may
also depend on $f$ and $\chi$.
\end{lemma}
\begin{proof}
We first integrate the selected part of the tilt-excess in
$t$ and $z$. This defines a nonnegative function of $w$,
\begin{equation*}
 g(w)=\int_{\mathbb R^2}\chi
 f\left(\frac{D_xu_\varepsilon}{|D_xu_\varepsilon|}\right)
 \varepsilon|D_xu_\varepsilon|^2\,dt\,dz.
\end{equation*}
The left-hand side of \eqref{pui} is $\int_A g(w)\,dw$.
To estimate this integral using stability, we consider the symmetric
$(n-1)\times(n-1)$ matrix
\begin{equation*}
 S(w)=\int_{\mathbb R^2}\chi
 f\left(\frac{D_xu_\varepsilon}{|D_xu_\varepsilon|}\right)
 \varepsilon\bigl(|D_xu_\varepsilon|^2I_{n-1}
              -D_wu_\varepsilon\otimes D_wu_\varepsilon\bigr)\,dt\,dz.
\end{equation*}
Both integrands extend as $C^1$ functions through
$D_xu_\varepsilon=0$, with derivative zero there. Indeed, their
dependence on $D_xu_\varepsilon$ is homogeneous of degree two and
smooth away from zero. For each fixed $\varepsilon$, the cutoff
$\chi$ therefore ensures that $g$ and $S$ are $C^1$ and compactly
supported.

The restriction on the support of $f$ ensures that $S$ controls $g$
in every $w$ direction. For $\xi\in\mathbb R^{n-1}$, we have
\begin{equation*}
 |u_{\varepsilon,t}|^2|\xi|^2
 \le |D_xu_\varepsilon|^2|\xi|^2
       -(D_wu_\varepsilon\cdot\xi)^2
 \le |D_xu_\varepsilon|^2|\xi|^2.
\end{equation*}
Where $f$ is nonzero, we have
$|u_{\varepsilon,t}|\ge c|D_xu_\varepsilon|$.
After multiplication by the nonnegative weight and integration in
$t,z$, the preceding inequalities yield
$cg(w)I_{n-1}\le S(w)\le g(w)I_{n-1}$.

We next estimate the divergence of $S$. We start with the full
horizontal matrix, whose divergence has the identity and bound
\begin{align}\label{tdiv}
 &\Div_x\bigl(|D_xu_\varepsilon|^2I_n
             -D_xu_\varepsilon\otimes D_xu_\varepsilon\bigr)
 =D_x^2u_\varepsilon D_xu_\varepsilon
                  -(\Delta_xu_\varepsilon)D_xu_\varepsilon,\notag\\
 &\left|\Div_x\bigl(|D_xu_\varepsilon|^2I_n
             -D_xu_\varepsilon\otimes D_xu_\varepsilon\bigr)\right|
 \le\sqrt{n-1}\,|D_xu_\varepsilon|\sqrt{|D(D_xu_\varepsilon)|^2-|D|D_xu_\varepsilon||^2}.
\end{align}
For the second line, we take coordinates at the point in question
with $D_xu_\varepsilon$ in the last horizontal direction.
Cauchy--Schwarz bounds the trace over the remaining directions.
Differentiation of $f(D_xu_\varepsilon/|D_xu_\varepsilon|)$ contributes
at most $C/|D_xu_\varepsilon|$ times the square root in \eqref{tdiv}.
The matrix multiplying this derivative has norm at most
$|D_xu_\varepsilon|^2$, so this term satisfies the same bound.

Let us multiply the full horizontal matrix in \eqref{tdiv} by
$\varepsilon\chi f(D_xu_\varepsilon/|D_xu_\varepsilon|)$.
For each row corresponding to a $w$ coordinate, its horizontal
divergence also contains the $t$ derivative of the entry in the last
column. Its integral in $t$ is zero because $\chi$ has compact support.
Thus the first $n-1$ components of the full divergence, integrated
in $t,z$, equal $\Div_w S$. The terms with a derivative of $\chi$
have integral at most $CE_\varepsilon$. Cauchy--Schwarz and the
second bound in \eqref{gloca} therefore yield
\begin{align*}
 \int_{\mathbb R^{n-1}}|\Div_w S|\,dw
 &\le CE_\varepsilon+CE_\varepsilon^{1/2}
  \left(\int_{\C_{32}}\varepsilon\left(|D(D_xu_\varepsilon)|^2-|D|D_xu_\varepsilon||^2\right)\,dX\right)^{1/2}\\
 &\le CE_\varepsilon.
\end{align*}

We first consider $n\ge3$. The matrix $S$ is symmetric,
nonnegative, and compactly supported. We apply
Serre \cite[Theorem~2.3, (11)]{Serre} on a ball containing its support
in its interior, so the boundary term is zero.
Since $\det S\ge c g^{n-1}$, the determinant inequality implies
\begin{align*}
 \|g\|_{L^{(n-1)/(n-2)}(\mathbb R^{n-1})}
 &\le C\left(\int_{\mathbb R^{n-1}}(\det S)^{1/(n-2)}\,dw\right)^{(n-2)/(n-1)}\\
 &\le C\int_{\mathbb R^{n-1}}|\Div_w S|\,dw
 \le CE_\varepsilon.
\end{align*}
H\"older's inequality now bounds $\int_A g(w)\,dw$ by
$CE_\varepsilon|A|^{1/(n-1)}$, as required.
For $n=2$, the matrix $S$ is scalar. Its compact support and the
bound for its derivative imply
\begin{equation*}
 \|g\|_{L^\infty(\mathbb R)}
 \le C\|S\|_{L^\infty(\mathbb R)}
 \le C\int_\mathbb R |S'(w)|\,dw
 \le CE_\varepsilon.
\end{equation*}
We obtain \eqref{pui} by integration over $A$ in this case as well.
\end{proof}

We next estimate the tilt-excess at transition points where
$\varepsilon|D_xu_\varepsilon|$ is bounded below.
The thin-slab assumption forces the associated level curves to have
short length, so a change in the tangential derivative has a definite
geometric cost. Combined with the preceding directional estimate,
this shows that these points carry a vanishing fraction of
$E_\varepsilon$ as $E_\varepsilon\to0$.
We will use the estimate to bound the number of covering balls in
Proposition~\ref{cover}.

\begin{proposition}\label{lhori}
Fix $0<b_1<\bar b$ and $\kappa>0$. Under \eqref{nassu}, for
sufficiently small $h_0$, we have
\begin{equation}\label{lhest}
 \int_{\C_{20}\cap\{|u_\varepsilon|\le b_1,
                       \,\varepsilon|D_xu_\varepsilon|\ge\kappa\}}
              \varepsilon|D_xu_\varepsilon|^2\,dX
 \le C E_\varepsilon^{\,1+1/[2(n-1)]}.
\end{equation}
\end{proposition}
\begin{proof}
The estimates along level curves follow the argument of
Schoen--Simon \cite[Section~6, pp.~785--787]{SS81}.
Related Allen--Cahn slice and coarea arguments appear in
Tonegawa--Wickramasekera \cite[Section~4]{TW12} and
Florit-Simon--Serra \cite[Appendix~B, Lemmas~B.4--B.5]{FSS25}.

We take a finite smooth even partition of unity on $\mathbb S^{n-1}$
subordinate to rotated copies of the caps in Lemma~\ref{plemm}.
We fix one term and rotate the horizontal coordinates so that
$x=(w,t)$. At the selected points in its support,
$|\varepsilon u_{\varepsilon,t}|\ge c\kappa$. Let $A$ be their
$w$-projection. By \eqref{ellipt}, varying $t$ through a sufficiently
small fixed multiple of $\varepsilon$ preserves
$|\varepsilon u_{\varepsilon,t}|\ge k$, for some fixed $k>0$, and
sweeps an interval of level values of fixed positive length.
We choose the variation so that the new points lie in $\C_{24}$ and
their values in $[-b_2,b_2]$, where $b_1<b_2<\bar b$.

The measure of the resulting pairs $(w,s)$ is therefore at least
$c|A|$. For almost every such pair, we follow the level curve
$u_\varepsilon(w,t,z)=s$ in the direction of increasing $z$.
We stop when $|\varepsilon u_{\varepsilon,t}|=k/2$ or when the curve
leaves $\C_{28}$. Before either event,
\begin{equation*}
 \frac{dz}{d\ell}
 =\frac{|u_{\varepsilon,t}|}
        {(u_{\varepsilon,t}^2+u_{\varepsilon,z}^2)^{1/2}}
 \ge c_k,
 \qquad \frac{|D_xu_\varepsilon|}{|Du_\varepsilon|}\ge c_k.
\end{equation*}
The implicit function theorem continues the curve while
$u_{\varepsilon,t}\ne0$. Since $|s|<\bar b$, the slab in
\eqref{nassu} bounds this length by $C_kh_0$. For small $h_0$
the curve stays in $\C_{28}$ and reaches
$|\varepsilon u_{\varepsilon,t}|=k/2$.

If $T$ is a unit tangent to this curve, then
\begin{equation*}
 |D^2u_\varepsilon T|^2\le\left(|D^2u_\varepsilon|^2-|D|Du_\varepsilon||^2\right).
\end{equation*}
Indeed the expression on the right is the sum of these squares over
an orthonormal basis perpendicular to $Du_\varepsilon$. The change
of $\varepsilon u_{\varepsilon,t}$ is at least $k/2$. Therefore
\begin{equation*}
 \int_{\mathrm{arc}}\varepsilon\sqrt{|D^2u_\varepsilon|^2-|D|Du_\varepsilon||^2}
       \frac{|D_xu_\varepsilon|}{|Du_\varepsilon|}\,d\ell\ge c_k.
\end{equation*}
The Jacobian of $X\mapsto(w,u_\varepsilon(X))$ is
$(u_{\varepsilon,t}^2+u_{\varepsilon,z}^2)^{1/2}\le|Du_\varepsilon|$.
Each of these slices has a full nonnegative integral at least
$c_k$. The coarea formula and \eqref{gloca} consequently yield
\begin{equation*}
 |A|\le C_k\int_{\C_{28}}\varepsilon|D_xu_\varepsilon|
        \sqrt{|D^2u_\varepsilon|^2-|D|Du_\varepsilon||^2}\,dX
       \le C E_\varepsilon^{1/2}.
\end{equation*}
Lemma~\ref{plemm} bounds the contribution of this cap by
$CE_\varepsilon|A|^{1/(n-1)}$. We sum over the partition of unity and obtain
\eqref{lhest}.
\end{proof}

We now prove Proposition~\ref{cover} by combining the
small-gradient cost with the estimate for large horizontal gradients.

\begin{proof}[Proof of Proposition~\ref{cover}]
Lemma~\ref{gtfene}, applied on nested interior cylinders, converts
the gradient-energy bound in \eqref{nassu} into the full energy
bounds needed for the local estimates below.
At a point where $|u_\varepsilon|\le b$ and
$\varepsilon|Du_\varepsilon|\le \sqrt{2W(u_\varepsilon)}/2$, we apply
\eqref{sdete} to $u_\varepsilon(X+\varepsilon Y)$.
For sufficiently small $\varepsilon$ its required domain lies in
$\C_{64}$. We have
\begin{equation}\label{scost}
 \int_{B_\varepsilon(X)}\varepsilon\left(|D(D_xu_\varepsilon)|^2-|D|D_xu_\varepsilon||^2\right)\,dX
       \ge c\varepsilon^{n-2}.
\end{equation}
We choose a maximal disjoint family of balls of this radius centered
at these points.
Their twofold dilates cover all such points, and
\begin{equation*}
 \#\{\text{balls}\}\,\varepsilon^n\le C\varepsilon^2E_\varepsilon.
\end{equation*}
If the set is nonempty and $n>2$, we also obtain from one ball in \eqref{scost}
that $\varepsilon^2\le C E_\varepsilon^{2/(n-2)}$.
The total $n$th power of these radii is consequently at most
$CE_\varepsilon^{1+2/(n-2)}$.
For $n=2$, a sufficiently small $E_0$ excludes such a point.

At each remaining point of \eqref{bset},
\begin{equation*}
 \varepsilon|D_xu_\varepsilon|
 >\frac{\tau}{2}\min_{[-b,b]}\sqrt{2W}>0.
\end{equation*}
By \eqref{ellipt}, a fixed-radius multiple $B_{r_0\varepsilon}(X)$
has tilt-excess at least $c\varepsilon^n$. Its fivefold dilate
lies in $\C_{20}\cap\{|u_\varepsilon|\le b_1\}$, where
$b<b_1<\bar b$, and has $\varepsilon|D_xu_\varepsilon|\ge\kappa$
for some fixed $\kappa>0$. A maximal disjoint family of the smaller
balls has fivefold dilates covering
these points. Proposition~\ref{lhori} bounds their number
times $\varepsilon^n$ by
$CE_\varepsilon^{1+1/[2(n-1)]}$.
Taking $E_0\le1$, we use $2/(n-2)>1/[2(n-1)]$ when $n>2$.
For $n=2$ the first set is empty. The two coverings therefore prove
\eqref{ncove}. Finally \eqref{ellipt} implies
$\left(\frac{\varepsilon}{2}|Du_\varepsilon|^2+
                   \frac{W(u_\varepsilon)}{\varepsilon}\right)\le C/\varepsilon$ on each covering ball. Integration and
projection yield \eqref{nener}.
\end{proof}

\begin{remark}
The same maximal-ball argument and Lemma~\ref{detect}
yield the following estimate. For fixed $0\le b<1$, $L>0$, and
$\delta>0$, the points in
$\{|u_\varepsilon|\le b\}\cap\C_{40}$ at which
$u_\varepsilon(X+\varepsilon\,\cdot)$ has $C^2(B_L)$ distance at
least $\delta$ from every rotated heteroclinic have a covering by
$C\varepsilon$ balls with total $n$th power at most
$C\varepsilon^2E_\varepsilon$. The covering balls and any fixed
enlargement of them lie in
$\C_{44}$ for sufficiently small $\varepsilon$. Their total volume
is at most $C\varepsilon^3E_\varepsilon$, by the lower and upper
bounds on their radii. This assertion follows from
\eqref{gloca} for solutions with arbitrary transition sets.
\end{remark}

\section{Harmonic approximation and excess decay}
\label{vertcomp}\label{tiltpart}

We use the estimates of Section~\ref{locenergy} to recover all the
tangential energy from the layer heights. We first count transitions
on vertical lines and prove strong compactness of the corresponding
energy integrals. We then construct harmonic limits of the normalized
heights and show that their Dirichlet energies equal the limiting
tilt-excess. These steps require the squared transition thickness to
be negligible compared with the excess.

The harmonic estimates lead to decay on smaller balls. We keep track
of the thickness error and determine when one rotation improves every
layer. The resulting iteration either reduces the multiplicity,
continues with improved excess, or reaches the transition thickness.
In the last case, it produces an entire solution to which
Proposition~\ref{crit} will apply.

\subsection{Strong compactness of the tilt-excess integrals}\label{vertlines}
Our objective is to prove strong compactness of the normalized
tilt-excess integrals on the horizontal base. For this purpose, the
covering from Subsection~\ref{badtilt} must control complete vertical
lines, since a line meeting an exceptional ball may carry energy
outside that ball. We obtain this control before discarding the
corresponding base points in the construction of the layer heights.

We first control the energy near the wells and count the transitions
on almost every vertical line. We then combine the intersection counts
with the stress-tensor identity to obtain a BV bound for a scalar
combination of vertical energy integrals. Proposition~\ref{bvcmp}
uses this bound to prove uniform integrability and strong local
$L^1$ precompactness of the normalized vertical integrals of the
tilt-excess. It also shows that the lines with fewer than the limiting
number of transitions have negligible base measure at the excess scale.

Throughout this subsection and Subsection~\ref{harm}, we assume
\eqref{acn}, \eqref{stab}, the energy bound \eqref{bound},
\begin{equation}\label{plimi}
 \varepsilon|Du_\varepsilon|^2\,dX
 \rightharpoonup Q\sigma_0\mathcal H^n\lfloor\{z=0\}
 \quad\hbox{locally in }\C_{32},\qquad Q\in\mathbb N,
\end{equation}
and
\begin{equation}\label{thickn}
 E_\varepsilon\longrightarrow0,\qquad
              \varepsilon^2/E_\varepsilon\longrightarrow0.
\end{equation}
By Lemma~\ref{gtfene}, for every fixed $b\in(0,1)$,
$\{|u_\varepsilon|\le b\}\cap\C_{24}$ lies in a slab whose width
tends to zero, and the full energy measures have the limit
\eqref{plimi}. Counting intersections and controlling an integer
multiplicity by BV and Poincar\'e estimates also appear in
Schoen--Simon \cite[Section~3, (3.11)--(3.13)]{SS81}.

We choose $\zeta\in C_c^\infty((-4,4))$, with $0\le\zeta\le1$ and
$\zeta=1$ on $[-3,3]$. On the base balls used below, all points with values in a fixed
compact subinterval of $(-1,1)$ eventually lie in $|z|<1$. In $2<|z|<5$ the solution is in a
convex well neighborhood. The maximum principle for
$1\mp u_\varepsilon$, followed by interior estimates, yields
\begin{equation}\label{cerro}
 1-|u_\varepsilon|+\varepsilon|Du_\varepsilon|
                  \le C e^{-c/\varepsilon}
 \quad\text{on the supports of }\zeta'\text{ and }\zeta''.
\end{equation}
All estimates denoted by $o(E_\varepsilon)$ are on fixed interior
base domains, with the normal interval kept equal to $(-4,4)$.

For the applications of Proposition~\ref{cover} here and in
Subsection~\ref{harm}, we make one fixed localization.
We set $\lambda=3/4$ and apply that proposition to
$\widetilde u_\varepsilon(Y)=u_\varepsilon(\lambda Y)$, with parameter
$\widetilde\varepsilon=\varepsilon/\lambda$ and upper normalization
$\widetilde E_\varepsilon=\lambda^{-n}E_\varepsilon$. Its domain
$\C_{64}$ corresponds to the original $\C_{48}$, and its slab
hypothesis on $\C_{32}$ corresponds to the original $\C_{24}$.
The preceding Hausdorff convergence verifies the slab hypothesis.
The resulting estimate covers the exceptional points in the original
$\C_{12}$ by balls in $\C_{15}$, which is sufficient for all uses
below. The original covering radii remain comparable to $\varepsilon$,
and their total $n$th power is bounded by
$C\lambda^{-n/[2(n-1)]}E_\varepsilon^{1+1/[2(n-1)]}$.
The fixed factor is included in $C$.

Proposition~\ref{cover} concerns the transition region.
To control complete vertical lines, we also need an estimate for the
tilt-excess where $u_\varepsilon$ is close to either well.
The next lemma makes this contribution small relative to
$E_\varepsilon$ and provides the error bounds used in counting
intersections with vertical lines.

\begin{lemma}
For every sufficiently small $\delta>0$,
\begin{equation}\label{rtail}
 \int_{\C_{12}\cap\{|u_\varepsilon|>1-\delta\}}
                  \varepsilon|D_xu_\varepsilon|^2\,dX
 \le C(\delta^c+\varepsilon^2)E_\varepsilon.
\end{equation}
Furthermore, for every fixed $\tau>0$,
\begin{equation}\label{fangl}
 \int_{\C_{10}\cap\{|D_xu_\varepsilon|\ge\tau|Du_\varepsilon|\}}
                  \varepsilon|D_xu_\varepsilon|^2\,dX=o(E_\varepsilon),
\end{equation}
and
\begin{equation}\label{ndsmal}
 \int_{\C_{10}}\frac{\left(\sqrt{2W(u_\varepsilon)}-\varepsilon|u_{\varepsilon,z}|\right)^2}{\varepsilon}\,dX
                          =o(E_\varepsilon).
\end{equation}
\end{lemma}
\begin{proof}
The estimate near the wells uses the decay of tangential derivatives
in a convex well, as in Wang \cite[Appendix~B, Lemma~B.4]{Wang17}.

We choose $\delta_0>0$ such that $W''\ge c>0$ in the two well
intervals of length $2\delta_0$. We also choose $\eta=1$ on $\C_{12}$,
supported in $\C_{18}$, with bounded derivative. In the positive well region, let $d(X)$ be the distance to
\begin{equation*}
 \{Y\in\C_{20}\mid u_\varepsilon(Y)\le1-\delta_0\}
                    \cup(\mathbb R^{n+1}\setminus\C_{20}).
\end{equation*}
For $T\ge0$, we choose $\chi_T$ zero when $d\le T$, equal to one when
$d\ge T+L\varepsilon$, and linear between these values. The number
$L$ is a sufficiently large fixed constant. We test the Jacobi
equation with $\eta^2\chi_T^2u_{\varepsilon,x_i}$ and obtain
\begin{equation*}
 \int\left\{\varepsilon^2|D(\eta\chi_Tu_{\varepsilon,x_i})|^2+
 W''(u_\varepsilon)(\eta\chi_Tu_{\varepsilon,x_i})^2\right\}\,dX
 =\varepsilon^2\int u_{\varepsilon,x_i}^2|D(\eta\chi_T)|^2\,dX.
\end{equation*}
Its support is in the positive well. After summing over $i$ and
choosing $L$ large, we obtain
\begin{equation*}
 \int_{\{d>T+L\varepsilon\}}\eta^2\varepsilon|D_xu_\varepsilon|^2\,dX
 \le\frac12\int_{\{d>T\}}\eta^2\varepsilon|D_xu_\varepsilon|^2\,dX
                                      +C\varepsilon^2E_\varepsilon.
\end{equation*}
By iteration, we obtain, for every $T\ge0$,
\begin{equation*}
 \int_{\{d>T\}}\eta^2\varepsilon|D_xu_\varepsilon|^2\,dX
 \le C\bigl(e^{-cT/\varepsilon}+\varepsilon^2\bigr)E_\varepsilon.
\end{equation*}
The cutoff errors are summed by a geometric series.

By \eqref{relip}, $|D\log(1-u_\varepsilon)|\le C/\varepsilon$
in $\C_{20}$. For $X\in\C_{12}$ with $u_\varepsilon(X)>1-\delta$,
\begin{equation*}
 d(X)\ge\min\{c\varepsilon\log(\delta_0/\delta),\,8\}.
\end{equation*}
Indeed, the line segment from $X$ to a point in the transition set stays in the convex
cylinder $\C_{20}$, and the distance to its boundary is at least $8$.
The preceding estimate is consequently bounded by
$C(\delta^c+e^{-c/\varepsilon}+\varepsilon^2)E_\varepsilon$ on this
set. We absorb $e^{-c/\varepsilon}$ into $\varepsilon^2$ and repeat the
argument near $-1$ to obtain \eqref{rtail}.

On each fixed set $\{|u_\varepsilon|\le b\}$ with $b<1$, Proposition~\ref{cover} controls the integral
in \eqref{fangl} by $CE_\varepsilon^{1+1/[2(n-1)]}$.
We first let $\varepsilon\to0$ and then let $b\uparrow1$ in \eqref{rtail}. This proves \eqref{fangl}.
The elementary inequality
\begin{equation*}
 \left(\sqrt{2W(u_\varepsilon)}-\varepsilon|u_{\varepsilon,z}|\right)^2\le2\left(\sqrt{2W(u_\varepsilon)}-\varepsilon|Du_\varepsilon|\right)^2+
 2\varepsilon^2|D_xu_\varepsilon|^2
                      \frac{|D_xu_\varepsilon|^2}{|Du_\varepsilon|^2}
\end{equation*}
holds with the last term set to zero at $Du_\varepsilon=0$.
We split the last integral at a fixed angle $\tau$, use
\eqref{fangl}, and then let $\tau\downarrow0$.
Proposition~\ref{discr} and
$\varepsilon^2=o(E_\varepsilon)$ prove
\eqref{ndsmal}.
\end{proof}

We now prepare the energy comparison used to bound the number of
transitions on a vertical line. For $\epsilon>0$ and a smooth function
$v$, we use the densities
\begin{equation*}
 e_\epsilon(v)=\frac\epsilon2|Dv|^2+\frac{W(v)}\epsilon,
 \qquad \xi_\epsilon(v)=\frac\epsilon2|Dv|^2-\frac{W(v)}\epsilon.
\end{equation*}

\begin{lemma}\label{hfset}
Fix $M\ge2$, $\Lambda_0>0$, and $t\in(0,1)$. There is $\eta_*>0$
with the following property. Let $\mathcal Y$ be a finite subset of one
vertical line in $\mathbb R^{n+1}$, with $\#\mathcal Y\le M$ and
$\operatorname{diam}\mathcal Y\le\eta_*$. Suppose $0<a<1$ and
$|Y-Y'|>3a$ for distinct $Y,Y'\in\mathcal Y$.
Let $\epsilon>0$ and let $v$ be a smooth solution of
$\epsilon^2\Delta v=W'(v)$ on
$\{\operatorname{dist}(\cdot,\mathcal Y)<1\}$, with $|v|\le1$.
Assume $\xi_\epsilon(v)\le\eta_*$ on that neighborhood and, for every
$Y\in\mathcal Y$ and $a\le r\le1$,
\begin{equation*}
 \int_{B_r(Y)}\bigl(|\xi_\epsilon(v)|+\epsilon|D_xv|^2\bigr)\,dX
 \le\eta_*r^n,
 \qquad
 \int_{B_r(Y)}\epsilon|Dv|^2\,dX\le\Lambda_0r^n.
\end{equation*}
Then
\begin{equation}\label{ajqwwb}
 \sum_{Y\in\mathcal Y}a^{-n}\int_{B_a(Y)}e_\epsilon(v)\,dX
 \le t+(1+t)\int_{\{\operatorname{dist}(\cdot,\mathcal Y)<1\}}e_\epsilon(v)\,dX.
\end{equation}
\end{lemma}
\begin{proof}
We apply Hutchinson--Tonegawa \cite[Proposition~5.5]{HT} with outer
radius one, zero forcing, and their cardinality parameter equal to
$M-1$, after horizontal translation. Their tilt term is
$\epsilon(|Dv|^2-|v_z|^2)=\epsilon|D_xv|^2$,
which vanishes at critical points. Their full energy is
exactly $e_\epsilon(v)$ above. The convention
$\sigma_0=2\sigma_{\mathrm{HT}}$ changes the density assigned to one
limiting transition. Both energy integrals in \eqref{ajqwwb} keep
their original normalization.
\end{proof}

The preceding comparison will bound the number of crossings once
we exclude points with excessive discrepancy or tilt at some scale.
Outside this covering, local approximation implies a nonzero vertical
derivative, and the planar energy limit prevents an additional
transition. The next lemma states both conclusions, with covering
energy negligible relative to $E_\varepsilon$.

\begin{lemma}\label{ccove}
Fix a compact interval $I\Subset(-1,1)$. The points with $u_\varepsilon\in I$ in
$B_6^n\times(-4,4)$ can be divided into covered and uncovered points,
where the covering consists of balls contained in $\C_8$ and
\begin{equation*}
 \sum_\ell r_\ell^n=o(E_\varepsilon),
\end{equation*}
with the following properties. At every uncovered point with $u_\varepsilon\in I$,
\begin{equation}\label{gcros}
 \varepsilon|u_{\varepsilon,z}|\ge c_I\sqrt{2W(u_\varepsilon)}.
\end{equation}
For each $x\in B_6^n$ and $s\in I$, the interval $-4<z<4$ contains
at most $Q$ uncovered roots of $u_\varepsilon(x,z)=s$.
The energy of the covering, its unsigned normal integral
$\int \sqrt{2W(u_\varepsilon)}|u_{\varepsilon,z}|\,dX$, and its projected base volume
are all $o(E_\varepsilon)$.
\end{lemma}
\begin{proof}
By \eqref{lmearl} and \eqref{bound}, for centers in $\C_{10}$ and
$0<r<1$,
\begin{equation}\label{luppe}
 \frac{\varepsilon}{2}|Du_\varepsilon|^2
                   -\frac{W(u_\varepsilon)}{\varepsilon}\le C,
 \qquad
 \int_{B_r(X)}\left(\frac{\varepsilon}{2}|Du_\varepsilon|^2+
                   \frac{W(u_\varepsilon)}{\varepsilon}\right)\,dX\le Cr^n
 \quad\text{for }0<r<1.
\end{equation}
The constants are uniform in these centers.

Let us choose $r_0<1/10$, $L>1$, and $\eta>0$ in the order described
at the end of this proof. We cover the indicated points in the transition set for which,
at some radius $L\varepsilon\le r\le r_0$,
\begin{equation}\label{witnes}
 \int_{B_r(X)}\left[
 \left|\frac{W(u_\varepsilon)}{\varepsilon}
                 -\frac{\varepsilon}{2}|Du_\varepsilon|^2\right|
                +\varepsilon|D_xu_\varepsilon|^2\right]\,dX>\eta r^n.
\end{equation}
By \eqref{relip} and Cauchy--Schwarz,
\begin{equation*}
 \int_{B_r(X)}\left|
 \frac{W(u_\varepsilon)}{\varepsilon}
                 -\frac{\varepsilon}{2}|Du_\varepsilon|^2\right|\,dX
 \le \frac{\eta r^n}{4}
             +C_\eta\int_{B_r(X)}\frac{\left(\sqrt{2W(u_\varepsilon)}-\varepsilon|Du_\varepsilon|\right)^2}{\varepsilon}\,dX.
\end{equation*}
Here \eqref{luppe} controls $\int 2W(u_\varepsilon)/\varepsilon\,dX$.
The part of the tilt-excess where
$|D_xu_\varepsilon|<\tau|Du_\varepsilon|$ is at most $C\tau^2r^n$.
We choose a fixed $\tau$ so that this is at most $\eta r^n/4$.
Every ball satisfying \eqref{witnes} consequently satisfies
\begin{equation*}
 r^n\le C_\eta\int_{B_r(X)}\frac{\left(\sqrt{2W(u_\varepsilon)}-\varepsilon|Du_\varepsilon|\right)^2}{\varepsilon}\,dX
 +C_\eta\int_{B_r(X)\cap\{|D_xu_\varepsilon|\ge\tau|Du_\varepsilon|\}}
                    \varepsilon|D_xu_\varepsilon|^2\,dX.
\end{equation*}
We choose disjoint balls satisfying \eqref{witnes} whose fivefold
dilates cover all points with such a radius. Their centers lie in
$B_6^n\times(-4,4)$ and $r_0<1/10$, so the enlarged balls stay in
$\C_8$. Equations
\eqref{adest} and \eqref{fangl} imply
\begin{equation*}
 \sum_\ell r_\ell^n\le C_\eta\varepsilon^2+o(E_\varepsilon)
                              =o(E_\varepsilon).
\end{equation*}
Their full energy and projected volume obey the same bound by
\eqref{luppe}. Also
$\sqrt{2W(u_\varepsilon)}|u_{\varepsilon,z}|\le\left(\frac{\varepsilon}{2}|Du_\varepsilon|^2+
                   \frac{W(u_\varepsilon)}{\varepsilon}\right)$, so their unsigned
normal integral does too.

It remains to control the uncovered crossings. We first choose the
parameters so that local approximation ensures a definite energy cost
at each crossing. We then apply Lemma~\ref{hfset} to show that an
additional crossing would exceed the planar limiting energy.
We choose $t\in(0,1)$ and $\delta>0$ so small that
\begin{equation}\label{cngap}
 (Q+1)\omega_n(\sigma_0-\delta)
 >t+(1+t)Q\sigma_0\omega_n.
\end{equation}
We use $M=Q+1$ and the bound for
$r^{-n}\int_{B_r}\varepsilon|Du_\varepsilon|^2\,dX$
in \eqref{luppe} to fix $\Lambda_0$ and then $\eta_*$.
We choose $r_0<1/10$ so that $Cr_0\le\eta_*$, where $C$ is the
constant in the first inequality in \eqref{luppe}.
Next we choose $L$ large enough that every vertical heteroclinic centered
at a level in $I$ has normalized energy at least
$\omega_n(\sigma_0-\delta/2)$ on $B_L$.
Finally we choose $\eta\le\eta_*$ small enough
for the following fixed-scale compactness implication.

At an uncovered point $X$, failure of \eqref{witnes} at
$r=4L\varepsilon$ implies approximation by a vertical heteroclinic in
$C^2(B_{7L/2})$, by the compactness argument used in
Hutchinson--Tonegawa \cite[Proposition~5.6]{HT}.
Indeed, along $\eta\to0$, every limit on $B_{4L}$ satisfies
$D_xv=0$ and $W(v)=|Dv|^2/2$. Its central value lies in $I$,
so it is an increasing or decreasing heteroclinic.
For this fixed $L$, the derivative of every such heteroclinic is bounded
away from zero on the vertical segment $[-3L,3L]$, uniformly in its
central level value in $I$. Consequently the approximation yields
\eqref{gcros}, energy ratio at least
$\omega_n(\sigma_0-\delta)$ on $B_{L\varepsilon}(X)$, and uniqueness
of the root of that level within distance $3L\varepsilon$ on its
vertical line. The radius $4L\varepsilon$ is admissible once
$4L\varepsilon<r_0$.

Suppose one vertical line with base coordinate $x_*$ contained $Q+1$
uncovered roots of the same level. We denote this finite set by
$\mathcal X$, and rescale by
\begin{equation*}
 v(Y)=u_\varepsilon((x_*,0)+r_0Y),\qquad
 \epsilon=\varepsilon/r_0,\qquad
 \mathcal Y=r_0^{-1}(\mathcal X-(x_*,0)),\qquad
 a=L\varepsilon/r_0.
\end{equation*}
The set $\mathcal Y$ lies on a vertical line, has $Q+1$ points
separated by more than $3a$, and $v$ satisfies the required equation
with $|v|\le1$.
Under this rescaling, we have
$\xi_\epsilon(v)(Y)=r_0\xi_\varepsilon(u_\varepsilon)((x_*,0)+r_0Y)$,
so its pointwise upper bound is at most $\eta_*$ on the full required
neighborhood. For every $Y\in\mathcal Y$ and $a\le r\le1$, the
corresponding ball before rescaling has radius between $L\varepsilon$ and
$r_0$. Its center is uncovered, so failure of \eqref{witnes}
implies precisely the first integral bound in Lemma~\ref{hfset},
after multiplication by $r_0^{-n}$. Equation~\eqref{luppe}
implies the second integral bound with the fixed $\Lambda_0$.

Since the slab width tends to zero,
$\operatorname{diam}\mathcal X=o(1)$ and
$\operatorname{diam}\mathcal Y\le\eta_*$ for small $\varepsilon$.
The parameters have been fixed in the order $t,\delta$, $\eta_*$,
$r_0$, $L$, and $\eta$. The unit neighborhood of $\mathcal Y$ corresponds
to a neighborhood contained in a ball of radius $r_0+o(1)$
centered at distance $o(1)$ from the limiting hyperplane. Weak full-energy
convergence therefore bounds its rescaled energy above by
$Q\sigma_0\omega_n+o(1)$. This assertion is uniform in the moving
base coordinates. To see this, we pass to convergent subsequences of
the centers and use nearby-ball inclusions, since the limiting planar
measure assigns zero mass to each boundary sphere.
The lower energy ratio at every selected root and
\eqref{ajqwwb} now yield
\begin{equation*}
 (Q+1)\omega_n(\sigma_0-\delta)
 \le t+(1+t)\bigl(Q\sigma_0\omega_n+o(1)\bigr),
\end{equation*}
contradicting \eqref{cngap}. 
\end{proof}

We now relate the number of transitions on a vertical line to the
integral of $\sqrt{2W(u_\varepsilon)}|u_{\varepsilon,z}|$ along that
line, with the cutoff $\zeta^2$. The following lemma shows that this
integral equals, up to a small $L^1$ error, $\sigma_0$ times an integer
between zero and $Q$. This relation will be used in
Proposition~\ref{bvcmp} to obtain strong compactness.

\begin{lemma}\label{count}
On $B_5^n$ there are measurable functions
$N_\varepsilon(x)\in\{0,1,\ldots,Q\}$ for which
\begin{align}\label{nsurp}
 &\int\left(\int\zeta^2\sqrt{2W(u_\varepsilon)}|u_{\varepsilon,z}|\,dz
                                    -Q\sigma_0\right)_+\,dx=o(E_\varepsilon),\\
 &\int\left|\int\zeta^2\sqrt{2W(u_\varepsilon)}|u_{\varepsilon,z}|\,dz
                              -\sigma_0N_\varepsilon(x)\right|\,dx
                                      =o(E_\varepsilon),\label{ninte}\\
 &|\{x\mid N_\varepsilon(x)<Q\}|\longrightarrow0.
                                      \label{cmacr}
\end{align}
Outside a set of base volume $o(E_\varepsilon)$, $N_\varepsilon$
is the actual number of roots of every level in one fixed compact subinterval of $(-1,1)$.
\end{lemma}
\begin{proof}
We use the one-dimensional inequalities of
Lemma~\ref{odtail} in the appendix. They apply to a
normal line of the present solution by \eqref{relip}.
Their errors are controlled by
$\int(\sqrt{2W(u_\varepsilon)}-\varepsilon|u_{\varepsilon,z}|)^2/\varepsilon\,dz$.

Choose two fixed narrow intervals of level values near the wells, both inside
a compact interval to which Lemma~\ref{ccove} applies.
Vertical coarea permits a boundary level value to be chosen in each band
so that the integrated number of covered crossings at that level
is $o(E_\varepsilon)$. The positive lower bound for $\sqrt{2W}$
on each band and the small unsigned normal integral of the covering
justify this assertion. The chosen level values may depend on $\varepsilon$.
The constants in Lemma~\ref{odtail} are uniform on the fixed bands.
Between these two boundary level values, uncovered roots contribute at
most $Q\int\sqrt{2W(s)}\,ds$. The covered contribution is
$o(E_\varepsilon)$. The upper bounds near the wells add the two missing
well integrals, at most $Q$ times each, together with the covered
crossing error. By \eqref{ndsmal} and \eqref{cerro},
all remaining terms are $o(E_\varepsilon)$. The three principal
integrals add to $Q\sigma_0$. This proves \eqref{nsurp}.

We project the covering for the fixed larger interval of level values onto the
base. Its image $B_\varepsilon$ has volume $o(E_\varepsilon)$.
On every line outside $B_\varepsilon$, all roots in the interval
satisfy \eqref{gcros}. Their number is at most $Q$ and
is constant in the level. A root cannot stop or fold while
$u_{\varepsilon,z}\ne0$, and cannot leave through the pure-phase
upper and lower boundaries of the normal interval. We denote this
number by $N_\varepsilon(x)$ and apply the upper and lower bounds near the wells,
including every boundary crossing in the lower bound. For
$x\notin B_\varepsilon$, we obtain
\begin{equation*}
 \left|\int\zeta^2\sqrt{2W(u_\varepsilon)}|u_{\varepsilon,z}|\,dz
                         -\sigma_0N_\varepsilon(x)\right|
 \le C\int\frac{\left(\sqrt{2W(u_\varepsilon)}-\varepsilon|u_{\varepsilon,z}|\right)^2}{\varepsilon}\,dz
                                  +Ce^{-c/\varepsilon}.
\end{equation*}
The extension used for the upper bound near each well lies inside the
pure-phase boundary regions. Set $N_\varepsilon=Q$ on $B_\varepsilon$. We have
\begin{equation*}
 \left|\int\zeta^2\sqrt{2W(u_\varepsilon)}|u_{\varepsilon,z}|\,dz-Q\sigma_0\right|
 \le\left(\int\zeta^2\sqrt{2W(u_\varepsilon)}|u_{\varepsilon,z}|\,dz-Q\sigma_0\right)_+
                                      +Q\sigma_0.
\end{equation*}
Integration proves \eqref{ninte}.

The exact normal-line identity is
\begin{equation}\label{ciden}
\begin{split}
 \int\zeta^2\left(\frac{\varepsilon}{2}|Du_\varepsilon|^2+
                   \frac{W(u_\varepsilon)}{\varepsilon}\right)\,dz
 &=\int\zeta^2\sqrt{2W(u_\varepsilon)}|u_{\varepsilon,z}|\,dz\\
 &\quad+\frac12\int\zeta^2\varepsilon|D_xu_\varepsilon|^2\,dz\\
 &\quad+\frac1{2\varepsilon}\int\zeta^2\left(\sqrt{2W(u_\varepsilon)}-\varepsilon|u_{\varepsilon,z}|\right)^2\,dz.
\end{split}
\end{equation}
It shows that the unsigned normal integral has weak limit
$Q\sigma_0\,dx$. Equation \eqref{ninte} transfers this
limit to $\sigma_0N_\varepsilon\,dx$. Since $0\le N_\varepsilon\le Q$,
we test on an interior base ball with a nonnegative smooth cutoff
and obtain \eqref{cmacr}.
\end{proof}

\begin{remark}\label{cntupper}
For sequences satisfying \eqref{acn}, \eqref{stab},
\eqref{bound}, and \eqref{plimi}, the root-count conclusion of
Lemma~\ref{count} also holds under the weaker condition
\begin{equation*}
 E_\varepsilon\le\widehat E_\varepsilon\longrightarrow0,
 \qquad
 \frac{\varepsilon^2}{\widehat E_\varepsilon}\longrightarrow0,
\end{equation*}
where $\widehat E_\varepsilon>0$ is an upper bound for the actual
tangential energy. We verify the errors needed for this conclusion.
Proposition~\ref{cover} bounds its covering contribution by
\begin{equation*}
 CE_\varepsilon^{1+1/[2(n-1)]}
 \le C\widehat E_\varepsilon^{1+1/[2(n-1)]}
 =o(\widehat E_\varepsilon).
\end{equation*}
Proposition~\ref{discr} and the boundary terms in
\eqref{cerro} contribute $O(\varepsilon^2)+O(e^{-c/\varepsilon})
=o(\widehat E_\varepsilon)$. For each sufficiently small fixed
$\rho>0$, the proof of \eqref{rtail} yields
\begin{equation*}
 \int_{\C_{12}\cap\{|u_\varepsilon|>1-\rho\}}
 \varepsilon|D_xu_\varepsilon|^2\,dX
 \le C(\rho^c+\varepsilon^2)E_\varepsilon.
\end{equation*}
After division by $\widehat E_\varepsilon$, we first let
$\varepsilon\to0$ and then $\rho\downarrow0$. The proof of
\eqref{fangl} therefore yields that estimate with right side
$o(\widehat E_\varepsilon)$. The inequality preceding
\eqref{ndsmal}, followed first by $\varepsilon\to0$ and then
by $\tau\downarrow0$, yields the same replacement in
\eqref{ndsmal}.

These are all the small errors used in Lemmas~\ref{ccove}
and~\ref{count}. The exclusion of $Q+1$ uncovered roots uses
the planar full-energy limit and Lemma~\ref{hfset}.
Equation~\eqref{ciden} and that limit then imply that the set
where $N_\varepsilon<Q$ has measure tending to zero.
For any prescribed compact interval $I\Subset(-1,1)$, we choose
the two auxiliary bands of level values on opposite sides of
$I$. Outside a subset of $B_4^n$ whose measure tends to zero,
every level $s\in I$ consequently has exactly $Q$ roots in
$(-4,4)$. This argument concerns the count alone and keeps
$E_\varepsilon$ equal to the actual tangential energy.
\end{remark}

The preceding result controls the number of transitions on most
vertical lines. We now prove compactness for the tilt-excess
integrated along vertical lines with the cutoff $\zeta^2$, including
lines meeting the covering. A BV bound for a suitable scalar
combination implies strong local $L^1$ precompactness after division by
$E_\varepsilon$. The same bound improves the estimate for the set of
lines with fewer than $Q$ transitions to the scale required in the
next subsection.

\begin{proposition}\label{bvcmp}
On $B_4^n$, the following estimate holds.
\begin{equation}\label{sbv}
 \left\|\frac1{E_\varepsilon}\left[Q\sigma_0+
 \int\zeta^2\left(\frac\varepsilon2|D_xu_\varepsilon|^2
      -\frac\varepsilon2 u_{\varepsilon,z}^2
      -\frac{W(u_\varepsilon)}\varepsilon\right)\,dz\right]\right\|_{BV}
 \le C.
\end{equation}
The functions
$E_\varepsilon^{-1}\int\zeta^2\varepsilon|D_xu_\varepsilon|^2\,dz$
are strongly precompact in $L^1_{\rm loc}(B_4^n)$. On each compactly
contained base ball, one also has
\begin{equation}\label{rsatu}
 \int\left|\int\zeta^2\sqrt{2W(u_\varepsilon)}|u_{\varepsilon,z}|\,dz-Q\sigma_0\right|\,dx
                =o(E_\varepsilon),
 \qquad |\{N_\varepsilon<Q\}|=o(E_\varepsilon).
\end{equation}
\end{proposition}
\begin{proof}
We use the domain-variation identity for the Allen--Cahn stress tensor
\cite[(2.6)]{Wang17}. The BV compactness and Poincar\'e--Sobolev
estimates can be found in \cite[Chapter~3]{AFP}.
The proof has three steps. We first bound the variation and mean of
the scalar expression. The BV Sobolev inequality then implies uniform
integrability, and the transition count allows us to express the
normalized tangential integral up to an error that tends to zero.

The horizontal rows of the Allen--Cahn stress yield
\begin{equation*}
\begin{split}
 \Div_x\int\zeta^2\left[\left(\frac{\varepsilon}{2}|Du_\varepsilon|^2+
                   \frac{W(u_\varepsilon)}{\varepsilon}\right) I_n
        -\varepsilon D_xu_\varepsilon\otimes D_xu_\varepsilon\right]\,dz
 =-\int(\zeta^2)'\varepsilon u_{\varepsilon,z}D_xu_\varepsilon\,dz.
\end{split}
\end{equation*}
The right side is exponentially small. We subtract the full horizontal stress from the tangential tensor in
\eqref{tdiv}, integrated against $\varepsilon\zeta^2$.
The difference is a scalar multiple of $I_n$, namely the scalar in
\eqref{sbv} before adding its constant and dividing by
$E_\varepsilon$. Thus its distributional derivative is controlled
by the difference of the two divergences. After division by $E_\varepsilon$, the total variation
of the normalized scalar is bounded by
\begin{equation*}
 \frac C{E_\varepsilon}
 \left(\int\varepsilon|D_xu_\varepsilon|^2\,dX\right)^{1/2}
 \left(\int\varepsilon\left(|D(D_xu_\varepsilon)|^2-|D|D_xu_\varepsilon||^2\right)\,dX\right)^{1/2}+o(1)\le C.
\end{equation*}
We work first on $B_5^n$ and then restrict to $B_4^n$. On
$\{N_\varepsilon=Q\}$, whose measure tends to that of the base ball, \eqref{ciden} and \eqref{ninte}
bound the $L^1$ norm of this scalar expression by $C$.
The Poincar\'e inequality in $BV$ bounds the oscillation from the mean.
On the same set, the triangle inequality bounds the absolute mean
by the $L^1$ norm there and the oscillation, divided by
$|\{N_\varepsilon=Q\}|$. This measure is bounded below, which proves
\eqref{sbv}.

By \eqref{ciden}, the nonnegative normalized tangential
integral is bounded above by twice the positive part of the scalar
function in \eqref{sbv}, plus
\begin{equation*}
 \frac2{E_\varepsilon}
   \left(\int\zeta^2\sqrt{2W(u_\varepsilon)}|u_{\varepsilon,z}|\,dz-Q\sigma_0\right)_+
       +\frac1{\varepsilon E_\varepsilon}\int\zeta^2\left(\sqrt{2W(u_\varepsilon)}-\varepsilon|u_{\varepsilon,z}|\right)^2\,dz.
\end{equation*}
The last two functions have $L^1$ norm tending to zero. From the Sobolev
inequality for $BV$ functions, we therefore obtain, uniformly over measurable base sets $A$,
\begin{equation}\label{uinte}
 \frac1{E_\varepsilon}\int_A\int\zeta^2
                  \varepsilon|D_xu_\varepsilon|^2\,dz\,dx
 \le C|A|^{1/n}+o(1).
\end{equation}
The estimate applies to full columns, including those meeting the
covering.

The same scalar identity, using nonnegativity of the tangential
integral, leads to
\begin{equation*}
\begin{split}
 \left\|\frac{\sigma_0}{E_\varepsilon}(Q-N_\varepsilon)\right\|_1
 &\le C|\{N_\varepsilon<Q\}|^{1/n}+o(1),\\
 |\{N_\varepsilon<Q\}|
 &\le\frac{E_\varepsilon}{\sigma_0}
        \left\|\frac{\sigma_0}{E_\varepsilon}(Q-N_\varepsilon)\right\|_1.
\end{split}
\end{equation*}
We substitute the second bound into the first and apply Young's
inequality. This yields
\begin{equation*}
 \left\|\frac{\sigma_0}{E_\varepsilon}(Q-N_\varepsilon)\right\|_1
 \le C E_\varepsilon^{1/(n-1)}+o(1)=o(1).
\end{equation*}
Together with \eqref{ninte}, we obtain
\eqref{rsatu}. Finally, \eqref{ciden}
expresses the normalized tangential integral as twice the scalar
BV function in \eqref{sbv}, up to an error tending to zero
in $L^1$. BV compactness therefore implies the asserted strong
precompactness and completes the proof.
\end{proof}

\subsection{Harmonic approximation and strong convergence of tilt-excess}
\label{harm}
\label{harmapprox}
The aim of this subsection is to determine the compactness limit from
Proposition~\ref{bvcmp} and prove the harmonic approximation in
Proposition~\ref{hcomp}. We show that the normalized layer heights
have harmonic limits and that their Dirichlet energy densities recover
the full normalized tilt-excess. Strong convergence of the slopes then
yields an excess decay estimate.

We first select and order the roots of intermediate level sets.
These heights need not be continuous, since roots may enter or leave
the selected list. We estimate the resulting boundary terms and use
the transformed normal flux to show that the normalized,
level-averaged heights have harmonic limits. We then control the
variation of the slopes in the level variable and recover the full
tilt-excess, including the regions near the wells.

Finally, from interior estimates for the harmonic limits, we obtain excess
decay after subtracting the linear part of each height. A common
change of normal requires these linear parts to agree. We state
this condition explicitly. The next subsection will determine when it
holds and use it in the iteration of excess decay.
 \begin{proposition}\label{hcomp}
Assume \eqref{wells}, \eqref{acn}, \eqref{stab},
\eqref{bound}, \eqref{plimi}, and \eqref{thickn}.
After passing to a subsequence, there are harmonic functions
$h_1,\ldots,h_Q$ on $B_2^n$ such that
\begin{equation}\label{stren}
 \frac1{E_\varepsilon}\int_{-4}^{4}
       \varepsilon|D_xu_\varepsilon(x,z)|^2\,dz
 \longrightarrow
 \sigma_0\sum_{\alpha=1}^Q|Dh_\alpha(x)|^2
 \quad\text{strongly in }L^1_{\mathrm{loc}}(B_2^n).
\end{equation}
These functions satisfy the energy bound
\begin{equation*}
 \sigma_0\sum_{\alpha=1}^Q
       \int_{B_{3/2}}|Dh_\alpha|^2\,dx\le1.
\end{equation*}
Along the same subsequence, for every fixed sufficiently small
$\theta>0$, we can choose unit vectors
$e_{\alpha,\varepsilon}$, $1\le\alpha\le Q$,
satisfying the following estimates with $C$ independent of
$\theta$:
\begin{gather}
 \max_{1\le\alpha\le Q}
       |e_{\alpha,\varepsilon}-e_{n+1}|
 \le C\sqrt{E_\varepsilon},
 \label{nchan}\\
 \limsup_{\varepsilon\to0}
 \frac1{\theta^n E_\varepsilon}
 \int_{B_\theta^n\times(-4,4)}
 \min_{1\le\alpha\le Q}
 \varepsilon\left(
       |Du_\varepsilon|^2
       -|Du_\varepsilon\cdot e_{\alpha,\varepsilon}|^2
       \right)\,dX
 \le C\theta^2.
 \label{mdeca}
\end{gather}
If
\begin{equation}\label{cgrad}
 Dh_1(0)=\cdots=Dh_Q(0),
\end{equation}
we can choose the same unit vector for every $\alpha$
in \eqref{nchan} and \eqref{mdeca}.
\end{proposition}

The functions $h_\alpha$ arise from the ordered roots of intermediate
levels after subtracting separate additive constants.
For a single normal, the estimate also contains the variance of
$Dh_1(0),\ldots,Dh_Q(0)$, as shown in \eqref{spvari}.
Condition \eqref{cgrad} removes this term and holds automatically
when $Q=1$.

We follow Wang's method of normalizing the height functions and
proving strong convergence of the excess
\cite[Sections~6--8]{Wang17}.
Related estimates for several layers appear in Florit-Simon--Serra
\cite[Proposition~7.6, Lemma~7.9 and Section~7.5]{FSS25}.

We begin with the construction of the selected heights. We work on
$B_4^n\times(-4,4)$ and fix a compact interval
$I\Subset(-1,1)$ containing the central interval in
Lemma~\ref{count}. We enlarge it slightly
when a cutoff in the level variable is used. We choose a small
$t_0>0$ and a smooth nonnegative weight $\rho\in C_c^\infty((t_0,2t_0))$, with
$\int\rho(t)\,dt=1$. For almost every $t\in(t_0,2t_0)$, select those
roots of $u_\varepsilon(x,z)=s$ for which
\begin{equation}\label{rroot}
 \frac{\varepsilon|u_{\varepsilon,z}|}{\sqrt{2W(u_\varepsilon)}}>t.
\end{equation}
Order the roots by height and denote the first $Q$ by
$f_{\alpha,\varepsilon}^t(x,s)$, $1\le\alpha\le Q$.
If there are between one and $Q-1$ roots, we repeat the last one.
If the root set is empty, we assign zero to every index. We use the same
conventions for ordinary derivatives and fluxes.
Counting selected roots below each height shows that the roots and
their ranks are measurable.

On a compact interval of level values, we obtain from \eqref{rroot}
$|u_{\varepsilon,z}|\ge c/\varepsilon$ at each selected root.
By \eqref{ellipt}, this derivative keeps its sign on a vertical
segment of length $c\varepsilon$. Thus distinct selected roots of
the same level are separated by $c\varepsilon$, and there are
finitely many roots per line. The selected roots over a compact base
and a compact set of level values have compact closure in the normal variable and satisfy
$|u_{\varepsilon,z}|\ge c/\varepsilon$. Away from the images of the
boundary where equality holds in \eqref{rroot}, properness and the
inverse function theorem show that the number of selected roots is
locally constant and each root is smooth. On this set we have
\begin{equation}\label{rslop}
 D_xf_{\alpha,\varepsilon}^t=-\frac{D_xu_\varepsilon}{u_{\varepsilon,z}},
 \qquad |D_xf_{\alpha,\varepsilon}^t|\le C_{I,t_0},
\end{equation}
where values of $u_\varepsilon$ are evaluated on the indicated root.
All these heights are bounded by
$\sup\{|z|\mid (x,z)\in\C_6,\ u_\varepsilon(x,z)\in I\}$, which tends to
zero. The supremum is set equal to zero when the set is empty.
The selected heights can jump when roots enter or leave the list.
Before averaging the heights, we estimate these jumps and the
boundary terms in the transformed normal flux. The following lemma
shows that both errors vanish after normalization.

\begin{lemma}\label{tboun}
With all integrals restricted to fixed interior cylinders and an
enlarged compact subinterval of $(-1,1)$,
\begin{equation}\label{tarea}
 \int\rho(t)\mathcal H^n\left(
 \left\{\frac{\varepsilon|u_{\varepsilon,z}|}{\sqrt{2W(u_\varepsilon)}}=t\right\}
                         \cap\{u_\varepsilon\in I\}\right)\,dt
 \le C\varepsilon E_\varepsilon^{\,1/2+1/[4(n-1)]}.
\end{equation}
After integration in the level variable, the horizontal distributional derivative
of each selected height differs from its ordinary derivative by a
measure whose variation, averaged against $\rho(t)\,dt$, is at most
\begin{equation}\label{hberro}
 C\left(\sup_{\C_6\cap\{u_\varepsilon\in I\}}|z|\right)
                    E_\varepsilon^{\,1/2+1/[4(n-1)]}.
\end{equation}
The divergence of its transformed normal flux, in the variables
$(x,s)$, has variation, averaged against $\rho(t)\,dt$, at most
\begin{equation}\label{fberro}
 C E_\varepsilon^{\,1/2+1/[4(n-1)]}.
\end{equation}
\end{lemma}
\begin{proof}
Take $2t_0<1/4$. If the ratio in \eqref{rroot} is between
$t_0$ and $2t_0$, either $\varepsilon|Du_\varepsilon|\le \sqrt{2W(u_\varepsilon)}/2$,
or $|u_{\varepsilon,z}|/|Du_\varepsilon|<1/2$.
The region where this ratio lies in $(t_0,2t_0)$ is therefore contained
in the exceptional set of Proposition~\ref{cover}, for a fixed angle. Its full energy is
$O(E_\varepsilon^{1+1/[2(n-1)]})$.
Coarea and Lemma~\ref{rlemm} yield
\begin{align*}
 &\int\rho(t)\int_{\{\varepsilon|u_{\varepsilon,z}|/\sqrt{2W(u_\varepsilon)}=t\}}
                   \frac{2W(u_\varepsilon)}{\varepsilon}\,d\mathcal H^n\,dt\\
 \le& C
 \left(\int\frac{2W(u_\varepsilon)}{\varepsilon}
       \left|D\left(\frac{\varepsilon u_{\varepsilon,z}}{\sqrt{2W(u_\varepsilon)}}\right)\right|^2\,dX\right)^{1/2}
 \left(\int_{\{t_0<\varepsilon|u_{\varepsilon,z}|/\sqrt{2W(u_\varepsilon)}<2t_0\}}
                 \frac{2W(u_\varepsilon)}{\varepsilon}\,dX\right)^{1/2}
 \\
 \le& C E_\varepsilon^{1/2+1/[4(n-1)]}.
\end{align*}
The positive lower bound for $\sqrt{2W}$ on $I$ proves
\eqref{tarea}.

The map $(x,z)\mapsto(x,u_\varepsilon(x,z))$ is a local
diffeomorphism at selected roots and on the boundary where equality
holds in \eqref{rroot}. If $(N_x,N_z)$ is a unit normal to this boundary
in the original variables,
its cofactor image is
\begin{equation}\label{bcofa}
 (u_{\varepsilon,z}N_x-D_xu_\varepsilon N_z,\ N_z).
\end{equation}
Thus on the image boundary, with unit normal $(n_x,n_s)$, we obtain from the area
formula with multiplicities
\begin{equation}\label{aboun}
 \int\left(|n_x|+\varepsilon^{-1}|n_s|\right)\,d\mathcal H^n
 \le\frac C\varepsilon\mathcal H^n
 \left(\left\{\frac{\varepsilon|u_{\varepsilon,z}|}
                         {\sqrt{2W(u_\varepsilon)}}=t\right\}
                         \cap\{u_\varepsilon\in I\}\right).
\end{equation}
Here we have used \eqref{ellipt} in the cofactor formula.

On a smooth branch the normal row of the stress is transformed into
\begin{equation}\label{bflux}
 \left(\varepsilon|u_{\varepsilon,z}|D_xf_{\alpha,\varepsilon}^t,
 \quad\operatorname{sgn}(u_{\varepsilon,z})
       \left[\frac{W(u_\varepsilon)}{\varepsilon}
                     -\frac\varepsilon2|Du_\varepsilon|^2\right]\right).
\end{equation}
Its divergence in $(x,s)$ is zero. To see this, recall that the
original divergence-free row is
\begin{equation*}
 \left(-\varepsilon u_{\varepsilon,z}D_xu_\varepsilon,
 \quad\frac{W(u_\varepsilon)}{\varepsilon}
       -\frac\varepsilon2|Du_\varepsilon|^2
       +\varepsilon|D_xu_\varepsilon|^2\right).
\end{equation*}
We multiply by the differential of
$(x,z)\mapsto(x,u_\varepsilon(x,z))$ and divide by
$|u_{\varepsilon,z}|$ to obtain
\eqref{bflux}. Its horizontal entries are bounded by $C$,
and its last entry is bounded by $C/\varepsilon$.

On a compact interior base and a compact interval of level values,
the selected roots admit finitely many inverse-function charts for
each fixed $\varepsilon$. The roots keep their order. For almost
every $t$, the inequalities \eqref{rroot} define finite-perimeter
pieces, by coarea. We apply the BV product rule on these pieces
\cite[Chapter~3]{AFP}. Contributions from subdivisions of the charts
cancel, and ordered labels change only on the image of the boundary
where equality holds in \eqref{rroot}.

At each insertion or deletion, at most $Q$ selected labels change.
The sum of the absolute values of their height jumps is bounded by
$2Q\sup_{\C_6\cap\{u_\varepsilon\in I\}}|z|$ per boundary preimage.
The normal flux jumps are bounded by
$CQ(|n_x|+\varepsilon^{-1}|n_s|)$ per preimage.
The area formula, with these multiplicities, and
\eqref{tarea}, \eqref{bcofa}, and \eqref{aboun} yield
\eqref{hberro}--\eqref{fberro}. Exhaustion removes the restriction
to compact chart domains. The selected roots stay away from the normal
endpoints, and the tests are compactly supported in the level
variable, so all boundary terms have been accounted for.
\end{proof}

We now average the selected heights in the level variable. Choose a
nonzero function $0\le\chi\in C_c^\infty(I)$ and define
the level-averaged height
\begin{equation}\label{aheig}
 \bar f_{\alpha,\varepsilon}(x)=
 \frac{\displaystyle\int\rho(t)\int\chi(s)\sqrt{2W(s)}
                         f_{\alpha,\varepsilon}^t(x,s)\,ds\,dt}
      {\displaystyle\int\chi(s)\sqrt{2W(s)}\,ds}.
\end{equation}
In the following formulas $D_x f_{\alpha,\varepsilon}^t$ means the
ordinary derivative away from the images of these boundaries, with the repetition
and zero conventions already stated. From Lemma~\ref{tboun},
we obtain
\begin{equation}\label{ahderi}
 D\bar f_{\alpha,\varepsilon}=
 \frac{\displaystyle\int\rho(t)\int\chi(s)\sqrt{2W(s)}
                         D_x f_{\alpha,\varepsilon}^t\,ds\,dt}
      {\displaystyle\int\chi(s)\sqrt{2W(s)}\,ds}\,dx
                         +\xi_{\alpha,\varepsilon},
 \qquad
 \frac{\|\xi_{\alpha,\varepsilon}\|_{\mathcal M}}
                  {\sqrt{E_\varepsilon}}\longrightarrow0.
\end{equation}
The measure $\xi_{\alpha,\varepsilon}$ includes its absolutely
continuous part.

Coarea $ds=|u_{\varepsilon,z}|\,dz$ and the Cauchy--Schwarz inequality
show that the squares of the slope averages in \eqref{ahderi}
and of the flux averages
\begin{equation*}
 \int\rho(t)\int\chi(s)\varepsilon|u_{\varepsilon,z}|
                      D_x f_{\alpha,\varepsilon}^t\,ds\,dt
\end{equation*}
are bounded by
$C\int\zeta^2\varepsilon|D_xu_\varepsilon|^2\,dz$.
A repeated root is used at most $Q$ times. Furthermore, we have
\begin{equation}\label{aferro}
\begin{split}
 &\int\left|\int\rho(t)\int\chi(s)
       \left(\varepsilon|u_{\varepsilon,z}|-\sqrt{2W(s)}\right)
                 D_x f_{\alpha,\varepsilon}^t\,ds\,dt\right|^2\,dx\\
 \le& C\int\frac{\left(\sqrt{2W(u_\varepsilon)}-\varepsilon|u_{\varepsilon,z}|\right)^2}{\varepsilon}\,dX\\
 =&o(E_\varepsilon).
\end{split}
\end{equation}
The bounded slopes in \eqref{rslop} and
$|u_{\varepsilon,z}|\le C/\varepsilon$ justify this direct estimate.

We integrate \eqref{bflux} against $\chi(s)$ and obtain
\begin{equation}\label{afdive}
\begin{split}
 &\left\|\Div_x\int\rho(t)\int\chi(s)\varepsilon|u_{\varepsilon,z}|D_x f_{\alpha,\varepsilon}^t\,ds\,dt\right\|_{\mathcal M}\\
 \le& C E_\varepsilon^{1/2+1/[4(n-1)]}
    +C\int_{\{u_\varepsilon\in I\}}
     \left|\frac{W(u_\varepsilon)}{\varepsilon}
            -\frac\varepsilon2|Du_\varepsilon|^2\right|
    |u_{\varepsilon,z}|\,dX\\
 \le& C E_\varepsilon^{1/2+1/[4(n-1)]}+C\varepsilon\\
=&o(\sqrt{E_\varepsilon}).
\end{split}
\end{equation}
The term containing $\chi'$ is bounded by $C\varepsilon$, by
\eqref{relip} and \eqref{adest}. After normalization,
$\varepsilon/\sqrt{E_\varepsilon}\to0$ by \eqref{thickn}.

The estimates above imply a small error in the derivative identity
and a small divergence for the averaged flux. The next lemma uses
these two facts, together with uniform integrability, to prove
strong convergence of the averaged slopes and harmonicity of the
limiting heights.

\begin{lemma}\label{ecomp}
Let $h_j\to h$ locally in $L^1$ on a domain in $\mathbb R^n$, $n\ge2$.
Suppose
\begin{equation*}
 Dh_j=v_j\,dx+\xi_j,\qquad
 \|\xi_j\|_{\mathcal M}\to0,
 \qquad \Div v_j=f_j+\Div e_j,
 \qquad \|f_j\|_{\mathcal M}+\|e_j\|_1\to0
\end{equation*}
on interior subdomains. If $v_j$ is bounded in $L^2$ and $|v_j|^2$
is uniformly integrable there, then $h$ is harmonic and
$v_j\to Dh$ strongly in $L^2_{\rm loc}$.
\end{lemma}
\begin{proof}
We use the weak estimates for Newtonian potentials and the estimate
for convolutions with finite measures in Ambrosio--Ponce--Rodiac
\cite[Theorem~4.1]{APR20}. Weak convergence and the derivative identity yield $h\in W^{1,2}_{\rm loc}$
and $\Delta h=0$. On nested interior balls, multiply the sources
in $\Delta h_j=f_j+\Div(e_j\,dx+\xi_j)$ by a cutoff equal to one
on the larger ball. Their Newtonian potential tends to zero in
$L^1$ on bounded sets, by local integrability of the fundamental
solution and its gradient. This includes the logarithmic fundamental
solution when $n=2$. Subtraction leaves a harmonic function whose
gradients converge smoothly on the smaller ball.

The weak estimates of Calder\'on--Zygmund \cite{CZ} and Stein
\cite{Stein} control the gradients of the mollified potentials.
For the terms containing $e_j\,dx+\xi_j$, the bound for finite
measures follows from \cite[Theorem~4.1]{APR20}, applied to each
derivative of the fundamental solution. All these bounds tend to
zero uniformly in the mollification parameter. On an interior
ball, the distributional gradient of the potential is $Dh_j$
minus a smooth gradient. Lebesgue differentiation and Fatou's
lemma therefore pass the estimates to its absolutely continuous
density, which tends to zero in measure.
The absolutely continuous part of $\xi_j$ tends to zero in $L^1$.
It follows that $v_j\to Dh$ in measure. Uniform integrability of
$|v_j|^2$ and Vitali's theorem yield the asserted strong convergence.
\end{proof}

We can now prove Proposition~\ref{hcomp}. We first determine the
harmonic limits of the averaged heights, then recover the excess
from the individual slopes, and finally estimate the excess after
changing the normal.

\begin{proof}[Proof of Proposition~\ref{hcomp}]
We first apply Lemma~\ref{ecomp} to the averaged heights in
\eqref{aheig}. For each $\alpha$, we let $c_{\alpha,\varepsilon}$ be
the mean of $\bar f_{\alpha,\varepsilon}$ over $B_3^n$ and divide
$\bar f_{\alpha,\varepsilon}-c_{\alpha,\varepsilon}$ by
$\sqrt{E_\varepsilon}$. Equation \eqref{ahderi} and the Poincar\'e
inequality in $BV$ imply local $L^1$ compactness. We divide the flux
average preceding \eqref{aferro} by the fixed denominator in
\eqref{aheig}. After normalization by $\sqrt{E_\varepsilon}$, the
slope and flux averages have uniformly integrable squares, by their
pointwise bounds and \eqref{uinte}. Their difference tends to zero
in $L^2$ by \eqref{aferro}, and \eqref{afdive} provides the required
divergence estimate. After we pass to a subsequence, we obtain
harmonic functions $h_\alpha$ on $B_2$ such that
\begin{align}\label{hconv}
 &\frac{\bar f_{\alpha,\varepsilon}-c_{\alpha,\varepsilon}}
              {\sqrt{E_\varepsilon}}
          \longrightarrow h_\alpha
                       \quad\text{in }L^1_{\rm loc}(B_2),\notag\\
 &\frac{1}{\sqrt{E_\varepsilon}}
 \frac{\displaystyle\int\rho(t)\int\chi(s)\sqrt{2W(s)}
                          D_xf_{\alpha,\varepsilon}^t\,ds\,dt}
      {\displaystyle\int\chi(s)\sqrt{2W(s)}\,ds}
          \longrightarrow Dh_\alpha
                       \quad\text{strongly in }L^2_{\rm loc}(B_2).
\end{align}
The convergence in \eqref{hconv} concerns averages over the level
values. To recover the tilt-excess, we must show that the individual
slopes have the same limits. We keep the weight $\chi$
fixed and consider a compact interval $I\Subset(-1,1)$ containing
its support and the central interval in Lemma~\ref{count}.
We denote by $A_{I,\varepsilon}$ the union of the projection of
\begin{equation*}
 \{u_\varepsilon\in I,\quad
       \varepsilon|u_{\varepsilon,z}|/\sqrt{2W(u_\varepsilon)}\le2t_0\},
\end{equation*}
the projected covering in Lemma~\ref{ccove}, and the set
$\{N_\varepsilon<Q\}$. Proposition~\ref{cover} controls the first
projection, as in the proof of Lemma~\ref{tboun}. Lemma~\ref{ccove}
and \eqref{rsatu} control the other two sets. Consequently,
\begin{equation}\label{dbase}
 |A_{I,\varepsilon}|=o(E_\varepsilon).
\end{equation}
Outside $A_{I,\varepsilon}$, $N_\varepsilon$ is the actual crossing
number, and every level in $I$ has exactly $Q$ roots. These roots
continue across $I$ in a fixed vertical order. Each occurs once in
the selected list, which is independent of $t\in(t_0,2t_0)$.

We now estimate the contributions from the deleted columns.
Equation \eqref{uinte} shows that their tilt-excess is
$o(E_\varepsilon)$. The pointwise bounds used above yield the same
conclusion for the squared slope averages. To control the individual
slopes as well, we use coarea and \eqref{rslop}. Each selected root
occurs at most $Q$ times, so for every measurable base set $A$,
\begin{equation*}
 \int_A\int_I\int\rho(t)
          \sum_{\alpha=1}^Q|D_xf_{\alpha,\varepsilon}^t|^2\,dt\,ds\,dx
 \le C_I\int_{A\times(-4,4)}\varepsilon|D_xu_\varepsilon|^2\,dX.
\end{equation*}
With $A=A_{I,\varepsilon}$, \eqref{uinte} makes this bound
$o(E_\varepsilon)$. We also need the normal energy on these columns.
Equations \eqref{rsatu} and \eqref{dbase} bound their unsigned normal
integral by $o(E_\varepsilon)$. The identity \eqref{ciden}, together
with \eqref{uinte} and \eqref{ndsmal}, bounds their full energy by
the same quantity. In particular,
\begin{equation}\label{dnorm}
 \int_{A_{I,\varepsilon}}\int\zeta^2
                        \varepsilon u_{\varepsilon,z}^2\,dz\,dx=o(1).
\end{equation}
It remains to compare the slopes on the columns that have not been
deleted. On each selected branch, we have
$|u_{\varepsilon,z}|\ge c_I/\varepsilon$ and a fixed positive lower
bound for $|u_{\varepsilon,z}|/|Du_\varepsilon|$. By differentiation of
\eqref{rslop}, we therefore obtain
\begin{equation*}
 \left|\partial_sD_x f_{\alpha,\varepsilon}^t\right|
 \le C_I\varepsilon
                  \left|D\left(\frac{Du_\varepsilon}{|Du_\varepsilon|}\right)\right|.
\end{equation*}
Coarea, \eqref{gloca}, and $\varepsilon^2/E_\varepsilon\to0$
from \eqref{thickn} yield
\begin{equation}\label{psvari}
 \frac1{E_\varepsilon}
 \sum_{\alpha=1}^Q\int_{(B_2\setminus A_{I,\varepsilon})\times I}
             |\partial_sD_xf_{\alpha,\varepsilon}^t|^2\,dx\,ds
 \le\frac{C_I\varepsilon^2}{E_\varepsilon}
                       \int\varepsilon\left(|D^2u_\varepsilon|^2-|D|Du_\varepsilon||^2\right)\,dX
 \longrightarrow0
\end{equation}
on a fixed smaller base ball when needed. The one-dimensional
Poincar\'e inequality in $s$, with the fixed nonzero weight
$\chi(s)\sqrt{2W(s)}$, shows that each slope has the same limit as
its average in \eqref{hconv}. We extend the slopes by zero over
$A_{I,\varepsilon}$. The estimates above justify this extension and
yield
\begin{equation}\label{pslimi}
 \frac{D_xf_{\alpha,\varepsilon}^t(x,s)}{\sqrt{E_\varepsilon}}
       \longrightarrow Dh_\alpha(x)
 \quad\text{strongly in }L^2_{\rm loc}(dx\,ds)
\end{equation}
for the extended functions.

We now recover the tilt-excess from the slope convergence. On each
remaining column, vertical coarea yields
\begin{equation*}
 \int_{\{u_\varepsilon\in I\}}\varepsilon|D_xu_\varepsilon|^2\,dz
 =\sum_{\alpha=1}^Q\int_I
       \varepsilon|u_{\varepsilon,z}|\,
                        |D_xf_{\alpha,\varepsilon}^t|^2\,ds.
\end{equation*}
To use \eqref{pslimi}, we replace $\varepsilon|u_{\varepsilon,z}|$
by $\sqrt{2W(s)}$. Cauchy--Schwarz, \eqref{ndsmal}, and the gradient
bound \eqref{ellipt} bound the integrated absolute error by
\begin{align*}
 &C_I\int\varepsilon\left|\sqrt{2W(u_\varepsilon)}-\varepsilon|u_{\varepsilon,z}|\right|\,|D_xu_\varepsilon|^2\,dX\\
 \le& C_I\left(\int\frac{\left(\sqrt{2W(u_\varepsilon)}-\varepsilon|u_{\varepsilon,z}|\right)^2}{\varepsilon}\,dX\right)^{1/2}
          \left(\int\varepsilon^3|D_xu_\varepsilon|^4\,dX\right)^{1/2}\\
 \le& o(\sqrt{E_\varepsilon})\,C\sqrt{E_\varepsilon}\\
 =&o(E_\varepsilon).
\end{align*}
By \eqref{pslimi} and the estimates on the deleted columns,
the normalized tangential integral over these level values converges
strongly in local $L^1$ to
\begin{equation*}
 \left(\int_I\sqrt{2W(s)}\,ds\right)
                          \sum_{\alpha=1}^Q|Dh_\alpha|^2.
\end{equation*}
We first let $\varepsilon\to0$ for fixed $I$, and then increase $I$
to $(-1,1)$. Since $\chi$ stays fixed, all these intervals yield the
same limiting gradients. Equation \eqref{rtail} controls the omitted
tilt-excess near the wells. We obtain \eqref{stren} with
$\zeta^2$, and \eqref{cerro} removes this cutoff. The definition of
$E_\varepsilon$ also yields the stated bound on the Dirichlet
energies of $h_1,\ldots,h_Q$.

We turn to the effect of changing the normal. In addition to the
tangential energy, we need the mixed integral containing
$u_{\varepsilon,z}D_xu_\varepsilon$ and the normal energy. On a selected
root, \eqref{rslop} expresses the
mixed integral in terms of $D_xf_{\alpha,\varepsilon}^t$.
Coarea, \eqref{pslimi}, and the coefficient estimate \eqref{ndsmal}
determine its limit on each fixed interval $I$. On the deleted
columns, Cauchy--Schwarz applies to the $o(E_\varepsilon)$ tangential
energy and the bounded normal energy. Near the wells, the same
inequality applies with \eqref{rtail} and the full energy bound.
Both contributions vanish after division by $\sqrt{E_\varepsilon}$,
with $I$ increased to $(-1,1)$ after taking the limit in
$\varepsilon$.

For the normal energy, we use the identity
\begin{equation*}
 \varepsilon u_{\varepsilon,z}^2
 =\varepsilon|Du_\varepsilon|^2-\varepsilon|D_xu_\varepsilon|^2.
\end{equation*}
The first term has the measure limit \eqref{plimi}, and the second
has total mass at most $E_\varepsilon\to0$. Integration in $z$
against $\zeta^2$, which equals one on the limiting hyperplane, therefore
determines the normal energy limit on the base. Together with the
mixed integral just considered, we have
\begin{align}\label{mnlimi}
 \frac1{\sqrt{E_\varepsilon}}
       \int\zeta^2\varepsilon u_{\varepsilon,z}D_xu_\varepsilon\,dz
 &\longrightarrow-\sigma_0\sum_{\alpha=1}^QDh_\alpha
                  &&\text{strongly in }L^1_{\rm loc},\notag\\
 \left(\int\zeta^2\varepsilon u_{\varepsilon,z}^2\,dz\right)\,dx
 &\rightharpoonup Q\sigma_0\,dx.
\end{align}
To allow a separate normal for each layer, we also need to control
the normal energy near the wells. For fixed
$I\Subset(-1,1)$, coarea on the selected roots expresses the normal
integral as the sum of $Q$ integrals of
$\varepsilon|u_{\varepsilon,z}|$. Cauchy--Schwarz and
\eqref{ndsmal} allow us to replace this coefficient by
$\sqrt{2W(s)}$ in $L^1(dx\,ds)$. Equation \eqref{dnorm} controls
the deleted columns. Hence, strongly in $L^1$ on smaller base balls,
\begin{equation*}
 \int_{\{z\mid u_\varepsilon(x,z)\in I\}}\zeta^2
                       \varepsilon u_{\varepsilon,z}^2\,dz
       \longrightarrow Q\int_I\sqrt{2W(s)}\,ds.
\end{equation*}
Let us now subtract this limit from the total normal measure in
\eqref{mnlimi}. The remaining measure is nonnegative and has limiting mass
$Q\int_{[-1,1]\setminus I}\sqrt{2W(s)}\,ds$ per unit base volume,
on balls whose boundaries have zero limiting mass. This mass tends
to zero as $I$ exhausts $(-1,1)$.

We first estimate the excess for one prescribed change of normal.
For a fixed vector $v\in\mathbb R^n$, we set
\begin{equation*}
 e_\varepsilon(v)=
 \frac{(-\sqrt{E_\varepsilon}v,1)}{\sqrt{1+E_\varepsilon|v|^2}}.
\end{equation*}
Orthogonal projection onto the hyperplane perpendicular to
$e_\varepsilon(v)$ yields
\begin{equation}\label{ncomp}
 |Du_\varepsilon|^2-|Du_\varepsilon\cdot e_\varepsilon(v)|^2
 \le\left|D_xu_\varepsilon+
                      \sqrt{E_\varepsilon}v u_{\varepsilon,z}\right|^2.
\end{equation}
We multiply \eqref{ncomp} by $\varepsilon/E_\varepsilon$ and
integrate. The squared terms and the mixed term on the right have
the limits established in \eqref{stren} and \eqref{mnlimi}.
The limiting normal measure assigns zero mass to the boundary of
each fixed ball $B_\theta\Subset B_2$, so its mass on that ball also
converges. Using \eqref{cerro} to remove $\zeta^2$, we obtain
\begin{equation}\label{onlimi}
\begin{split}
 &\limsup_{\varepsilon\to0}\frac1{E_\varepsilon}
 \int_{B_\theta\times(-4,4)}\varepsilon
       \left(|Du_\varepsilon|^2
                     -|Du_\varepsilon\cdot e_\varepsilon(v)|^2\right)\,dX\le\sigma_0\sum_{\alpha=1}^Q
                                  \int_{B_\theta}|Dh_\alpha-v|^2\,dx.
\end{split}
\end{equation}

Next we allow a different change of normal on each layer. For fixed
vectors $v_1,\ldots,v_Q$, we use
$e_\varepsilon(v_\alpha)$ on the $\alpha$th root outside the deleted
columns. This bounds the pointwise minimum over the $Q$ normals
from above. On a compact interval of level values, coarea reduces
the estimate to the squared slopes, the constant vectors, and their
cross terms. The coefficient replacement for the squared slopes is
the one used to prove \eqref{stren}. For the constant vectors, we
use their boundedness and
$\int\left|\sqrt{2W(u_\varepsilon)}-\varepsilon|u_{\varepsilon,z}|\right|\,ds\,dx\to0$.
Cauchy--Schwarz controls the cross terms by these two bounds.
Thus \eqref{pslimi} determines the limit on the selected roots.

We still have to estimate the deleted columns and the regions near
the wells. On either set, \eqref{ncomp} bounds the excess by twice the
tilt-excess plus $CE_\varepsilon$ times the normal energy.
On the deleted columns, \eqref{uinte} and \eqref{dnorm} make both
contributions negligible after division by $E_\varepsilon$.
Near the wells, \eqref{rtail} controls the first contribution, and
the preceding estimate for the normal energy controls the second.
We first let $\varepsilon\to0$ and then increase $I$ to $(-1,1)$.
We conclude that
\begin{equation}\label{snlimi}
\begin{split}
 &\limsup_{\varepsilon\to0}\frac1{E_\varepsilon}
 \int_{B_\theta\times(-4,4)}\min_{1\le\alpha\le Q}\varepsilon
       \left(|Du_\varepsilon|^2
    -|Du_\varepsilon\cdot e_\varepsilon(v_\alpha)|^2\right)\,dX\\
    \le&\sigma_0\sum_{\alpha=1}^Q
    \int_{B_\theta}|Dh_\alpha-v_\alpha|^2\,dx.
\end{split}
\end{equation}

We complete the decay estimate by choosing the linear part of each
harmonic height. Interior estimates and their Dirichlet energy
bound yield
\begin{equation*}
 \max_\alpha|Dh_\alpha(0)|\le C,
 \qquad
 \sigma_0\sum_{\alpha=1}^Q\int_{B_\theta}
                    |Dh_\alpha-Dh_\alpha(0)|^2\,dx\le C\theta^{n+2}
 \quad\text{for }0<\theta<1/4.
\end{equation*}
One may choose $v_\alpha=Dh_\alpha(0)$ and
$e_{\alpha,\varepsilon}=e_\varepsilon(v_\alpha)$.
The first bound yields \eqref{nchan}, and the second, together with
\eqref{snlimi}, yields \eqref{mdeca}.

If a single normal is required, we choose
$v=Q^{-1}\sum_{\beta=1}^QDh_\beta(0)$. We decompose each
$Dh_\alpha-v$ as $(Dh_\alpha-Dh_\alpha(0))+(Dh_\alpha(0)-v)$.
The mean-value property implies that the integral of
$Dh_\alpha-Dh_\alpha(0)$ over $B_\theta$ is zero, so the cross term
vanishes after integration. Equation \eqref{onlimi} and the preceding
harmonic estimate yield
\begin{equation}\label{spvari}
\begin{split}
 &\limsup_{\varepsilon\to0}\frac1{\theta^n E_\varepsilon}
 \int_{B_\theta\times(-4,4)}\varepsilon
       \left(|Du_\varepsilon|^2
                     -|Du_\varepsilon\cdot e_\varepsilon(v)|^2\right)\,dX\\
 \le& C\theta^2+
 \sigma_0|B_1^n|\sum_{\alpha=1}^Q
 \left|Dh_\alpha(0)-\frac1Q\sum_{\beta=1}^QDh_\beta(0)\right|^2.
\end{split}
\end{equation}
Under \eqref{cgrad}, the last term vanishes and all
$e_{\alpha,\varepsilon}$ coincide. This proves the final assertion
of Proposition~\ref{hcomp}.
\end{proof}

\subsection{Excess decay and a common rotation for the layers}
\label{scale}
\label{excessiter}
In this subsection, we turn harmonic approximation into the continuation
alternative used in the local sheeting proof. The main objective is
to determine when one rotation improves all the layers, since the
preceding subsection initially allows a separate normal for each harmonic
height. We also keep the thickness error through the iteration and
determine the entire limits that arise at its terminal scale.

We first extend harmonic approximation to limits supported on
several parallel hyperplanes, keeping a common base domain for the
limiting heights. We then prove the fixed-normal estimate and
iterate it to obtain a Morrey bound with a thickness error. This
bound controls limits at intermediate scales and implies the growth
estimates for entire limits obtained at the transition thickness.
Iteration down to a scale determined by the thickness already
appears in Tonegawa \cite[p.~35]{Ton03}; for the tangential excess,
Wang \cite[Lemma~10.1]{Wang17} includes the thickness error in a
Morrey estimate. We follow this scheme using the strong convergence
of Subsection~\ref{harm} at arbitrary multiplicity.
A related distinction between geometric scales and transition
thickness appears in the nonlocal argument of
Dipierro--Serra--Valdinoci \cite{DSV20}.

The remaining difficulty is to compare the linear parts of the
harmonic heights. If the hyperplane through a marked zero carries
the full multiplicity at the height scale, ordering and estimates
for the gaps force the limiting gradients to coincide. The final
proposition consequently establishes three alternatives: decay after a common
rotation, reduction to a smaller multiplicity, or a microscopic
entire limit with the growth bounds needed in Section~\ref{rigidpart}.

For the following lemma, let $u_j$ be bounded stable solutions with parameters $\varepsilon_j$
on a neighborhood of $\overline{B_1^n\times(-1,1)}$. Assume a uniform
local energy bound and, on this product region,
\begin{equation*}
 \varepsilon_j|Du_j|^2\,dX\rightharpoonup
 \sigma_0\sum_{\ell=1}^{L}m_\ell
       \mathcal H^n\lfloor\{z=c_\ell\},
 \qquad -1<c_1<\cdots<c_L<1,\quad m_\ell\in\mathbb N.
\end{equation*}
Suppose that $E_j>0$ bounds the tilt-excess on a fixed larger
neighborhood and
\begin{equation*}
 E_j\longrightarrow0,\qquad \varepsilon_j^2/E_j\longrightarrow0.
\end{equation*}

 We choose a fixed interval $(a,b)\Subset(-1,1)$ with
$a,b\notin\{c_1,\ldots,c_L\}$.

\begin{lemma}\label{pwind}
Under the preceding assumptions, a subsequence admits harmonic
functions $h_{\ell,\alpha}$ such that
\begin{equation*}
 \frac1{E_j}\int_a^b
       \varepsilon_j|D_xu_j(x,z)|^2\,dz
 \longrightarrow
 \sigma_0\sum_{\ell:\,a<c_\ell<b}
          \sum_{\alpha=1}^{m_\ell}|Dh_{\ell,\alpha}(x)|^2
\end{equation*}
strongly in $L^1$ on each smaller fixed base ball.
The harmonic functions have a common base domain chosen
independently of the distances between the hyperplanes.
\end{lemma}

Since $E_j$ bounds the tilt-excess on a larger neighborhood,
the convergence also yields
\begin{equation*}
 \sigma_0\sum_{\ell:\,a<c_\ell<b}
          \sum_{\alpha=1}^{m_\ell}
          \int_{B_r^n}|Dh_{\ell,\alpha}|^2\,dx\le1
\end{equation*}
on each smaller base ball $B_r^n$.
The same conclusions hold on any fixed product region,
after restriction to compactly contained products.
\begin{proof}
We choose disjoint fixed normal intervals about the relevant hyperplanes and
cutoffs equal to one on smaller such intervals. By Lemma~\ref{gtfene},
each fixed intermediate band belonging to a hyperplane lies in its smaller
interval for large $j$. The cutoff derivatives are in pure phase
regions, where the equation and interior estimates yield errors
$C\exp(-c/\varepsilon_j)=o(E_j)$. Constants may depend on the fixed
gaps.

We apply the proof of Proposition~\ref{hcomp} in each interval, keeping the
same lateral base and nested interior base balls. The geometric
stability and tensor estimates use these lateral cutoffs. Their
normal cutoff terms have the preceding exponential bound. We localize
the scalar estimates of Subsection~\ref{disc} by finitely many fixed
balls and obtain errors $O(\varepsilon_j^2)=o(E_j)$. Proposition~\ref{cover}
localizes in the same way. In Lemma~\ref{ccove}, we choose the outer
radius smaller than the distance to the adjacent hyperplanes.
The limiting density there is $m_\ell$, so the uncovered crossing
number in this interval is at most $m_\ell$.

We use all the roots in that interval in the BV estimate of
Proposition~\ref{bvcmp} and the ordered-height arguments of
Subsection~\ref{harm}. Their normal boundary terms are
exponentially small, and the term from the level cutoff divided by
$\sqrt{E_j}$ is $O(\varepsilon_j/\sqrt{E_j})=o(1)$.
Equations~\eqref{hconv}, \eqref{psvari}, and \eqref{pslimi}, with the relative
estimate near the wells in \eqref{rtail}, therefore express the limit of the
whole normal integral in terms of $m_\ell$ harmonic heights on the common smaller base.

The portions between these intervals have exponentially small energy.
The sum of the strong limits proves the assertion. The original
normalized integral bounds the sum of the limiting Dirichlet energies.
Consequently their interior harmonic estimates on the common base
are independent of the gaps.
\end{proof}

The next proposition controls the tilt-excess on a smaller ball,
up to an error proportional to the squared thickness. Keeping the
normal fixed will allow us to iterate this estimate uniformly in
the center and radius.

\begin{proposition}\label{fnstep}
Fix $n\ge2$ and $W$ as in \eqref{wells}, and $\Lambda<\infty$. There
are $A<\infty$ and $\theta_0>0$,  such
that for each $0<\theta\le\theta_0$ there are
$\eta_\theta,\varepsilon_\theta>0$ and $M_\theta<\infty$ with
the following property. If $u_\varepsilon$ is a smooth stable
solution of \eqref{acn} in $B_2^{n+1}$, with
$-1\le u_\varepsilon\le1$, and
\begin{equation*}
 \int_{B_2}\left(\frac\varepsilon2|Du_\varepsilon|^2+
                   \frac{W(u_\varepsilon)}\varepsilon\right)\,dX\le\Lambda,
 \qquad
 \int_{B_2}\varepsilon|D_xu_\varepsilon|^2\,dX\le\eta_\theta,
 \qquad \varepsilon\le\varepsilon_\theta,
\end{equation*}
then
\begin{equation}\label{fnfloo}
 \int_{B_{2\theta}}\varepsilon|D_xu_\varepsilon|^2\,dX
 \le A\theta^n\int_{B_2}\varepsilon|D_xu_\varepsilon|^2\,dX
                         +M_\theta\varepsilon^2.
\end{equation}
\end{proposition}

\begin{proof}
If the assertion failed for every simultaneous choice
of $\eta_\theta,\varepsilon_\theta,M_\theta$,
there would be counterexamples with
\begin{equation*}
 E_j=\int_{B_2}\varepsilon_j|D_xu_j|^2\,dX\longrightarrow0,
 \qquad \varepsilon_j\longrightarrow0,
 \qquad \varepsilon_j^2/E_j\longrightarrow0.
\end{equation*}
From the compactness theorem of Hutchinson--Tonegawa
\cite[Theorem~1]{HT}, we obtain a stationary integral varifold limit on
interior subdomains. Vanishing horizontal energy makes its tangent
hyperplane horizontal almost everywhere. Stationarity against horizontal
vector fields implies that its weight on a product region is
$\mathcal L^n\otimes\nu$. Rectifiability and integrality make
$\nu$ atomic, with masses in $\sigma_0\mathbb N$. The local energy
bound permits only finitely many atoms on a fixed smaller product.

We work in $B_{1/2}^n\times(-1,1)\Subset B_2$ and choose endpoints in
$(-3/4,-1/2)$ and $(1/2,3/4)$ avoiding the limiting hyperplanes.
From Lemma~\ref{pwind}, with the original normalization $E_j$, we obtain
harmonic heights on a common fixed smaller base with total Dirichlet
energy at most one. Thus, for universal $A_0$ and $\theta_0$,
\begin{equation*}
 \sigma_0\sum_{\ell,\alpha}
   \int_{B_{2\theta}^n}|Dh_{\ell,\alpha}|^2\,dx\le A_0\theta^n
 \qquad\text{for }0<\theta\le\theta_0.
\end{equation*}
If the interval between the endpoints contains no limiting
hyperplane, the normalized energy tends to zero on a smaller product.
In either case, strong convergence of the normal integral bounds the energy in
$B_{2\theta}^{n+1}$ and contradicts the counterexample when $A>2A_0$.

In the contradiction argument, we impose the simultaneous conditions
$E_j\le1/j$, $\varepsilon_j\le1/j$, and
$E_j\ge j\varepsilon_j^2$. If instead
$E_j<M_\theta\varepsilon_j^2$, we obtain \eqref{fnfloo} directly
by inclusion of the smaller ball. The case $E_j=0$ is immediate.
\end{proof}

We next iterate the fixed-normal estimate over a range of radii.
The following proposition keeps the thickness error explicit, with
constants uniform in the center and the number of scales. This is
the estimate used below both at intermediate scales and at the
transition thickness.

\begin{proposition}\label{mwfloo}
Fix $0<\gamma<1$. Let $u_\varepsilon$ be bounded and stable in
$B_8^{n+1}$, with uniformly bounded diffuse energy there. Suppose
\begin{equation*}
 \int_{B_4}\varepsilon|D_xu_\varepsilon|^2\,dX\le E_\varepsilon,
 \qquad E_\varepsilon\longrightarrow0.
\end{equation*}
There are constants $C,K$, depending only on the fixed data and
$\gamma$, such that for every $X\in B_1^{n+1}$ and
\begin{equation}\label{mrrang}
 K\max\{E_\varepsilon^{1/(2-2\gamma)},\varepsilon\}\le r\le1,
\end{equation}
one has
\begin{equation}\label{mfloo}
 \int_{B_{2r}(X)}\varepsilon|D_xu_\varepsilon|^2\,dX
 \le C E_\varepsilon r^{n-2+2\gamma}
                              +C\varepsilon^2r^{n-2}.
\end{equation}
\end{proposition}
\begin{proof}
By \eqref{lmearl} and the energy bound on $B_8$, we have
\begin{equation*}
 r^{-n}\int_{B_{2r}(X)}
 \left(\frac{\varepsilon}{2}|Du_\varepsilon|^2
       +\frac{W(u_\varepsilon)}{\varepsilon}\right)\,dX
 \le C
 \qquad\text{for }X\in B_1,\quad 0<r\le1.
\end{equation*}
Thus, after rescaling any of these balls to $B_2$, the energy
bound required in Proposition~\ref{fnstep} is uniform in $X$
and $r$.

Fix $0<\theta<1$, with $\theta\le\theta_0$, sufficiently
small that
\begin{equation*}
 A\theta^n\le\frac12\theta^{n-2+2\gamma}.
\end{equation*}
We consider the radii $r_k=\theta^k$. At radius $r_k$,
Proposition~\ref{fnstep} applies to
$Y\mapsto u_\varepsilon(X+r_kY)$ whenever
\begin{equation*}
 r_k^{-n}\int_{B_{2r_k}(X)}
       \varepsilon|D_xu_\varepsilon|^2\,dX\le\eta_\theta,
 \qquad
 \frac{\varepsilon}{r_k}\le\varepsilon_\theta.
\end{equation*}
The rescaled parameter is $\varepsilon/r_k$. Hence
\eqref{fnfloo}, after returning to the original variables, yields
\begin{equation*}
 \int_{B_{2r_{k+1}}(X)}
       \varepsilon|D_xu_\varepsilon|^2\,dX
 \le \frac12\theta^{n-2+2\gamma}
       \int_{B_{2r_k}(X)}
       \varepsilon|D_xu_\varepsilon|^2\,dX
       +M_\theta\varepsilon^2r_k^{n-2}.
\end{equation*}

Since $B_2(X)\subset B_4$, the integral at $r_0=1$ is at
most $E_\varepsilon$. As long as the preceding applications
of Proposition~\ref{fnstep} are valid, we can iterate this
inequality. The contribution from the initial integral is
at most $E_\varepsilon r_k^{n-2+2\gamma}$. The terms containing
$\varepsilon^2$ form a geometric series with ratio
$\frac12\theta^{2\gamma}<1$, after division by
$\varepsilon^2r_k^{n-2}$. We therefore obtain
\begin{equation}\label{frsolv}
 \int_{B_{2r_k}(X)}
       \varepsilon|D_xu_\varepsilon|^2\,dX
 \le E_\varepsilon r_k^{n-2+2\gamma}
       +C_\theta\varepsilon^2r_k^{n-2}.
\end{equation}
The constant $C_\theta$ is independent of $k$ and $X$.
The same calculation applies when $n=2$.

We now verify the two smallness conditions at every radius
in \eqref{mrrang}. We choose $K\ge1$ sufficiently large that
\begin{equation*}
 K^{-(2-2\gamma)}+C_\theta K^{-2}\le\eta_\theta,
 \qquad
 K^{-1}\le\varepsilon_\theta.
\end{equation*}
At $r_0=1$, both conditions hold for sufficiently small
$\varepsilon$, since $E_\varepsilon\to0$. Suppose that the
iteration has reached a radius $r_k$ satisfying
\eqref{mrrang}. By \eqref{frsolv} and the lower bound for $r_k$,
we have
\begin{equation*}
 r_k^{-n}\int_{B_{2r_k}(X)}
       \varepsilon|D_xu_\varepsilon|^2\,dX
 \le E_\varepsilon r_k^{-2+2\gamma}
       +C_\theta\left(\frac{\varepsilon}{r_k}\right)^2
 \le K^{-(2-2\gamma)}+C_\theta K^{-2}
 \le\eta_\theta.
\end{equation*}
From the same lower bound, we obtain
$\varepsilon/r_k\le K^{-1}\le\varepsilon_\theta$.
Proposition~\ref{fnstep} therefore applies again.
Induction proves \eqref{frsolv} at every radius $r_k$
in the stated range.

Finally, for any $r$ satisfying \eqref{mrrang}, we choose
$k$ so that $r_{k+1}<r\le r_k$. The radius $r_k$ also
satisfies \eqref{mrrang}, and
$B_{2r}(X)\subset B_{2r_k}(X)$. Since $r_k<r/\theta$,
estimate \eqref{frsolv} implies \eqref{mfloo} after enlarging
$C$. All choices are uniform in $X$ and in the number
of iterations.
\end{proof}

Under microscopic rescaling, the thickness term in \eqref{mfloo}
becomes $CR^{n-2}$. It will imply the tangential growth bound for the
entire limits below, as the thickness-scale term in
\cite[(10.8)--(10.9)]{Wang17} does in the multiplicity-one argument.
Before passing to this scale, we describe the limits at intermediate
scales and state the alternatives for their multiplicities.

\label{sectio}
Assume that $u_j$ are bounded stable solutions with parameters
$\varepsilon_j$, locally bounded diffuse energy, and
\begin{equation*}
 \varepsilon_j|Du_j|^2\,dX\rightharpoonup
 Q\sigma_0\mathcal H^n\lfloor\{z=0\}.
\end{equation*}
We work on fixed compact interior regions. By Lemma~\ref{gtfene},
the full energy measures have the same limit. The monotonicity
estimate \eqref{lmearl} also yields, uniformly in interior centers,
\begin{equation}\label{monoto}
 r^{-n}\mathcal E_{\varepsilon_j}(u_j;B_r(X))
 \le s^{-n}\mathcal E_{\varepsilon_j}(u_j;B_s(X))+Cs
 \qquad\text{for }0<r<s.
\end{equation}

The next lemma shows that an intermediate limit with vanishing
rescaled tilt consists of parallel hyperplanes, with bounded total
multiplicity. This description determines the case in which the
local sheeting argument can be reduced to a smaller multiplicity.
 \begin{lemma}\label{lspli}
Under the preceding assumptions, let $X_j$ be zeros in a fixed
compact interior region, and let $r_j\downarrow0$ satisfy
$\varepsilon_j/r_j\to0$. Suppose that, for every fixed $R>0$,
\begin{equation}\label{vtilt}
 r_j^{-n}\int_{B_{Rr_j}(X_j)}
       \varepsilon_j|D_xu_j|^2\,dX\longrightarrow0.
\end{equation}
Then, along a subsequence, the energy measures of
$U_j(Y)=u_j(X_j+r_jY)$, with parameter $\varepsilon_j/r_j$,
converge locally weakly to
\begin{equation}\label{xzqki}
 \mu_\infty=\sigma_0\sum_{\ell=1}^k
 m_\ell\mathcal H^n\lfloor\{Y_{n+1}=a_\ell\},
 \qquad m_\ell\in\mathbb N,
 \qquad \sum_{\ell=1}^k m_\ell\le Q.
\end{equation}
The heights $a_\ell$ are distinct, and one of them is zero.
\end{lemma}
\begin{proof}
The thickness of $U_j$ is $\varepsilon_j/r_j\to0$.
From \eqref{monoto}, we obtain local energy bounds on every fixed
rescaled ball. The compactness theorem \cite[Theorem~1]{HT}
and a diagonal subsequence yield a stationary integral varifold.
As in Proposition~\ref{fnstep},
\eqref{vtilt} makes its weight a product of $\mathcal L^n$ with
an atomic measure of masses in $\sigma_0\mathbb N$.

We pass to a subsequence for which $X_j\to X_*$. Hausdorff convergence of the
transition sets puts $X_*$ on $\{z=0\}$. For fixed $R$ and then
fixed small $s$, we compare $Rr_j$ with $s$ in \eqref{monoto}.
From convergence on the outer ball, using nearby-ball inclusions for
the moving center, we obtain
\begin{equation*}
 R^{-n}\mu_\infty(B_R)
 \le Q\sigma_0\omega_n+Cs.
\end{equation*}
We let $s\downarrow0$ and then $R\to\infty$. Any finite collection
of atoms has total multiplicity at most $Q$, so there are at most
$Q$ atoms altogether. Local Hausdorff convergence and $U_j(0)=0$
put one hyperplane through the origin. 
\end{proof}

If the hyperplane through the origin has multiplicity $m_0<Q$,
we choose a fixed product neighborhood meeting only this
limiting hyperplane. The restrictions of $U_j$ remain stable,
and their energy measures converge to this hyperplane with
multiplicity $m_0$. Suppose that the local sheeting theorem
holds for every multiplicity below $Q$. After a fixed
isotropic dilation, we can apply that theorem on a smaller
neighborhood. This dilation preserves graph slopes and
nonvanishing of the vertical derivative. The zero sets
therefore consist of graphs whose slopes tend to zero.
In particular, the angle between their tangent hyperplanes
at the origin and the horizontal hyperplane tends to zero.

\begin{remark}
A closed normal interval lying strictly between two hyperplanes in
\eqref{xzqki} is eventually zero-free over any fixed compact
lateral base, by Hausdorff convergence of the transition sets.
If $\delta_j=\varepsilon_j/r_j$, the pure-phase error on such
a rescaled product is $C\delta_j^{-p}e^{-c/\delta_j}=o(1)$ for
every fixed $p$.
\end{remark}

\begin{corollary}\label{morrey}
Under the hypotheses of Proposition~\ref{mwfloo}, let $r_j\to0$
satisfy, for some fixed $0<\gamma<1$,
\begin{equation}\label{schoi}
 E_jr_j^{-2+2\gamma}\to0,\qquad
 \varepsilon_j/r_j\to0.
\end{equation}
Then \eqref{vtilt} holds uniformly in interior centers.
In particular, if $\varepsilon_j^2=o(E_j)$, one may take
$r_j=\sqrt{E_j}$.
\end{corollary}
\begin{proof}
For each fixed $R>0$, \eqref{schoi} makes $Rr_j/2$ admissible
in Proposition~\ref{mwfloo}. Its normalized bound is
\begin{equation*}
 C_R E_jr_j^{-2+2\gamma}
       +C_R(\varepsilon_j/r_j)^2\longrightarrow0.
\end{equation*}
At $r_j=\sqrt{E_j}$, the first term is $C_RE_j^\gamma$.
\end{proof}

The other possibility is that the hyperplane through the marked
zero carries the full multiplicity. The next lemma shows that this
centered hyperplane persists at larger shrinking scales satisfying
\eqref{vtilt}. We will use this conclusion to compare the heights
at successive scales.

\begin{lemma}\label{satura}
Suppose that the limit at scales $r_j$ in Lemma~\ref{lspli} is
$Q\sigma_0\mathcal H^n\lfloor\{Y_{n+1}=0\}$.
Let $s_j\to0$, $s_j/r_j\to\infty$, and assume \eqref{vtilt}
also holds with $s_j$ in place of $r_j$. Every subsequential limit
at scales $s_j$ is the multiplicity-$Q$ horizontal hyperplane through
the center.
\end{lemma}
\begin{proof}
For fixed $R,t>0$, we have $Rr_j<ts_j$ for large $j$, and
from \eqref{monoto} we obtain
\begin{equation*}
 (ts_j)^{-n}\mathcal E_{\varepsilon_j}(u_j;B_{ts_j}(X_j))
 \ge (Rr_j)^{-n}\mathcal E_{\varepsilon_j}(u_j;B_{Rr_j}(X_j))-Cts_j.
\end{equation*}
The limiting ratio is at least $Q\sigma_0\omega_n$.
Lemma~\ref{lspli} establishes the opposite bound and expresses the
limit as finitely many parallel hyperplanes. Thus its normalized ball
mass is $Q\sigma_0\omega_n$ for every $t$. A hyperplane at height $a$
contributes
\begin{equation*}
 m\sigma_0\omega_n(1-a^2/t^2)_+^{n/2}.
\end{equation*}
We let $t\to\infty$ and obtain total multiplicity $Q$.
The equality for every $t$ then forces all heights to be zero.
\end{proof}

For the remainder of this subsection, $\delta_j$ denotes the
parameter in the equation satisfied by the solutions under
consideration:
\[
 \delta_j^2\Delta u_j=W'(u_j).
\]
When we obtain these solutions from the original sequence
with parameters $\varepsilon_j$ by rescaling a neighborhood
of radius $r_j$ to unit size, we have
$\delta_j=\varepsilon_j/r_j$. Thus $\delta_j$ measures the
transition thickness relative to the radius of that
neighborhood. We continue to use $u_j$ for the rescaled
solutions.

We now pass from intermediate scales to the transition thickness.
The next lemma converts Proposition~\ref{mwfloo} into growth bounds
for the resulting entire limits. These are the bounds used in the
rigidity argument of Section~\ref{rigidpart}.

\begin{lemma}\label{bterm}
Let $u_j$ be bounded stable solutions on a fixed larger cylinder,
with thickness $\delta_j\to0$, uniformly bounded diffuse energy,
a marked zero in a fixed compact interior subdomain, and excess
$E_j\to0$ relative to the normal direction. Suppose that for some
$0<\gamma<1$,
\begin{equation}\label{tentr}
 E_j\delta_j^{-2+2\gamma}\longrightarrow0.
\end{equation}
Every microscopic pointed limit $U$ is a bounded stable entire
solution satisfying, with constants uniform in $Y$,
\begin{equation}\label{tgbroa}
 \int_{B_R(Y)}|D_xU|^2\,dX\le C R^{n-2},\qquad
 \int_{B_R(Y)}\left(\tfrac12|DU|^2+W(U)\right)\,dX\le C R^n
 \quad\text{for }R\ge R_*.
\end{equation}
If the rescalings come from the original multiplicity-$Q$ sequence,
the density at infinity is at most $Q$.
\end{lemma}
\begin{proof}
The critical tangential growth in the lemma has appeared in Wang
\cite[Lemmas~11.4--11.5]{Wang17}, obtained there under a near-unit
density hypothesis. Here it follows from \eqref{mfloo} by rescaling.

From condition~\eqref{tentr}, we obtain
$E_j^{1/(2-2\gamma)}=o(\delta_j)$. Hence
Proposition~\ref{mwfloo}, on a fixed interior subdomain, is
available at all radii $R\delta_j$ with $R\ge R_*$. For
$V_j(Y)=u_j(X_j+\delta_jY)$ we have
\begin{equation*}
 \int_{B_R(Y)}|D_xV_j|^2\,dX
 \le C E_j\delta_j^{-2+2\gamma}R^{n-2+2\gamma}
                         +C R^{n-2}.
\end{equation*}
For each fixed $Y,R$, the first term tends to zero. Interior
elliptic compactness and stability yield the entire limit. The
local full-energy ratio bound implies the second estimate in
\eqref{tgbroa}. Every fixed microscopic center eventually lies
in the same parent interior subdomain, so the constants are
uniform over all centers of the limit.

For the last assertion, we compare a fixed microscopic ball with
a fixed outer ball in the original coordinates using
\eqref{monoto}. We pass to the microscopic limit and then let the
outer radius tend to zero, as in Lemma~\ref{lspli}. The resulting
energy ratio is at most $Q\sigma_0\omega_n$ at every radius.
\end{proof}

We return to the common-normal condition in
Proposition~\ref{hcomp}. When the marked hyperplane carries the full
multiplicity at the height scale, the selected heights are confined
to a thin slab at every larger scale. The next proposition uses
this confinement and the ordering of the heights to show that their
normalized gaps have constant limits, so their gradients agree.

\begin{proposition}\label{hcgrad}
Let $u_j$ be bounded stable solutions on a fixed larger cylinder,
with thickness $\delta_j\to0$, uniformly bounded diffuse energy,
and
\begin{equation*}
 \delta_j|Du_j|^2\,dX\rightharpoonup
 Q\sigma_0\mathcal H^n\lfloor\{z=0\}.
\end{equation*}
Let $E_j\to0$ bound their tangential excess on the larger cylinder,
and put $\lambda_j=\sqrt{E_j}$. Assume
\begin{equation}\label{hregi}
 \delta_j^2=o(E_j),\qquad
 \frac{\delta_j}{\lambda_j}\log(2/\lambda_j)\longrightarrow0.
\end{equation}
Let $X_j=(0,z_j)$ be interior zeros, after horizontal translation.
Suppose that the pointed limit at scale $\lambda_j$ is exactly
the multiplicity-$Q$ horizontal hyperplane through $X_j$.
Then the harmonic heights of Proposition~\ref{hcomp} have the same
gradient.
\end{proposition}

Condition~\eqref{hregi} holds, in particular, if $E_j\ge\delta_j$.
In that case the height-scale limit either has smaller multiplicity
at the marked hyperplane or provides the common gradient required for
the single-normal conclusion of Proposition~\ref{hcomp}.
\begin{proof}
We fix $\gamma\in(0,1)$ and apply Proposition~\ref{mwfloo}.

The argument follows the approach of Florit-Simon--Serra
\cite[Sections~7.4--7.5]{FSS25}, who use small-scale proximity and
ordering of the layers to compare their affine approximations.
Here we use the heights in \eqref{aheig}, with the same root selection
and weights at every scale, and estimate their nonnegative gaps in $BV$.

Fix a large numerical constant $A$ and a small $s_0>0$ so that
$B_{As_0}(X_j)$ lies in the parent domain. Uniformly for
$\lambda_j\le s\le s_0$, we obtain from Proposition~\ref{mwfloo}
\begin{equation*}
 s^{-n}\int_{B_{As}(X_j)}\delta_j|D_xu_j|^2\,dX
 \le C_A E_j^\gamma+C_A\delta_j^2/E_j\longrightarrow0.
\end{equation*}
The rescaled thickness also tends to zero uniformly. Every sequence
of scales tending to zero therefore has a parallel-hyperplane limit.
If $s/\lambda_j$ stays bounded, the assumed height-scale limit
forces it to be the centered $Q$-hyperplane. If $s/\lambda_j\to\infty$,
Lemma~\ref{satura} implies the same conclusion. At scales bounded
away from zero, we use the original planar convergence.
Hausdorff convergence of the transition sets and a contradiction
argument yield $\eta_j\to0$ such that
\begin{equation}\label{ucslab}
 \{u_j\in I\}\cap B_{As}(X_j)
        \subset\{|z-z_j|\le\eta_j s\}
 \qquad(\lambda_j\le s\le s_0),
\end{equation}
for a fixed compact interval $I\Subset(-1,1)$ enlarged once to
contain the level cutoffs.

We choose $r_0\ll s_0$ so that all enlarged balls below remain in
the parent domain. Original Hausdorff convergence puts every root
in the original fixed normal interval over the relevant base
within $|z-z_j|<s_0/2$ for large $j$. For
$\lambda_j\le r\le r_0$, $|x|\le Cr$, and any such root, we set
$d=|z-z_j|$ and $s=\max\{Cr,d\}$. The fixed constants can be
chosen so that the root lies in $B_{As}(X_j)$ and
$\lambda_j\le s\le s_0$. If $d\ge Cr$, \eqref{ucslab}
would yield $d\le\eta_jd$. Hence $d<Cr$, and
\begin{equation}\label{arconf}
 |z-z_j|\le C\eta_j r.
\end{equation}
Thus every root in the original fixed normal interval lies in the
local interval of length comparable to $r$. The two intervals
therefore contain the same roots over the smaller base. We apply
\eqref{rroot} with the same $t_0$ and the same values of $t$ in both
intervals, so the selected roots and their rank labels agree at every
dyadic scale.

We use the fixed weights $\rho(t)$ and $\chi(s)$ from
\eqref{aheig}, and set
\begin{equation*}
 G_j=\bar f_{Q,j}-\bar f_{1,j}\ge0.
\end{equation*}
On empty fibers, we assign $z_j$ to every index. This leaves $G_j$
unchanged, and all heights relative to $z_j$ satisfy
\eqref{arconf}. We set
\begin{equation*}
 \mathcal E_j(r)=C\bigl(E_jr^{-2+2\gamma}
                               +\delta_j^2r^{-2}\bigr),
\end{equation*}
with $C$ bounding the normalized tilt-excess on the fixed
enlargements used to prove Lemma~\ref{tboun}. This bound tends to zero
uniformly for $r\in[\lambda_j,r_0]$.

The quantitative boundary estimates in the proof of
Lemma~\ref{tboun} use Proposition~\ref{cover}, Lemma~\ref{rlemm},
and the local full-energy bound. Thus they apply here with small
$\mathcal E_j(r)$, small $\delta_j/r$, and the thin slab
\eqref{arconf}. Together with the coarea bound for the ordinary
slope density, they yield
\begin{equation*}
 \begin{gathered}
 DG_j=v_{j,r}\,dx+\xi_{j,r},\qquad
 \int_{B_r}|v_{j,r}|^2\,dx\le C r^n\mathcal E_j(r),\\
 \|\xi_{j,r}\|_{\mathcal M(B_r)}
       \le C r^n\mathcal E_j(r)^{1/2+1/[4(n-1)]}.
 \end{gathered}
\end{equation*}
To check the last scaling, we dilate lengths by $r^{-1}$ and heights
by $r^{-1}$. The ratio in \eqref{rroot} is invariant under this
dilation, so the selected roots are unchanged. Equation~\eqref{hberro}
then bounds the variation in the
unit base ball, since \eqref{arconf} bounds the rescaled heights
uniformly. Returning to the original variables multiplies this
variation by $r^{n-1}$ from the base and by $r$ from the height jump,
which yields the factor $r^n$. Both estimates concern restrictions
of the same function $G_j$, with the fixed average against
$\rho(t)\,dt$ at every scale.
The Cauchy--Schwarz inequality and $\mathcal E_j(r)\to0$ yield
\begin{equation*}
 |DG_j|(B_r)\le C r^n\sqrt{\mathcal E_j(r)}.
\end{equation*}

From the BV Poincar\'e inequality \cite[Chapter~3]{AFP}, we obtain
\begin{equation*}
 \left|\fint_{B_r}G_j\,dx-\fint_{B_{r/2}}G_j\,dx\right|
 \le C r^{1-n}|DG_j|(B_r)
 \le C\bigl(\lambda_j r^\gamma+\delta_j\bigr).
\end{equation*}
For fixed $0<r\le r_0$, we telescope down to a radius in
$[\lambda_j,2\lambda_j]$, where
$0\le G_j\le C\eta_j\lambda_j$. The geometric sum and the
number of dyadic scales yield
\begin{equation}\label{gmboun}
 0\le\fint_{B_r}\frac{G_j}{\lambda_j}\,dx
 \le Cr^\gamma+C\eta_j+
 C\frac{\delta_j}{\lambda_j}
                  \log\frac{2r}{\lambda_j}.
\end{equation}

By \eqref{hconv}, we have
\begin{equation*}
 \frac{\bar f_{\alpha,j}-c_{\alpha,j}}{\lambda_j}
               \longrightarrow h_\alpha\quad\text{in }L^1_{\rm loc}.
\end{equation*}
At fixed $r$, \eqref{gmboun} bounds
$(c_{Q,j}-c_{1,j})/\lambda_j$. We pass to a subsequence on which
it tends to $c$. Then
\begin{equation*}
 G_j/\lambda_j\longrightarrow G=c+h_Q-h_1
       \quad\text{in }L^1_{\rm loc},\qquad
 G\ge0,\quad\Delta G=0.
\end{equation*}
Condition~\eqref{hregi} removes the last two errors in
\eqref{gmboun}. From the mean-value property, we obtain
$0\le G(0)\le Cr^\gamma$ for every small $r$.
Thus $G(0)=0$, and the strong minimum principle implies $G\equiv0$.
Since
$0\le\bar f_{\alpha,j}-\bar f_{1,j}\le G_j$,
every normalized intermediate gap tends to zero in $L^1_{\rm loc}$.
Equation~\eqref{ahderi} and the slope convergence in \eqref{hconv}
then yield $Dh_\alpha=Dh_1$ for every $\alpha$.

Finally, $E_j\ge\delta_j$ implies
$\delta_j/\lambda_j\le\sqrt{\delta_j}$ and
$\log(2/\lambda_j)\le\log(2/\sqrt{\delta_j})$, proving
\eqref{hregi} in this case.
\end{proof}

We can now combine the preceding estimates into the alternative
used at each step of the sheeting argument. We seek a smaller
cylinder on which the excess decreases after a small change of
normal. The next proposition proves this decay when the harmonic
heights have a common gradient. The two remaining outcomes are a
limit that permits a reduction in multiplicity and an entire
solution satisfying the growth bounds of Lemma~\ref{bterm}.

We consider bounded stable rescaled solutions $u_j$ on the
fixed nested cylinders used above, with $u_j(0)=0$ and a
uniform bound on the full energy. Their thicknesses
$\delta_j$ and excesses $E_j$, measured relative to unit
normals $e_j$, tend to zero. In coordinates whose last axis
is $e_j$, we assume convergence to the multiplicity-$Q$
horizontal hyperplane through the origin on the domain of
Proposition~\ref{hcomp}. We also assume that every energy
measure $\mu_\infty$ obtained at intermediate scales
satisfies the inherited density bound
\begin{equation*}
 \mu_\infty(B_R)\le Q\sigma_0\omega_nR^n
 \qquad\text{for every }R>0.
\end{equation*}
When the cylinder radius is multiplied by $\theta$ and
the normal is changed to $e_j^+$, we denote by $E_j^+$
the excess on the corresponding rotated outer cylinder,
after rescaling the new radius to one.

\begin{proposition}\label{cont}
Under the preceding assumptions, for every $a\in(0,1/4)$
there are $\theta\in(0,1)$ and $C_0<\infty$, depending only
on $a$ and the fixed data, with the following properties
along a subsequence.

If $E_j\le\delta_j$, every microscopic pointed limit
satisfies \eqref{tgbroa}, uniformly in the center.

If $E_j>\delta_j$, either the hyperplane through the origin
in the limit at scale $\sqrt{E_j}$ has multiplicity
less than $Q$, permitting the restriction in
Lemma~\ref{lspli}, or there are unit normals $e_j^+$ such that
\begin{equation*}
 |e_j^+-e_j|\le C_0\sqrt{E_j},
 \qquad E_j^+\le aE_j.
\end{equation*}
\end{proposition}

The same conclusions and constants apply when the
iteration depths depend on $j$, provided the hypotheses
and uniform bounds on the larger cylinders remain valid.
 \begin{proof}
We work in coordinates whose last axis is $e_j$ and fix
$\gamma=3/4$. After passing to a subsequence, we may assume
that either $E_j\le\delta_j$ for every $j$ or
$E_j>\delta_j$ for every $j$.

Suppose first that $E_j\le\delta_j$. Then
\[
 E_j\delta_j^{-2+2\gamma}
 =E_j\delta_j^{-1/2}
 \le\delta_j^{1/2}\longrightarrow0.
\]
Lemma~\ref{bterm} therefore applies. Every microscopic
pointed limit satisfies the two growth bounds in
\eqref{tgbroa}, with constants uniform in the center.

We now suppose that $E_j>\delta_j$ and set
$\lambda_j=\sqrt{E_j}$. For all sufficiently large $j$,
we have
\[
 \frac{\delta_j^2}{E_j}\le\delta_j\longrightarrow0,
 \qquad
 \frac{\delta_j}{\lambda_j}\log\frac{2}{\lambda_j}
 \le\sqrt{\delta_j}\log\frac{2}{\sqrt{\delta_j}}
 \longrightarrow0.
\]
In particular, $\delta_j/\lambda_j\to0$. By
Corollary~\ref{morrey}, for every fixed $R>0$,
\[
 \lambda_j^{-n}
 \int_{B_{R\lambda_j}}
       \delta_j|D_xu_j|^2\,dX
 \le C_R\left(E_j^\gamma+\frac{\delta_j^2}{E_j}\right)
 \longrightarrow0.
\]
By Lemma~\ref{lspli}, the limit at scale $\lambda_j$ consists
of parallel hyperplanes of total multiplicity at most $Q$,
with one hyperplane through the origin. If that hyperplane
has multiplicity less than $Q$, we restrict to a product
neighborhood meeting only that hyperplane, as described
after Lemma~\ref{lspli}. This proves the alternative with
smaller multiplicity.

It remains to consider the case in which the hyperplane
through the origin has multiplicity $Q$. Since the total
multiplicity is at most $Q$, the limit consists entirely
of this hyperplane. After passing to a further subsequence,
Proposition~\ref{hcomp} provides harmonic heights
$h_1,\ldots,h_Q$. The two thickness estimates above verify
the remaining hypotheses of Proposition~\ref{hcgrad},
so these heights have the same gradient.

We can consequently remove the same linear part from
every harmonic height. Their uniform Dirichlet energy bound
and the interior estimates for harmonic functions imply
\[
 \sigma_0\sum_{\alpha=1}^Q
 \int_{B_\rho^n}
       |Dh_\alpha(x)-Dh_1(0)|^2\,dx
 \le C\rho^{n+2}
\]
for every fixed sufficiently small $\rho>0$.
The single-normal conclusion of Proposition~\ref{hcomp}
therefore provides unit vectors $e_j^+$ satisfying
\[
 |e_j^+-e_j|\le C_0\sqrt{E_j},
 \qquad
 \limsup_{j\to\infty}\frac{1}{E_j}
 \int_{B_\rho^n\times(-4,4)}
 \delta_j\left(|Du_j|^2-|Du_j\cdot e_j^+|^2\right)\,dX
 \le C\rho^{n+2}.
\]
The constants $C$ and $C_0$ depend only on the fixed data
and are independent of $\rho$.

We set $\theta=\rho/128$. The rotated outer cylinder
of radius $64\theta$ is contained in $B_\rho^{n+1}$,
since its distance from the origin is at most
$64\sqrt{2}\theta<\rho$. This ball is contained in
$B_\rho^n\times(-4,4)$. Rescaling the cylinder by
$\theta^{-1}$ multiplies its excess integral by
$\theta^{-n}$. Hence
\[
 \limsup_{j\to\infty}\frac{E_j^+}{E_j}
 \le C\theta^{-n}\rho^{n+2}
 =C128^n\rho^2.
\]
We first choose $\rho$ sufficiently small that
$C128^n\rho^2<a/2$, and then take $j$ sufficiently large.
We obtain $E_j^+\le aE_j$, together with the required
bound on the change of normal.
\end{proof}

This completes the estimates for the finite iteration.
The decay alternative will be used in the local sheeting argument,
while the microscopic alternative leads to the rigidity problem
considered in Section~\ref{rigidpart}.

\section{Rigidity under a tangential energy bound}\label{rigidpart}

We classify the entire solutions obtained when the excess iteration
reaches the transition thickness. Besides the full energy bound,
these solutions satisfy a bound for the horizontal gradient energy
that is uniform in the center of the ball. This additional information
will force a vertical one-dimensional transition throughout the
dimension range needed for local sheeting.

We first construct ordered layers outside controlled neighborhoods
of exceptional zeros. The equations and stability of these layers
yield scalar inequalities, which we extend across the omitted regions.
Spherical averaging and radial Hardy tests then exclude several layers.
The estimates include the different exponential rates that arise when
$W''(-1)$ and $W''(1)$ differ.

\begin{proposition}\label{crit}
Let $2\le n\le9$ and let $U\in C^2(\mathbb R^{n+1};[-1,1])$ be a stable
entire solution of $\Delta U=W'(U)$ with $U(0)=0$. Suppose that there
are constants $A,K>0$ and $R_*\ge1$ such that
\begin{equation*}
 \begin{aligned}
 \int_{B_R(Y)}|D_xU|^2\,dX&\le AR^{n-2}
       &&\text{for }Y\in\mathbb R^{n+1},\ R\ge R_*,\\
 \int_{B_R}\left(\frac12|DU|^2+W(U)\right)\,dX&\le KR^n
       &&\text{for }R\ge R_*.
 \end{aligned}
\end{equation*}
Then $U(x,z)=g(z)$ or $g(-z)$.
\end{proposition}

\subsection{The exceptional zeros}
\label{exzeros}

The objective of this subsection is to recover the geometry at infinity
needed for the analysis of interacting layers. We determine the planar
blowdown and bound the heights of all zeros. The horizontal stability
inequality then controls the number of balls covering the zeros where
the solution fails to approximate a vertical transition. Removing the
enlarged projected balls leaves an unbounded component on which the next subsection
will construct ordered layers.

We first determine the planar limit. The following lemma holds for every
$n\ge2$ and uses only the two energy bounds at the origin.
\begin{lemma}\label{plane}
Let $U\in C^2(\mathbb R^{n+1};[-1,1])$ be a stable entire solution
of $\Delta U=W'(U)$,
with $U(0)=0$. Suppose that for all sufficiently large $R$,
\begin{equation*}
 \int_{B_R}|D_xU|^2\,dX\le AR^{n-2},\qquad
 \int_{B_R}\left(\frac12|DU|^2+W(U)\right)\,dX\le KR^n.
\end{equation*}
There is a fixed integer $Q\ge1$ such that
\begin{equation}\label{bdlim}
 R\left(\frac12|DU(RX)|^2+W(U(RX))\right)\,dX
 \rightharpoonup Q\sigma_0\mathcal H^n\lfloor\{z=0\}
 \qquad\text{as }R\to\infty.
\end{equation}
The measures $R|DU(RX)|^2\,dX$ have the same limit. For every $\eta>0$
there is $C_\eta<\infty$ such that
\begin{equation}\label{cone}
 U(x,z)=0\quad\Longrightarrow\quad |z|\le\eta|x|+C_\eta.
\end{equation}
\end{lemma}
\begin{proof}
By Modica's estimate and monotonicity \cite{Modica85,Modica88}, the
energy ratio has a finite limit. This limit is positive because
$U(0)=0$ and the uniform gradient bound yield positive energy in a
fixed ball. For any $R_j\to\infty$, we apply Hutchinson--Tonegawa
\cite[Theorem~1]{HT} to $U(R_j\,\cdot)$. The horizontal gradient
mass on a fixed ball is $O(R_j^{-2})$. The stationary-measure
argument in the proof of Lemma~\ref{lspli} therefore shows that the limit is
a union of horizontal hyperplanes with integer multiplicities. Its centered
mass ratio is constant in the radius and equals the energy density
of $U$ at infinity. An offset hyperplane would make that ratio strictly
increase once the radius passed its height. Thus every hyperplane lies
at height zero, and the limiting density fixes one multiplicity for
all sequences. Equipartition implies the same convergence for
$R|DU(RX)|^2\,dX$.

If zeros $P_j=(x_j,z_j)$ satisfy $|P_j|\to\infty$ and
$|z_j|\ge a|P_j|$ for some $a>0$, the local gradient bound and
Modica monotonicity yield a fixed positive normalized energy on
$B_{a|P_j|/2}(P_j)$. After rescaling by $|P_j|$, this ball stays
away from the limiting hyperplane, contradicting \eqref{bdlim}. Hence
$|z|=o(|(x,z)|)$ along unbounded sequences of zeros. From this we obtain
\eqref{cone}.
\end{proof}

Fix $n\ge5$ and consider bounded stable
entire solutions of $\Delta U=W'(U)$ with fixed constants in
\begin{equation}\label{mwm}
\int_{B_R(Y)}|D_xU|^2\,dX\le A R^{n-2}
 \qquad\text{for }Y\in\mathbb R^{n+1},\ R\ge R_*.
\end{equation}

We apply the horizontal Sternberg--Zumbrun inequality \eqref{szum},
with $\varepsilon=1$, to a cutoff equal to one on $B_R(Y)$ and
supported in $B_{2R}(Y)$. This is the directional inequality used
by Tonegawa \cite[Proposition~4]{Ton05}. Combining it with \eqref{mwm}, we obtain
\begin{equation}\label{subv}
\int_{B_R(Y)}\left(|D(D_xU)|^2-|D|D_xU||^2\right)\,dX\le C A R^{n-4}
 \qquad\text{for }R\ge R_*,
\end{equation}
uniformly in $Y$.

Now fix $M>0$ and $\delta>0$. We call a zero $P$ exceptional if
\begin{equation}\label{pcgbpz}
 \min_{\pm}\|U(P+\cdot)-g(\pm z)\|_{C^2(B_{2M})}\ge\delta.
\end{equation}
Thus, after $P$ is moved to the origin, the solution has
$C^2(B_{2M})$ distance at least $\delta$ from each of $g(z)$ and
$g(-z)$. The next lemma shows that every exceptional zero contributes
a fixed positive amount to the integral in \eqref{subv} on a ball
of fixed radius.

\begin{lemma}\label{cost}
For fixed $M,\delta>0$, there are $L\ge1$ and $c>0$, depending
only on $n,W,A,R_*,M,\delta$, such that every exceptional zero $P$
satisfies
\begin{equation}\label{cvm}
\int_{B_L(P)}\left(|D(D_xU)|^2-|D|D_xU||^2\right)\,dX\ge c.
\end{equation}
\end{lemma}
\begin{proof}
Suppose that the conclusion fails for fixed $M,\delta$.
After translating exceptional zeros to the origin, we obtain
solutions $U_j$ satisfying \eqref{mwm}, $U_j(0)=0$, and
\begin{equation*}
 \int_{B_j}\left(|D(D_xU_j)|^2-|D|D_xU_j||^2\right)\,dX<1/j.
\end{equation*}
After we pass to a subsequence, $U_j$ converges to $U$ in
$C^2_{\rm loc}(\mathbb R^{n+1})$. The limit is stable, satisfies
$U(0)=0$, and inherits \eqref{mwm} with the same constants.
On every relatively compact subset of $\{|D_xU|>0\}$, the integrands
converge and are nonnegative. Their integrals tend to zero, so
$|D(D_xU)|^2-|D|D_xU||^2=0$ there. By Lemma~\ref{kzero},
$U=g(z)$ or $g(-z)$.
However, convergence in $C^2(B_{2M})$ preserves the lower bound
$\delta$ in \eqref{pcgbpz}, whereas either limiting solution has
distance zero from one of the two heteroclinics. This is a contradiction.
\end{proof}

The estimate \eqref{cvm} also holds at zeros satisfying
\begin{equation}\label{nxgpf}
|U_z(P)|\le\tfrac12g'(0)
  \quad\hbox{or}\quad |D_xU(P)|\ge|U_z(P)|.
\end{equation}
Indeed, we take $M=1$ and fix $\delta>0$ sufficiently small in terms
of $n,W$. Closeness to either vertical heteroclinic then implies
$|U_z(P)|>g'(0)/2$ and $|D_xU(P)|<|U_z(P)|$.
Thus every zero satisfying \eqref{nxgpf} is exceptional for these
fixed parameters, and Lemma~\ref{cost} applies with $L,c$ depending
only on $n,W,A,R_*$.

The following counting argument adapts the division into good and
bad cubes used by Liu--Wang--Wei--Wu \cite[Section~2]{LWWW24}.
Here the horizontal integral in \eqref{subv} controls the number of
balls centered at exceptional zeros.
We choose a maximal $2L$-separated set of exceptional zeros. The balls $B_L$
are disjoint, while their doubled balls cover the exceptional set. Equations
\eqref{subv}--\eqref{cvm} yield, for $R\gg L+R_*$,
\begin{equation*}
\#\{\text{chosen centers in }B_R\}\le C R^{n-4}.
\end{equation*}

To project the entire exceptional set onto a horizontal annulus, we
assume the confinement
\begin{equation}\label{ehltc}
U(x,z)=0\quad\Longrightarrow\quad |z|\le C(1+|x|).
\end{equation}
Under the centered full-energy bound of Lemma~\ref{plane},
its confinement conclusion implies \eqref{ehltc}.

Under \eqref{ehltc}, the horizontal projections of the covering balls form
fixed-radius balls with
\begin{equation}\label{ugiskt}
\#\{\text{projected centers in }B_R^n\}\le C R^{n-4}.
\end{equation}
They are locally finite. We choose smooth functions $\rho_i$, equal to
one on the projected $2L$ balls, supported on their $4L$ enlargements,
with $0\le\rho_i\le1$ and $|D\rho_i|\le C/L$. The Lipschitz cutoff
$\chi=1-\sup_i\rho_i$ vanishes on the projected exceptional set.
Since $|D\sup_i\rho_i|^2\le\sum_i|D\rho_i|^2$ almost everywhere,
on each large horizontal dyadic annulus $A_R=\{R<|x|<2R\}$,
\begin{equation}\label{xvd}
\int_{A_R}|x|^{-(n-2)}|D\chi|^2\,dx\le C R^{-2},
 \qquad
 \int_{A_R}|x|^{-n}\mathbf1_{\{\chi\ne1\}}\,dx\le C R^{-4}.
\end{equation}
A fixed enlargement of the annulus changes only the constants.

Consequently, for radial Hardy tests
\begin{equation*}
 \zeta(x)=|x|^{-(n-2)/2}\psi(\log|x|),\qquad
 |\psi|+|\psi'|\le1,\qquad \operatorname{supp}\zeta\subset\{|x|>R_1\},
\end{equation*}
the additional gradient cost of replacing $\zeta$ by $\chi\zeta$
is $O(R_1^{-2})$, uniformly in the logarithmic length of the support.
Indeed, the additional square term is bounded by the dyadic sum of
$CR^{-2}$. Since $|D\zeta|^2\le C|x|^{-n}$, Cauchy--Schwarz and
\eqref{xvd} bound the cross term on each annulus by
$C(R^{-2}R^{-4})^{1/2}=CR^{-3}$. By summation, we obtain the assertion.

\label{vnmu}

We need a connected region over which the zeros can be labeled by
their heights. The next lemma bounds the volume left outside one
unbounded component of the complement of the removed balls.
This will allow us to prove that the number of zero graphs equals the
multiplicity of the planar energy limit.

\begin{lemma}\label{geomet}
Let $n\ge5$. Suppose a locally finite set of centers $p_i\in\mathbb R^n$
satisfies
\begin{equation}\label{govvli}
\#\{i\mid |p_i|\le R\}\le C(1+R)^{n-4}.
\end{equation}
Let $H$ be the union of closed balls about these centers whose radii
$\rho_i$ satisfy
\begin{equation}\label{mli}
0<\rho_0\le\rho_i\le C_0\log(2+|p_i|).
\end{equation}
Then $\mathbb R^n\setminus H$ has an unbounded connected component
$\Omega_\infty$ such that, with
\begin{equation*}
 \kappa=\frac{2(n-1)}{n-2},\qquad 2<\kappa<3,
\end{equation*}
for almost every sufficiently large $R$,
\begin{equation}\label{nffq}
\frac{\mathcal H^{n-1}(S_R\setminus\Omega_\infty)}
      {\mathcal H^{n-1}(S_R)}
 \le C R^{-\kappa}(\log R)^{n-1}.
\end{equation}
Consequently
\begin{equation}\label{uyctu}
|(B_{2R}\setminus B_R)\setminus\Omega_\infty|
 \le C R^{n-\kappa}(\log R)^{n-1}.
\end{equation}
Constants depend on $n$ and the packing and radius bounds.
\end{lemma}
\begin{proof}
Only centers with $R/2<|p_i|<2R$ can have balls meeting $S_R$, once
$R$ is large. The balls about the finitely many remaining small centers lie inside $B_R$. Each spherical cap has area at most $C(\log R)^{n-1}$ and
boundary perimeter at most $C(\log R)^{n-2}$. After rescaling
$S_R$ to the unit sphere, the removed relative area and total
perimeter are bounded by
\begin{equation}\label{gmsyvj}
C R^{-3}(\log R)^{n-1},
 \qquad C R^{-2}(\log R)^{n-2},
\end{equation}
respectively. The sum of the perimeters of the components of the
open complement is bounded by the same total perimeter up to a fixed
factor. These estimates hold for almost every radius.

The classical spherical isoperimetric inequality, recalled in
B\"ogelein--Duzaar--Fusco \cite{BDF17}, bounds the volume $v$ of
each component of relative volume at most $1/2$ by
$C p^{(n-1)/(n-2)}$, where $p$ is its normalized perimeter.
Since the exponent exceeds one, summing this estimate bounds the
total volume of all such components by
\begin{equation*}
 C\left(\sum p\right)^{(n-1)/(n-2)}
 \le C R^{-\kappa}(\log R)^{n-1}.
\end{equation*}
As the total area of the complement tends to one, there must be a unique component
$G_R$ of relative volume greater than $1/2$. It occupies all but the
error in \eqref{nffq}, together with the smaller removed-area term in \eqref{gmsyvj}.

The directions of rays from the origin that meet a ball centered at
$p_i$, $|p_i|\gg1$, have spherical area at most
$C(\rho_i/|p_i|)^{n-1}$. We sum over dyadic annuli using \eqref{govvli} and obtain
\begin{equation}\label{krn}
|\{\theta\mid \{r\theta\mid r\ge R\}\cap H\ne\varnothing\}|
 \le C R^{-3}(\log R)^{n-1}.
\end{equation}
We take any two sufficiently large radii $R_2\ge R_1$ for which these
estimates hold.
The angular images of $G_{R_1}$ and $G_{R_2}$ both have measure
close to the full sphere. Their intersection contains a direction
outside the exceptional set in \eqref{krn}. The segment of that ray
between the two spheres avoids $H$, and joins the two largest
components. Thus all $G_R$ lie in one connected component
$\Omega_\infty$. This component is unbounded. Equation \eqref{nffq} follows,
and by radial integration we obtain \eqref{uyctu}.
\end{proof}

Fixed-radius balls satisfy the lemma after enlarging $C_0$.
For any two ball families satisfying \eqref{govvli}--\eqref{mli}, the
components constructed in the lemma intersect, since their traces on
almost every sufficiently large common sphere both have relative
measure tending to one.

\subsection{Ordered zero graphs and the choice of layer positions}
\label{lyeacx}\label{ordgraphs}

In this subsection, we pass from the geometry of the zero set to the layer
positions used in the interaction equations. We order the zeros on the
unbounded component from the preceding subsection and prove that their
number equals the blowdown multiplicity. We then adjust these heights
so that the error in the sum of heteroclinic profiles is orthogonal to
the derivatives of those profiles. Proposition~\ref{orth} controls
the squared slopes by the horizontal gradient energy, allowing the
next subsection to estimate the interaction errors in integral form.

Let $U\in C^2(\mathbb R^{n+1};[-1,1])$ be an entire stable double-well solution
satisfying \eqref{mwm}. Assume the all-scale planar full-energy blowdown
\begin{equation}\label{dofz}
R\left(\frac12|DU(RX)|^2+W(U(RX))\right)\,dX
 \rightharpoonup \theta\sigma_0\mathcal H^n\llcorner\{z=0\},
 \qquad 0<\theta<\infty.
\end{equation}
Here $e(U)=|DU|^2/2+W(U)$. The mass of the measure on the left
on a set $A$ is $R^{-n}\int_{RA}e(U)\,dX$. 

We apply the clearing-out argument in the proof of Lemma~\ref{plane}
to \eqref{dofz} and obtain
\begin{equation}\label{izhux}
U(x,z)=0\ \Longrightarrow\ |z|\le C(1+|x|).
\end{equation}
We fix $M\ge1$ and then choose $\delta>0$ sufficiently small.
At every zero outside the exceptional set \eqref{pcgbpz}, closeness to
$g(z)$ or $g(-z)$ implies
\begin{equation}\label{kdi}
|U_z|\ge\tfrac12g'(0),\qquad |D_xU|\le|U_z|.
\end{equation}
Lemma~\ref{cost} and \eqref{ugiskt}, with \eqref{izhux}, cover the
horizontal projections of all exceptional zeros by balls
$B_{2L}(p_i)$ satisfying \eqref{govvli}. We enlarge them to
\begin{equation}\label{hrpy}
B_{\rho_i}(p_i),\qquad
 \rho_i=2L+\lambda\log(2+|p_i|),
\end{equation}
where $\lambda$ is a fixed sufficiently large number chosen below.
Let $\Omega_\infty(M,\delta)$ be the component from
Lemma~\ref{geomet}.

Over each connected component of the complement, the zero-set
projection is a local diffeomorphism by \eqref{kdi}. It is proper
by \eqref{izhux}. Its fibers are compact and discrete, and thus
finite. By ordering their real heights, we obtain global smooth labels of
constant number, including zero for an empty fiber. Their slopes
are at most one by \eqref{kdi}.

We denote this number by $k(M,\delta)$ on $\Omega_\infty(M,\delta)$.
All these integers equal one integer $k$, since any two of the
components intersect and both count the complete zero fiber there.
The constants may depend on $M,\delta$, which are fixed while the
spatial radius tends to infinity.

The next lemma controls how a solution approaches a well away from
its zero set. Its exponential energy bound will allow us to prove that
the number of ordered zero graphs equals the multiplicity of the planar
energy limit. For a proof of the convergence, we refer to
Hamel--Liu--Sicbaldi--Wang--Wei \cite[Proposition~7]{HLSWW21}.

\begin{lemma}\label{wentr}
A bounded entire solution $v\in C^2(\mathbb R^{n+1};[0,1])$ of
$\Delta v=W'(v)$ is identically zero or identically one.
For every $b_0\in(0,1)$ there is $R_0<\infty$, depending only on
$W,b_0,n$, such that every bounded stable entire solution satisfies
\begin{equation*}
 \operatorname{dist}(Y,\{U=0\})\ge R_0\quad\Longrightarrow\quad |U(Y)|>b_0.
\end{equation*}
For $b_0<1$ sufficiently close to one,
\begin{equation}\label{sbplfj}
 \frac12|DU(Y)|^2+W(U(Y))
 \le C\exp[-c(\operatorname{dist}(Y,\{U=0\})-R_0)_+].
\end{equation}
If the zero set is empty, $U$ is a constant.
\end{lemma}

To determine $k$, we compare the energy near each layer with the planar
density, following Florit-Simon--Serra \cite[Lemma~4.9]{FSS25}.
Here the discarded base region must also be estimated. The zero set
is nonempty by \eqref{dofz} and Lemma~\ref{wentr}. By \eqref{izhux},
\begin{equation*}
 \operatorname{dist}((x,z),\{U=0\})
 \ge\frac{(|z|-C(1+|x|))_+}{\sqrt{1+C^2}}.
\end{equation*}
The global pointwise energy bound and \eqref{sbplfj} imply
\begin{equation}\label{cda}
F(x)=\int_{\mathbb R}e(U)(x,z)\,dz\le C(1+|x|).
\end{equation}
Thus from \eqref{uyctu} we obtain
\begin{equation}\label{dtee}
\int_{(B_{2R}\setminus B_R)\setminus\Omega_\infty(M,\delta)}
       F(x)\,dx
 \le C R^{n+1-\kappa}(\log R)^{n-1}=o(R^n).
\end{equation}
By a dyadic sum, we obtain the same $o(R^n)$ bound on $B_R$. 
For $x\in\Omega_\infty(M,\delta)$ and $r=|x|$ sufficiently large,
we obtain from \eqref{hrpy}
\begin{equation}\label{qojmg}
d(x)=\operatorname{dist}
   (x,\{\text{horizontal projections of exceptional zeros}\})
 \ge c\lambda\log r.
\end{equation}
Indeed any projected exceptional point is within $2L$ of a covering center.
Comparable-radius centers yield the logarithmic distance directly.
Noncomparable-radius centers are at distance comparable to $r$.
If the exceptional set is empty, interpret $d(x)=+\infty$.

The horizontal ball $B_{d(x)}(x)$ is disjoint from the projections of exceptional zeros.
Proper projection and \eqref{kdi} yield all its ordered root graphs, with
the same $k$ roots as at $x$. The straight segment between base
points remains in this ball, so the slope bound applies there.
It follows that, writing the roots at $x$ as $f_\ell(x)$,
\begin{equation}\label{dist}
\operatorname{dist}((x,z),\{U=0\})
 \ge \min\left\{\frac{d(x)}2,
       \frac1{\sqrt2}\min_\ell|z-f_\ell(x)|\right\}.
\end{equation}
A zero whose base lies outside $B_{d(x)/2}(x)$ satisfies the first
distance bound. Any other zero lies on one of the locally continued
graphs and satisfies the second bound. For $k=0$ the inner minimum is infinite.

On the vertical interval $|z|\le B(1+r)$, where $B$ is fixed larger
than the constant in \eqref{izhux}, we combine \eqref{sbplfj}--\eqref{dist} and sum the individual
exponential bounds associated with the roots. Outside that interval \eqref{izhux}--\eqref{sbplfj} yield an
exponentially small integral over the remaining vertical region. Consequently
\begin{equation}\label{bsi}
\int_{\{|z-f_\ell(x)|\ge M\ \forall\ell\}} e(U)(x,z)\,dz
 \le Ck e^{-cM}+C(1+r)e^{-c d(x)}+Ce^{-cr}
\end{equation}
for every sufficiently large fixed $M$. All well-entry constants
are fixed independently of $M,\delta$. The local graph slope is at most
one. Taking $\lambda$ large enough makes the last two terms $O(r^{-2})$.
The enlarged balls still satisfy the hypotheses of Lemma~\ref{geomet}.

At every root of $x\in\Omega_\infty(M,\delta)$, \eqref{pcgbpz} fails, so the
solution on $B_{2M}$ is $\delta$-close to a vertical heteroclinic.
Choosing $\delta$ sufficiently small after $M$ makes the normal
derivative nonzero throughout $[-2M,2M]$, and excludes another root
at distance less than $2M$. The intervals of radius $M$
are therefore disjoint, and each has full energy
\begin{equation*}
 \sigma_M+O(M\delta),\qquad
 \sigma_M=\int_{-M}^{M}g'(t)^2\,dt.
\end{equation*}
Together with \eqref{bsi}  we obtain
\begin{equation}\label{nkpr}
|F(x)-k\sigma_0|
 \le Ck\bigl(M\delta+\sigma_0-\sigma_M+e^{-cM}\bigr)+O(r^{-2})
 \quad\text{for }x\in\Omega_\infty(M,\delta),\ r\to\infty.
\end{equation}
The same argument applies to $\int_{\mathbb R}|DU(x,z)|^2\,dz$.

We integrate \eqref{nkpr} on $B_R\cap\Omega_\infty(M,\delta)$ and use
\eqref{cda}--\eqref{dtee} on the complement. For fixed $M,\delta$, we divide by
$R^n$ and let $R\to\infty$. From \eqref{dofz}, we obtain
\begin{equation*}
 R^{-n}\int_{B_R}F(x)\,dx\longrightarrow\theta\sigma_0\omega_n.
\end{equation*}
For this limit, we fix a sufficiently large $B$ and first integrate over
$|z|\le BR$. Equations \eqref{izhux}--\eqref{sbplfj} make the remaining
mass exponentially small. The boundary of this cylinder has zero
limiting measure. Hence
\begin{equation*}
 |\theta-k|\sigma_0
 \le Ck\bigl(M\delta+\sigma_0-\sigma_M+e^{-cM}\bigr).
\end{equation*}
The integer $k$ is independent of the allowed error and radius as proved
above. We let $\delta\downarrow0$ for fixed $M$ and then let $M\to\infty$.
This yields
\begin{equation}\label{pqmaw}
k=\theta.
\end{equation}
Thus $Q=k=\theta$ is a positive integer. For each fixed approximation
error bound, the zero set over the selected component consists of
exactly $Q$ globally ordered smooth graphs.

We now choose layer positions for which a sum of one-dimensional
transitions approximates $U$ with a small error.
The next proposition establishes the orthogonality conditions and the
integral bounds for the squared slopes needed in the layer equations
and stability inequality.
Let $U$ satisfy \eqref{mwm} and the all-scale planar energy convergence
\eqref{dofz}. Fix a large ball for approximation by heteroclinics and a small error bound, remove
the logarithmically enlarged balls described above, and let $\Omega$ be
the component of asymptotically full spherical measure. There are exactly
$Q$ ordered zero graphs $h_1<\cdots<h_Q$ over this component.

We choose the lower well to be $-1$, reflecting both $U$ and $W$
when necessary. For ordered layer functions $f_1<\cdots<f_Q$, we set
\begin{equation}\label{sumg}
 U=G+\phi,\qquad
 G(x,z)=\sum_{i=1}^Q g(s_i(z-f_i(x)))-\frac{1+(-1)^Q}{2},
 \qquad s_i=(-1)^{i-1}.
\end{equation}

\begin{proposition}\label{orth}
For each $D_0<\infty$ and small $\nu>0$, we can choose the radius for
approximation by heteroclinics, the error bound, enlargement, and
exterior radius so that there are smooth functions $f_1<\cdots<f_Q$
on the exterior part of $\Omega$ satisfying
\begin{equation}\label{komua}
 \begin{gathered}
 f_{i+1}-f_i\ge D_0,\qquad |Df_i|\le\nu,\qquad |\phi|\le\nu,\\
 \int_{\mathbb R}\phi(x,z)s_i g'(s_i(z-f_i(x)))\,dz=0,
 \end{gathered}
\end{equation}
\begin{equation*}
 \int_{\mathbb R}\phi(x,z)^2\,dz\le\nu^2,
 \qquad |f_i(x)|\le C(1+|x|).
\end{equation*}
Furthermore,
\begin{equation}\label{ofd}
 \begin{split}
 \sum_i|Df_i(x)|^2&\le C\int_{\mathbb R}|D_xU(x,z)|^2\,dz,\\
 \int_{\Omega\cap\{R<|x|<2R\}}\sum_i|Df_i|^2\,dx&\le C(A+1)R^{n-2}.
 \end{split}
\end{equation}
Here $C$ is independent of $D_0,\nu$ once the basic separation and smallness conditions hold.
Every fixed derivative order is bounded up to the artificial boundary
from within the domain. Derivatives of $U,G$, and $\phi$ decay
exponentially with the vertical distance from the layers.
\end{proposition} 
\begin{proof}
We first approximate $U$ using the zero graphs $h_i$, and then
adjust their heights to satisfy the orthogonality conditions in
\eqref{komua}. For this adjustment, we follow the implicit-function
construction of Wang--Wei
\cite[Proposition~4.1 and Remark~4.2]{WW}, using fixed vertical
coordinates and the values $W''(-1)$ and $W''(1)$ of the present
potential.

We choose $M$ large in terms of $D_0$ and $\nu$, and then choose
$\delta$ sufficiently small depending on $M$ and $\nu$.
At every zero outside the exceptional set, we require $U$ to be
$\delta$-close to a vertical heteroclinic in $C^2(B_{8M})$.
The choice of $\delta$ ensures that $U_z$ keeps its sign on
the vertical interval of radius $8M$ about that zero.
Consequently,
\[
 h_{i+1}-h_i\ge 8M,
 \qquad |Dh_i|\le C\delta.
\]
By \eqref{izhux}, the region below the cone containing the zero
set is connected, so the lower limiting well is independent of
$x$. With this well fixed as $-1$, as above, the signs of the
successive crossings alternate.

We first set $f_i=h_i$ in \eqref{sumg}. On each interval
$|z-h_i(x)|\le M$, we estimate $U-G$ using the approximation
by a heteroclinic and the separation of the other layers.
Outside these intervals, we use \eqref{dist},
Lemma~\ref{wentr}, and the exponential convergence of $g$
to its limiting values. To estimate the integral over the
whole vertical line, we first restrict to
$|z|\le B(1+|x|)$, where $B$ is fixed and larger than the
constant in \eqref{izhux}. Equations \eqref{izhux} and
\eqref{sbplfj} bound the integral over the remaining region
exponentially. We obtain
\begin{equation}\label{jgfb}
 \begin{split}
 \|U(x,\cdot)-G(x,\cdot)\|_\infty
 &\le C_Q(\delta+e^{-cM})+Ce^{-cd(x)},\\
 \int_{\mathbb R}|U-G|^2(x,z)\,dz
 &\le C_Q(M\delta^2+e^{-cM})
      +C(1+|x|)e^{-cd(x)}+Ce^{-c|x|}.
 \end{split}
\end{equation}
The choices of $M$ and $\delta$ make the first terms as small
as required in terms of $\nu$. By \eqref{qojmg}, we can next
fix the logarithmic enlargement sufficiently large that the
last two terms in the second inequality are
$O(|x|^{-K})$ for any prescribed $K$. We then increase the
exterior radius so that both errors in \eqref{jgfb} are
uniformly small.

We now solve the orthogonality equations in \eqref{komua}
for $f_i$ close to $h_i$. Their Jacobian with respect to the
layer positions has entries
\[
 \int_{\mathbb R}
 s_i g'(s_i(z-f_i))s_jg'(s_j(z-f_j))\,dz
 -
 \delta_{ij}\int_{\mathbb R}
 (U-G)g''(s_i(z-f_i))\,dz.
\]
The diagonal entries of the first integral equal $\sigma_0$,
and its off-diagonal entries tend uniformly to zero as the
minimum separation increases. The second integral is small
by \eqref{jgfb}. These bounds also hold for positions
sufficiently close to $h_i$, since translations of $g$
depend continuously on their shifts. Thus the Jacobian
has a uniformly bounded inverse.

By the quantitative implicit-function argument in
\cite[Proposition~4.1]{WW}, we obtain a unique solution
near the zero heights, with
$\max_i|f_i-h_i|$ bounded by $C_Q$ times the initial
$L^2(\mathbb R)$ error in \eqref{jgfb}.
Uniqueness ensures that the local solutions agree on
overlaps. Translations of $g$ are Lipschitz in both
$L^2$ and $L^\infty$, with constants $\|g'\|_2$ and
$\|g'\|_\infty$. Hence, after the preceding choices,
the resulting $\phi=U-G$ satisfies both
\[
 \|\phi(x,\cdot)\|_\infty\le\nu,
 \qquad
 \|\phi(x,\cdot)\|_2\le\nu.
\]
The small shifts preserve the required separation
$f_{i+1}-f_i\ge D_0$. They also preserve the linear height
bound, by \eqref{izhux}.

It remains to estimate the derivatives of $f_i$.
Differentiating the orthogonality equations with respect
to $x$, while keeping $z$ fixed, we obtain
\begin{equation}\label{jgvcmu}
 \begin{split}
 \sum_j\biggl[
 &\int_{\mathbb R}
 s_i g'(s_i(z-f_i))s_jg'(s_j(z-f_j))\,dz\\
 &\qquad
 -\delta_{ij}\int_{\mathbb R}
 \phi g''(s_i(z-f_i))\,dz
 \biggr]D_xf_j
 =
 -\int_{\mathbb R}
 D_xU\,s_i g'(s_i(z-f_i))\,dz.
 \end{split}
\end{equation}
Using the uniform bound for the inverse matrix and
Cauchy--Schwarz on the right-hand side, we deduce the
first inequality in \eqref{ofd}.

The same equation also controls the pointwise slopes.
Interior estimates bound $D_xU$ globally. Near each
layer, the heteroclinic approximation yields
$D_xU=O(\delta)$, and the small shifts from $h_i$ to $f_i$
keep these intervals inside the regions of approximation.
Outside an interval of radius $M$, the integral of
$|g'(s_i(z-f_i))|$ is $O(e^{-cM})$.
Splitting the right-hand side of \eqref{jgvcmu} into
these two regions, we obtain $|Df_i|\le\nu$.

We next integrate the first inequality in \eqref{ofd}
over $\Omega\cap\{R<|x|<2R\}$.
For $|z|\le BR$, the resulting cylinder lies in a ball
of radius $CR$, so \eqref{mwm} bounds the integral of
$|D_xU|^2$ by $CAR^{n-2}$.
For $|z|>BR$, we choose $B$ larger than the confinement
constant and use \eqref{izhux} and \eqref{sbplfj}
to bound the remaining integral exponentially.
We therefore obtain the second inequality in \eqref{ofd}.
The constants in these two estimates depend on fixed
bounds for the inverse matrix and for $g'$.
They are independent of $D_0$ and $\nu$ once the basic
separation and smallness conditions hold.

Finally, interior estimates bound the derivatives of $U$.
Repeated differentiation of the orthogonality equations
uses the same invertible matrix, while all remaining
terms contain lower derivatives of $f_i$ and integrable
derivatives of $g$. By induction, as in
\cite[Remark~4.2]{WW}, every fixed derivative order of
$f_i$ is bounded. The logarithmic enlargement separates
the artificial boundary from the projections of exceptional
zeros. The same local construction therefore applies
across that boundary, and the derivative bounds hold
up to it from within the domain.

In the two unbounded vertical intervals, \eqref{izhux}
and Lemma~\ref{wentr}, together with interior estimates,
imply exponential convergence of $U$ to the corresponding
well and exponential decay of its derivatives.
The corresponding estimates for $G$ follow from those
for $g$. Thus $\phi$ and its derivatives also decay
exponentially there, which justifies integration by
parts in $z$.
\end{proof}

\subsection{Interaction estimates and stability for the layers}
\label{layers}\label{duuc}\label{layerint}

The aim of this subsection is to compute the interaction between adjacent
layers and transfer stability to their displacements. We will use these
estimates to derive the Toda equations and scalar inequalities in
Subsection~\ref{scalarred}. Using the decomposition in
Proposition~\ref{orth}, we
estimate the errors through first derivatives of the layer functions,
whose squared integrals are controlled by \eqref{ofd}. We keep the
coefficients determined by $W''(-1)$ and $W''(1)$ separate to account
for the different exponential decay rates at the two wells.

The heteroclinic solution approaches the two wells $1,-1$ with the
following exponential asymptotics:
\begin{equation}\label{gtail}
 \begin{split}
 g(t)&=1-A_+e^{-\sqrt{W''(1)}t}
                    +O(e^{-2\sqrt{W''(1)}t})\quad\text{as }t\to+\infty,\\
 g(t)&=-1+A_-e^{\sqrt{W''(-1)}t}
                    +O(e^{2\sqrt{W''(-1)}t})\quad\text{as }t\to-\infty.
 \end{split}
\end{equation}
Here $A_+,A_->0$ are determined by $W$ and $g(0)=0$.
The expansions, including any fixed number of derivatives, follow by
integrating $g'=\sqrt{2W(g)}$ near each well, as in
\cite[Appendix~A]{WW}.
When an index is $1$ or $-1$, we also write $A_1=A_+$ and $A_{-1}=A_-$.

Let $f_1<\cdots<f_Q$, with $f_{i+1}-f_i\ge D$. We use $s_i$,
$G$ and $\phi$ as in \eqref{sumg}.
At a descending layer the heteroclinic is $g(f_i-z)$, which solves
the same one-dimensional equation as the increasing heteroclinic.
Derivatives with respect to a layer position satisfy
\begin{equation*}
 \partial_{f_i}G=-s_i g'(s_i(z-f_i)),\qquad
 \partial_{f_i}^2G=g''(s_i(z-f_i)).
\end{equation*}
On the gap between $f_i$ and $f_{i+1}$ the well is $s_i$ and its
exponential rate is $\sqrt{W''(s_i)}$.

The next proposition computes the leading interaction between
neighboring layers and bounds the remaining integrals in the vertical
variable. These estimates determine the equations for the layer
positions and control the errors in their stability inequality.

In the formulas below, the right-neighbor term is omitted when $i=Q$,
and the left-neighbor term is omitted when $i=1$. We also estimate
the functions
\begin{equation}\label{tizqr}
 \begin{gathered}
 [W''(G)-W''(g(s_i(z-f_i)))]g'(s_i(z-f_i)),\\
 [W''(G)-W''(g(s_i(z-f_i)))]g''(s_i(z-f_i)),\\
 [W'''(G)-W'''(g(s_i(z-f_i)))]g'(s_i(z-f_i))^2,\\
 g'(s_i(z-f_i))g'(s_j(z-f_j))\quad\text{for }i\ne j
 \end{gathered}
\end{equation}

\begin{proposition}\label{quanti}
For sufficiently large $D$, the following estimates hold uniformly
in the ordered layer positions. For every fixed multi-index $\alpha$,
\begin{equation}\label{suu}
 \begin{split}
 &\sum_i\bigg|\partial_f^\alpha\bigg\{
 \int_{\mathbb R}
 [W'(G)-\sum_j W'(g(s_j(z-f_j)))]s_i g'(s_i(z-f_i))\,dz\\
 &\qquad\qquad-2W''(s_i)A_{s_i}^2
       e^{-\sqrt{W''(s_i)}(f_{i+1}-f_i)}\\
 &\qquad\qquad+2W''(s_{i-1})A_{s_{i-1}}^2
       e^{-\sqrt{W''(s_{i-1})}(f_i-f_{i-1})}\bigg\}\bigg|\\
 \le& C_{W,Q,\alpha}e^{-c_WD}
       \sum_{j=1}^{Q-1}e^{-\sqrt{W''(s_j)}(f_{j+1}-f_j)}.
 \end{split}
\end{equation}
Moreover,
\begin{equation}\label{drtf}
 \int_{\mathbb R}|W'(G)-\sum_iW'(g(s_i(z-f_i)))|^2\,dz
 \le Ce^{-c_WD}\sum_{i=1}^{Q-1}
                     e^{-\sqrt{W''(s_i)}(f_{i+1}-f_i)}.
\end{equation}
The right side of \eqref{drtf} also bounds the squared normal
$L^2$ norms of the functions in \eqref{tizqr} and of their derivatives
of any fixed order with respect to the layer positions.

For $i\ne j$, the integrals of products of fixed positive-order normal
derivatives of $g(s_i(z-f_i))$ and $g(s_j(z-f_j))$ tend uniformly to
zero as $D\to\infty$. The same holds for their fixed-order derivatives
with respect to the layer positions.
\end{proposition}
We prove this in Appendix~\ref{inter}. In the layer estimates below,
$\epsilon_D$ denotes a constant bounded by $Ce^{-c_WD}$. Its value can
increase by a factor depending only on $W,Q$. 

Suppose that \eqref{sumg} is defined on
$\mathcal O\times\mathbb R$, that $U$ solves $\Delta U=W'(U)$ and is
stable, and that
\begin{equation}\label{iwh}
 \begin{gathered}
 f_{i+1}-f_i\ge D,\qquad |Df_i|\le\delta,\qquad |\phi|\le\delta,\\
 \int_{\mathbb R}\phi(x,z)s_i g'(s_i(z-f_i(x)))\,dz=0
       \qquad\text{for }1\le i\le Q.
 \end{gathered}
\end{equation}
The normal decay in Proposition~\ref{orth} justifies the integrations
below by truncation. By direct differentiation, we obtain
\begin{equation}\label{wkjm}
 \begin{split}
 (-\Delta+W''(G))\phi={}&
 \sum_i g''(s_i(z-f_i))|Df_i|^2
 -\sum_i s_i g'(s_i(z-f_i))\Delta f_i\\
 &-[W'(G)-\sum_iW'(g(s_i(z-f_i)))]\\
 &-[W'(G+\phi)-W'(G)-W''(G)\phi].
 \end{split}
\end{equation}

To estimate the orthogonal error, we use the following standard
spectral gap.

\begin{lemma}
There are $D_*,c>0$, depending only on $W,Q$, such that if all gaps
are at least $D_*$ and $v\in H^1(\mathbb R)$ is orthogonal to every
$s_i g'(s_i(z-f_i))$, then
\begin{equation*}
 \int_{\mathbb R}(v_z^2+W''(G)v^2)\,dz
          \ge c\int_{\mathbb R}(v_z^2+v^2)\,dz.
\end{equation*}
\end{lemma}

For $D\ge D_*$, $\delta$ sufficiently small depending on $W,Q$, and
$\rho\in C_c^1(\mathcal O)$, we test \eqref{wkjm} with $\rho(x)^2\phi$.
Orthogonality cancels the $\Delta f_i$ terms. The spectral gap absorbs
the Taylor remainder, so we obtain
\begin{equation}\label{vnb}
 \begin{split}
 \int\rho^2\int_{\mathbb R}(|D\phi|^2+\phi^2)\,dz\,dx
 \le{}&C\int\rho^2\sum_i|Df_i|^4\,dx\\
 &+C\epsilon_D\int\rho^2\sum_{i=1}^{Q-1}
                 e^{-\sqrt{W''(s_i)}(f_{i+1}-f_i)}\,dx
 +C\int|D\rho|^2\int_{\mathbb R}\phi^2\,dz\,dx.
 \end{split}
\end{equation}

The next proposition expresses the stability of $U$ as an inequality
for layer displacements. It bounds the interaction between neighboring
layers by the derivatives of these displacements, with errors involving
$Df_i$ and $\phi$.

\begin{proposition}\label{rstat}
Assume \eqref{iwh}, with $D$ sufficiently large and $\delta$
sufficiently small. For every fixed $\tau\in(0,1)$ and every
$\eta\in C_c^1(\mathcal O;\mathbb R^Q)$, we have
\begin{equation}\label{evjhr}
 \begin{split}
 &\int_{\mathcal O}\sum_{i=1}^{Q-1}
 2W''(s_i)^{3/2}A_{s_i}^2
 e^{-\sqrt{W''(s_i)}(f_{i+1}-f_i)}
 (\eta_{i+1}-\eta_i)^2\,dx\\
 &\le
 (\sigma_0+C\epsilon_D+C\tau)
 \int_{\mathcal O}|D\eta|^2\,dx\\
 &\quad+
 C(\epsilon_D+\tau)
 \int_{\mathcal O}|\eta|^2\sum_i|Df_i|^2\,dx\\
 &\quad+
 C_\tau\int_{\mathcal O}|\eta|^2
 \int_{\mathbb R}(|D_x\phi|^2+\phi^2)\,dz\,dx\\
 &\quad+
 C\epsilon_D\int_{\mathcal O}|\eta|^2
 \sum_{j=1}^{Q-1}
 e^{-\sqrt{W''(s_j)}(f_{j+1}-f_j)}\,dx.
 \end{split}
\end{equation}
\end{proposition}

\begin{proof}
We use the derivatives of the one-dimensional transitions
as test functions, following del Pino--Kowalczyk--Wei
\cite[Lemma~9.2]{DPKW13} and Wang--Wei
\cite[Proposition~8.1]{WW}.
More precisely, we apply the stability inequality of $U$ to
\[
 \sum_i\eta_i(x)s_i g'(s_i(z-f_i(x))).
\]
A cutoff in $z$, followed by a limit, makes this test
admissible. The exponential decay in the vertical variable
justifies this limit and the integrations below.

We first compute the contribution of the horizontal
derivatives. Since
\[
 D_x\bigl(s_i g'(s_i(z-f_i))\bigr)
 =-g''(s_i(z-f_i))Df_i,
\]
this contribution is
\begin{equation}\label{shexac}
 \begin{split}
 \int_{\mathcal O}\sum_{i,j}\biggl[
 &D\eta_i\cdot D\eta_j
 \int_{\mathbb R}
 s_i g'(s_i(z-f_i))s_jg'(s_j(z-f_j))\,dz\\
 &-2\eta_jD\eta_i\cdot Df_j
 \int_{\mathbb R}
 s_i g'(s_i(z-f_i))g''(s_j(z-f_j))\,dz\\
 &+\eta_i\eta_jDf_i\cdot Df_j
 \int_{\mathbb R}
 g''(s_i(z-f_i))g''(s_j(z-f_j))\,dz
 \biggr]\,dx.
 \end{split}
\end{equation}
For $i=j$, we have
$\int_{\mathbb R}(g')^2\,dz=\sigma_0$ and
$\int_{\mathbb R}g'g''\,dz=0$.
Thus the diagonal terms consist of
$\sigma_0\int_{\mathcal O}|D\eta|^2\,dx$ and
\begin{equation}\label{bal}
 \int_{\mathcal O}\sum_i\eta_i^2|Df_i|^2
 \int_{\mathbb R}g''(s_i(z-f_i))^2\,dz\,dx.
\end{equation}
To obtain the small coefficient of $\sum_i|Df_i|^2$ in
\eqref{evjhr}, we must cancel the second term. We use the equation for
$\phi$ to do so.

We next consider the vertical derivatives and the potential.
Differentiating the interaction integral with respect to
$f_j$, and using the equation for $g$, we obtain
\begin{equation}\label{hodtv}
 \begin{split}
 &\int_{\mathbb R}\bigl[
 g''(s_i(z-f_i))g''(s_j(z-f_j))
 +W''(G)s_i g'(s_i(z-f_i))s_jg'(s_j(z-f_j))
 \bigr]\,dz\\
 =&-\partial_{f_j}\int_{\mathbb R}
 \bigl[W'(G)-\sum_kW'(g(s_k(z-f_k)))\bigr]
 s_i g'(s_i(z-f_i))\,dz\\
 &\quad-\delta_{ij}\int_{\mathbb R}
 \bigl[W'(G)-\sum_kW'(g(s_k(z-f_k)))\bigr]
 g''(s_i(z-f_i))\,dz.
 \end{split}
\end{equation}
The derivative on the right will produce the interaction
term in \eqref{evjhr}. The remaining integral will cancel
another term from the equation for $\phi$.

To account for the difference between $U$ and $G$, we expand
\[
 W''(U)=W''(G)+W'''(G)\phi+O_W(\phi^2).
\]
In the diagonal part of the term linear in $\phi$, we first
replace $W'''(G)$ by $W'''(g(s_i(z-f_i)))$.
We will estimate the difference after the cancellations.
Differentiating the equation for $g$ twice, we have
\begin{equation}\label{ldhl}
 \bigl[-\partial_{zz}+W''(g(s_i(z-f_i)))\bigr]
 g''(s_i(z-f_i))
 =
 -W'''(g(s_i(z-f_i)))g'(s_i(z-f_i))^2.
\end{equation}
We multiply \eqref{wkjm} by $g''(s_i(z-f_i))$ and integrate
twice by parts in $z$. Using \eqref{ldhl}, we obtain
\[
 \begin{split}
 &\int_{\mathbb R}
 W'''(g(s_i(z-f_i)))\phi g'(s_i(z-f_i))^2\,dz\\
 =&-\int_{\mathbb R}
 g''(s_i(z-f_i))\Delta_x\phi\,dz\\
 &\quad-\sum_j|Df_j|^2\int_{\mathbb R}
 g''(s_i(z-f_i))g''(s_j(z-f_j))\,dz\\
 &\quad+\sum_j\Delta f_j\int_{\mathbb R}
 g''(s_i(z-f_i))s_jg'(s_j(z-f_j))\,dz\\
 &\quad+\int_{\mathbb R}
 \bigl[W'(G)-\sum_jW'(g(s_j(z-f_j)))\bigr]
 g''(s_i(z-f_i))\,dz\\
 &\quad+\int_{\mathbb R}
 \bigl[W'(G+\phi)-W'(G)-W''(G)\phi\bigr]
 g''(s_i(z-f_i))\,dz\\
 &\quad+\int_{\mathbb R}
 \bigl[W''(G)-W''(g(s_i(z-f_i)))\bigr]
 \phi g''(s_i(z-f_i))\,dz.
 \end{split}
\]

We multiply this identity by $\eta_i^2$ and sum over $i$.
The terms with $j=i$ in the negative sum containing
$|Df_j|^2$ cancel \eqref{bal}.
The terms containing $\Delta f_i$ vanish because
$\int_{\mathbb R}g'g''\,dz=0$.
The positive integral containing
$W'(G)-\sum_jW'(g(s_j(z-f_j)))$ cancels the last term
of \eqref{hodtv}.
These are exact cancellations.

After these cancellations, the leading interaction
contribution is
\[
 -\int_{\mathcal O}\sum_{i,j}\eta_i\eta_j
 \partial_{f_j}\int_{\mathbb R}
 \bigl[W'(G)-\sum_kW'(g(s_k(z-f_k)))\bigr]
 s_i g'(s_i(z-f_i))\,dz\,dx.
\]
We apply \eqref{suu} with one derivative with respect to
the layer positions. Summing the leading terms over $i,j$,
we obtain the negative of the left-hand side of
\eqref{evjhr}.
The differentiated remainder is bounded in absolute value by
\[
 C\epsilon_D\int_{\mathcal O}|\eta|^2
 \sum_{j=1}^{Q-1}
 e^{-\sqrt{W''(s_j)}(f_{j+1}-f_j)}\,dx.
\]

We now estimate the remaining terms.
By Proposition~\ref{quanti}, the integrals involving
derivatives of two distinct layers, together with their
first derivatives with respect to the layer positions,
are $O(\epsilon_D)$.
Consequently, the terms with $i\ne j$ in \eqref{shexac}
and in the sum containing $|Df_j|^2$ are bounded by
\[
 C\epsilon_D\int_{\mathcal O}
 \left(|D\eta|^2+|\eta|^2\sum_i|Df_i|^2\right)\,dx.
\]
For the terms containing $\Delta f_j$, we integrate by
parts in $x$. Differentiation of the vertical integrals
introduces only first derivatives of the layer functions,
so these terms satisfy the same bound.

For the term containing $\Delta_x\phi$, integration by
parts in $x$ yields
\[
 \begin{split}
 &-\int_{\mathcal O}\sum_i\eta_i^2
 \int_{\mathbb R}g''(s_i(z-f_i))\Delta_x\phi\,dz\,dx\\
 =&2\int_{\mathcal O}\sum_i\eta_iD\eta_i\cdot
 \int_{\mathbb R}g''(s_i(z-f_i))D_x\phi\,dz\,dx\\
 &\quad-\int_{\mathcal O}\sum_i\eta_i^2Df_i\cdot
 \int_{\mathbb R}s_i g'''(s_i(z-f_i))D_x\phi\,dz\,dx.
 \end{split}
\]
The vertical $L^2$ norms of $g''$ and $g'''$ are uniformly
bounded. By Cauchy--Schwarz and Young's inequality, the
absolute value of this expression is at most
\[
 C\tau\int_{\mathcal O}
 \left(|D\eta|^2+|\eta|^2\sum_i|Df_i|^2\right)\,dx
 +
 C_\tau\int_{\mathcal O}|\eta|^2
 \int_{\mathbb R}|D_x\phi|^2\,dz\,dx.
\]

The other terms linear in $\phi$ contain the coefficient
differences in \eqref{tizqr}.
This also includes the terms with $i\ne j$ from the
expansion of $W''(U)$, whose coefficients contain
$g'(s_i(z-f_i))g'(s_j(z-f_j))$.
Using the squared $L^2$ bounds in Proposition~\ref{quanti}
and Young's inequality, we bound all these terms by
\[
 C\int_{\mathcal O}|\eta|^2
 \int_{\mathbb R}\phi^2\,dz\,dx
 +
 C\epsilon_D\int_{\mathcal O}|\eta|^2
 \sum_{j=1}^{Q-1}
 e^{-\sqrt{W''(s_j)}(f_{j+1}-f_j)}\,dx.
\]
Finally, the Taylor remainders in the expansions of
$W''(G+\phi)$ and $W'(G+\phi)$ are bounded by $C\phi^2$.
Their contributions are therefore bounded by
$C\int_{\mathcal O}|\eta|^2
\int_{\mathbb R}\phi^2\,dz\,dx$.

The full quadratic form is nonnegative by stability.
Its leading terms are
$\sigma_0\int_{\mathcal O}|D\eta|^2\,dx$ and the negative
interaction term computed above.
Combining the error estimates and moving this interaction
term to the other side, we obtain \eqref{evjhr}.
\end{proof}

\subsection{Scalar reduction and radial Hardy tests}
\label{scalarred}

The objective of this subsection is to reduce the layer equations and
stability to scalar inequalities suitable for spherical averaging and
radial Hardy tests. Assuming $Q\ge2$, we choose a weighted mean gap
for which the equation and stability contain
the same positive interaction function $a$. An exponential lower bound
for $a$ then allows us to deduce the dimension restriction in
Lemma~\ref{lemma} by spherical averaging and radial Hardy tests.

We first derive the individual layer equations with their errors, and
then choose the weights and establish the three scalar estimates.
The calculations take place on $\mathcal O$, while
Lemma~\ref{lemma} concerns a full exterior domain.
Proposition~\ref{extens} will account for the regions removed from
$\mathcal O$.

For each gap, we denote its leading interaction by
\begin{equation}\label{hidefi}
 a_i(x)=\frac{2W''(s_i)A_{s_i}^2}{\sigma_0}
 e^{-\sqrt{W''(s_i)}(f_{i+1}(x)-f_i(x))},
 \qquad a_0=a_Q=0.
\end{equation}
We project \eqref{wkjm} onto
$s_i g'(s_i(z-f_i))$.
The one-dimensional Jacobi equation removes the
corresponding vertical derivative terms.
To handle the horizontal derivatives, we differentiate
the orthogonality condition in \eqref{iwh} and obtain
\[
 \int_{\mathbb R}
 s_i g'(s_i(z-f_i))D_x\phi\,dz
 =
 Df_i\int_{\mathbb R}\phi g''(s_i(z-f_i))\,dz.
\]
For $i\ne j$, we also have
\[
 \begin{split}
 &D_x\int_{\mathbb R}
 s_i g'(s_i(z-f_i))s_jg'(s_j(z-f_j))\,dz\\
 &=-Df_i\int_{\mathbb R}
 g''(s_i(z-f_i))s_jg'(s_j(z-f_j))\,dz\\
 &\quad-Df_j\int_{\mathbb R}
 s_i g'(s_i(z-f_i))g''(s_j(z-f_j))\,dz.
 \end{split}
\]
We use this identity to place the terms containing
$\Delta f_j$, for $j\ne i$, inside a divergence.
The terms containing $|Df_j|^2$ that arise in this step
cancel the corresponding terms in \eqref{wkjm}.
For $j=i$, we use
$\int_{\mathbb R}(g')^2\,dz=\sigma_0$ and
$\int_{\mathbb R}g'g''\,dz=0$.
By Proposition~\ref{quanti}, the resulting equation is
\begin{equation}\label{hsequa}
 \Delta f_i=a_{i-1}-a_i+e_i+\operatorname{div}F_i,
\end{equation}
where
\[
 \begin{split}
 F_i=\frac1{\sigma_0}\biggl[
 &Df_i\int_{\mathbb R}
 \phi g''(s_i(z-f_i))\,dz-\sum_{j\ne i}Df_j\int_{\mathbb R}
 s_i g'(s_i(z-f_i))s_jg'(s_j(z-f_j))\,dz
 \biggr].
 \end{split}
\]

We next estimate $F_i$ and the remaining scalar term $e_i$.
The bound for $F_i$ follows from $|\phi|\le\delta$ and
the estimates for integrals involving distinct layers
in Proposition~\ref{quanti}.
In $e_i$, these estimates also control the products
of derivatives of distinct layer positions.
For the term containing $D_x\phi$, Young's inequality yields
\[
 \left|Df_i\cdot\int_{\mathbb R}
 g''(s_i(z-f_i))D_x\phi\,dz\right|
 \le
 \tau|Df_i|^2
 +C_\tau\int_{\mathbb R}|D_x\phi|^2\,dz.
\]
The coefficient estimates in \eqref{tizqr} control the
remaining term linear in $\phi$, and the Taylor remainder
is bounded by $C\phi^2$.
Together with the interaction remainder in \eqref{suu},
we obtain, for every fixed $\tau>0$,
\begin{equation}\label{hserro}
 \begin{aligned}
 \sum_i|F_i|
 &\le C(\delta+\epsilon_D)
       \left(\sum_i|Df_i|^2\right)^{1/2},\\
 \sum_i|e_i|
 &\le C(\epsilon_D+\tau)\sum_i|Df_i|^2
 +C_\tau\int_{\mathbb R}(|D_x\phi|^2+\phi^2)\,dz
 +C\epsilon_D\sum_{i=1}^{Q-1}a_i.
 \end{aligned}
\end{equation}

We next choose a weighted sum of the layer equations. The weights
must produce the same interaction term in \eqref{hsequa} and in
\eqref{evjhr} tested with $\eta_i=c_i\zeta$.
We therefore choose $c_1,\ldots,c_Q$ by
\begin{equation}\label{gtcoef}
 \sum_{i=1}^Qc_i=0,
 \qquad
 c_{i+1}-c_i=\frac1{\sqrt{W''(s_i)}}
 \quad\text{for }1\le i<Q.
\end{equation}
These coefficients depend only on $W,Q$.
They are strictly increasing, and
$\sum_{j=1}^i c_j<0$ for $i<Q$.
When $W''(-1)=W''(1)$, they are equally spaced.
We define
\begin{equation}\label{clp}
 \begin{split}
 b&=\frac{\sum_i c_i f_i}{\sum_i c_i^2},\\
 a&=\frac1{\sigma_0\sum_i c_i^2}
 \sum_{i=1}^{Q-1}
 2\sqrt{W''(s_i)}A_{s_i}^2
 e^{-\sqrt{W''(s_i)}(f_{i+1}-f_i)}.
 \end{split}
\end{equation}

The leading terms in the weighted sum satisfy
\[
 \sum_i c_i(a_{i-1}-a_i)
 =
 \sum_{i=1}^{Q-1}(c_{i+1}-c_i)a_i
 =
 \left(\sum_i c_i^2\right)a.
\]
We therefore multiply \eqref{hsequa} by $c_i$, sum over
$i$, and divide by $\sum_i c_i^2$.
With
\[
 F=\frac{\sum_i c_iF_i}{\sum_i c_i^2},
 \qquad
 e=\frac{\sum_i c_ie_i}{\sum_i c_i^2},
\]
we obtain
\[
 \Delta b=a+e+\operatorname{div}F.
\]
The coefficients in these sums depend only on $W,Q$,
and $a$ is bounded above and below by positive fixed
multiples of $\sum_i a_i$.
Thus \eqref{hserro} implies
\begin{equation}\label{hpiano}
 \begin{split}
 |F|&\le C(\delta+\epsilon_D)
       \left(\sum_i|Df_i|^2\right)^{1/2},\\
 |e|&\le C(\epsilon_D+\tau)\sum_i|Df_i|^2
 +C_\tau\int_{\mathbb R}(|D_x\phi|^2+\phi^2)\,dz
 +C\epsilon_D a.
 \end{split}
\end{equation}
In particular, the differential inequality that we
will use is
\begin{equation}\label{mny}
 \Delta b\le
 (1+C\epsilon_D)a
 +C(\epsilon_D+\tau)\sum_i|Df_i|^2
 +C_\tau\int_{\mathbb R}(|D_x\phi|^2+\phi^2)\,dz
 +\operatorname{div}F.
\end{equation}

The function $b$ also has a direct interpretation in
terms of the layer gaps.
By summation by parts and \eqref{gtcoef}, we have
\[
 \begin{split}
 \sum_i c_i f_i
 &=-\sum_{i=1}^{Q-1}
   \left(\sum_{j=1}^i c_j\right)(f_{i+1}-f_i),\\
 \sum_i c_i^2
 &=\sum_{i=1}^{Q-1}
   \frac{-\sum_{j=1}^i c_j}{\sqrt{W''(s_i)}}.
 \end{split}
\]
Consequently, $b$ is a convex combination of
$\sqrt{W''(s_i)}(f_{i+1}-f_i)$, with coefficients
\[
 \frac{-\sum_{j=1}^i c_j}
 {\sqrt{W''(s_i)}\sum_j c_j^2},
 \qquad 1\le i<Q.
\]
These coefficients are positive and sum to one.
All coefficients of the exponentials in $a$ are also
positive and depend only on $W,Q$.
Jensen's inequality therefore yields
\begin{equation}\label{ojyzt}
 a\ge c_{W,Q}e^{-b},
 \qquad c_{W,Q}>0.
\end{equation}

We next obtain the scalar stability inequality.
For $\zeta\in C_c^1(\mathcal O)$, we take
$\eta_i=c_i\zeta$ in \eqref{evjhr} and divide by
$\sigma_0\sum_i c_i^2$.
By \eqref{gtcoef}, the left-hand side becomes exactly
$\int_{\mathcal O}a\zeta^2\,dx$.
The last term on the right is bounded by
$C\epsilon_D\int_{\mathcal O}a\zeta^2\,dx$.
For sufficiently large $D$, we absorb this term and obtain
\begin{equation}\label{ctyq}
 \begin{split}
 \int_{\mathcal O}a\zeta^2\,dx
 &\le [1+C(\epsilon_D+\tau)]
       \int_{\mathcal O}|D\zeta|^2\,dx\\
 &\quad+C(\epsilon_D+\tau)
       \int_{\mathcal O}\zeta^2\sum_i|Df_i|^2\,dx\\
 &\quad+C_\tau\int_{\mathcal O}\zeta^2
       \int_{\mathbb R}(|D_x\phi|^2+\phi^2)\,dz\,dx.
 \end{split}
\end{equation}
The absorption changes the coefficients by a factor
$1+O(\epsilon_D)$. The function $a$ remains the one
defined in \eqref{clp}.

We have now obtained the three scalar estimates:
\eqref{mny}, \eqref{ojyzt}, and \eqref{ctyq}.
We will also need to control the contribution of
$\operatorname{div}F$ when taking spherical averages.
If
\[
 \int_{\mathcal O\cap\{R<|x|<2R\}}
 \sum_i|Df_i|^2\,dx\le A_1R^{n-2},
\]
then Cauchy--Schwarz and \eqref{hpiano} imply
\begin{equation}\label{vtnenm}
 \int_{\mathcal O\cap\{R<|x|<2R\}}
 |F|\,|x|^{1-n}\,dx
 \le C(\delta+\epsilon_D)\sqrt{A_1}.
\end{equation}
Summing this estimate over dyadic annuli and dividing
by the logarithm of the outer radius yields a bound
of the same size. Thus the averaged contribution of
$F$ becomes small with $\delta$ and $\epsilon_D$.

The next lemma states the radial estimate that we will
apply to these scalar inequalities.
When the coefficients in the differential and stability
inequalities equal one and the averaged errors vanish,
its conclusion requires $n\ge10$.
The lemma quantifies how the errors affect this restriction.
The radial argument follows Wang
\cite[Proposition~4.1 and Section~6]{WangToda19}
and the spherical-average argument of Dancer--Farina
\cite[proof of Theorem~1, Step~4]{DF09}.
Gui--Wang--Wei \cite{Gui-Wang-Wei} used related radial estimates
for axially symmetric Allen--Cahn solutions.
The dimension restriction agrees with the result for
the exponential equation in Farina \cite{Farina07}.

We consider general functions on a full
exterior domain. Let $n>2$, $d=n-2$, and
$\Omega=\mathbb R^n\setminus\overline{B_{R_0}}$.
We take
$b\in C^2(\Omega)$,
$a,e,g\in C^0(\Omega)$, and
$F\in C^1(\Omega;\mathbb R^n)$.
Here $e$ denotes the scalar error in the differential
inequality, and $g$ denotes the error in the stability
inequality. In the application, we use the upper bounds
for the errors in \eqref{mny} and \eqref{ctyq}.

Fix $R_1>R_0$ and set $S=|S^{n-1}|$.
The radial tests involve the weight $|x|^{2-n}$,
so we measure the positive parts of the scalar errors by
\begin{equation}\label{ohk}
 \begin{aligned}
 E_*&=\limsup_{R\to\infty}\frac1{S\log R}
 \int_{R_1<|x|<R}e_+(x)|x|^{2-n}\,dx,\\
 G_*&=\limsup_{R\to\infty}\frac1{S\log R}
 \int_{R_1<|x|<R}g_+(x)|x|^{2-n}\,dx.
 \end{aligned}
\end{equation}
For the divergence term, we average the radial component
of $F$, keeping its sign:
\begin{equation}\label{bpke}
 \begin{split}
 q(t)&=e^t\fint_{S^{n-1}}
 F(e^t\theta)\cdot\theta\,d\theta,\\
 F_*&=\limsup_{T\to\infty}
 \frac1T\int_{\log R_1}^Tq(t)\,dt.
 \end{split}
\end{equation}
\begin{lemma}\label{lemma}
For the functions defined above, suppose $a>0$, the upper limits
$E_*,G_*,F_*$ are finite, and, for constants $c,\beta,\lambda,\kappa>0$,
\begin{equation}\label{dbdve}
a\ge c e^{-\beta b},\qquad
 \Delta b\le\lambda a+e+\operatorname{div}F.
\end{equation}
Assume also that every compactly supported radial test satisfies
\begin{equation}\label{uufske}
\int_\Omega a\zeta^2\,dx
 \le\kappa\int_\Omega|D\zeta|^2\,dx+\int_\Omega g\zeta^2\,dx.
\end{equation}
Then we have
\begin{equation}\label{fsjoj}
\beta\left\{\frac{\lambda\kappa d}{4}
           +\frac{\lambda G_*+E_*}{d}+F_*\right\}\ge2.
\end{equation}
\end{lemma}

\begin{proof}
We set $t=\log r$ and $\gamma=d/2$, and define
\begin{equation*}
 \begin{aligned}
 m(t)&=\fint_{S^{n-1}} b(e^t\theta)\,d\theta,
 &\omega(t)&=e^{2t}\fint_{S^{n-1}} a(e^t\theta)\,d\theta,\\
 E(t)&=e^{2t}\fint_{S^{n-1}} e(e^t\theta)\,d\theta,
 &G(t)&=e^{2t}\fint_{S^{n-1}} g_+(e^t\theta)\,d\theta.
 \end{aligned}
\end{equation*}
We first use stability to bound the average of $\omega$.
For $\chi\in C_c^1((\log R_1,\infty))$, we take
$\zeta(r)=r^{-\gamma}\chi(\log r)$ in \eqref{uufske}.
After the change of variables $r=e^t$ and division by $S$,
the Dirichlet integral becomes
$\int(\chi'-\gamma\chi)^2\,dt$.
Since $\chi$ has compact support, $\int\chi\chi'\,dt=0$.
Using $g\le g_+$, we therefore obtain
\begin{equation*}
 \int\omega\chi^2\,dt
 \le\kappa\int\bigl((\chi')^2+\gamma^2\chi^2\bigr)\,dt
       +\int G\chi^2\,dt.
\end{equation*}

We fix $s>\log R_1+1$.
For $T>s+1$, we choose $\chi$ between zero and one, equal to
one on $[s,T]$ and zero outside $[s-1,T+1]$.
The function varies only on $[s-1,s]$ and $[T,T+1]$,
where $|\chi'|\le C$ with $C$ independent of $T$.
Since $\omega>0$ and $G\ge0$, the preceding inequality implies
\begin{equation*}
 \int_s^T\omega(t)\,dt
 \le\kappa\gamma^2(T-s)+C
       +\int_{\log R_1}^{T+1}G(t)\,dt.
\end{equation*}
The two intervals where $\chi$ varies contribute a bounded
amount to the Dirichlet integral.
By \eqref{ohk} and the change of variables $r=e^t$, the upper
limit of the last integral divided by $T$ is $G_*$.
Here replacing $T$ by $T+1$ leaves the upper limit unchanged,
since $(T+1)/T\to1$.
We divide by $T$ and obtain
\begin{equation}\label{gxhbmb}
\limsup_{T\to\infty}\frac1T\int_s^T\omega(t)\,dt
 \le\kappa\gamma^2+G_*.
\end{equation}

We next estimate the growth of $m$.
Taking spherical averages in the differential inequality
in \eqref{dbdve}, and using the definition of $q$ in
\eqref{bpke}, we obtain
\begin{equation*}
 m''+dm'\le\lambda\omega+E+q'+dq.
\end{equation*}
Thus $(m'-q)'+d(m'-q)\le\lambda\omega+E$.
We multiply by $e^{dt}$ and integrate from $s$ to $t$.
A second integration from $s$ to $T$ yields
\begin{equation}\label{opt}
\begin{split}
 m(T)\le {}&m(s)+\frac{m'(s)-q(s)}d(1-e^{-d(T-s)})
              +\int_s^Tq(t)\,dt\\
 &+\frac1d\int_s^T(1-e^{-d(T-t)})
                    (\lambda\omega(t)+E(t))\,dt.
 \end{split}
\end{equation}
Since $0\le1-e^{-d(T-t)}\le1$ for $s\le t\le T$,
the last line is bounded above by
$d^{-1}\int_s^T(\lambda\omega+E_+)\,dt$.
Moreover,
\[
 E_+(t)\le e^{2t}\fint_{S^{n-1}}e_+(e^t\theta)\,d\theta.
\]
Consequently, \eqref{ohk} bounds the upper limit of
$T^{-1}\int_s^T E_+(t)\,dt$ by $E_*$, while
\eqref{bpke} shows that the corresponding upper limit
for $q$ equals $F_*$.
The first two terms on the right-hand side of \eqref{opt}
remain bounded as $T\to\infty$ and vanish after division by $T$.
Together with \eqref{gxhbmb} and $\gamma=d/2$, these facts imply
\begin{equation}\label{tcxs}
\limsup_{T\to\infty}\frac{m(T)}T
 \le B_*=\frac{\lambda\kappa d}{4}
                 +\frac{\lambda G_*+E_*}{d}+F_*.
\end{equation}

Finally, the lower bound $a\ge ce^{-\beta b}$ in
\eqref{dbdve} and Jensen's inequality imply
$\omega(t)\ge c\exp(2t-\beta m(t))$.
If $\beta B_*<2$, we choose $B>B_*$ such that $\beta B<2$.
By \eqref{tcxs}, we have $m(t)\le Bt$ for all sufficiently
large $t$, and hence
\[
 \omega(t)\ge c e^{(2-\beta B)t}.
\]
The integral $\int_s^T\omega(t)\,dt$ then grows faster than $T$, contradicting
\eqref{gxhbmb}.
Thus $\beta B_*\ge2$, which is \eqref{fsjoj}.
\end{proof}  

For the exact coefficients $\lambda=\kappa=\beta=1$ and zero
averaged flux, the lower bound \eqref{fsjoj} is violated whenever
\begin{equation*}
E_*+G_*<\frac{(n-2)(10-n)}4.
\end{equation*}
Thus the base dimensions $3\le n\le9$, corresponding to ambient
dimensions $4\le N\le10$, admit a strictly positive error allowance.
At $n=9$, the sum $E_*+G_*$ may be any number below $7/4$ under
the normalization in \eqref{ohk}. The case $n=2$ is treated by the
geometric stability inequality.

In particular, if an Allen--Cahn reduction produces
$\lambda,\kappa,\beta\to1$ as a fixed approximation error
$\delta\downarrow0$, and
$E_*+G_*+\max(F_*,0)\le C(A,Q,n)\delta^\alpha$ for some
$\alpha>0$, one fixed sufficiently small $\delta$ leads to the
contradiction. The required averaged bounds are verified in
the proof of Proposition~\ref{crit} below.

\subsection{Extending the scalar estimates across the excluded regions}
\label{scalext}

In this subsection, we make the scalar estimates available for averaging
over complete spheres. The layer functions are defined only on the
component constructed in Lemma~\ref{geomet}, so their equations and
stability inequalities acquire boundary and cutoff terms when extended
across the complement. Our objective is to show that these added terms
have bounded weighted integrals and vanish in the logarithmic averages.
Proposition~\ref{extens} then permits the application of
Lemma~\ref{lemma} to the rigidity problem.

Let $n\ge5$ and let $H$ be a locally finite union of closed balls with
center count $\#\{|p_i|\le R\}\le C(1+R)^{n-4}$ and radii between
a positive constant and $C\log(2+|p_i|)$. Let $\Omega$ be the
component of asymptotically full measure constructed in
Lemma~\ref{geomet}. For functions on $\Omega$, we define the averaged
positive scalar errors by \eqref{ohk}, integrating only over $\Omega$,
and the signed radial-averaged flux by \eqref{bpke}, with $F1_\Omega$
in place of $F$.
\begin{proposition}\label{extens}
Let $\Omega$ be the component defined above. Suppose that
$b,a,e,g$ are smooth real-valued functions on $\Omega$ and
$F$ is a smooth vector-valued function on $\Omega$.
Assume that $E_*,G_*,F_*$ are finite. For some positive constants $c,\beta,\lambda,\kappa,C_a,C_b$,
suppose that, outside a fixed ball,
\begin{equation}\label{frhn}
0<a\le C_a,\qquad
a\ge c e^{-\beta b},\qquad
\Delta b\le\lambda a+e+\operatorname{div}F.
\end{equation}
Assume also that every $\xi\in C_c^1(\Omega)$ supported
outside this ball satisfies
\begin{equation}\label{djcrh}
\int_\Omega a\xi^2\,dx
\le
\kappa\int_\Omega|D\xi|^2\,dx
+\int_\Omega g\xi^2\,dx.
\end{equation}
Suppose that $b$ and $F$ extend smoothly to a neighborhood
of each point of $\partial\Omega$, and that their boundary
values satisfy
\begin{equation}\label{uhzotk}
|Db|+|F|\le C_b,\qquad
|b|\le C_b(1+|x|).
\end{equation}
Then
\[
\beta\left\{
\frac{\lambda\kappa(n-2)}{4}
+\frac{\lambda G_*+E_*}{n-2}
+F_*
\right\}\ge2.
\]
\end{proposition} 

\begin{proof}
Fix a sufficiently large interior radius. Terms from its boundary
are bounded independently of the outer radius. We write
$A_R=\{R<|x|<2R\}$ and
$\kappa_0=2(n-1)/(n-2)>2$, reserving $\kappa$ for the stability
coefficient. Lemma~\ref{geomet} and the ball count yield
\begin{equation}\label{rgg}
|A_R\setminus\Omega|\le C R^{n-\kappa_0}(\log R)^{n-1},
 \qquad
 \mathcal H^{n-1}(\partial\Omega\cap A_R)
 \le C R^{n-4}(\log R)^{n-1}.
\end{equation}
In particular,
\begin{equation}\label{gzsj}
\int_{\Omega^c\cap\{|x|>R_1\}}|x|^{2-n}\,dx<\infty.
\end{equation}
The complement here includes every other component of
$\mathbb R^n\setminus H$.

We define the extensions by
\begin{equation}\label{qdjwb}
\widetilde b=b1_\Omega,
 \qquad \widetilde a=a1_\Omega+c1_{\Omega^c}.
\end{equation}
Then $\widetilde a\ge c e^{-\beta\widetilde b}$ almost everywhere.
From the Gauss formula, with $\nu$ the outward normal of $\Omega$, we obtain
\begin{equation*}
 \Delta(b1_\Omega)
 =1_\Omega\Delta b-\partial_\nu b\,
        \mathcal H^{n-1}\llcorner\partial\Omega
       -\operatorname{div}
        (b\nu\,\mathcal H^{n-1}\llcorner\partial\Omega),
\end{equation*}
\begin{equation*}
 1_\Omega\operatorname{div}F
 =\operatorname{div}(F1_\Omega)
    +(F\cdot\nu)\mathcal H^{n-1}\llcorner\partial\Omega.
\end{equation*}
Consequently \eqref{frhn} implies
\begin{equation}\label{moqdzv}
\Delta\widetilde b\le\lambda\widetilde a+
                  \widetilde e+\operatorname{div}\widetilde F,
\end{equation}
where the scalar and vector measures are
\begin{equation}\label{alr}
\widetilde e=e_+1_\Omega\,dx+
       (F\cdot\nu-\partial_\nu b)\,
         \mathcal H^{n-1}\llcorner\partial\Omega,
 \qquad
 \widetilde F=F1_\Omega\,dx-
       b\nu\,\mathcal H^{n-1}\llcorner\partial\Omega.
\end{equation}
Here we have used $a1_\Omega\le\widetilde a$ and $e\le e_+$.
By \eqref{uhzotk} and \eqref{rgg}, the boundary terms satisfy
\begin{equation}\label{kwwk}
 \begin{split}
 &\int_{\partial\Omega\cap A_R}
    |F\cdot\nu-\partial_\nu b|\,|x|^{2-n}\,d\mathcal H^{n-1}(x)+\int_{\partial\Omega\cap A_R}
    |b\nu\cdot(x/|x|)|\,|x|^{1-n}\,d\mathcal H^{n-1}(x)\\
 \le& C R^{-2}(\log R)^{n-1}.
 \end{split}
\end{equation}
Thus these contributions have finite total weighted variation and
contribute zero to both asymptotic averages.

For the stability estimate, we choose $0\le\chi\le1$, vanishing on a neighborhood of $H$,
equal to one outside its fixed-thickness enlargements, and with
$|D\chi|\le C$. We use $\chi\zeta$ on $\Omega$ and set it equal to zero
on other components. This is an admissible compactly supported
Sobolev test in \eqref{djcrh}, by approximation. The collar volume on $A_R$ is at most $C R^{n-4}(\log R)^n$.
For radial Hardy plateaus
\begin{equation*}
 \zeta=r^{-(n-2)/2}\psi(\log r),\qquad
 |\psi|+|\psi'|\le C,
\end{equation*}
the lost left-hand term, the extra cutoff Dirichlet term, and the
Dirichlet cross term all have a summable dyadic bound. For example,
\begin{equation*}
 \int_{\Omega\cap A_R}a(1-\chi^2)\zeta^2\,dx
 +\int_{\Omega\cap A_R}\zeta^2|D\chi|^2\,dx
 \le C R^{-2}(\log R)^n.
\end{equation*}
On the collar, $|D\zeta|^2\le Cr^{-n}$, so Cauchy--Schwarz bounds
the cross term by $CR^{-3}(\log R)^n$. We sum over dyadic annuli
and obtain, uniformly in the plateau length,
\begin{equation}\label{kltvp}
\int\widetilde a\zeta^2\,dx
 \le\kappa\int|D\zeta|^2\,dx+
       \int_\Omega g_+\zeta^2\,dx+O_{R_1}(1).
\end{equation}
The added term $c\int_{\Omega^c}\zeta^2\,dx$ is bounded by \eqref{gzsj}.
Thus from \eqref{kltvp} we deduce \eqref{gxhbmb} for the full spherical mean
of $\widetilde a$.

We push the scalar measure in \eqref{alr}, weighted by $|x|^{2-n}$,
and the radial component of the vector measure, weighted by
$|x|^{1-n}$, forward by $x\mapsto\log|x|$. Let us divide both measures
by $|S^{n-1}|$ and denote them by $E$ and $q$, respectively.
Their interior parts agree with the densities used in
Lemma~\ref{lemma}, with $e$ replaced by $e_+$. Their boundary parts have finite total variation.
Indeed, the dyadic sum in \eqref{kwwk} is bounded by
\begin{equation*}
 C\sum_{j\ge0}(2^jR_1)^{-2}
                      [\log(2^jR_1)]^{n-1}<\infty.
\end{equation*}
Since $\widetilde b$ is locally BV on each exterior annulus, its
spherical mean $m$ is locally BV in $t=\log r$. Radial testing of
\eqref{moqdzv} yields, in distributions and with $d=n-2$,
\begin{equation*}
 m''+dm'\le\lambda\omega+E+q'+dq.
\end{equation*}
We set $p=m'-q$ as a distribution. Then
$(e^{dt}p)'\le e^{dt}(\lambda\omega\,dt+E)$.
The difference between the right side and the left side is a positive
distribution, hence a locally finite Radon measure. Thus
$(e^{dt}p)'$ is a Radon measure and $p$ has a locally BV
representative. In particular, $m'=p\,dt+q$ as measures.
We choose $s$ at a common Lebesgue point of $m,p$ outside the atoms
of $E,q$. We integrate the inequality for $p$ and use the
identity for $m'$ to obtain, for almost every $T>s$,
\begin{equation*}
 \begin{split}
 m(T)\le{}&m(s)+\frac{p(s)}d(1-e^{-d(T-s)})+q((s,T))\\
 &+\frac{\lambda}{d}\int_s^T
                  (1-e^{-d(T-t)})\omega(t)\,dt\\
 &+\frac1d\int_{(s,T)}(1-e^{-d(T-t)})\,dE(t).
 \end{split}
\end{equation*}
This is the measure form of \eqref{opt}. The scalar kernel is
nonnegative and bounded by $1/d$, so its boundary contribution is
bounded by $d^{-1}$ times the total variation of the boundary part of
$E$. The boundary contribution to $q((s,T))$ is bounded by the total
variation of its boundary part. Both bounds are independent of $T$.
After division by $T$, these terms are $O(T^{-1})$.
Consequently \eqref{tcxs} holds for almost every large $T$ with
the stated averages.
Jensen's inequality for the spherical means of \eqref{qdjwb} implies
the exponential lower bound for $\omega$ for almost every large $T$,
contradicting \eqref{kltvp}.
\end{proof}

\subsection{Rigidity from the averaged error bounds}
\label{averigid}

Our objective is to complete the directional rigidity theorem,
Proposition~\ref{crit}, by combining the geometric estimates with the
scalar argument. For $2\le n\le4$, cutoffs in the horizontal stability
inequality force the integral in \eqref{subv} to vanish, and
Lemma~\ref{kzero} applies. For $5\le n\le9$, we estimate the
logarithmic averages of the errors and use the extension from the
preceding subsection. Choosing the separation and approximation parameters
in order makes the left-hand side of \eqref{fsjoj} smaller than $2$
and excludes multiple layers.
 
We begin with the lower-dimensional cases, which require only the
horizontal energy bound at the origin.

\begin{proposition}\label{small}
Let $2\le n\le4$, and let $U\in C^2(\mathbb R^{n+1};[-1,1])$ be a
stable entire solution of $\Delta U=W'(U)$ with $U(0)=0$.
If $\int_{B_R}|D_xU|^2\,dX\le AR^{n-2}$ for all sufficiently large
$R$, then $U(x,z)=g(z)$ or $g(-z)$.
\end{proposition}
\begin{proof}
For the lower-dimensional cases, we apply the horizontal inequality
of Tonegawa \cite[Proposition~4, (4.7)]{Ton05}, recalled in
\eqref{szum}, with radial cutoff functions. At $n=4$ we use a
logarithmic cutoff, as recalled in \cite[Section~2.2]{FSS25}.
Take $\varepsilon=1$.
For $n<4$, we use a cutoff equal to one on $B_R$, supported in $B_{2R}$,
and with gradient bounded by $C/R$ to obtain
\begin{equation*}
 \int_{B_R}\left(|D(D_xU)|^2-|D|D_xU||^2\right)\,dX
          \le CA R^{n-4}\longrightarrow0.
\end{equation*}
For $n=4$, we take $\eta=1$ on $B_R$, zero outside $B_{R^2}$, and
linear in $\log|X|$ between these balls. By integration in radius, we obtain
\begin{equation*}
 \int |D_xU|^2|D\eta|^2\,dX
 \le\frac{C}{(\log R)^2}
       \int_{R<|X|<R^2}\frac{|D_xU|^2}{|X|^2}\,dX
 \le\frac{CA}{\log R}\longrightarrow0.
\end{equation*}
In either case the integrand in the horizontal Kato inequality vanishes wherever
$|D_xU|>0$. Lemma~\ref{kzero} proves the assertion.
\end{proof}

\begin{proof}[Proof of Proposition~\ref{crit}]
Proposition~\ref{small} treats $2\le n\le4$. We assume $5\le n\le9$.
From Lemma~\ref{plane}, we obtain the planar energy limit \eqref{bdlim} with
a fixed multiplicity $Q\ge1$. If $Q=1$, then
\begin{equation*}
 \lim_{R\to\infty}R^{-n}
    \int_{B_R}\left(\frac12|DU|^2+W(U)\right)\,dX=\sigma_0\omega_n.
\end{equation*}
From Wang's theorem \cite[Theorem~11.1]{Wang17}, we obtain a one-dimensional
heteroclinic solution.  The tangential growth and $U(0)=0$
fix the vertical direction and the translation, as in the last
argument of Lemma~\ref{kzero}.

Suppose now that $Q\ge2$. For any sufficiently large $D$ and
sufficiently small $\nu>0$, we obtain from Proposition~\ref{orth} the
orthogonal decomposition on a component $\Omega$ of asymptotically
full spherical measure. We use its convention for the lower well,
reflecting both $U$ and $W$ if necessary. The layer functions satisfy
\begin{equation}\label{kmod}
 \begin{gathered}
 f_{i+1}-f_i\ge D,\qquad |Df_i|,|\phi|\le\nu,\qquad
 \int_{\mathbb R}\phi^2\,dz\le\nu^2,\\
 \int_{\Omega\cap\{R<|x|<2R\}}\sum_i|Df_i|^2\,dx
                  \le A_1R^{n-2}.
 \end{gathered}
\end{equation}
Here $A_1$ is independent of $D,\nu$. The covering constants and
the starting radius may depend on them. We define $a,b$ by
\eqref{clp}, using the coefficients \eqref{gtcoef}.

We take
\begin{equation*}
 \zeta(x)=|x|^{-(n-2)/2}\psi(\log|x|),
\end{equation*}
where $\psi$ equals one on $[s,T]$ and has unit ramps at both ends.
We multiply $\zeta$ by the cutoff around the removed balls and set it
zero on the other components, obtaining an admissible test $\rho$.
For a sufficiently large fixed $s$, the cutoff estimates from
Subsection~\ref{scalext} and \eqref{kmod} yield
\begin{equation}\label{ehwhw}
 \begin{split}
 \frac1{|\mathbb S^{n-1}|}\int|D\rho|^2\,dx
       &\le\frac{(n-2)^2}{4}T+O(1),\\
 \int\rho^2\sum_i|Df_i|^2\,dx&\le CA_1T+O(1).
 \end{split}
\end{equation}
By approximation we may use $\rho$ in \eqref{vnb}. It follows that
\begin{equation}\label{tfayyf}
 \begin{split}
 \int\rho^2\int_{\mathbb R}(|D_x\phi|^2+\phi^2)\,dz\,dx
 \le{}&C\nu^2\int\rho^2\sum_i|Df_i|^2\,dx
       +C\epsilon_D\int a\rho^2\,dx
       +C\nu^2\int|D\rho|^2\,dx.
 \end{split}
\end{equation}
By substitution in \eqref{ctyq}, we obtain, for every finite $T$,
\begin{equation}\label{qgr}
 \begin{split}
 (1-C_\tau\epsilon_D)\int a\rho^2\,dx
 \le{}&[1+C(\epsilon_D+\tau)+C_\tau\nu^2]\int|D\rho|^2\,dx\\
 &+[C(\epsilon_D+\tau)+C_\tau\nu^2]
                          \int\rho^2\sum_i|Df_i|^2\,dx.
 \end{split}
\end{equation}
For fixed $\tau$, we choose $D$ large enough that
$C_\tau\epsilon_D<1/2$ and absorb the interaction term.
For each $\eta>0$, we first choose $\tau$ small, then $D$ large, and
finally $\nu$ small. The constants $C_\tau$ in \eqref{mny},
\eqref{ctyq}, and \eqref{qgr} are fixed before the latter two choices.
We fix the resulting exterior domain and $s$. Equations
\eqref{ehwhw}, \eqref{tfayyf}, and \eqref{qgr} yield
\begin{equation}\label{arnv}
 \begin{split}
 &\limsup_{T\to\infty}\frac1{|\mathbb S^{n-1}|T}\int a\rho^2\,dx
       \le\frac{(n-2)^2}{4}+\eta,\\
 &\limsup_{T\to\infty}\frac{C_\tau}{|\mathbb S^{n-1}|T}
       \int\rho^2\int_{\mathbb R}(|D_x\phi|^2+\phi^2)\,dz\,dx
       \le\eta.
 \end{split}
\end{equation}

The same bounds hold on the full weighted plateau in $\Omega$.
Indeed, $a$ and $\sum_i|Df_i|^2$ are bounded, and
\eqref{cda}, \eqref{sumg}, and \eqref{kmod} yield
\begin{equation*}
 \begin{split}
 \int_{\mathbb R}|D_x\phi|^2\,dz
 &\le2\int_{\mathbb R}|D_xU|^2\,dz+C\sum_i|Df_i|^2
       \le C(1+|x|),\\
 \int_{\mathbb R}\phi^2\,dz&\le\nu^2.
 \end{split}
\end{equation*}
On $\{R<|x|<2R\}$ the cutoff collar has volume at most
$CR^{n-4}(\log R)^n$. The lost integrals with weight $|x|^{2-n}$
are bounded by $CR^{-2}(\log R)^n$ for the interaction and
the squared slopes, and by $CR^{-1}(\log R)^n$ for the remainder. These bounds are summable over dyadic annuli.

Equations \eqref{mny}, \eqref{ctyq}, and \eqref{ojyzt} yield the
scalar inequalities on $\Omega$ with
\begin{equation*}
 \beta=1,\qquad \lambda=1+O(\epsilon_D),\qquad
 \kappa=1+O(\epsilon_D+\tau).
\end{equation*}
We use their nonnegative error bounds for $e,g$ in
Proposition~\ref{extens}. By \eqref{ehwhw} and \eqref{arnv}, the
normalized averages in \eqref{ohk} satisfy
\begin{equation*}
 E_*+G_*\le C(\epsilon_D+\tau)A_1+2\eta.
\end{equation*}
From \eqref{vtnenm}, we obtain
\begin{equation*}
 F_*\le C(\nu+\epsilon_D)\sqrt{A_1}.
\end{equation*}
Proposition~\ref{orth} provides the local boundary extensions and
the bounds $|b|\le C(1+|x|)$ and $|Db|+|F|\le C$ required in
Proposition~\ref{extens}. The interaction density $a$ is also bounded.
Since
\begin{equation*}
 \frac{n-2}{4}\le\frac74<2,
\end{equation*}
one fixed sufficiently small choice of the parameters makes
the left-hand side of \eqref{fsjoj} smaller than $2$.
Proposition~\ref{extens} leads to a contradiction. This excludes
$Q\ge2$ and completes the proof.
\end{proof}

\section{Sheeting and consequences for entire solutions}\label{sheetproof}

We combine the excess estimates with Proposition~\ref{crit} to
recover the normal direction at every interior zero and obtain
exactly $Q$ disjoint graphs. We then apply the resulting geometric
control on expanding balls. A planar blowdown leads to
one-dimensionality, and the regularity of stable limiting cones
provides planarity in ambient dimensions at most seven.

We also examine what remains valid in higher dimensions under the
additional tangential growth bound. The layer equations bound the
possible multiplicity and the amount of tangential energy required
by several layers. These estimates describe the entire limits that
can arise when the local graphical argument stops.

\subsection{Local sheeting and one-dimensional symmetry}
\label{proof}

The proof of the first theorem follows the sheeting approach of Schoen--Simon
\cite{SS81}, with an induction on multiplicity as in Wickramasekera
\cite[Section~6 and Lemma~15.1]{Wic14}. At smaller radii, either the
limiting hyperplane through the selected zero has lower multiplicity, or
we improve the excess relative to a common hyperplane for all the layers. 
For minimal hypersurfaces, dilations preserve minimality and stability.
For Allen--Cahn, rescaling also changes the transition thickness
relative to the radius. Wang \cite{Wang17} already accounts for this
thickness in excess decay and uses microscopic classification to
recover the normal direction. Here the estimates of Section~\ref{tiltpart}
control the full horizontal gradient energy throughout the finite
iteration, while the total change of the approximating normal stays
small. When the excess becomes small relative to the thickness,
they produce an entire limit satisfying the directional growth bound
in Proposition~\ref{crit}. Its classification determines
the normal at the selected zero and completes this part of the
argument. Thus the local proof starts from weak planar convergence
and obtains the geometry of every zero layer. The restriction
$n\le9$ enters through the classification of these entire limits.

For a unit vector $e$, we will denote orthogonal projection onto $e^\perp$
by $P_{e^\perp}$ and set
\begin{equation*}
 \mathcal C_r(X;e)=\{X+y+te\mid y\in e^\perp,\ |y|<r,\ |t|<r\}.
\end{equation*}
Under this notation, we have $\mathcal C_r(0;e_{n+1})=\C_r$.

We now prove Theorem~\ref{sheet}. It suffices to consider
$2\le n\le9$.
Indeed, when $n=1$, we may extend each solution by setting
$\widetilde u_j(x,y,z)=u_j(x,z)$ in $B_1^3$. 
By Fubini's theorem, the measure convergence in
\eqref{planar} extends to the horizontal plane in $B_1^3$
with the same multiplicity $Q$.
Once the case $n=2$ has been proved, we can then apply it to
$\widetilde u_j$ and restrict the resulting graphs to $y=0$. 

\begin{proof}[Proof of Theorem~\ref{sheet}]
We prove the theorem by induction on $Q$.
At the induction step, we assume the full conclusion for every
positive multiplicity less than $Q$. We first show that all interior
zeros are regular and that their normal directions converge uniformly
to the vertical direction.  

The convergence in \eqref{planar} concerns fixed compact regions.
To control every zero, we also examine neighborhoods whose radii tend
to zero along the sequence. When the rescaled thickness and tangential
excess tend to zero, the limit can contain several distinct parallel
hyperplanes, although the original limit is a single hyperplane of multiplicity
$Q$. Distances that tend to zero in the original coordinates can have
positive limits after rescaling. Monotonicity bounds the multiplicity of the hyperplane through the
selected zero by $Q$. If this multiplicity is smaller, restriction to
a neighborhood containing only that hyperplane yields the same local
problem with lower multiplicity. The induction hypothesis then
implies regularity and convergence of the normal direction at the
selected zero. When this hyperplane still has multiplicity $Q$, we use the
excess improvement and the further rescaling described below.

Let us set $e_0=e_{n+1}$ and, for a unit vector $e$, define
\begin{equation*}
 \beta(p,e)=
 \begin{cases}
 |p\cdot e|/|p|,&p\ne0,\\
 0,&p=0.
 \end{cases}
\end{equation*}
 We will prove that
$\beta(Du_j,e_0)\to1$ uniformly on the zeros of $u_j$ in each compact
interior region. This implies $u_{j,z}\ne0$ there for large $j$ and,
since
\begin{equation*}
 \frac{|D_xu_j|}{|u_{j,z}|}
       =\sqrt{\beta(Du_j,e_0)^{-2}-1},
\end{equation*}
also proves the required convergence of the slopes.
Suppose that this assertion fails. We can then choose zeros $X_j$
in one fixed compact interior region and $\eta\in(0,1)$ such that
\begin{equation}\label{mbad}
 \beta(Du_j(X_j),e_0)\le1-\eta.
\end{equation}

Let us keep $X_j$ at the center of every rescaling. The smaller
neighborhoods allow us to examine the regularity and normal direction
at this zero. The points $X_j$ may vary throughout the fixed compact
region. Excluding every sequence satisfying \eqref{mbad} will therefore
prove the desired estimate uniformly at all zeros in that region,
even though the radii used in the argument may depend on the
selected points.

Choose $r_0>0$ so that all the balls
$\overline{B_{256r_0}(X_j)}$ lie in a fixed compact subset of
$B_1^{n+1}$. At a radius $r\le r_0$ and with a unit normal $e$,
we rescale about $X_j$ by $r$ and take $e$ as the last coordinate
direction. The thickness and normalized tilt-excess become
\begin{equation*}
 \delta=\varepsilon_j/r,\qquad
 E=r^{-n}\int_{\mathcal C_{64r}(X_j;e)}
                   \varepsilon_j|P_{e^\perp}Du_j|^2\,dX.
\end{equation*}
We denote their initial values at $(r_0,e_0)$ by
$\delta_{0,j}$ and $E_{0,j}$. From Lemma~\ref{gtfene}, we obtain
$E_{0,j}\to0$ and $\delta_{0,j}\to0$. Every rotated cylinder
$\mathcal C_{128r}(X_j;e)$ lies in the chosen ambient ball. Observe that
\eqref{lmearl} bounds the full energy on these cylinders uniformly
in the chosen radius and normal.

We also need the original multiplicity to bound the mass of every
limit obtained at shrinking radii. After a subsequence, the centers
$X_j$ converge to a point of $\{z=0\}$, by Lemma~\ref{gtfene}.
For $r_j\to0$ and fixed $R>0$, we apply \eqref{monoto} between
$Rr_j$ and a fixed small radius $s$. We first pass to the limit
$j\to\infty$ on the outer ball, using nearby-ball inclusions for
the moving centers, and then let $s\downarrow0$. Thus every
limiting full-energy measure satisfies
\begin{equation}\label{dcap}
 \mu_\infty(B_R)\le Q\sigma_0\omega_nR^n
 \qquad\text{for }R>0.
\end{equation}
The same argument applies to all later choices of shrinking radii.

Before choosing the successive radii, we describe the limits that
can occur while the marked zero still satisfies a fixed angle bound.
Consider any sequence of the above rescalings for which
\begin{equation}\label{nsmal}
 E_j\to0,\qquad e_j\to e_0,\qquad \delta_j\to0,\qquad
 \beta(Du_j(X_j),e_j)\le1-\eta/2.
\end{equation}
If the radii stay bounded away from zero, the original planar
convergence yields the centered multiplicity-$Q$ hyperplane.
If the radii tend to zero, vanishing tilt-excess and stationarity
yield horizontal hyperplanes with integer multiplicities on interior
products in $\mathcal C_{64}$. Here we use the local product argument
in the proof of Lemma~\ref{lspli}. Hausdorff convergence places one
hyperplane through the origin, and \eqref{dcap} bounds its multiplicity
$m$ by $1\le m\le Q$. If $m<Q$, we restrict to a fixed small
cylinder containing only that limiting hyperplane and dilate isotropically.
From the induction hypothesis, we then deduce
$\beta(Du_j(X_j),e_j)\to1$, contradicting \eqref{nsmal}.
Consequently $m=Q$. This central hyperplane already contributes
$Q\sigma_0\omega_nR^n$ to each ball $B_R$ with $R<64$.
Any additional hyperplane meeting $\mathcal C_{32}$ would increase the
mass of one such ball, contrary to \eqref{dcap}. We have therefore
proved that the limit in $\mathcal C_{32}$ is the centered
multiplicity-$Q$ hyperplane.

We now choose the radii and normals. Fix $a\in(0,1/4)$ and
take $\theta,C_0$ from Proposition~\ref{cont}. At radius $r$, the
estimates required to pass to radius $\theta r$ are
\begin{equation}\label{istep}
 |e^+-e|\le C_0\sqrt E,\qquad E^+\le aE,
\end{equation}
where $E^+$ is measured on the rotated outer cylinder at the new
radius. We also choose $\delta_*>0$ so that the constant-zero
solution is unstable on a fixed interior ball whenever
$0<\delta\le\delta_*$. Such a choice is possible because
$W''(0)<0$. Indeed, for one fixed nonzero compactly supported test
function $\varphi$, we have
\begin{equation*}
 \delta^2\int|D\varphi|^2\,dX
       +W''(0)\int\varphi^2\,dX<0
 \qquad\text{for }0<\delta\le\delta_*.
\end{equation*}

For each sufficiently large $j$, we start at $(r_0,e_0)$.
At each step we stop if $E\le\delta$ or
$\delta\ge\theta\delta_*$. Otherwise, if a normal satisfying
\eqref{istep} exists, we choose it and replace $r$ by $\theta r$.
If every normal fails \eqref{istep}, we stop at the current radius.
This procedure is finite, since $\delta$ increases by
$\theta^{-1}$ at each step. Moreover, $\delta\le\delta_*$
throughout the procedure. At depth $k$, we obtain from the energy contraction
and the sum of the successive changes of normal
\begin{equation}\label{anboun}
 E_{k,j}\le a^kE_{0,j},\qquad
 |e_{k,j}-e_0|
 \le\frac{C_0\sqrt{E_{0,j}}}{1-\sqrt a}\longrightarrow0.
\end{equation}
This summation follows the same argument as in Wang
\cite[Lemmas~10.3--10.4]{Wang17}.
The bound on $|e_{k,j}-e_0|$ is uniform in the number of steps.
Thus the reference direction at the final radius still converges
to the original vertical direction, even when that radius tends
to zero and the number of steps increases with $j$. Since
\begin{equation*}
 |\beta(p,e)-\beta(p,e_0)|\le|e-e_0|
 \qquad\text{for every }p,
\end{equation*}
including $p=0$, the bound in \eqref{mbad} implies
$\beta(Du_j(X_j),e_{k,j})\le1-\eta/2$ at every selected depth
for large $j$.

We pass to a subsequence of the final rescalings and denote their
thicknesses, excesses, and normals by $\delta_j,E_j,e_j$.
From \eqref{anboun}, we obtain $E_j\to0$ and $e_j\to e_0$.
If $\delta_j\to0$, then $\delta_j<\theta\delta_*$ for large $j$.
The procedure must then have stopped because $E_j\le\delta_j$
or because every normal fails \eqref{istep}.
After a further subsequence, there are three possible cases.
\begin{enumerate}
\item Suppose that $\delta_j\to\delta_\infty>0$.
By elliptic compactness, we obtain a stable limit on a fixed interior
cylinder. The vanishing tilt-excess makes this limit
independent of the horizontal variables, and its value at the
origin is zero. If its vertical derivative also vanished there,
uniqueness for the one-dimensional equation, together with
$W'(0)=0$, would force the limit to be identically zero.
This contradicts the choice of $\delta_*$, since
$\delta_\infty\le\delta_*$. Its vertical derivative at the origin
is therefore nonzero. Smooth convergence implies
$\beta(Du_j(X_j),e_j)\to1$, contradicting the bound at the
marked zero.

\item Suppose that $\delta_j\to0$ and $E_j\le\delta_j$.
We apply Lemma~\ref{bterm} with $\gamma=3/4$, since
\begin{equation*}
 E_j\delta_j^{-2+2\gamma}\le\delta_j^{1/2}\longrightarrow0.
\end{equation*}
After rescaling once more by $\delta_j$ about the origin, we obtain
a bounded stable entire limit $U$ with $U(0)=0$ and both growth
bounds in \eqref{tgbroa}. Proposition~\ref{crit} implies
$U(x,z)=g(z)$ or $g(-z)$. Thus $|U_z(0)|=g'(0)>0$ and
$D_xU(0)=0$. Smooth convergence again yields
$\beta(Du_j(X_j),e_j)\to1$, since multiplication of the gradient
by a positive dilation factor leaves $\beta$ unchanged.
This contradicts the bound at the marked zero.

\item Suppose that $\delta_j\to0$, $E_j>\delta_j$, and every
normal fails \eqref{istep}. These rescalings satisfy
\eqref{nsmal}, so the preceding compactness argument shows that
their limit in $\mathcal C_{32}$ is the centered
multiplicity-$Q$ hyperplane. Their intermediate limits also satisfy
\eqref{dcap}. We can therefore apply Proposition~\ref{cont}
to this sequence. If the hyperplane through the marked zero at the
height scale has smaller multiplicity, restriction to a fixed
cylinder and the induction hypothesis contradict the angle bound.
The other alternative produces a normal satisfying \eqref{istep},
contradicting the stopping condition. The larger-domain bounds
and constants remain uniform when the final depths depend on $j$,
as required in that proposition.
\end{enumerate}
All three cases contradict \eqref{mbad}. Hence $u_{j,z}\ne0$
at every zero in each fixed compact interior region for large
$j$, and $|D_xu_j|/|u_{j,z}|\to0$ uniformly on these zeros.

It remains to obtain exactly $Q$ graphs. Fix
$B_r^n\times(-h,h)\Subset B_1^{n+1}$ and use the normal estimate
on a slightly larger product. Hausdorff convergence keeps all
zeros away from the upper and lower faces. Since $u_{j,z}\ne0$,
vertical projection of the zero set is a local diffeomorphism,
and it is proper over compact subsets of $B_r^n$. Its fibers
are finite and their cardinality is locally constant. The base
$B_r^n$ is connected, so this cardinality is constant.
Ordering the heights therefore defines smooth disjoint graphs
over the whole base. By implicit differentiation, we obtain
\begin{equation*}
 Df_{\alpha,j}(x)
 =-\frac{D_xu_j(x,f_{\alpha,j}(x))}
          {u_{j,z}(x,f_{\alpha,j}(x))}.
\end{equation*}
The uniform normal estimate implies that all these slopes tend
uniformly to zero.

We determine the number of graphs by counting the zeros on one
vertical line. We choose a fixed $s>0$ such that
$\C_{128s}\Subset B_1^{n+1}$, $B_{5s}^n\Subset B_r^n$,
and $4s<h$. We set $v_j(Y)=u_j(sY)$ and
$\delta_j=\varepsilon_j/s$, and use the upper bound
\begin{equation*}
 \widehat E_j
 =\max\left\{\int_{\C_{64}}\delta_j|D_xv_j|^2\,dY,\,
                   \delta_j\right\}.
\end{equation*}
The rescaled solutions satisfy \eqref{bound}, and their gradient
measures converge to the multiplicity-$Q$ horizontal hyperplane.
By Lemma~\ref{gtfene}, we have $\widehat E_j\to0$ and
$\delta_j^2/\widehat E_j\le\delta_j\to0$.
Remark~\ref{cntupper}, applied with an interval containing zero,
shows that, outside a subset of $B_4^n$ of measure $o(1)$,
each vertical line has exactly $Q$ zero roots in $(-4,4)$.
Hausdorff convergence places every zero of the target product
over $B_{4s}^n$ in $|z|<s$ for large $j$. Thus any one of these
lines counts the complete zero fiber in the target product.
Local constancy now makes the number of graphs equal to $Q$
on all of $B_r^n$. Finally, from Hausdorff convergence we obtain
\begin{equation*}
 \max_\alpha\|f_{\alpha,j}\|_{L^\infty(B_r)}\longrightarrow0.
\end{equation*}
This completes the induction for the graphical decomposition and
the $C^1$ convergence.

When $W''(-1)=W''(1)=k^2$ with $k>0$, the second-order estimate follows
from Wang--Wei \cite[Corollary~1.2]{WW}. Indeed, replacing $W$
by $W/k^2$ and $\varepsilon_j$ by $\varepsilon_j/k$ preserves
the equation and stability, and divides the energy by $k$.
The new potential has second derivative one at both wells.
On a slightly larger interior product, Lemma~\ref{gtfene}
provides the local energy bound. The argument above provides
exactly $Q$ ordered zero graphs and their $C^1$ convergence.
These are the three hypotheses of that corollary. Its dimension
is the ambient dimension, which is at most ten here. Applying
it on finitely many smaller cylinders, we obtain
\begin{equation*}
 \max_\alpha\|f_{\alpha,j}\|_{C^{2,\theta}(B_r)}
 \le C(r,\theta),\qquad 0<\theta<1.
\end{equation*}

For unequal values of $W''(-1)$ and $W''(1)$, we adapt the proof
of the same corollary. We explain the estimates that change.
We first impose the preliminary geometric bound in
\cite[Theorem~1.1]{WW} and use the microscopic coordinates of
\cite[Sections~4--11]{WW}. In these coordinates, the second
fundamental forms and the tangential derivatives of
$\log|\nabla u|$ are $O(\varepsilon)$. When $Q=1$, the interaction
terms vanish. The linear estimates in \cite[Section~6]{WW}
then bound the remainder and the mean curvature by
$C\varepsilon^2$, and the final paragraph below completes the
argument. We now assume $Q\ge2$. For the gap through the well $s_i$,
we set
\begin{equation*}
 \begin{gathered}
 k_i=\sqrt{W''(s_i)},\qquad
 a_i=\frac{2k_i^2A_{s_i}^2}{\sigma_0}e^{-k_i d_i},\\
 \delta=\frac{\min\{\sqrt{W''(-1)},\sqrt{W''(1)}\}}
                 {8\max\{\sqrt{W''(-1)},\sqrt{W''(1)}\}}>0.
 \end{gathered}
\end{equation*}
Here $d_i$ is the adjacent-layer distance used in the local normal
coordinates. We denote by $\mathcal A$ the largest value of
$e^{-k_i d_i}$, allowing both choices of the adjacent reference
surface, on the enlarged region for each local estimate.
The enlargements have width $C_W|\log\varepsilon|^2$, as in
\cite[Propositions~6.1 and~7.1]{WW}. All constants may depend on
$W,Q,\theta$ and the fixed interior domains, and are independent
of $\varepsilon$.

The linear estimates use positivity of $W''$ near either well
and the spectral gap on the orthogonal complement of $g'$.
These properties hold under \eqref{wells}. On a region close
to the well $s$, the outer operator is $-\Delta+W''(s)$.
Truncating the heteroclinics at distance $C_W|\log\varepsilon|$
makes the truncation error and its required derivatives
$O(\varepsilon^3)$.

  In an adjacent gap of
length $d_i$, the two differences from the common well are bounded
by $Ce^{-k_i t}$ and $Ce^{-k_i(d_i-t)}$ for $0\le t\le d_i$.
Their product is therefore bounded by $Ce^{-k_i d_i}$, and its
integral by $C(1+d_i)e^{-k_i d_i}$. On the half-gap nearest one
layer, the difference contributed by its neighbor is at most
$Ce^{-k_i d_i/2}\le C\mathcal A^{1/2}$. A farther layer has
crossed its first complete gap and at least half of a second gap.
The first gap contributes at most $\mathcal A$, and the second
at most $\mathcal A^{4\delta}$. Polynomial factors in the gap
lengths are absorbed using the unused exponential decay in each
gap. Thus the nonadjacent terms are bounded by
$C\mathcal A^{1+\delta}$.
These bounds also hold for the fixed-order derivatives of $g$
that occur in the calculation. For position derivatives, we
differentiate the full integral before dividing it at the
midpoints, as in Appendix~\ref{inter}. This avoids boundary
terms from moving midpoints.
 Some coefficients require a separate estimate away from the
common gap. The difference $W''(u)-W''(g_i)$ may approach a
nonzero constant beyond a neighboring layer. We integrate over
the part of the normal line beyond that layer. Since this region
is separated from the $i$th layer by the whole adjacent gap,
the exponential decay of $g_i'$ yields
\begin{equation}\label{farint}
 \int |g_i'|\,dt \le Ce^{-k_i d_i},
 \qquad
 \int |g_i'|^2\,dt \le Ce^{-2k_i d_i}.
\end{equation}
On the common-gap part, Taylor expansion bounds the coefficient
difference by $C|\phi|$ and the neighboring differences from the
common well. Estimate~\eqref{farint} controls the farther terms.
These regional estimates, together with the preceding
pointwise bounds for the heteroclinics, replace the estimates in
\cite[(7.4)--(7.6), (8.7), and (8.15)--(8.17)]{WW}.

For example, using
$\|\phi\|_{C^{2,\theta}}\le C(\varepsilon^2+\mathcal A)$,
we obtain the following bound for the integral considered
after \cite[(8.16)]{WW}:
\begin{equation}\label{mixcoef}
 \begin{split}
 &\int |W''(u)-W''(g_i)|\,|g_i'\phi_t|\,dt\\
 &\quad\le C(\varepsilon^2+\mathcal A)
 \left(\varepsilon^2+\mathcal A+
       \sum_{j=i-1}^{i}(1+d_j)e^{-k_jd_j}\right)\\
 &\quad\le C(\varepsilon^2+\mathcal A^{1+\delta}).
 \end{split}
\end{equation}
Terms from farther layers have already been included in the
$C\mathcal A$ bound inside the parentheses.
At the earlier stage of \cite[Appendix~B, case~(6)]{WW}, the
analogous integral contains $\phi$ in place of $\phi_t$. Before
using the estimate for $\phi$, Young's inequality bounds it by
$C\|\phi\|_{C^{2,\theta}}^2+C\mathcal A^{1+\delta}$.
This is the bound used to obtain \cite[Proposition~6.1]{WW}.
Thus the estimate for $\phi$ is available before we use it
in \eqref{mixcoef}.

The geometric estimates in \cite[Lemmas~3.1--3.5]{WW} use the
preliminary geometric bound and the equation, and apply to the
fixed potential $W$. On the truncated support, (8.1) bounds the
Jacobian error by $C\varepsilon|t|$, (8.6) bounds the error
in replacing a signed distance by its affine expression by
$C\varepsilon^{1/3}$, and (8.21) bounds the relevant normal
and projection errors by $C\varepsilon^{1/6}$.
The Hessian of each signed distance is $O(\varepsilon)$.
Thus the tangential derivative of the signed-distance error
in (8.6) is bounded by $C\varepsilon|t|$.
This bound and the $C\varepsilon$ bound for derivatives of the
second fundamental forms in their Lemma~3.1 provide the corresponding local
$C^\theta$ estimates for the distance and Jacobian errors.
The normal and projection errors have the same control by
Lemmas~3.4--3.5. In each interaction integral,
these errors multiply either a full adjacent interaction or a
nonadjacent term. Since $|t|\le C_W|\log\varepsilon|$ and
$\delta\le1/8$, their total contribution is bounded by
$C\varepsilon^\delta\mathcal A+C\mathcal A^{1+\delta}$.
The same reasoning controls the local $C^\theta$ norm in
Appendix~B, case~(6), and the tangentially differentiated
terms. The projections remain inside the enlarged domains
specified above.

These estimates verify the modifications needed in
\cite[Sections~5--8]{WW}. More precisely, the adjacent
$\mathcal A^{1/2}$ bound above is used in their Lemma~6.6
and in the second term of (7.2), so the Schauder absorption
still yields $\|\phi\|_{C^{2,\theta}}\le
C(\varepsilon^2+\mathcal A)$. For the tangential derivatives,
the chain rule introduces the derivatives of the normal
coordinates or the layer corrections. The former are bounded
by $C\varepsilon^{1/6}$, as in their Lemma~3.4, and the latter
by $C(\varepsilon^2+\mathcal A)$, as in their Lemma~4.6.
They multiply a full adjacent interaction or one of the
nonadjacent terms estimated above. The terms in their
Lemma~7.4 are consequently bounded by
$C(\varepsilon^2+\varepsilon^\delta\mathcal A+
\mathcal A^{1+\delta})$.
For the quadratic form, adjacent derivative overlaps contribute
at most $C(1+d_i)e^{-k_i d_i}$ to the coefficient of
$|D\eta_i|^2$. The same regional estimates, followed by the
normal integrations by parts in their Lemmas~8.3--8.6, bound
the remaining coefficient of $\eta_i^2$ by
\begin{equation*}
 C\bigl(\varepsilon^2+\varepsilon^\delta\mathcal A
                       +\mathcal A^{1+\delta}\bigr).
\end{equation*}
The coefficient of $\sum_i\int|D\eta_i|^2$ is
$1+O(\varepsilon^\delta+\mathcal A^\delta)$.
Thus their Propositions~6.1, 7.1, and~8.1 and Corollary~7.6
hold with these errors. The leading terms of the layer
equations are $a_{i-1}-a_i$. In the signed displacement
convention, the interaction term in the stability inequality
is $\sum_i\int k_i a_i(\eta_{i+1}-\eta_i)^2$.
Both coefficients follow from \eqref{gfwron} and differentiation
with respect to the gap. For the negative well, we apply that
calculation to $-g(-t)$ and $W(-s)$.

 Subtracting the adjacent layer equations
and multiplying by $k_i$ produces $2k_i a_i$ in the equation
for $k_i d_i$. The interactions with the exterior neighbors
have the favorable sign. In the stability inequality, we take
the two signed displacements to be $-\zeta$ and $\zeta$ and
the others to be zero. Division by two produces the same
coefficient $2k_i a_i$. Subtracting
$\log(4k_i^3A_{s_i}^2/\sigma_0)$ from $k_i d_i$
therefore makes the two leading coefficients equal to one.
At interaction size $\kappa\ge M\varepsilon^2|\log\varepsilon|$,
the rescaled errors in \cite[Section~9]{WW} are bounded by
\begin{equation*}
 C\left(\frac{1}{M|\log\varepsilon|}
                 +\varepsilon^\delta+\kappa^\delta\right).
\end{equation*}
Their integral argument requires $q<2$ and
$n-2(2q+1)<0$, which is possible for every $n\le9$.

The iteration in \cite[Section~10.1]{WW} now yields
$\mathcal A\le\varepsilon^{1+\theta_0}$ for every fixed
$\theta_0<1$. We choose $\theta_0>1-\delta$.
The errors in the layer equations are then $O(\varepsilon^2)$,
and the subsequent iterations yield
$\mathcal A\le C\varepsilon^2|\log\varepsilon|$.
For the final improvement in \cite[Section~10.2]{WW}, we use
$k_i$ times the intrinsic distance between the two layers.
The geometric term remains nonnegative. Their relative
gradient estimate and the bound
$C\varepsilon^3|\log\varepsilon|^2$ for the nonlinear
geometric error remain valid, since $k_i$ is a fixed positive
number. The differential and stability inequalities again
have matching coefficients. The same rescaling therefore
implies $\mathcal A\le C\varepsilon^2$.

As in \cite[Section~11]{WW}, this bounds the microscopic
$C^\theta$ norm of the mean curvature by $C\varepsilon^2$.
In the original coordinates, the graph mean-curvature equation
yields the required uniform $C^{2,\theta}$ estimate.
 \end{proof}

\begin{remark}
The compactness and continuation arguments are independent of the
dimension. For the graphical decomposition, the restriction in
Theorem~\ref{sheet} enters through Proposition~\ref{crit} at the
microscopic terminal scale. For several layers, the second-order
estimate uses the same dimension range through the scalar
stability argument in \cite[Section~9]{WW}.
\end{remark}

The next proposition extends local sheeting for multiplicity-one
limits to every $n\ge2$.  

\begin{proposition}
For $Q=1$, the conclusion of Theorem~\ref{sheet} holds for every $n\ge2$.
\end{proposition}
\begin{proof}
We use the finite iteration from the proof of Theorem~\ref{sheet}.
For $Q=1$, the density bound forces every pointed limit for which the
thickness tends to zero to be the centered multiplicity-one hyperplane on the
required interior cylinder. There is only one limiting harmonic height,
so the condition that all these heights have the same gradient holds
automatically. All continuation estimates are therefore
available in every dimension.

At a stopping scale with $E_j\le\delta_j$ and $\delta_j\to0$,
Lemma~\ref{bterm} yields a bounded stable entire solution with critical
tangential growth.
The original monotonicity argument bounds its density at infinity by
one. The existence of a zero implies positive integer multiplicity, so its blowdown
has multiplicity one. Wang's theorem \cite[Theorem~11.1]{Wang17}
classifies it as a heteroclinic. The tangential growth forces a vertical
normal. This proves the stopping conclusion used in the local proof,
independently of Proposition~\ref{crit}. The other stopping cases,
uniform bound on the total rotation, and final graph counting are
unchanged. Hausdorff convergence implies convergence of the graph heights.
The single-layer second-order argument at the end of the preceding
proof is independent of dimension and completes the uniform
$C^{2,\theta}$ estimate.
\end{proof}
 
We first classify stable entire solutions that have a planar blowdown.
We then use regularity of limiting varifolds to prove the classification
under the energy growth assumption and discuss the finite-index case.

The following half-space result shows that a bounded positive solution
with zero boundary values equals $g(z)$. It will determine the solution
above its highest zero hyperplane in Proposition~\ref{prigi}.

\begin{lemma}\label{hrigi}
Suppose $v$ is a bounded classical solution of
\begin{equation*}
 \Delta v=W'(v)\quad\text{in }\mathbb R^{N-1}\times(0,\infty),
 \qquad 0<v\le1,\qquad v(x,0)=0.
\end{equation*}
Then $v(x,z)=g(z)$ for $z\ge0$.
\end{lemma}
\begin{proof}
Farina \cite[Lemma~2.3]{Farina03}, applied to $f=-W'$, yields
$\sup v=1$. The function $-W'$ has a globally Lipschitz extension
from $[0,1]$, and $W'(1)=0$. Thus the half-space symmetry theorem
of Berestycki--Caffarelli--Nirenberg \cite[Theorem~1.4]{BCN97}
yields $v(x,z)=h(z)$ with $h'>0$ and $h(z)\to1$.
The first integral $h'^2/2=W(h)$ and the boundary value $h(0)=0$
imply $h=g$.
For a negative solution, we apply the same argument to its negative
with potential $W(-s)$. The original solution is then $g(-z)$ for the
original potential.
\end{proof}

For $3\le N\le10$, the next proposition shows that a stable entire
solution with a planar blowdown is a single heteroclinic.
This consequence of local sheeting reduces the proof of
Theorem~\ref{global} to showing that the limiting varifold is
supported on a hyperplane.
\begin{proposition}\label{prigi}
Suppose $3\le N\le10$ and $U\in C^2(\mathbb R^N;[-1,1])$ is stable and
solves $\Delta U=W'(U)$. Suppose there are $R_j\to\infty$, a hyperplane $P$ through
the origin, and $Q\in\mathbb N$ such that
\begin{equation}\label{epblow}
 R_j^{-1}|D(U(R_j\,\cdot))|^2\,dX
 \rightharpoonup Q\sigma_0\mathcal H^{N-1}\lfloor P
 \quad\text{locally in }\mathbb R^N.
\end{equation}
Then $Q=1$ and $U(X)=g(e\cdot X-b)$ for some
$e\in\mathbb S^{N-1}$ and $b\in\mathbb R$.
\end{proposition}
\begin{proof}
After a rotation, we have $P=\{z=0\}$. We apply Theorem~\ref{sheet} to
$U(R_j\,\cdot)$ on a fixed product in $B_1$ and obtain $Q$ zero graphs
whose slopes tend uniformly to zero. In particular, $\{U=0\}$ is
nonempty. Every fixed zero $(x,z)$ of $U$ belongs, after rescaling,
to one of these graphs for all large $j$, and
\begin{equation*}
 U_{x_i}(x,z)+U_z(x,z)\partial_i f_{\alpha,j}(x/R_j)=0.
\end{equation*}
We let $j\to\infty$ and obtain $D_xU(x,z)=0$. The same theorem implies
$U_z(x,z)\ne0$. The implicit function theorem therefore represents
the zero set locally by constant graphs.
For each height $b$ at which a zero occurs, $\{x\mid U(x,b)=0\}$ is
closed by continuity and open by this local description. It equals
$\mathbb R^{N-1}$. There are at most $Q$ such heights, since $Q+1$ fixed
zero hyperplanes would produce $Q+1$ intersections with a vertical line in
the rescaled product for large $j$.

We denote the height of the highest zero hyperplane by $b$. On $z>b$ the solution has one sign,
so Lemma~\ref{hrigi} implies that it equals $g(z-b)$ or $g(b-z)$.
We denote this entire one-dimensional solution by $V$. Then
\begin{equation*}
 \Delta(U-V)=\left[\int_0^1W''(V+t(U-V))\,dt\right](U-V).
\end{equation*}
By unique continuation \cite{Aronszajn}, we have $U=V$ throughout
$\mathbb R^N$. The limiting measure of $R|DU(RX)|^2\,dX$ has multiplicity one,
and hence $Q=1$.
\end{proof}

We can now prove the second main theorem. The argument below also
covers the lower-dimensional cases by adding independent coordinates.
 
\begin{proof}[Proof of Theorem~\ref{global}]
It suffices to prove the result for $5\le N\le7$. Indeed, if $N<5$,
we extend $U$ to $\mathbb R^5$ by setting
$\widetilde U(X,Y)=U(X)$. We apply stability on each slice with
$Y$ fixed and integrate in $Y$. Adding the nonnegative term from
the $Y$ derivatives proves stability of $\widetilde U$.
Fubini's theorem bounds its energy in $B_R^5$ by
$CR^{5-N}R^{N-1}=CR^4$. The five-dimensional conclusion
then applies. Since $\widetilde U$ is independent of $Y$ and $g'>0$,
the resulting one-dimensional direction lies in $\mathbb R^N$.

We therefore assume $5\le N\le7$. By the range of $U$, we have
$W(U)>0$. Modica's inequality and monotonicity
\cite{Modica85,Modica88}, together with \eqref{natgrowth}, imply
\begin{equation*}
 0<\lim_{R\to\infty}R^{1-N}
 \int_{B_R}\left(\frac12|DU|^2+W(U)\right)\,dX<\infty.
\end{equation*}
For any $R_j\to\infty$, we obtain from the compactness theorem of
Hutchinson--Tonegawa \cite[Theorem~1 and Proposition~4.3]{HT}, after a
diagonal extraction, a stationary integral varifold $V$ on
$\mathbb R^N$. Its weight is $\sigma_0^{-1}$ times the limiting full
energy measure, by equipartition. Thus
$r^{1-N}\|V\|(B_r)$ is the same positive constant for every $r>0$.
Equality in the stationary monotonicity formula implies that $V$ is a cone.

The assumptions on $W$ satisfy (A1) of
Tonegawa--Wickramasekera \cite{TW12}. The bound \eqref{natgrowth}, the
range $[-1,1]$, and stability of the dilated solutions verify (A2)--(A3)
on every fixed ball. Their Theorem~2.1 makes the support of $V$ smooth
at the origin when $N\le7$. A cone smooth at its vertex is a hyperplane.
Stationarity and integrality then imply a constant positive integer
multiplicity. Equipartition yields \eqref{epblow}, and
Proposition~\ref{prigi} proves the assertion in these dimensions and
completes the proof.
\end{proof}

\subsection{Transition layers in higher dimensions}
\label{higher}\label{highlayers}

The objective of this subsection is to determine what the total energy
and uniform horizontal energy bounds imply beyond the dimension range
of Proposition~\ref{crit}. The same averaged equations and stability tests
imply restrictions on the blowdown multiplicity and a positive lower
bound for the horizontal energy growth of multiple layers. When
$W''(-1)=W''(1)$, the multiplicity $Q\ge2$ must satisfy
$n\ge2+4Q(Q-1)$. A sufficiently small horizontal growth coefficient
still forces one-dimensionality.

Finally, we return to local sheeting. If regularity of the zeros or
uniform convergence of their normal lines fails, the rescaling
argument of Subsection~\ref{proof} produces a stable entire solution
subject to these restrictions. For $10\le n\le25$, this
solution must have planar blowdown multiplicity two, even when
$W''(-1)\ne W''(1)$. Proposition~\ref{reduce} establishes the precise
reduction.  

Throughout this subsection $n\ge5$, and $W$ satisfies \eqref{wells}.
Unless stated otherwise, $U$ is a bounded stable entire solution,
$U(0)=0$, and
\begin{equation}\label{hassu}
 \begin{aligned}
 \int_{B_R(Y)}|D_xU|^2\,dX&\le AR^{n-2}
       &&\text{for }Y\in\mathbb R^{n+1},\ R\ge R_*,\\
 \int_{B_R(0)}\left(\frac12|DU|^2+W(U)\right)\,dX&\le KR^n
       &&\text{for }R\ge R_*.
 \end{aligned}
\end{equation}
The constants $A,K,R_*$ are fixed, and the tangential estimate is
uniform in the center.

The geometric conclusions of Subsections~\ref{exzeros}
and~\ref{ordgraphs} remain available
under these assumptions. Lemma~\ref{plane} establishes the planar limit
\eqref{bdlim}, with integer multiplicity $Q$, and  
\eqref{cone}. The measures $R|DU(RX)|^2\,dX$ have the same limit.
If $Q=1$, Wang's theorem \cite[Theorem~11.1]{Wang17} makes $U$
one-dimensional. Its normal is vertical, since a nonzero horizontal
component would produce tilt-excess of order $R^n$, contrary
to \eqref{hassu}. Since $U(0)=0$, we have $U=g(z)$ or $g(-z)$.

We also use the zero graphs constructed in
Subsection~\ref{ordgraphs}.
For any fixed $M,\delta>0$, Lemmas~\ref{cost} and \ref{geomet}
cover the zeros satisfying
\begin{equation*}
 \min_{\pm}\|U(P+\cdot)-g(\pm z)\|_{C^2(B_{2M})}\ge\delta
\end{equation*}
by at most $C_{M,\delta}R^{n-4}$ fixed-radius balls in $B_R$.
We may increase the approximation radius beyond $M$ and decrease
the error below $\delta$ as required for the graph construction.
After projection onto $\mathbb R^n$ and enlargement by a fixed
multiple of $\log(2+|x|)$, the complement of these balls has an
unbounded component $\Omega$ satisfying \eqref{nffq} and \eqref{uyctu}.
The zero set above $\Omega$ consists of exactly $Q$ ordered smooth
graphs, by \eqref{pqmaw}. These arguments apply for every $n\ge5$,
since $2(n-1)/(n-2)>2$. The component and its constants may depend
on the chosen approximation.

We choose the lower limiting well to be $-1$, replacing $(U,W)$ by
$(-U,W(-\,\cdot))$ once if necessary. Thus $s_i=(-1)^{i-1}$ is the well
between the $i$th and $(i+1)$st transition layers. We use the constants
$c_i$ from \eqref{gtcoef}.
In particular, $-\sum_{j=1}^i c_j>0$ for $i<Q$.

For the following lemma, assume that the integer $Q$ in
\eqref{bdlim} satisfies $Q\ge2$. Fix a component $\Omega_0$ from
the preceding construction, with sufficiently accurate approximation by
vertical heteroclinics. Denote its zero graphs by $h_1<\cdots<h_Q$
and extend them by zero outside $\Omega_0$ when taking spherical means.
For any sufficiently large fixed radius $R_1$, we consider
\begin{equation}\label{hmdefi}
 \begin{aligned}
 &\frac1T\fint_{\mathbb S^{n-1}}
           (h_i\mathbf1_{\Omega_0})(e^T\theta)\,d\theta,\\
 &\frac1{|\mathbb S^{n-1}|T}
 \int_{\Omega_0\cap\{R_1<|x|<e^T\}}
 \frac{2W''(s_i)A_{s_i}^2}{\sigma_0}
 e^{-\sqrt{W''(s_i)}(h_{i+1}-h_i)}|x|^{2-n}\,dx.
 \end{aligned}
\end{equation}
We take $T\to\infty$ through regular radii. Limits of spherical means
are essential limits of their almost-everywhere representatives.

The next lemma bounds these averages and relates their limits
through the layer equations and stability. These conclusions will
yield the multiplicity restriction in Proposition~\ref{multi} and the
tilt-excess lower bound in Proposition~\ref{egap}.

\begin{lemma}\label{haver}
The quantities in \eqref{hmdefi} are bounded. Every simultaneous
subsequential limit, denoted respectively by
$\beta_1,\ldots,\beta_Q\in\mathbb R$ and
$\alpha_1,\ldots,\alpha_{Q-1}\ge0$, satisfies, with
$\alpha_0=\alpha_Q=0$,
\begin{equation}\label{hlequa}
 (n-2)\beta_i=\alpha_{i-1}-\alpha_i,\qquad
 \beta_{i+1}-\beta_i\ge\frac2{\sqrt{W''(s_i)}},
\end{equation}
\begin{equation}\label{hlstab}
 \sum_{i=1}^{Q-1}\sqrt{W''(s_i)}\alpha_i(v_{i+1}-v_i)^2
 \le\frac{(n-2)^2}{4}\sum_{i=1}^Q v_i^2
 \qquad\text{for }v\in\mathbb R^Q.
\end{equation}
Moreover,
\begin{equation}\label{hgmean}
 \liminf_{R\to\infty}
 \frac{1}{\log R}\fint_{\mathbb S^{n-1}}
 [(h_{i+1}-h_i)\mathbf1_{\Omega_0}](R\theta)\,d\theta
 \ge\frac2{\sqrt{W''(s_i)}}.
\end{equation}
\end{lemma}

\begin{proof}
The statement concerns the actual zero heights $h_i$ on the
fixed domain $\Omega_0$. To use the projected layer equations,
we work with the functions $f_i$ obtained from the orthogonality
conditions on a domain $\Omega$. We can improve this approximation
by changing $\Omega$ and $f_i$, while keeping $\Omega_0$ and $h_i$
fixed. We first estimate the averages involving $f_i$.
We then use the small difference between $f_i$ and $h_i$ on
$\Omega\cap\Omega_0$, together with the bounds on the omitted
regions, to transfer these estimates to $h_i$. In this last step,
we first let the radius tend to infinity and then let the
approximation error tend to zero.

For any large $D$ and small $\nu$, we obtain from Proposition~\ref{orth}
a component $\Omega$ and orthogonal layer functions $f_i$ whose
distance from the actual ordered zero heights is at most $\nu$.
In particular, $|f_i-h_i|\le\nu$ on $\Omega\cap\Omega_0$.
We have $|Df_i|,|\phi|\le\nu$, $\int\phi^2\,dz\le\nu^2$, and
\begin{equation*}
 \int_{\Omega\cap\{R<|x|<2R\}}\sum_i|Df_i|^2\,dx\le A_1R^{n-2},
\end{equation*}
where $A_1$ is independent of $D,\nu$.

We use the interaction densities \eqref{hidefi} and the projected
equations \eqref{hsequa}, with the estimates \eqref{hserro} at approximation
error $\nu$. Constants below may depend on $n,W,Q,A,K$.
We choose $\psi\in C_c^1(\mathbb R)$ with fixed support and equal to one
on $[0,1]$. We multiply $|x|^{-(n-2)/2}\psi(\log|x|-T)$ by the cutoff
around the removed balls and denote the result by $\rho$. Its Dirichlet integral
and $\int\rho^2\sum_i|Df_i|^2\,dx$ are bounded independently of $T$.
The inequality for $\phi$ in \eqref{vnb} and scalar stability
\eqref{ctyq} yield
\begin{align*}
 \int\rho^2\left(\int(|D\phi|^2+\phi^2)\,dz\right)\,dx
 &\le C\nu^2\int\rho^2\sum_i|Df_i|^2\,dx
     +C\epsilon_D\int\rho^2\sum_i a_i\,dx\\
 &\quad+C\nu^2\int|D\rho|^2\,dx,\\
 \int\rho^2\sum_i a_i\,dx
 &\le C\int|D\rho|^2\,dx+C(\epsilon_D+\tau)
       \int\rho^2\sum_i|Df_i|^2\,dx\\
 &\quad+C_\tau\int\rho^2\left(\int(|D_x\phi|^2+\phi^2)\,dz\right)\,dx.
\end{align*}
The scalar density in \eqref{clp} is comparable to $\sum_i a_i$.
Fix $\tau$ and substitute the first inequality into the second.
The resulting term $C_\tau\epsilon_D\int\rho^2\sum_i a_i\,dx$ can be
absorbed by taking $D$ sufficiently large. We then choose $\nu$ small.
On every unit logarithmic annulus sufficiently far out, we obtain
\begin{equation}\label{huannu}
 \int_{\Omega\cap\{e^T<|x|<e^{T+1}\}}
                  \sum_i a_i\,|x|^{2-n}\,dx\le C.
\end{equation}
Moreover the weighted integral of $\sum_i|e_i|$ there and the integral
of $\sum_i|F_i||x|^{1-n}$ can be made at most any prescribed
$\eta>0$, apart from summable errors from the removed balls. To obtain
this estimate, we choose $\tau$ first, then $D$, then $\nu$, including
$C_\tau$ in the choice of the last two parameters.

These statements hold without the cutoff around the removed balls.
In fact the interactions and $\sum_i|Df_i|^2$ are bounded, and
\begin{equation*}
 \int_{\mathbb R}(|D_x\phi|^2+\phi^2)\,dz\le C(1+|x|).
\end{equation*}
The latter follows from the full normal energy estimate for $U$ and
$D_x\phi=D_xU+\sum_i s_i g'(s_i(z-f_i))Df_i$.
On a dyadic annulus the regions where the cutoff differs from one have volume
at most $CR^{n-4}(\log R)^n$. The lost weighted integrals are therefore
$O(R^{-2}(\log R)^n)$ for bounded integrands and
$O(R^{-1}(\log R)^n)$ for the last integrand. Both are summable.
The flux estimate also follows directly by Cauchy--Schwarz and the
annular bound for $\sum_i|Df_i|^2$.

The vector inequality \eqref{evjhr}, tested with
$\eta_i=v_i\rho$ on a long logarithmic annulus, tells us that
\begin{equation}\label{hvaver}
 \begin{split}
 &\limsup_{T\to\infty}\frac1{|\mathbb S^{n-1}|T}
 \int_{\Omega\cap\{R_1<|x|<e^T\}}
 \sum_{i=1}^{Q-1}\sqrt{W''(s_i)}a_i(v_{i+1}-v_i)^2|x|^{2-n}\,dx\\
 \le&
 \left(\frac{(n-2)^2}{4}+C\eta\right)\sum_{i=1}^Q v_i^2.
 \end{split}
\end{equation}
The end cutoffs have fixed logarithmic length. Their costs are bounded
independently of $T$ by \eqref{huannu} and the estimates
just proved.

We extend $f_i$ by zero across the complement of $\Omega$.
The distributional identities in Proposition~\ref{extens}
apply to each equation \eqref{hsequa}. The added scalar
boundary measure is
$(F_i\cdot\nu-\partial_\nu f_i)\mathcal H^{n-1}\lfloor\partial\Omega$.
The added vector measure is
$-f_i\nu\mathcal H^{n-1}\lfloor\partial\Omega$.
Their total variations with weights $|x|^{2-n}$ and $|x|^{1-n}$,
respectively, are finite outside a fixed ball. This uses the bounded
traces of $Df_i,F_i$, the bound $|f_i|\le C(1+|x|)$, and
\eqref{rgg}. These facts hold for every $n\ge5$.

Let us set $d=n-2$ and
\begin{equation*}
 m_i(t)=\fint_{\mathbb S^{n-1}}(f_i\mathbf1_\Omega)(e^t\theta)\,d\theta,
 \qquad
 \omega_i(t)=e^{2t}\fint_{\mathbb S^{n-1}}
                     (a_i\mathbf1_\Omega)(e^t\theta)\,d\theta.
\end{equation*}
The spherical averages of the extended equations read
\begin{equation*}
 m_i''+d m_i'=\omega_{i-1}-\omega_i+E_i+q_i'+d q_i
\end{equation*}
in distributions. Here $E_i$ and $q_i$ are the weighted pushforwards
of the scalar errors and radial fluxes. Their absolute integrals on
$[s,T]$ are at most $C\eta T+O(1)$, including the finite boundary
measures. Fix $s$ sufficiently large and integrate as in \eqref{opt},
with equality for the present equation. This yields an exponential convolution
of $\omega_{i-1}-\omega_i+E_i$ and the accumulated flux $q_i$.
By \eqref{huannu}, the exponential convolution of each $\omega_i$ is
bounded. The error and flux integrals, divided by $T$, are
$O(\eta)+o(1)$. Consequently
\begin{equation}\label{hmfrel}
 \frac{m_i(T)}T=
 \frac1{dT}\int_s^T(\omega_{i-1}-\omega_i)\,dt+O(\eta)+o(1).
\end{equation}
In particular the normalized means are bounded.

We next prove the lower bound for each gap. The first derivative of a
zero extension is
\begin{equation*}
 D(f_i\mathbf1_\Omega)=\mathbf1_\Omega Df_i\,dx
             -f_i\nu\mathcal H^{n-1}\lfloor\partial\Omega.
\end{equation*}
The annular gradient bound and Cauchy--Schwarz inequality yield
\begin{equation*}
 \int_{\Omega\cap\{e^T<|x|<e^{T+1}\}}|Df_i|\,|x|^{1-n}\,dx
 \le C A_1^{1/2}.
\end{equation*}
Together with the finite weighted boundary variation, this bounds the
variation of $m_i$ on every unit interval sufficiently far out. We set
$\ell_i=m_{i+1}-m_i$ and add the constant
$2W''(s_i)A_{s_i}^2/\sigma_0$ to the interaction density on
$\Omega^c$. Its weighted integral at infinity is finite by
\eqref{gzsj}. We then use Jensen's inequality on the sphere to obtain
\begin{equation*}
 e^{2t}\fint_{\mathbb S^{n-1}}\left[
 a_i\mathbf1_\Omega+
 \frac{2W''(s_i)A_{s_i}^2}{\sigma_0}\mathbf1_{\Omega^c}\right](e^t\theta)\,d\theta
 \ge \frac{2W''(s_i)A_{s_i}^2}{\sigma_0}
      e^{2t-\sqrt{W''(s_i)}\ell_i(t)}.
\end{equation*}
Its integral on $[T,T+1]$ is bounded by
\eqref{huannu} and \eqref{nffq} and \eqref{uyctu}.
The bounded oscillation of $\ell_i$ on this interval therefore implies
\begin{equation}\label{hmgmea}
 \sqrt{W''(s_i)}\ell_i(T)\ge2T-C
 \quad\text{for almost every sufficiently large }T.
\end{equation}
The constant may depend on the fixed approximation parameters.

Every component used above and $\Omega_0$ contains asymptotically all
of each large sphere. On their intersection the $i$th actual zero has
the same label, and its distance from $f_i$ is at most $\nu$.
On the rest of the sphere all heights are bounded by $C(1+R)$ and the
relative measure is $O(R^{-2(n-1)/(n-2)}(\log R)^{n-1})$.
Consequently the means of $f_i\mathbf1_\Omega$ and
$h_i\mathbf1_{\Omega_0}$ differ by at most $\nu+o(1)$.
On the intersection the corresponding interaction densities differ
by a multiplicative factor between $e^{-C\nu}$ and $e^{C\nu}$.
On its complement their weighted integrals at infinity are finite,
since the densities are bounded and \eqref{gzsj} holds for both
components. Thus their normalized logarithmic integrals differ by
$O(\nu)+o(1)$, using their uniform boundedness.

It follows that the quantities in \eqref{hmdefi} are
bounded and that \eqref{hgmean} holds. We select any subsequence
on which all those finitely many quantities converge. We apply
\eqref{hmfrel} and \eqref{hvaver} using successively more accurate
approximations and then let $\eta,\nu$ tend to zero. A countable choice of these approximations leaves one
common set of full measure of admissible radii. The limits satisfy
\eqref{hlequa} and \eqref{hlstab}.
The same argument applies to every convergent subsequence.
\end{proof}

When $W''(-1)=W''(1)$, Wang \cite[Definition~1.2, Remark~1.3 and
Proposition~4.1]{WangToda19} obtains the dimension bound $n\ge2+4Q(Q-1)$
for homogeneous Toda solutions. The averaged layer equations imply
the following necessary condition for the Allen--Cahn solution with
possibly different values of $W''$ at the two wells.

\begin{proposition}\label{multi}
Assume \eqref{hassu}, and let $Q\ge2$ be the multiplicity
of the limiting hyperplane. With the constants in
\eqref{gtcoef}, one necessarily has
\begin{equation}\label{hscond}
 \begin{split}
 &2(n-2)\sum_{i=1}^{Q-1}\sqrt{W''(s_i)}
       \left(-\sum_{j=1}^i c_j\right)(v_{i+1}-v_i)^2\le\frac{(n-2)^2}{4}\sum_{i=1}^Qv_i^2
       \qquad\text{for }v\in\mathbb R^Q.
 \end{split}
\end{equation}
If $W''(-1)=W''(1)$, this is equivalent to
\begin{equation}\label{hecthr}
 n\ge 2+4Q(Q-1),\qquad
 Q\le\frac{1+\sqrt{n-1}}2.
\end{equation}
\end{proposition}
\begin{proof}
We subtract consecutive equations in \eqref{hlequa} and obtain
\begin{equation*}
 2\alpha_i-\alpha_{i-1}-\alpha_{i+1}
 \ge\frac{2(n-2)}{\sqrt{W''(s_i)}},\qquad \alpha_0=\alpha_Q=0.
\end{equation*}
The numbers $-\sum_{j=1}^i c_j$, with zero values at $i=0,Q$,
satisfy
\begin{equation*}
 2\left(-\sum_{j=1}^i c_j\right)
 -\left(-\sum_{j=1}^{i-1}c_j\right)
 -\left(-\sum_{j=1}^{i+1}c_j\right)
 =\frac1{\sqrt{W''(s_i)}}.
\end{equation*}
From the discrete maximum principle, we therefore obtain
\begin{equation*}
 \alpha_i\ge2(n-2)\left(-\sum_{j=1}^i c_j\right).
\end{equation*}
Substitution into \eqref{hlstab} proves
\eqref{hscond}.

We next suppose that $W''(-1)=W''(1)=k^2$ with $k>0$. Then
$c_i=(i-(Q+1)/2)/k$ and
$-\sum_{j=1}^i c_j=i(Q-i)/(2k)$. Condition
\eqref{hscond} becomes
\begin{equation*}
 \sum_{i=1}^{Q-1}i(Q-i)(v_{i+1}-v_i)^2
 \le\frac{n-2}{4}\sum_{i=1}^Qv_i^2.
\end{equation*}
The matrix of this quadratic form has largest eigenvalue $Q(Q-1)$,
as computed in Wang \cite[Remark~1.3]{WangToda19}. This proves
\eqref{hecthr}.

For general $W$ satisfying \eqref{wells}, we first consider three layers and set
$k_1=\sqrt{W''(s_1)}$ and $k_2=\sqrt{W''(s_2)}$ in this calculation.
The matrix with quadratic form
\begin{equation*}
 2\sum_{i=1}^{2}k_i\left(-\sum_{j=1}^i c_j\right)
                  (v_{i+1}-v_i)^2
\end{equation*}
has edge weights
\begin{equation*}
 \frac23\left(2+\frac{k_1}{k_2}\right),\qquad
 \frac23\left(2+\frac{k_2}{k_1}\right).
\end{equation*}
Its eigenvalues are
\begin{equation*}
 0,\qquad 2,\qquad
 \frac{10+4(k_1/k_2+k_2/k_1)}3\ge6.
\end{equation*}
If $Q>3$, the first two numbers $-\sum_{j\le i}c_j$ are at least as large
as the corresponding three-layer values. Indeed, on the indices
$0,1,2,3$ they solve the same second-difference equations with the
same zero left boundary value. Their right boundary value is nonnegative,
whereas the three-layer value is zero. The discrete maximum principle applies.
We test the full quadratic form with a vector supported on its first
three entries. The remaining edge contributes a nonnegative term.
Its largest Rayleigh quotient is therefore at least six. Condition
\eqref{hscond} implies $6\le(n-2)/4$, as required.
The case $Q=1$ was settled at the beginning of this subsection.
\end{proof}

In dimension $n=10$, the preceding proposition and
Lemma~\ref{haver} determine the averages completely.
Suppose that $U$ satisfies \eqref{hassu} and is non-one-dimensional.
Then $Q=2$. We set $k=\sqrt{W''(1)}$, with the lower well fixed as
above. Equation~\eqref{hlequa} implies $\alpha_1\ge8/k$,
while \eqref{hlstab}, tested with $v=(-1,1)$, implies
$2k\alpha_1\le16$. Hence every subsequential limit satisfies
$\alpha_1=8/k$, $\beta_1=-1/k$, and $\beta_2=1/k$.
The bounded quantities in \eqref{hmdefi} therefore converge, and
the actual zero heights on the fixed component $\Omega_0$ satisfy
\begin{equation*}
 \begin{aligned}
 \frac1{\log R}\fint_{\mathbb S^9}(h_1\mathbf1_{\Omega_0})(R\theta)\,d\theta
     &\longrightarrow-\frac1k,\\
 \frac1{\log R}\fint_{\mathbb S^9}(h_2\mathbf1_{\Omega_0})(R\theta)\,d\theta
     &\longrightarrow\frac1k.
 \end{aligned}
\end{equation*}
These are essential limits, as specified before Lemma~\ref{haver}.
Multiplying the interaction average in \eqref{hmdefi} by $2k$ also yields
\begin{equation*}
 \frac1{|\mathbb S^9|\log R}
 \int_{\Omega_0\cap\{R_1<|x|<R\}}
 \frac{4k^3A_+^2}{\sigma_0}e^{-k(h_2-h_1)}|x|^{-8}\,dx
 \longrightarrow16.
\end{equation*}

We return to arbitrary $n\ge5$ and estimate the tilt-excess
required by several transition layers.
Wang--Wei \cite[Proposition~4.1 and Remark~4.2]{WW} differentiate the
orthogonality conditions to estimate the layer functions.
We combine that calculation with the averaged gaps to determine
the leading constant in the tilt-excess.

\begin{proposition}\label{egap}
Under \eqref{hassu}, if $Q\ge2$, then
\begin{equation}\label{hlegap}
 \begin{split}
 &\liminf_{R\to\infty}\frac1{|\mathbb S^{n-1}|\log R}
 \int_{\{1<|x|<R\}}|x|^{2-n}
             \left(\int_{\mathbb R}|D_xU(x,z)|^2\,dz\right)\,dx\ge4\sigma_0\sum_{i=1}^Q c_i^2.
 \end{split}
\end{equation}
Consequently,
\begin{equation}\label{hbegap}
 \limsup_{R\to\infty}R^{2-n}\int_{B_R^{n+1}}|D_xU|^2\,dX
 \ge\frac{4\sigma_0|\mathbb S^{n-1}|}{n-2}\sum_{i=1}^Q c_i^2.
\end{equation}
In particular, under the same uniform bound in the center, either
\begin{equation}\label{hsgrow}
 \int_{B_R(0)}|D_xU|^2\,dX=o(R^{n-2}),
\end{equation}
or the more general quantitative condition
\begin{equation}\label{hscoef}
 \limsup_{R\to\infty}R^{2-n}\int_{B_R(0)}|D_xU|^2\,dX
 <\frac{2\sigma_0|\mathbb S^{n-1}|}
 {(n-2)\max\{W''(-1),W''(1)\}}
\end{equation}
forces $U=g(z)$ or $g(-z)$.
\end{proposition}

\begin{proof}
We use arbitrarily accurate orthogonal layer functions on $\Omega$ as in
Lemma~\ref{haver} and let $b$ be the function in
\eqref{clp}. The identity after \eqref{mny} expresses $b$ as a
convex combination of
$\sqrt{W''(s_i)}(f_{i+1}-f_i)$. Hence
we obtain from \eqref{hmgmea}
\begin{equation*}
 \fint_{\mathbb S^{n-1}}(b\mathbf1_\Omega)(e^T\theta)\,d\theta
 \ge2T-C
\end{equation*}
for almost every large $T$.

The boundary part of $D(b\mathbf1_\Omega)$ has finite total variation
with weight $|x|^{1-n}$. We integrate its radial derivative from a fixed
$s$ to $T$ and apply Cauchy--Schwarz to the interior part. This yields
\begin{equation*}
 2T-C\le O(1)+
 \left[\frac{T-s}{|\mathbb S^{n-1}|}
 \int_{\Omega\cap\{e^s<|x|<e^T\}}|Db|^2|x|^{2-n}\,dx\right]^{1/2}.
\end{equation*}
Thus the lower limit of the weighted integral of $|Db|^2$, divided by
$|\mathbb S^{n-1}|T$, is at least four.

We need the leading constant in the derivative estimate for $b$.
We differentiate the orthogonality equations and obtain \eqref{jgvcmu}.
The Gram matrix of the functions $s_i g'(s_i(z-f_i))$ is
$\sigma_0 I+o(1)$ in operator norm as $D\to\infty$. The matrix on the
left of \eqref{jgvcmu} differs from it by a diagonal matrix of norm
$O(\nu)$, and is therefore also $\sigma_0 I+o(1)$.
For each horizontal derivative, Bessel's inequality in this finite
family and inversion of this matrix yield
\begin{equation*}
 \left|\frac{\sum_i c_i\partial_{x_j}f_i}{\sum_i c_i^2}\right|^2
 \le\frac{1+o(1)}{\sigma_0\sum_i c_i^2}
                      \int_{\mathbb R}|U_{x_j}|^2\,dz.
\end{equation*}
Indeed, the projection map has norm at most
$(\sigma_0+o(1))^{1/2}$ and the inverse matrix has norm at most
$(\sigma_0-o(1))^{-1}$. We sum over $j$ and let $D\to\infty$,
$\nu\to0$ to obtain \eqref{hlegap}. Since the integrand is nonnegative,
we may enlarge the base integration to the whole annulus.

We compare normal columns with ambient balls. By
\eqref{cone} and the exponential decay away from the zero
set, for every fixed $\eta>0$ and all sufficiently large $R$,
\begin{equation}\label{hcball}
 \int_{B_R^n}\int_{\mathbb R}|D_xU|^2\,dz\,dx
 \le\int_{B_{\sqrt{1+\eta^2}R}^{n+1}}|D_xU|^2\,dX+o(R^{n-2}).
\end{equation}
To justify the error, we use \eqref{cone} with slope
$\eta/4$. For $|x|\le R$ and $|z|>\eta R$, the distance to the zero
set is at least $c_\eta R+c_\eta(|z|-\eta R)$, after increasing $R$.
The integral in $z$ over $|z|>\eta R$ and its integral over $B_R^n$ are exponentially small.
The remaining cylinder is contained in the displayed ball.
The reverse inequality with radius $R$ follows from inclusion in the
full column. It follows, by letting $\eta\downarrow0$, that the upper
limits of the two unweighted energies divided by $R^{n-2}$ agree.

To compare the weighted and unweighted bounds, we set
$F(R)=\int_{B_R^n}\int_{\mathbb R}|D_xU|^2\,dz\,dx$.
It is finite, nondecreasing, and $O(R^{n-2})$.
By integration by parts in the radius, we obtain
\begin{equation*}
 \int_{1<|x|<R}|x|^{2-n}\left(\int_{\mathbb R}|D_xU|^2\,dz\right)\,dx
 =R^{2-n}F(R)-F(1)+(n-2)\int_1^Rr^{1-n}F(r)\,dr.
\end{equation*}
The upper limit of this expression divided by $\log R$ is at most
$(n-2)\limsup R^{2-n}F(R)$. Together with
\eqref{hlegap} and \eqref{hcball}, this proves
\eqref{hbegap}.

Finally, if $Q\ge2$, two consecutive $c_i$ differ by at least
$1/\sqrt{\max\{W''(-1),W''(1)\}}$. Hence
\begin{equation*}
 \sum_i c_i^2\ge\frac1{2\max\{W''(-1),W''(1)\}}.
\end{equation*}
Either \eqref{hsgrow} or
\eqref{hscoef} contradicts
\eqref{hbegap}. We are left with $Q=1$, which was
settled at the beginning of this subsection.
\end{proof}

If $W''(-1)=W''(1)=k^2$ with $k>0$, then
\begin{equation*}
 \sum_i c_i^2=\frac{Q(Q^2-1)}{12k^2}.
\end{equation*}
Thus \eqref{hbegap} takes the explicit form
\begin{equation*}
 \limsup_{R\to\infty}R^{2-n}\int_{B_R}|D_xU|^2\,dX
 \ge\frac{\sigma_0|\mathbb S^{n-1}|Q(Q^2-1)}{3k^2(n-2)}.
\end{equation*}
These conclusions assume the uniform tangential bound in \eqref{hassu}.

We now apply these bounds to the local sheeting problem. The following
proposition treats zeros whose regularity or normal convergence fails.
Uniform convergence of the normal lines would yield exactly $Q$
graphs with slopes tending to zero, by the graph-counting argument
in the proof of Theorem~\ref{sheet}. A uniform second-order estimate
is a separate issue in the dimensions considered here.
 
\begin{proposition}\label{reduce}
Let $n\ge10$, and suppose that a sequence satisfies the assumptions
of Theorem~\ref{sheet} apart from the restriction $n\le9$.
Suppose that, after a subsequence, there are zeros $X_j$ in a fixed
compact interior set and $\eta>0$ satisfying \eqref{mbad}.
Then there exists a bounded stable
entire solution $U$ of $\Delta U=W'(U)$ satisfying \eqref{hassu},
with $U(0)=0$, distinct from both $g(z)$ and $g(-z)$.
The planar blowdown of $U$ has integer multiplicity $q$ with
$2\le q\le Q$. If $10\le n\le25$, then $q=2$.
\end{proposition}

\begin{proof}
Among sequences satisfying the preceding hypotheses, with multiplicity
at most $Q$ and zeros satisfying \eqref{mbad}, we choose one with
the smallest multiplicity.
For every sequence with smaller multiplicity, the zeros are therefore
eventually regular on compact interior sets and their normal lines
converge uniformly to the vertical direction. These are the properties
needed in the iteration below.

We apply the finite iteration from that proof, keeping $X_j$ at the
center. By \eqref{anboun}, the excess tends to zero and the total
change of normal tends to zero, uniformly in the number of steps.
In particular, the bound at the selected zero remains
$\beta(Du_j(X_j),e_j)\le1-\eta/2$ at the final radius.
Whenever a rescaling has a planar limit of smaller multiplicity, our
choice of the sequence implies regularity and uniform normal
convergence for that rescaled sequence. Its normals would then
converge to the vertical direction,
contrary to the bound on $\beta$. This excludes the smaller-multiplicity
alternative in Proposition~\ref{cont}.
A positive limiting thickness is also excluded:
the first stopping case in the local proof yields a one-dimensional
limit whose vertical derivative at the origin is nonzero, contradicting
the bound on $\beta$. Finally, when the thickness tends to zero and
the excess is larger than the thickness, Proposition~\ref{cont}
implies a further contraction. The iteration can therefore end only,
after a subsequence, with
\begin{equation*}
 \delta_j\longrightarrow0,\qquad E_j\le\delta_j.
\end{equation*}

We now rescale by $\delta_j$ about the selected zero. Since
$E_j\delta_j^{-1/2}\le\delta_j^{1/2}\to0$, Lemma~\ref{bterm}
applies with $\gamma=3/4$. After a further subsequence, the rescaled
solutions converge in $C^2_{\mathrm{loc}}(\mathbb R^{n+1})$ to a
bounded stable entire solution $U$ with $U(0)=0$ and the bounds
\eqref{hassu}. The bound on $\beta$ is unchanged by dilation.
Since the chosen normals tend to the vertical direction, smooth
convergence yields
\begin{equation*}
 |U_z(0)|\le(1-\eta/2)|DU(0)|.
\end{equation*}
Both $g(z)$ and $g(-z)$ have zero horizontal gradient and vertical
derivative of absolute value $g'(0)>0$ at the origin. Hence $U$
differs from both of them.

Lemma~\ref{plane} determines the planar blowdown of $U$ and its
positive integer multiplicity $q$. The density bound in
Lemma~\ref{bterm} implies $q\le Q$. If $q=1$, the multiplicity-one
argument at the beginning of this subsection would imply that $U$ equals
$g(z)$ or $g(-z)$. Thus $q\ge2$. The three-layer estimate in the
proof of Proposition~\ref{multi} shows that $q\ge3$ requires
$n\ge26$, so $q=2$ when $10\le n\le25$.
\end{proof}

\section{Finite Morse index solutions}
\label{finproof}

In this section, we prove
Theorem~\ref{finindex}.
We first show that all blowdowns have the same multiplicity.
We then apply Theorem~\ref{sheet} on the regular parts of their
supports and exclude higher multiplicity separately for nonflat
cones and hyperplanes. In the planar case, we use the layer
equations and the scalar estimates established in
Subsections~\ref{layerint} and~\ref{scalarred}.
The one-dimensional conclusion for $4\le N\le7$ follows at the end.

Throughout this section, $U$ satisfies the assumptions of
Theorem~\ref{finindex}.
In particular, $4\le N\le10$, and we write
\begin{equation}\label{equalwells}
 W''(-1)=W''(1)=k^2,\qquad k>0.
\end{equation}
We use $n=N-1$ and $\sigma_0$ as defined in the Introduction.

Finite Morse index implies that $U$ is stable outside some ball
$B_{R_0}$. Indeed, otherwise we could choose compactly supported
functions with pairwise disjoint supports, each outside a ball
containing all the preceding supports, on which the second variation
is negative. Their linear span would contain negative subspaces of
arbitrarily large dimension. Thus
\begin{equation}\label{extstab}
 \int_{\mathbb R^N}\bigl(|D\zeta|^2+W''(U)\zeta^2\bigr)\,dX\ge0
 \qquad
 \text{for }\zeta\in C_c^1(\mathbb R^N\setminus\overline{B_{R_0}}).
\end{equation}
Only this consequence of finite Morse index will be used below.

For $R>0$, set $u_R(X)=U(RX)$ and $\varepsilon_R=R^{-1}$.
We use the normalized measures
\begin{equation*}
 \mu_R=\frac1{\sigma_0}
 \left(\frac1{2R}|Du_R|^2+RW(u_R)\right)dX
\end{equation*}
and the associated varifolds
\begin{equation}\label{finvar}
 V_R(\Phi)=\frac1{\sigma_0R^n}
 \int_{\mathbb R^N}
 \Phi\bigl(X/R,(DU(X))^\perp\bigr)|DU(X)|^2\,dX,
 \qquad \Phi\in C_c(\mathbb R^N\times G(N,n)).
\end{equation}
The choice of plane at a critical point of $U$ does not affect this
formula. We write $|C|$ for the varifold of a hypersurface $C$ with
multiplicity one on its regular part.

We first describe the limits obtained by rescaling at infinity.
From Modica's inequality and the monotonicity formula
\cite{Modica85,Modica88}, we obtain
\begin{equation}\label{finmono}
 \begin{split}
 \xi&=W(U)-\frac12|DU|^2\ge0,\\
 \frac{d}{dR}\left(R^{-n}\mathcal E_1(U;B_R)\right)
 &=R^{-n}\int_{\partial B_R}U_r^2\,d\mathcal H^n
   +R^{-n-1}\int_{B_R}\xi\,dX.
 \end{split}
\end{equation}
Consequently,
\begin{equation}\label{findens}
 D_\infty=\lim_{R\to\infty}
 \frac{\mathcal E_1(U;B_R)}{\sigma_0\omega_n R^n}
 \quad\text{satisfies}\quad 0<D_\infty<\infty.
\end{equation}
The positivity follows from $W(U)>0$ at finite points, and the
upper bound follows from \eqref{fingrowth}.
Monotonicity also yields
\begin{equation}\label{finmass}
 \|V_R\|(B_\rho)\le\mu_R(B_\rho)
 \le \omega_nD_\infty\rho^n
 \qquad(R,\rho>0).
\end{equation}
If $X(R)=\int_{B_R}\xi\,dX$, we integrate
\eqref{finmono} from $R$ to $2R$ and obtain
\begin{equation*}
 (2R)^{-n}\mathcal E_1(U;B_{2R})
 -R^{-n}\mathcal E_1(U;B_R)
 \ge\frac{1-2^{-n}}n R^{-n}X(R).
\end{equation*}
It follows that $R^{-n}X(R)\to0$.
Thus $\mu_R-\|V_R\|$ tends to zero locally in total variation.

\begin{lemma}\label{conelimits}
Every sequence $R_j\to\infty$ has a subsequence such that
\begin{equation}\label{coneconv}
 V_{R_j}\rightharpoonup Q|C|,
 \qquad
 \mu_{R_j}\rightharpoonup Q\mathcal H^n\lfloor C,
\end{equation}
where $Q\in\mathbb N$ and $C$ is a nonempty stable minimal
hypercone with connected, two-sided regular part.
If $n\le6$, then $C$ is a hyperplane. If $n\ge7$, then
$\dim_{\mathcal H}\operatorname{sing}C\le n-7$.
Moreover,
\begin{equation}\label{conedens}
 Q\Theta_C(0)=D_\infty,
 \qquad
 \Theta_C(0)=\frac{\mathcal H^n(C\cap B_1)}{\omega_n}.
\end{equation}
\end{lemma}

\begin{proof}
As in the proof of Theorem~\ref{global}, we obtain from the compactness theorem
\cite[Theorem~1 and Proposition~4.3]{HT} a stationary
integral varifold limit $V$ on $\mathbb R^N$.
By the discrepancy estimate above, its weight equals the
limit of $\mu_{R_j}$. From \eqref{findens}, first at
continuity radii and then by approximation, we obtain
\begin{equation*}
 \|V\|(B_\rho)=\omega_nD_\infty\rho^n\qquad(\rho>0).
\end{equation*}
Equality in the stationary monotonicity formula makes $V$ a cone.

On every compact subset of $\mathbb R^N\setminus\{0\}$, the
rescaled solutions are stable for large $j$ by
\eqref{extstab}. The theorem of
Tonegawa--Wickramasekera \cite[Theorem~2.1]{TW12} therefore implies
stability and embedded regularity of the limiting support away
from a singular set of dimension at most $n-7$ there.
For $n\le6$, the support is smooth away from the origin.
For $n\ge7$, adjoining the origin preserves the stated
bound on the singular set. In both cases, with
$S=\operatorname{sing}(\operatorname{spt}\|V\|)$, we have
$\mathcal H^{n-1}(S)=0$.

Each connected component $M$ of the regular part defines a
stationary varifold $|M|$ across $S$ by the cutoff argument in
\cite[Section~5, proof of Theorem~1.1]{Wic14max}.
Its hypotheses hold because $\mathcal H^{n-1}(S)=0$ and the
local mass bounds for $V$ also hold for $M$.
Positive dilations preserve each component, so the closures of
two distinct components would meet at the origin and only in $S$.
The intersection theorem \cite[Theorem~3.2]{Wic14max} excludes
this possibility. Hence the regular part is connected.

The regular part is also two-sided. Indeed, every closed curve
in $\mathbb R^N\setminus S$ bounds a disk in that open set.
One can first fill the curve in $\mathbb R^N$, translate the
filling by an arbitrarily small vector to avoid $S$, and join
its boundary back to the original curve. Such translations
exist because the set of forbidden translations has Hausdorff
dimension at most $\dim_{\mathcal H}S+2<N$.
A one-sided properly embedded hypersurface would lead to a closed
curve with odd intersection number modulo two. If we made a filling
disk transverse to the hypersurface, we would then obtain a compact
one-dimensional manifold with an odd number of boundary points,
which is impossible.

From the constancy theorem, we now obtain $V=Q|C|$ for a single positive
integer $Q$, where $C=\operatorname{spt}\|V\|$.
The singular set carries no varifold mass.
The regular part of $C$ is stable: if the origin is singular,
every compactly supported test on the regular part avoids it;
if the origin is regular, $C$ is already a hyperplane.
For $n\le6$, Simons' theorem \cite{Simons} makes the stable cone,
smooth away from its vertex, a hyperplane.
The identity \eqref{conedens} follows from its mass on $B_1$.
Finally, \eqref{finmass} excludes any additional mass at the
origin.
\end{proof}

\begin{lemma}\label{fixedmult}
The integer $Q$ in \eqref{coneconv} is independent of
the sequence $R_j$. Either every limiting support is a hyperplane,
or every limiting support is nonflat.
\end{lemma}

\begin{proof}
Fix $L<\infty$. The class of varifolds $|C|$ of nonempty stable
minimal hypercones with the regularity in Lemma~\ref{conelimits}
and $\mathcal H^n(C\cap B_1)\le L$ is compact.
For $n\le6$, this is the compactness of the family of hyperplanes.
For $n\ge7$, we first obtain from Schoen--Simon compactness \cite[Theorem~2]{SS81}
a limit $m|C_\infty|$, with the same regularity and
stability. The connectedness and two-sidedness arguments in the
preceding proof apply to this limit as well.
The multiplicity estimate for stable cones
\cite[Lemma~3.2]{ZW24} excludes $m\ge2$, so $m=1$.
In particular, convergence in this class is smooth on compact
subsets of the regular part of the limit.

Every such cone has $\Theta_C(0)\ge1$. This follows by applying
monotonicity at a regular point, whose density is one, and letting
the radius tend to infinity. Equality implies that the cone is
a hyperplane. For fixed $L$, consider the sets
\begin{equation*}
 \mathcal K_q=
 \{q|C|:\ C\text{ is as above and }
                 q\mathcal H^n(C\cap B_1)\le L\},
 \qquad 1\le q\le\lfloor L/\omega_n\rfloor,
\end{equation*}
The nonempty sets among them form a finite collection of disjoint
compact sets.
They are disjoint because the integer density at a regular point
of a common support determines $q$.
Thus they have positive pairwise distance in any metric inducing
local varifold convergence.

Take $L=\omega_nD_\infty$. Every subsequential limit of $V_R$ as
$R\to\infty$ belongs to $\bigcup_q\mathcal K_q$.
Consequently, the distance of $V_R$ from this union tends to zero.
On the other hand, $R\mapsto V_R$ is continuous.
This follows directly from \eqref{finvar} by dominated
convergence, since $R$ in a compact interval determines a fixed compact
integration region and the plane field there is independent of $R$.
Choose disjoint neighborhoods of the finitely many compact sets
$\mathcal K_q$. For all sufficiently large $R$, the continuous
image of $[R,\infty)$ lies in their union, and hence in just one
of them. Every limiting varifold belongs to the corresponding
$\mathcal K_Q$. This proves the first assertion.

By \eqref{conedens}, every limiting support now has density
$D_\infty/Q$. If one of them is a hyperplane, this density is one,
so all of them are hyperplanes. Otherwise all are nonflat.
\end{proof}

We next state the consequence of Theorem~\ref{sheet} needed to
compare the layers on a limiting cone.

\begin{lemma}\label{localgraph}
Let $v_j$ be stable solutions of \eqref{acn}, with $|v_j|\le1$
and parameters $\varepsilon_j\to0$, on a fixed open subset of
$\mathbb R^N$. Assume that their energies are locally bounded and their
varifolds converge to $Q|\Gamma|$, where $\Gamma$ is a smooth
embedded minimal hypersurface. On every compact interior
coordinate neighborhood of $\Gamma$, the complete zero set
consists, for large $j$, of exactly $Q$ ordered normal graphs
converging to $\Gamma$ in $C^2$.
Their second derivatives are uniformly H\"older continuous on
smaller neighborhoods, and
\begin{equation}\label{localmean}
 |H_{\{v_j=0\}}|\le C\varepsilon_j.
\end{equation}
At zeros in a smaller neighborhood, we obtain one-dimensional heteroclinic
limits by rescaling by $\varepsilon_j$.
In particular, distinct zero graphs have separation tending
to infinity in these rescaled coordinates.
\end{lemma}

\begin{proof}
Suppose first that zeros $P_j\to P\in\Gamma$ are critical or
that their normal lines fail to converge to the normal of
$\Gamma$ at $P$. Choose $\rho_j\downarrow0$ sufficiently slowly
that $\varepsilon_j/\rho_j\to0$, $|P_j-P|/\rho_j\to0$, and
$v_j(P+\rho_j\,\cdot)$ converges in gradient measure to
$Q\sigma_0\mathcal H^n\lfloor T_P\Gamma$ on every fixed ball.
Such a choice follows by first fixing the radius and passing
to the limit, and then letting the radius decrease.
The normal estimate in the proof of Theorem~\ref{sheet},
obtained before the final graph count, contradicts either failure.
The same argument is uniform over a compact subset of $\Gamma$.

In a small fixed graph cylinder, projection of the zero set
onto the base is therefore a local diffeomorphism.
The convergence of the transition regions \cite[Theorem~1]{HT}
keeps all zeros away from the upper and lower caps. Projection is proper
over compact subsets of the base, with finite fibers of locally constant
cardinality. By ordering their heights, we obtain smooth graphs with
$C^1$ convergence. To count them, fix a central fiber and make
a second slow diagonal choice of radius dominating the height
displacements of all its roots. Theorem~\ref{sheet} counts
this fiber as exactly $Q$, and local constancy implies the same
number throughout the smaller cylinder.

Set $\widetilde W=W/k^2$ and
$\widetilde\varepsilon_j=\varepsilon_j/k$.
Then $\widetilde W''(\pm1)=1$ and the equation has the
normalization of \cite{WW}.
The local energy bound, the complete set of $Q$ graphs, and their
$C^1$ convergence verify the three hypotheses of
\cite[Corollary~1.2]{WW}. Its dimension denotes the ambient
dimension, and our assumption $N\le10$ is within its range.
That result establishes a uniform $C^{2,\theta}$ bound on smaller
patches, for each fixed $\theta\in(0,1)$, and
\eqref{localmean}. By interpolation, we obtain the asserted
$C^2$ convergence. The first step in the proof of that corollary
establishes the heteroclinic limit at the scale
$\widetilde\varepsilon_j$, equivalently at the scale
$\varepsilon_j$ after a fixed change of coordinates.
The normal derivative of the limiting heteroclinic is strictly positive.
Two distinct zero graphs at bounded distance after this rescaling would
then have two intersections with the same transverse fiber in
a region of strict monotonicity. This proves the last assertion.
\end{proof}

We first exclude higher multiplicity when the limiting cones are nonflat.

\begin{proposition}\label{nonflat}
If every limiting support in \eqref{coneconv} is
nonflat, then $Q=1$.
\end{proposition}

\begin{proof}
By Lemma~\ref{conelimits}, this case can occur only when
$n\ge7$. Suppose $Q\ge2$.
Let $\mathcal C_U$ denote the family of limiting supports,
each taken with multiplicity one. This is compact by
Lemma~\ref{fixedmult} and a diagonal choice of
blowdown sequences. For $C\in\mathcal C_U$, write
$\Sigma_C=\operatorname{reg}C\cap S^n$.

We will compare layers on subsets where the link is uniformly
regular. For $\theta\in\Sigma_C$, define $r_C(\theta)$ to be the
supremum of $s\in(0,10^{-3})$ such that, in geodesic normal
coordinates of $S^n$ centered at $\theta$, the whole link in
$B_{4s}^{n-1}\times(-4s,4s)$ is a single graph $v$ over
$T_\theta\Sigma_C$, with
\begin{equation*}
 v(0)=Dv(0)=0,
 \qquad \|Dv\|_\infty+s\|D^2v\|_\infty\le\frac1{10}.
\end{equation*}
Set
\begin{equation*}
 \Sigma_C(\rho)=\{\theta\in\Sigma_C:r_C(\theta)>\rho\}.
\end{equation*}
By compactness, there exists $\rho>0$ such that
$\Sigma_C(8\rho)\ne\varnothing$ for every $C\in\mathcal C_U$.
Indeed, otherwise a convergent sequence of cones would contradict
smooth convergence near any regular point of its limit.
Moreover, if $C_j\to C$ and $\theta_j\in\Sigma_{C_j}(\rho)$,
then every limit of $\theta_j$ is regular in $C$.
This follows from the single graph description on a fixed smaller
ball and the minimal surface equation. Conversely, every fixed
point of $\Sigma_C(8\rho)$ and a small neighborhood of it are
represented in $\Sigma_{C_j}(\rho)$ for large $j$.

Choose a unit normal $\nu_C$ on $\Sigma_C$ and use the coordinates
\begin{equation}\label{sphereco}
 \Psi_C(r,\theta,h)
 =r\cos(h/r)\theta+r\sin(h/r)\nu_C(\theta).
\end{equation}
They satisfy $|\Psi_C(r,\theta,h)|=r$, and $h$ is arc length
along the corresponding normal geodesic on $\partial B_r$.
The graphs constructed in Lemma~\ref{localgraph} can be expressed
in these coordinates on each compact subset of
$\operatorname{reg}C$. Local orderings agree on overlaps because
the normal has been fixed.

We also need to compare gaps measured with respect to different
reference cones converging to the same cone. Fix the same point $p$ on the actual lower graph.
In the first coordinates, write the actual upper graph as
$s=b(x)$ and put $F(x,s)=s-b(x)$. Thus
$F(p)=-w_1$, where $w_1$ is the gap on the first normal fiber
through $p$. Along the second normal fiber from $p$ to the
upper graph, with arc length parameter $t$, we have
\begin{equation*}
 \frac{d}{dt}F(\gamma(t))=1+o(1).
\end{equation*}
The error is uniform on each fixed compact regular neighborhood,
since both reference cones and both zero graphs converge there
in $C^1$, with consistently chosen normals.
We integrate along this fiber and obtain
\begin{equation}\label{gapbound}
 w_1=(1+o(1))w_2.
\end{equation}
No estimate of the individual graph heights relative to their
separation is required.

For each sufficiently large integer $\ell$, choose $C_\ell\in\mathcal C_U$
such that the local varifold distance between $V_{2^\ell}$ and
$Q|C_\ell|$ tends to zero as $\ell\to\infty$.
Such a choice is possible by compactness.
For all sufficiently large $\ell$, we obtain from Lemma~\ref{localgraph}
exactly $Q$ ordered graph heights $h_{1,\ell}<\cdots<h_{Q,\ell}$ in
\eqref{sphereco} over a neighborhood of
$\{1\}\times\Sigma_{C_\ell}(\rho)$.
The assertion is uniform on this set: a sequence of failures
would converge to a regular point of a limiting cone, where
Lemma~\ref{localgraph} applies.
Define
\begin{equation}\label{conewidth}
 D_\ell=2^\ell\sup_{\theta\in\Sigma_{C_\ell}(\rho)}
       \bigl(h_{Q,\ell}(1,\theta)-h_{1,\ell}(1,\theta)\bigr).
\end{equation}
By the same compactness argument, we obtain
\begin{equation*}
 2^{-\ell}D_\ell\to0,
 \qquad D_\ell\to\infty.
\end{equation*}
If $D_\ell$ were bounded along a subsequence, the lowest and
highest graphs over a point of $\Sigma_{C_\ell}(8\rho)$ would
have bounded separation in the original coordinates.
This contradicts the last assertion of Lemma~\ref{localgraph},
applied on the uniformly regular neighborhoods just described.

Choose integers $\ell_j\to\infty$ such that
\begin{equation*}
 M_j:=D_{\ell_j}=\max_{\ell_0\le \ell\le \ell_j}D_\ell\to\infty,
 \qquad R_j=2^{\ell_j},\qquad \delta_j=M_j/R_j\to0.
\end{equation*}
After passing to a subsequence, $V_{R_j}\rightharpoonup Q|C_*|$
and $|C_{\ell_j}|\rightharpoonup |C_*|$.
On every compact subset of $\operatorname{reg}C_*$, denote the
highest and lowest graph heights in \eqref{sphereco} by
$h_j^+$ and $h_j^-$. Put
\begin{equation*}
 v_j=(h_j^+-h_j^-)/\delta_j>0.
\end{equation*}
Let $H(h)$ be scalar mean curvature, pulled back to $C_*$ with
consistent normal. From the first variation of mean curvature
\cite[Appendix~A]{CM20}, we obtain $DH(0)=-J_{C_*}$, where
$J_{C_*}=\Delta_{C_*}+|A_{C_*}|^2$.
Using the exact identity
\begin{equation*}
 H(h_j^+)-H(h_j^-)
 =\int_0^1 DH\bigl(h_j^-+t(h_j^+-h_j^-)\bigr)
                    [h_j^+-h_j^-]\,dt
\end{equation*}
and dividing by $\delta_j$, we obtain
\begin{equation}\label{lineareq}
 L_jv_j=f_j,
 \qquad L_j\to J_{C_*},
 \qquad \|f_j\|_{L^\infty(K)}\le C_K/M_j\to0
\end{equation}
for every compact coordinate neighborhood $K$ in the regular part.
Here the coefficients converge by the $C^2$ convergence of the
graphs. The last bound follows from
\eqref{localmean}, since
$|H(h_j^\pm)|\le C_K/R_j$ and $R_j\delta_j=M_j$.
The operators are uniformly elliptic with uniformly H\"older
continuous coefficients on each such neighborhood.

Choose points approaching the supremum in \eqref{conewidth}
with relative error tending to zero. Their limits are regular
points of $C_*$. We express the corresponding lower zero in the
fixed coordinates of $C_*$ and use \eqref{gapbound}
to obtain points $\theta_j'\to\theta_*\in\Sigma_{C_*}$ such that
\begin{equation*}
 v_j(1,\theta_j')\to1.
\end{equation*}
From the inhomogeneous Harnack inequality and interior elliptic
estimates \cite[Chapters~8 and~9]{GT}, applied to
\eqref{lineareq}, we obtain bounds on every fixed connected
regular neighborhood.
The regular part is connected by Lemma~\ref{conelimits}.
By an exhaustion and diagonal extraction, we therefore obtain
\begin{equation}\label{posjacobi}
 v_j\to v\quad\text{locally in }C^1(\operatorname{reg}C_*),
 \qquad J_{C_*}v=0,
 \qquad v>0,
 \qquad v(\theta_*)=1.
\end{equation}

Fix $\theta\in\Sigma_{C_*}(8\rho)$ and an integer $m\ge1$.
Write $r_m=2^{-m}$.
At the radii $R_jr_m$, the varifolds still converge to
$Q|C_*|$, because $C_*$ is a cone. Consequently the chosen
reference cones $C_{\ell_j-m}$ also converge to $C_*$.
The lower zero above $(r_m,\theta)$, rescaled to unit radius,
has a base point $\theta_{j,m}$ in the corresponding
$\Sigma_{C_{\ell_j-m}}(\rho)$ for large $j$.
Applying \eqref{gapbound} at this same lower zero yields
\begin{equation*}
 \frac{h_j^+(r_m,\theta)-h_j^-(r_m,\theta)}{r_m}
 =(1+o_{j,m}(1))
 \bigl(h_{Q,\ell_j-m}(1,\theta_{j,m})
       -h_{1,\ell_j-m}(1,\theta_{j,m})\bigr).
\end{equation*}
Since $2^{\ell_j-m}=R_jr_m$, it follows that
\begin{equation*}
 v_j(r_m,\theta)
 \le(1+o_{j,m}(1))\frac{D_{\ell_j-m}}{M_j}
 \le1+o_{j,m}(1).
\end{equation*}
First fixing $m$ and then sending $j\to\infty$, we conclude
that the same function in \eqref{posjacobi} satisfies
\begin{equation}\label{boundvals}
 0<v(2^{-m}\theta)\le1\qquad(m\ge1).
\end{equation}

The cone $C_*$ is nonflat, so its vertex is singular.
Theorem~1.4 of \cite{ZW24} applies to the positive solution $v$
on the two-sided stable regular part of $C_*$, whose singular set
has codimension at least seven.
We obtain
\begin{equation*}
 \liminf_{\operatorname{reg}C_*\ni X\to0}v(X)=+\infty,
\end{equation*}
contradicting \eqref{boundvals}.
Thus $Q=1$.
\end{proof}

Suppose now that every limiting support is a hyperplane.
Write $Z=\{U=0\}$.
From Lemma~\ref{localgraph} and a contradiction argument at points
$X_j\in Z$, $|X_j|\to\infty$, we obtain
\begin{equation}\label{anncurv}
 |X|\,|A_Z(X)|+|X|^2|H_Z(X)|\le C
 \qquad\text{on }Z\setminus B_{R_1},
\end{equation}
and $Z\setminus B_{R_1}$ is smoothly embedded.
Indeed, by rescaling by $|X_j|$, we obtain a planar limit, and
$X_j/|X_j|$ stays in its regular part away from the origin.
The mean curvature bound in the rescaled coordinates is
$O(|X_j|^{-1})$, which becomes $O(|X_j|^{-2})$ in the original
coordinates. From the proof of \cite[Corollary~1.2]{WW}, we also obtain,
on each set $\{|U|\le1-b\}$, for fixed $b\in(0,1)$,
outside a sufficiently large ball, the bounds
\begin{equation*}
 |DU|\ge c>0,\qquad
 \frac{|D^2U|^2-|D|DU||^2}{|DU|^2}\le C|X|^{-2}.
\end{equation*}

For every $\eta>0$ and all sufficiently large $R$, some unit
vector $e_R$ satisfies
\begin{equation}\label{annslabs}
 Z\cap(B_{8R}\setminus B_{R/8})
 \subset\{|X\cdot e_R|\le\eta R\}.
\end{equation}
Otherwise, rescaling a sequence of violating annuli and using
the convergence of the transition regions would contradict
planarity of the limit.
The relatively closed embedded hypersurface outside $B_{R_1}$,
\eqref{anncurv}, and \eqref{annslabs}
satisfy the hypotheses of the geometric result
\cite[Definition~3.1 and Proposition~3.2]{FS26}.
After one fixed rotation and enlargement of $R_1$, we obtain
smooth ordered functions $h_1<\cdots<h_m$ on
$\Omega_{R_1}=\mathbb R^n\setminus\overline{B_{R_1}^n}$ such that
\begin{equation}\label{extgraphs}
 Z\cap(\Omega_{R_1}\times\mathbb R)
 =\bigcup_{i=1}^m\{(x,h_i(x)):x\in\Omega_{R_1}\},
\end{equation}
the remaining part of $Z$ is bounded, and, for every
$\alpha\in(0,1)$,
\begin{equation}\label{subheight}
 |h_i(x)|+|x|\,|Dh_i(x)|\le C_\alpha|x|^\alpha.
\end{equation}
Here $m=Q$. To check this, rescale about the origin and fix a
nonzero base point of the limiting plane. The corresponding
fiber contains exactly $m$ roots by \eqref{extgraphs}
and \eqref{subheight}, whereas
Theorem~\ref{sheet} counts exactly $Q$ roots in a fixed small
cylinder about that point.

\begin{lemma}\label{heightbd}
After increasing $R_1$, the functions in
\eqref{extgraphs} satisfy
\begin{equation}\label{logheight}
 |h_i(x)|\le C\log(2+|x|),\qquad
 |Dh_i(x)|\le C|x|^{-1},\qquad
 |D^2h_i(x)|\le C|x|^{-1}.
\end{equation}
\end{lemma}

\begin{proof}
Set $r=|x|$.
From the curvature bound and the small slopes, we obtain
$|D^2h_i|\le C/r$.
For the upward normal, choose the scalar mean curvature
$H_i=-\operatorname{div}(Dh_i/\sqrt{1+|Dh_i|^2})$.
Then
\begin{equation*}
 \Delta h_i=-\sqrt{1+|Dh_i|^2}\,H_i
       +\frac{Dh_i\otimes Dh_i:D^2h_i}{1+|Dh_i|^2}.
\end{equation*}
Taking $\alpha=1/2$ in \eqref{subheight}, we obtain
$|\Delta h_i|\le Cr^{-2}$.
Extend $h_i$ smoothly through a fixed ball and denote the
extension by $\widetilde h_i$.
The function $F_i^0=\Delta\widetilde h_i$ satisfies
$|F_i^0(x)|\le C(1+|x|)^{-2}$.
Since $n>2$, define
\begin{equation*}
 \Gamma_n(x)=-\frac{|x|^{2-n}}{(n-2)|S^{n-1}|},
 \qquad
 P_i(x)=\int_{\mathbb R^n}
       [\Gamma_n(x-y)-\Gamma_n(-y)]F_i^0(y)\,dy.
\end{equation*}
The integral converges absolutely: at infinity the kernel
difference is $O(|x||y|^{1-n})$, and its singularities at $0$
and $x$ are locally integrable.
Moreover, $\Delta P_i=F_i^0$ in distributions.

For $r\ge2$, split the integral into $|y|<r/2$,
$r/2\le|y|\le2r$, and $|y|>2r$.
The subtraction term on $1<|y|<2r$ is bounded by
$C\int_1^{2r}s^{-1}\,ds$, and all remaining terms are bounded.
Differentiating the first kernel is justified by local
integrability and the integrable tail of $D\Gamma_n F_i^0$.
Each of the same three regions contributes at most $C/r$.
Consequently,
\begin{equation*}
 |P_i(x)|\le C\log(2+r),\qquad |DP_i(x)|\le C/r.
\end{equation*}
The entire harmonic function $\widetilde h_i-P_i$ has
sublinear growth by \eqref{subheight}.
The interior gradient estimate on expanding balls makes it
constant. This proves \eqref{logheight}.
\end{proof}

\begin{lemma}\label{vertdecomp}
After increasing $R_1$, there are smooth ordered functions
$f_1<\cdots<f_Q$ on $\Omega_{R_1}$ for which the decomposition
$U=G+\phi$ of \eqref{sumg} satisfies the orthogonality conditions
of \eqref{iwh}. For some fixed finite $M$, with $r=|x|$, we have
\begin{equation}\label{phasebd}
 \begin{split}
 |f_i-h_i|+|D(f_i-h_i)|&\le Cr^{-2}(1+\log r)^M,\\
 |f_i|&\le C\log(2+r),\qquad |Df_i|\le C/r,\\
 e^{-k(f_{i+1}-f_i)}&\le Cr^{-2},\\
 B_\phi(x):=\int_{\mathbb R}
       (|D_x\phi|^2+\phi^2)\,dz
       &\le Cr^{-4}(1+\log r)^M.
 \end{split}
\end{equation}
Also $\sup_z|\phi(x,z)|\to0$ as $r\to\infty$.
All normal integrations by parts have exponentially decaying
tails, locally uniformly in the base variable.
\end{lemma}

\begin{proof}
Replacing $(U,W)$ by $(-U,W(-\,\cdot))$ if necessary, we may
assume that the lower phase is $-1$.
We then use $s_i=(-1)^{i-1}$ as in \eqref{sumg}.
This replacement preserves \eqref{wells},
\eqref{equalwells}, and all the geometric conclusions.

Fix a large $r=|x_0|$ and work in balls of radius $cr$ in the
original coordinates, with larger concentric balls still contained
in the stable region.
The graphical estimates already established verify the
hypotheses of the profile estimates in \cite{WW}.
With $\widetilde W=W/k^2$ and parameter $1/(kcr)$, their
stretched coordinates differ from the physical coordinates by
the fixed factor $k$.
Thus from \cite[Proposition~10.1, (4.13), and (11.1)]{WW} we obtain
\begin{equation}\label{normtrans}
 e^{-kD_i}\le Cr^{-2},\qquad
 \|\alpha_i\|_{C^{2,\theta}(B_1^{\Gamma_i})}
       +\|\Phi\|_{C^{2,\theta}(B_1)}\le Cr^{-2}.
\end{equation}
Here $\Gamma_i=\operatorname{graph}h_i$, $D_i$ is the pointwise
distance to the nearest adjacent layer, $\alpha_i$ is the normal
phase correction, and $\Phi$ is the remainder in the
normal-coordinate representation. If $Q=1$, we set $D_i=+\infty$.
The bounds hold uniformly on physical unit
neighborhoods in the smaller ball, for fixed $\theta\in(0,1)$.
The truncation in \cite[Section~4.1]{WW}, together with
\eqref{gtail}, has error $O(r^{-4})$ in $C^1$, so the profiles
may be replaced by $g$.
A vertical gap is at least the distance to the neighboring
hypersurface. Hence the first estimate in
\eqref{normtrans} also bounds
$e^{-k(h_{i+1}-h_i)}$.

We pass to vertical coordinates. For one of the graphs, put
$t=z-h_i(x)$ and let $d$ be signed distance from that graph.
If $(y,h_i(y))$ is the normal footpoint, then
\begin{equation*}
 x=y-\frac{dDh_i(y)}{\sqrt{1+|Dh_i(y)|^2}},
 \qquad
 z=h_i(y)+\frac d{\sqrt{1+|Dh_i(y)|^2}}.
\end{equation*}
For $|t|\le L\log r$, we obtain from \eqref{logheight}
$|y-x|\le C|t|/r$ and
\begin{equation}\label{distconv}
 |d-t|+|D(d-t)|\le Cr^{-2}(1+|t|)^2.
\end{equation}
For the derivative estimate, use
\begin{equation*}
 D_xd=-\frac{Dh_i(y)}{\sqrt{1+|Dh_i(y)|^2}},
 \qquad \partial_zd=(1+|Dh_i(y)|^2)^{-1/2}.
\end{equation*}
The difference between $Dh_i(y)$ and $Dh_i(x)$ is at most
$C|y-x|/r\le C|t|r^{-2}$.
This proves \eqref{distconv} using only the
second derivative bound already in
\eqref{logheight}.
Together with \eqref{normtrans}, it shows that replacing
the normal profiles by vertical profiles about the actual zero
heights creates a $C^1$ error bounded by
$Cr^{-2}(1+\log r)^M$ on $|z|\le L\log r$.

We explain the control outside this interval.
Extend each exterior graph height to a Lipschitz function on
$\mathbb R^n$, with a common finite Lipschitz constant.
Distance to its graph is bounded below by a fixed multiple
of the vertical distance. The part of $Z$ not covered by the
exterior graphs is contained in a fixed compact set.
For $r$ large, distance to that set is bounded below by a
fixed multiple of $r+|z|$, whereas
$\min_i|z-h_i(x)|\le |z|+C\log(2+r)$.
It follows that, for all $z\in\mathbb R$,
\begin{equation}\label{zerodist}
 \operatorname{dist}((x,z),Z)
 \ge c\min_i|z-h_i(x)|.
\end{equation}
The convergence and decay argument in Lemma~\ref{wentr}
applies outside a larger compact set, since all fixed
physical balls sufficiently far away are stable.
In a well, the barrier and interior estimates imply exponential
decay of $U$ to that well and of its derivatives in the distance
to $Z$. When using these estimates on a far column, the fixed
excluded compact set causes no loss: its distance is also
bounded below by a fixed multiple of
$\min_i|z-h_i(x)|$.
By \eqref{logheight} and
\eqref{zerodist}, increasing $L$ makes the integrals
of the squared errors and their first derivatives on
$|z|>L\log r$ smaller than any prescribed fixed inverse power
of $r$.

Let $G_h$ be the expression in \eqref{sumg} with $f_i=h_i$,
and put $\phi_h=U-G_h$.
The central interval has length $O(\log r)$, so from the preceding
estimates we obtain
\begin{equation*}
 \int_{\mathbb R}(|D_x\phi_h|^2+\phi_h^2)\,dz
 \le Cr^{-4}(1+\log r)^M.
\end{equation*}
They also control the integrals
$\int\phi_h s_i g'(s_i(z-h_i))\,dz$ and their horizontal
derivatives by $Cr^{-2}(1+\log r)^M$.
For the derivatives, differentiate under the integral and use
the bounds for $D_x\phi_h$, $Dh_i$, and the exponential tails.

Apply the implicit-function argument in the proof of
Proposition~\ref{orth} to \eqref{iwh}, with displacements
$f_i-h_i$. The same Jacobian is uniformly invertible, and the
preceding bounds for the initial integrals imply
$|f_i-h_i|\le Cr^{-2}(1+\log r)^M$.
For the derivative estimate, differentiate as in \eqref{jgvcmu}.
The horizontal derivatives at zero displacement have just been
bounded by $Cr^{-2}(1+\log r)^M$. Their change at the selected
displacements is $O(\max_i|f_i-h_i|)$, since the mixed derivatives
are uniformly bounded by the profile and tail estimates.
Using the same inverse matrix, we therefore obtain
$|D(f_i-h_i)|\le Cr^{-2}(1+\log r)^M$.
Uniqueness makes the local choices agree on overlaps.
The small corrections preserve the gap estimate, and substitution
in $U=G+\phi$ proves the remaining bounds and the tail decay.
\end{proof}

\begin{proposition}\label{finplanar}
If every limiting support in \eqref{coneconv} is
a hyperplane, then $Q=1$.
\end{proposition}

\begin{proof}
Suppose $Q\ge2$.
We use the decomposition of Lemma~\ref{vertdecomp}
on the full exterior base $\Omega_{R_1}$.
The calculations of Subsections~\ref{layerint}
and~\ref{scalarred} apply on this domain.
Indeed, every test has compact base support in $\Omega_{R_1}$,
so its support in the ambient space is disjoint from $B_{R_0}$. Normal cutoffs may be removed using the
exponential tails. Thus \eqref{hsequa}, \eqref{hserro}, and
Proposition~\ref{rstat} hold here without requiring stability
inside that ball.

Choose the constants $c_i$ by \eqref{gtcoef} and define $b,a$
by \eqref{clp}. In particular, from \eqref{ojyzt} we obtain
\begin{equation*}
 a\ge c_{W,Q}e^{-b}>0.
\end{equation*}
By \eqref{phasebd},
\begin{equation*}
 a+\sum_i|Df_i|^2\le Cr^{-2},
 \qquad B_\phi=o(r^{-2}).
\end{equation*}
It follows from \eqref{hpiano}, \eqref{mny}, and
\eqref{ctyq} that, for every fixed $\eta>0$, there is an
exterior radius $R_\eta$ on which
\begin{equation}\label{finscalar}
 \Delta b\le a+\eta r^{-2}+\operatorname{div}F,
 \qquad |F|\le\eta r^{-1},
\end{equation}
and
\begin{equation*}
 \int_{\Omega_{R_\eta}}a\zeta^2\,dx
 \le(1+\eta)\int_{\Omega_{R_\eta}}|D\zeta|^2\,dx
       +\eta\int_{\Omega_{R_\eta}}r^{-2}\zeta^2\,dx
 \qquad(\zeta\in C_c^1(\Omega_{R_\eta})).
\end{equation*}
For these choices, first take $\tau$ small relative to the
fixed constants in the $Cr^{-2}$ bounds. Then take $D$ large
so that the terms with $\epsilon_D$ are small. Finally increase
$R_\eta$ so that the actual gaps exceed $D$, the slopes and
$\phi$ are small, and $C_\tau B_\phi\le\eta r^{-2}$ after
adjusting the fixed constants. We also obtain the bound
for $F$ in \eqref{finscalar}.

Apply Lemma~\ref{lemma} with
\begin{equation*}
 \beta=\lambda=1,\qquad \kappa=1+\eta,
 \qquad e(x)=g(x)=\eta|x|^{-2}.
\end{equation*}
Here $g$ denotes the error function in that lemma.
From its definitions, we obtain $E_*=G_*=\eta$ and $F_*\le\eta$.
The last assertion follows from $|F|\le\eta/r$ through
the signed radial flux, independently of any bound for
$\operatorname{div}F$.
Writing $d=n-2=N-3$, \eqref{fsjoj} therefore implies
\begin{equation*}
 2\le\frac d4+
       \left(1+\frac d4+\frac2d\right)\eta.
\end{equation*}
Since $4\le N\le10$, we have $1\le d\le7$ and hence $d/4<2$.
We choose $\eta>0$ sufficiently small and obtain a contradiction.
This proves $Q=1$.
\end{proof}

We can now prove Theorem~\ref{finindex}.

\begin{proof}[Proof of Theorem~\ref{finindex}]
By Lemmas~\ref{conelimits} and~\ref{fixedmult},
every blowdown is $Q|C|$ with the same integer $Q$, and the
limiting supports are either all planar or all nonflat.
Propositions~\ref{nonflat} and~\ref{finplanar} imply $Q=1$
in the two cases. The convergence of the full energy measures
with the same multiplicity is already included in
\eqref{coneconv}. This proves the theorem.
\end{proof}

For dimensions at most seven, the preceding result also determines
the original solution.

\begin{corollary}\label{finoned}
Let $4\le N\le7$, and let $U$ satisfy the assumptions of
Theorem~\ref{finindex}.
Then
\begin{equation*}
 U(X)=g(e\cdot X-b)
 \quad\text{for some }e\in\mathbb S^{N-1},\ b\in\mathbb R.
\end{equation*}
\end{corollary}

\begin{proof}
By Theorem~\ref{finindex}
and Lemma~\ref{conelimits}, every varifold blowdown is a
hyperplane with multiplicity one. From \eqref{findens}
and~\eqref{conedens}, we therefore obtain
\begin{equation*}
 \lim_{R\to\infty}R^{1-N}\mathcal E_1(U;B_R)
 =\sigma_0\omega_{N-1}.
\end{equation*}
The potential extension described in
Section~\ref{locenergy} verifies the assumptions on $W$
in \cite{Wang17}, while preserving the equation and energy of $U$.
The conclusion follows from \cite[Theorem~11.1]{Wang17}.
\end{proof}

\section{Multiplicity one on closed manifolds}
\label{multproof}

In this section, we prove Theorem~\ref{multone}
and derive its geometric applications. We follow
Chodosh--Mantoulidis \cite[Sections~4--7]{CM20}, using
Theorem~\ref{global} and the graphical estimates of Wang--Wei
\cite{WW} to obtain the local estimates in dimensions
$3\le N\le7$. Away from finitely many points, the highest and
lowest zero graphs determine a positive Jacobi field after normalization.
The comparison argument of \cite[Section~4.2]{CM20} bounds this
field near the excluded points. Once multiplicity one is established,
the convergence and index results of \cite[Section~5]{CM20} apply
directly, and the applications follow as in its Section~6.

We first work with an arbitrary smooth metric $\mathbf g$ on $M$.
Throughout this section, $W$ is even and satisfies \eqref{wells}.
We continue to use $n=N-1$, the heteroclinic $g$, and its energy $\sigma_0$,
and write $k=\sqrt{W''(1)}=\sqrt{W''(-1)}$.
Since $W$ is even and $g(0)=0$, uniqueness of the increasing
heteroclinic implies that $g$ is odd. Thus $g'$ is even and
\begin{equation*}
 \int_{\mathbb R}t\bigl(g'(t)\bigr)^2\,dt=0.
\end{equation*}
The same cancellation holds with an even cutoff. It is used in
\cite[Lemma~7.10]{CM20} to estimate the error in the Dirichlet
barrier construction. Evenness also ensures that $W'$ is odd,
so changing the sign of a solution preserves the equation.
We use both facts below when applying those barriers.

The replacement
\begin{equation*}
 \overline W=\frac{2W}{k^2},\qquad
 \overline\varepsilon=\frac{\sqrt2}{k}\varepsilon
\end{equation*}
puts the potential in the normalization of \cite{CM20}.
It multiplies the energy, $\sigma_0$, and the second variation
by the same positive constant. Thus it changes neither the
normalized varifolds nor the Morse index. The estimates below
are stated in the original normalization.

For a solution of \eqref{closedac}, the second variation is
\begin{equation*}
 \int_M\left(\varepsilon|\nabla\zeta|_{\mathbf g}^2
       +\varepsilon^{-1}W''(u)\zeta^2\right)d\mu_{\mathbf g}.
\end{equation*}
Stability on an open set means that this expression is nonnegative
for every smooth test function compactly supported there.
The normalized varifolds are
\begin{equation*}
 V_j(\Phi)=\frac1{\sigma_0}
 \int_M\Phi\bigl(X,(\nabla u_j(X))^\perp\bigr)
       \varepsilon_j|\nabla u_j|_{\mathbf g}^2\,d\mu_{\mathbf g},
 \qquad \Phi\in C(G_n(M)),
\end{equation*}
where $G_n(M)$ is the bundle of unoriented $n$-planes.
The choice of plane at a critical point does not affect this formula.

\subsection{Local estimates}

All local estimates are taken on smaller coordinate neighborhoods
with a fixed positive distance from the boundary of the equation
domain. Their constants may depend on these neighborhoods, the
energy bound, $W$, and fixed smooth bounds for the metric.
They also apply to smoothly converging sequences of metrics.
At a regular point of a level set, put
\begin{equation}\label{levelcurv}
 B(u)^2=
 \frac{|\nabla^2u|_{\mathbf g}^2
       -|\nabla|\nabla u|_{\mathbf g}|_{\mathbf g}^2}
      {|\nabla u|_{\mathbf g}^2}.
\end{equation}
As in \cite[(1.5)]{WW}, this is the sum of the squared norm of
the second fundamental form and
$|\nabla^{\top}\log|\nabla u|_{\mathbf g}|^2$.

We first recall the graphical estimates under an a priori bound
for \eqref{levelcurv}. The separation estimate used
here is the Riemannian estimate in \cite[Remark~10.9]{WW}.

\begin{lemma}\label{closedest}
Let $3\le N\le7$, and let $v_j$ be stable solutions with parameters
$\varepsilon_j\to0$ in a fixed coordinate neighborhood.
Assume a uniform local energy bound and, on a fixed central phase
region,
\begin{equation}\label{priorbd}
 \varepsilon_j|\nabla v_j|\ge c>0,\qquad B(v_j)\le C.
\end{equation}
On smaller graph neighborhoods, the zero graphs have uniformly
bounded $C^{2,\theta}$ norms for each fixed $\theta\in(0,1)$.
Denote them by $\Gamma_{i,j}$ and their distances to the adjacent
graphs by $D_{i,j}$, taking $D_{i,j}=+\infty$ when there is no
adjacent graph. On each compact interior graph domain $K$,
\begin{equation}\label{closedsep}
 \max_i\sup_K e^{-kD_{i,j}/\varepsilon_j}
     =o(\varepsilon_j^2|\log\varepsilon_j|),
\end{equation}
and their mean curvatures satisfy
\begin{equation}\label{smallmean}
 \max_i\frac{\|H_{i,j}\|_{L^\infty(K)}}
                  {\varepsilon_j|\log\varepsilon_j|}\longrightarrow0.
\end{equation}
In particular, for all sufficiently large $j$,
\begin{equation}\label{logsep}
 D_{i,j}\ge\frac{\varepsilon_j}{k}
 \bigl(2|\log\varepsilon_j|-\log|\log\varepsilon_j|-C_K\bigr).
\end{equation}
\end{lemma}

\begin{proof}
The bounds in \eqref{priorbd} yield ordered zero graphs with bounded
second derivatives on smaller cylinders. Their number is bounded by
the local energy, since each graph contributes one transition across
a fixed smaller base disk. After replacing $W$ by $W/k^2$ and
$\varepsilon_j$ by $\varepsilon_j/k$, these graphs satisfy the
hypotheses of the local graphical analysis of Wang--Wei.
The Riemannian estimate \cite[Remark~10.9 and (10.24)]{WW} implies
\eqref{closedsep}. Its dimension bound concerns the ambient dimension
and is satisfied for $N\le7$. Taking logarithms, we obtain
\eqref{logsep}. The local estimate
\cite[Remark~2.6 and (2.22)]{CM20} then yields the stated
$C^{2,\theta}$ bounds.

Let $A_j$ be the largest interaction on a slightly larger graph
domain. We have $A_j=o(\varepsilon_j^2|\log\varepsilon_j|)$ by
\eqref{closedsep}. From the projected equation
\cite[(2.35)--(2.36)]{CM20}, we obtain, for some fixed $\kappa>0$,
\begin{equation*}
 \varepsilon_j\|H_{i,j}\|_{L^\infty(K)}
 \le C A_j+C\varepsilon_j^2
       +C\varepsilon_j^\kappa A_j+C A_j^{1+\kappa}.
\end{equation*}
Division by $\varepsilon_j^2|\log\varepsilon_j|$ proves
\eqref{smallmean}.
\end{proof}

We now obtain the a priori bounds from Theorem~\ref{global}.

\begin{lemma}\label{stableest}
Let $3\le N\le7$, and let $v_j$ be stable solutions of the
Allen--Cahn equation on a fixed open set, with parameters
$\varepsilon_j\to0$ and locally bounded energy.
For each compact interior set and each fixed $\beta\in(0,1)$,
there are $c,C>0$ such that, for all large $j$,
\begin{equation}\label{stablecurv}
 \varepsilon_j|\nabla v_j|\ge c,\qquad B(v_j)\le C
 \quad\text{on }\{|v_j|\le1-\beta\}.
\end{equation}
Consequently, Lemma~\ref{closedest} applies in
smaller zero-set neighborhoods.
\end{lemma}

\begin{proof}
From the local monotonicity estimates
\cite[Appendix~B]{Gua18}, we obtain
\begin{equation}\label{localdens}
 \mathcal E_{\varepsilon_j}(v_j;B_r(p))\le Cr^{N-1}
 \qquad(\varepsilon_j\le r\le r_0)
\end{equation}
on every fixed smaller open set.
Theorem~\ref{global} therefore verifies the hypothesis of
\cite[Proposition~3.7]{CM20}, and implies the lower gradient bound.
Indeed, the metrics become Euclidean under rescaling by
$\varepsilon_j$, and every limit at a central-phase point is a
stable entire solution with the energy growth in
Theorem~\ref{global}. It is a heteroclinic. From interior estimates
in the same rescaling, we also obtain
\begin{equation}\label{microcurv}
 \varepsilon_j B(v_j)\longrightarrow0
\end{equation}
uniformly on smaller central phase regions.

Suppose the curvature bound in \eqref{stablecurv}
fails. By point selection for $B(v_j)$ times distance to the
boundary, we obtain $p_j$ and
$\lambda_j=B(v_j)(p_j)\to\infty$ such that, after rescaling
lengths by $\lambda_j$, the domains exhaust $\mathbb R^N$,
$B\le2$ on each fixed ball in the central phase region, and
$B(0)=1$. We use a slightly larger phase interval when making
this selection. The rescaled parameters satisfy
\begin{equation*}
 \widehat\varepsilon_j=\lambda_j\varepsilon_j\longrightarrow0
\end{equation*}
by \eqref{microcurv}. The metrics converge
smoothly to the Euclidean metric, and \eqref{localdens}
implies uniform area growth bounds on fixed balls.

The a priori curvature and gradient bounds now allow us to
apply Lemma~\ref{closedest} to the rescaled
solutions. Their zero graphs have uniform $C^{2,\theta}$ bounds
and mean curvatures tending to zero. The limiting support is
therefore a smooth embedded minimal hypersurface, possibly
with several components and integer multiplicities.
Each component is complete, stable, and has area growth at
most $CR^{N-1}$; stability follows from \cite{Ton05,TW12}.
The stable Bernstein theorem with this area growth bound in
$N\le7$, following from Simons' cone theorem and
Schoen--Simon sheeting \cite{SS81,Simons}, makes these components
parallel hyperplanes.

The zero graphs converge to those planes in $C^2$ on smaller
neighborhoods. The remaining tangential part of $B$ tends to
zero as well: \cite[(2.25)--(2.30)]{CM20} bound it in these
coordinates by
\begin{equation*}
 o(1)+C\widehat\varepsilon_j
 +C\widehat\varepsilon_j^{-1}
       (\widehat\varepsilon_j^2
                   |\log\widehat\varepsilon_j|)^{1+\kappa}
 \longrightarrow0.
\end{equation*}
The $o(1)$ includes the Christoffel symbols in the limiting
planar coordinates. This contradicts $B(0)=1$ and proves
\eqref{stablecurv}.
Thus the curvature rescaling argument of
\cite[Theorem~3.4]{CM20} applies in the present dimensions,
with Theorem~\ref{global} and
Lemma~\ref{closedest} establishing the two required estimates.
\end{proof}

We also use the additional derivative estimate
\cite[(2.7)]{CM20}. On smaller stable neighborhoods, we obtain
\begin{equation}\label{curvderiv}
 \varepsilon_j|\nabla A_{\Gamma_{i,j}}|
 +\varepsilon_j^{1+\theta}
           [\nabla A_{\Gamma_{i,j}}]_\theta\le C_\theta.
\end{equation}
The derivatives and seminorms are computed in the graph
coordinates. This weighted bound will be used in the Dirichlet
construction below.

\subsection{Convergence away from finitely many points}

Let $u_j$ satisfy \eqref{closedac} and
\eqref{closedbd}, with no additional assumption on
the metric. From compactness and bounded-index regularity
\cite{Gua18,HT,TW12}, we obtain, after a subsequence,
\begin{equation}\label{limitdecom}
 V_j\rightharpoonup V
      =\sum_{\ell=1}^J Q_\ell|\Sigma_\ell|,
 \qquad Q_\ell\in\mathbb N.
\end{equation}
We consider a nonzero limit. For $3\le N\le7$, its support is
a smooth embedded closed minimal hypersurface and has finitely
many connected components. Regularity across the possible
points of index concentration follows from
\cite[Theorem~3.8 and Proposition~3.10]{Gua18}.
This concerns the limiting support, not the convergence of the
zero sets at these points.

There is a finite set $S\subset M$, with at most $I_0$ points,
such that every point of $M\setminus S$ has a neighborhood on
which $u_j$ is stable for all sufficiently large $j$.
To see this, pass diagonally over a countable basis of balls so
that stability on each ball either holds eventually or fails
eventually. The set of points without an eventually stable
neighborhood has at most $I_0$ elements: otherwise $I_0+1$
disjoint such neighborhoods would admit that many independent
negative variations. Only local stability on these
neighborhoods is asserted.

Fix a component $\Sigma$ of the limiting support, with
multiplicity $Q$, and a tubular neighborhood disjoint from the
other components. On the component of $M$ containing $\Sigma$,
we have $-1<u_j<1$ for all large $j$. Indeed, a solution
attaining a well is constant there and contributes no mass.
For the analysis on $\Sigma$, replace $S$ by $S\cap\Sigma$.
On compact subsets of $\Sigma\setminus S$,
we obtain from Lemmas~\ref{closedest} and~\ref{stableest}
complete ordered zero graphs. Their number is $Q$:
proper transverse projection implies a locally constant finite
number of roots; the microscopic heteroclinic convergence and
the separation estimate imply energy $\sigma_0$ per transition.
Exponential decay between the transitions and equipartition
then show that this number equals the multiplicity in
\eqref{limitdecom}.

Suppose first that $\Sigma$ is two-sided, with unit normal $\nu$.
In its normal coordinates
\begin{equation*}
 (y,z)\longmapsto\exp_y(z\nu(y)),
\end{equation*}
the local orderings agree on overlaps. We obtain heights on
compact subsets of $\Sigma\setminus S$ satisfying
\begin{equation}\label{ordheights}
 f_{1,j}<\cdots<f_{Q,j},\qquad
 f_{i,j}\longrightarrow0\quad\text{locally in }C^2.
\end{equation}
The convergence also holds in each fixed $C^{2,\theta}$ norm,
by first using a larger H\"older exponent in the local estimates.
The graphs are defined on successively larger members of an
exhaustion of $\Sigma\setminus S$.

When $Q\ge2$, set
\begin{equation*}
 w_j=f_{Q,j}-f_{1,j}>0,\qquad L_j=|\log\varepsilon_j|.
\end{equation*}
For each $K\Subset\Sigma\setminus S$, we obtain from the preceding estimates
\begin{equation}\label{heightlow}
 w_j\ge c_K\varepsilon_jL_j
\end{equation}
and
\begin{equation}\label{heightpde}
 \frac{\|H_{Q,j}-H_{1,j}\|_{L^\infty(K)}}
      {\varepsilon_jL_j}\longrightarrow0.
\end{equation}
Here and below mean curvatures are pulled back to the fixed
hypersurface $\Sigma$.

\subsection{Comparison near the excluded points}

For $p\in\Sigma$, denote by $D_\rho(p)$ the intrinsic ball
of radius $\rho$ in $\Sigma$. The next estimate is the local
comparison used in \cite[Section~4.2]{CM20}. We verify the
hypotheses of its barrier construction and keep track of the
independence from the inner radius.

\begin{lemma}\label{heightest}
Assume that $\Sigma$ is two-sided and $Q\ge2$.
For $p\in S$, choose $\rho>0$ sufficiently small that
$\overline{D_{4\rho}(p)}$ is a coordinate disk containing no
other point of $S$. There is $C_\rho<\infty$ such that, for
every fixed $r\in(0,\rho/2)$ and all sufficiently large $j$,
\begin{equation}\label{heightsup}
 \sup_{D_\rho(p)\setminus D_r(p)}w_j
 \le C_\rho\left(\varepsilon_jL_j
           +\sup_{\partial D_{2\rho}(p)}w_j\right).
\end{equation}
The constant is independent of $r$ and $j$.
\end{lemma}

\begin{proof}
The zero graphs converge in $C^{2,\theta}$ on the fixed annulus
$D_{4\rho}(p)\setminus\overline{D_{\rho/2}(p)}$.
After rescaling the disk by $\rho^{-1}$, we obtain from the implicit function
theorem for the minimal graph equation minimal disks
\begin{equation*}
 d_{j,t}^{+},\ d_{j,t}^{-}:D_{2\rho}(p)\longrightarrow\mathbb R,
 \qquad |t|<t_0,
\end{equation*}
with boundary values $f_{Q,j}+t$ and $f_{1,j}+t$, respectively.
Their derivatives in $t$ solve the corresponding linear
Dirichlet equations with boundary value one. On a sufficiently
small disk these derivatives are bounded above and below by
positive constants. Each family therefore foliates a fixed
smaller tubular neighborhood. The difference of the two disks
at $t=0$ satisfies a uniformly elliptic equation, and from the
Dirichlet estimate we obtain
\begin{equation}\label{diskbound}
 \|d_{j,0}^{+}-d_{j,0}^{-}\|_{L^\infty(D_{2\rho}(p))}
 \le C_\rho\|w_j\|_{L^\infty(\partial D_{2\rho}(p))}.
\end{equation}

We next verify the smallness conditions in
\cite[(7.3)--(7.6)]{CM20}. We choose exponents
$\theta<\theta_+<1$, with $\theta$ above the lower bound
required in \cite[Theorem~7.4]{CM20}, and use
\eqref{curvderiv} with exponent $\theta_+$ on the fixed stable
annulus containing the boundary of the disk. We take
$t_0=\rho\tau_0$, where $\tau_0>0$ is fixed and sufficiently
small, and then fix $\rho$. Let $A_{j,t}$ denote
the second fundamental form of either family, pulled back to
the fixed disk by its graph parametrization. The boundary
values are $f_{Q,j}+t$ and $f_{1,j}+t$, so their tangential
derivatives are independent of $t$. By \eqref{curvderiv}, the
weighted third derivatives of these boundary values are bounded.
The differentiated minimal graph equations are uniformly
elliptic for $|t|\le t_0$. Their boundary Schauder estimates
therefore yield
\begin{equation*}
 \sup_{|t|\le t_0}
 \left(\varepsilon_j\|\nabla A_{j,t}\|_{C^0}
 +\varepsilon_j^{1+\theta_+}
       [\nabla A_{j,t}]_{\theta_+}\right)\le C_\rho.
\end{equation*}
The constant is independent of $j$ and $t$.

The minimal disks converge in $C^{2,\theta_+}$ to a smooth
family, uniformly for $|t|\le t_0$. We denote its second
fundamental forms by $A_{\infty,t}$ and set
\begin{equation*}
 \delta_{j,\rho}
 =\sup_{|t|\le t_0}\|A_{j,t}-A_{\infty,t}\|_{C^0}
 \longrightarrow0.
\end{equation*}
Interpolation on the fixed disk, using the preceding weighted
bound, implies
\begin{align*}
 \varepsilon_j\|D(A_{j,t}-A_{\infty,t})\|_{C^0}
 &\le C_\rho\left(
 \delta_{j,\rho}^{\theta_+/(1+\theta_+)}
 +\varepsilon_j\delta_{j,\rho}\right),\\
 \varepsilon_j^{1+\theta}
 [D(A_{j,t}-A_{\infty,t})]_\theta
 &\le C_\rho\left(
 \delta_{j,\rho}^{(\theta_+-\theta)/(1+\theta_+)}
 +(\varepsilon_j\delta_{j,\rho})^{1-\theta/\theta_+}\right).
\end{align*}
The limiting family is smooth, with uniform bounds in $t$.
The connection terms are controlled by the $C^{2,\theta_+}$
convergence. Hence, for this fixed $\rho$,
\begin{equation*}
 \sup_{|t|\le t_0}
 \left(\varepsilon_j\|\nabla A_{j,t}\|_{C^0}
 +\varepsilon_j^{1+\theta}[\nabla A_{j,t}]_\theta\right)
 \longrightarrow0.
\end{equation*}

After lengths are rescaled by $\rho^{-1}$, the parameter is
$\varepsilon_j/\rho$. The second fundamental form and its
covariant derivative acquire factors $\rho$ and $\rho^2$,
respectively. Each of the two weighted terms in the last
display consequently acquires a factor $\rho$ and still tends
to zero. The induced metrics have uniformly bounded
$C^{1,\theta}$ norms on the rescaled disk, so their weighted
derivatives tend to zero as well. The remaining, unweighted
geometry in \cite[(7.3) and (7.5)]{CM20} is small when
$\rho$ is small and then $j$ is large. Indeed, the limiting
disk families become flat under this rescaling as
$\rho\downarrow0$, uniformly for $|t|\le\rho\tau_0$.
For fixed $\rho$, the Dirichlet Jacobi operators converge as
$j\to\infty$ to those of the limiting disk families. As
$\rho\downarrow0$, these latter operators approach the
Euclidean Dirichlet Laplacian on the rescaled disk. Its positive spectral gap
persists for these small perturbations, which verifies
\cite[(7.9)]{CM20}.

We first choose $\rho$ to meet these geometric and spectral
conditions. With $\rho$ fixed, we choose $j$ so that the weighted
derivatives and $\varepsilon_j/\rho$ meet the bounds in
\cite[Theorem~7.4]{CM20}. These choices are uniform in $t$
and depend only on the fixed outer annulus, hence are independent
of the inner radius $r$. The theorem applies to both families
in the present ambient dimension.

We use its boundary data from \cite[(4.19)--(4.28)]{CM20},
with a fixed large constant $B$. Returning from the rescaled
disk changes $\varepsilon_j|\log\varepsilon_j|$ by
$O_\rho(\varepsilon_j)$. With $\rho$ fixed, we absorb this
change by increasing $B$ once and taking $j$ sufficiently large.
On the constant parts of the boundary, the error
$\varepsilon_j^3=\rho^3(\varepsilon_j/\rho)^3$ is
$o((\varepsilon_j/\rho)^2)$, as required by the rescaled
boundary norm in that theorem.
For the upper family, we change the sign of $u_j$ if necessary so that the phase above its
highest zero is positive. In Fermi coordinates $(y,s)$ about
the disk, the resulting solutions $b_{j,t}$ depend continuously
on $t$ and have boundary values
\begin{equation}\label{barrierbd}
 b_{j,t}=1-\varepsilon_j^3\quad\text{for }s\ge2B\varepsilon_jL_j,
 \qquad
 b_{j,t}=-1\quad\text{for }s\le-2B\varepsilon_jL_j.
\end{equation}
We also need a strict upper bound on the whole boundary.
The explicit boundary values in that construction and the
exponential decay of $1-g$ yield fixed
$m>3$ and $c>0$ such that
\begin{equation}\label{upperbd}
 -1\le b_{j,t}\le1-c\varepsilon_j^m
 \quad\text{on the whole boundary}.
\end{equation}
Indeed, the central part is $g(s/\varepsilon_j)$ plus a
nonnegative $o(1)$ multiple of $g'(s/\varepsilon_j)$.
Since $g'/(1-g)$ is bounded on $[0,\infty)$, its deficit
from $1$ on $|s|\le2B\varepsilon_jL_j$ is at least
$c\varepsilon_j^{2kB}$. The remaining part interpolates
with $1-\varepsilon_j^3$.
We take $m>\max\{3,2kB\}$ and obtain
\eqref{upperbd}.

The same theorem shows that the zero set of $b_{j,t}$ is at
distance $o(\varepsilon_j)$ from the disk and that its central
phase region has width $O(\varepsilon_j)$. The constants are
uniform in $j,t$. The normal shift is
$O(\varepsilon_j^{2-2\alpha})$, where we choose $\alpha<1/2$
sufficiently small for \cite[Theorem~7.4]{CM20}.
This is the transition estimate \cite[(4.29)]{CM20}.

We now apply the sliding argument of
\cite[(4.30)--(4.39)]{CM20}. For each $t$, we pull $u_j$ back
using the Fermi map of $d_{j,t}^{+}$. The coordinate cylinder
is fixed, while this map, the pulled-back metric, and the
pulled-back solution vary continuously with $t$. We use $u_j$
for this pullback throughout the sliding argument. For each fixed
$t$, it and $b_{j,t}$ solve the same equation with the same
metric. The parameter dependence in \cite[Theorem~7.4]{CM20}
also makes $b_{j,t}$ continuous in $t$.
At $t=t_0$, we have $b_{j,t_0}<u_j$: the transition regions are separated,
the boundary inequalities are strict, and the maximum principle
applies in the two well regions. The inequality fails for
$t=-t_0$. At the first contact value $t_j$, we have
$b_{j,t_j}\le u_j$. The strong maximum principle for their
difference implies a boundary contact, unless the solutions
coincide, in which case a boundary contact can again be used.
Also $t_j\to0$, since the initial strict comparison holds
for every fixed positive $t$ when $j$ is sufficiently large.

We next prove $t_j\le C_\rho\varepsilon_jL_j$.
From exponential decay away from the zero set, we obtain
\begin{equation*}
 1-u_j\le C\exp\bigl(-c_0\varepsilon_j^{-1}
                   \operatorname{dist}_{\mathbf g}(\,\cdot\,,\{u_j=0\})\bigr)
\end{equation*}
in the upper phase. A boundary contact cannot have
$s\le-2B\varepsilon_jL_j$, by \eqref{barrierbd}
and $u_j>-1$. On the remaining lateral boundary, the positive
speed of the disks and the small slopes of the zero graphs
show that the signed distance above the highest zero graph
is at least
\begin{equation*}
 c_\rho t_j-C_\rho B\varepsilon_jL_j.
\end{equation*}
The comparison is made on the fixed annulus, where all zeros
are graphs. Points outside this annulus and the horizontal
caps either satisfy the lower boundary condition in
\eqref{barrierbd} or stay a fixed distance above
the transition region. Thus, if
$t_j>C_\rho\varepsilon_jL_j$ with a sufficiently large fixed
constant, exponential decay implies
$u_j>1-c\varepsilon_j^m\ge b_{j,t_j}$ at every possible
contact point. This contradiction proves the bound for $t_j$.

At a zero of $u_j$, the inequality $b_{j,t_j}\le u_j$ and the
transition estimate place that zero below the upper disk,
up to $C_\rho\varepsilon_j$. Consequently,
\begin{equation*}
 f_{Q,j}\le d_{j,0}^{+}+C_\rho\varepsilon_jL_j
 \quad\text{on }D_\rho(p)\setminus D_r(p).
\end{equation*}
We repeat the argument below the lowest graph, with the normal and
phase reversed, and obtain
\begin{equation*}
 f_{1,j}\ge d_{j,0}^{-}-C_\rho\varepsilon_jL_j.
\end{equation*}
Subtract these inequalities and use \eqref{diskbound}.
All constructions use only the fixed outer annulus. The inner
radius $r$ only specifies where the zero graphs are already
defined. Hence the constants are independent of $r$.
\end{proof}

\subsection{Positive Jacobi fields}

\begin{proposition}\label{closedjac}
Let $V$ be the nonzero limit in \eqref{limitdecom}
for an arbitrary smooth metric, and let $\Sigma$ be a connected
component of its support. If $\Sigma$ is two-sided and has
multiplicity at least two, it carries a smooth positive Jacobi
field. If $\Sigma$ is one-sided, its multiplicity is even and
its normal double cover carries a smooth positive Jacobi field.
\end{proposition}

\begin{proof}
First suppose that $\Sigma$ is two-sided and $Q\ge2$.
Let $H(f)$ be scalar mean curvature of a normal graph, pulled
back to $\Sigma$, with
\begin{equation*}
 DH(0)=-J_\Sigma,\qquad
 J_\Sigma=\Delta_\Sigma+|A_\Sigma|^2
                  +\operatorname{Ric}_{\mathbf g}(\nu,\nu).
\end{equation*}
From the normal graph formula \cite[Appendix~A]{CM20} and the
fundamental theorem of calculus, we obtain
\begin{equation*}
 H(f_{Q,j})-H(f_{1,j})
 =\int_0^1DH\bigl(f_{1,j}+t(f_{Q,j}-f_{1,j})\bigr)[w_j]dt.
\end{equation*}
Thus, as in \cite[(4.3)--(4.5)]{CM20},
\begin{equation}\label{diffeq}
 \mathcal L_jw_j=-(H_{Q,j}-H_{1,j}),\qquad
 \mathcal L_j\longrightarrow J_\Sigma
\end{equation}
on compact subsets of $\Sigma\setminus S$.
The operators are uniformly elliptic, with uniformly H\"older
continuous coefficients, by \eqref{ordheights}.

Fix $y_0\in\Sigma\setminus S$ and set
\begin{equation*}
 a_j=w_j(y_0),\qquad v_j=w_j/a_j.
\end{equation*}
Then $v_j>0$, $v_j(y_0)=1$, and
$a_j\ge c\varepsilon_jL_j$ by \eqref{heightlow}.
Equations~\eqref{heightpde} and \eqref{diffeq} yield
\begin{equation*}
 \mathcal L_jv_j=F_j,\qquad
 \|F_j\|_{L^\infty(K)}\longrightarrow0
 \quad(K\Subset\Sigma\setminus S).
\end{equation*}
We follow the local argument for these normalized differences in
\cite[proof of Theorem~4.1 and (4.5)--(4.15)]{CM20}.
Since $n\ge2$, the punctured hypersurface $\Sigma\setminus S$
is connected. The inhomogeneous Harnack inequality and interior
estimates then yield, after a diagonal subsequence,
\begin{equation}\label{punctjac}
 v_j\longrightarrow v\quad\text{locally in }C^1(\Sigma\setminus S),
 \qquad J_\Sigma v=0,\qquad v>0,\qquad v(y_0)=1.
\end{equation}
For $p\in S$, divide \eqref{heightsup} by $a_j$.
Equation~\eqref{heightlow} controls its first term, and
\eqref{punctjac} controls the term on the fixed boundary
$\partial D_{2\rho}(p)$. Hence $v$ is bounded near $p$.
A finite set has zero $W^{1,2}$ capacity when $n\ge2$.
The standard cutoff argument for bounded weak solutions
extends the Jacobi equation across $S$. Elliptic regularity
and the strong maximum principle extend $v$ to a smooth
positive solution on $\Sigma$. In particular, $\Sigma$ is stable, since
\begin{equation}\label{jacstab}
 \int_\Sigma\left(|\nabla\zeta|^2
       -( |A_\Sigma|^2+\operatorname{Ric}_{\mathbf g}(\nu,\nu))\zeta^2\right)
                   d\mu_\Sigma
 =\int_\Sigma v^2|\nabla(\zeta/v)|^2\,d\mu_\Sigma\ge0.
\end{equation}

If $\Sigma$ is one-sided, pass to its connected normal double
cover $\pi:\widetilde\Sigma\to\Sigma$ and lift the tubular
neighborhood and the functions $u_j$. The deck transformation
maps $(y,z)$ to $(\tau(y),-z)$ and leaves each lifted $u_j$ invariant.
The parity argument in \cite[proof of Theorem~4.1]{CM20}
uses only this invariance and the constancy of the two limiting
phases on the connected punctured cover. An odd value of $Q$
would exchange opposite phases, so $Q$ is even and at least two.
Applying the preceding two-sided construction to the highest and
lowest lifted graphs yields a smooth positive Jacobi function on
$\widetilde\Sigma$. This step uses only local stability off
$\pi^{-1}(S)$ and is independent of the global Morse index
of the lifted functions.
\end{proof}

\begin{proof}[Proof of Theorem~\ref{multone}]
Let $V$ be a nonzero subsequential limit in
\eqref{limitdecom}.
If a two-sided component has multiplicity at least two,
Proposition~\ref{closedjac} produces a nonzero positive
Jacobi field on it. If a component is one-sided, that
proposition produces such a field on its normal double cover.
For a bumpy metric both conclusions are impossible, in the
second case by applying the definition to the closed immersed
hypersurface $\widetilde\Sigma\to M$.
For $\operatorname{Ric}_{\mathbf g}>0$, stability on the
component or its double cover contradicts
\eqref{jacstab} with $\zeta=1$. In the one-sided case, we first
integrate on $\widetilde\Sigma$ and then divide by two. In both cases
we obtain
\begin{equation*}
 0\le-\int_\Sigma
       (|A_\Sigma|^2+\operatorname{Ric}_{\mathbf g}(\nu,\nu))\,d\mu_\Sigma<0.
\end{equation*}
Thus every component is two-sided and has multiplicity one.
By equipartition, the full energy measures converge to
$\sigma_0\|V\|$, as asserted.
\end{proof}

\subsection{Geometric applications}

We now follow \cite[Section~6]{CM20}. The results for smooth
multiplicity-one limits in Section~5 of that paper are valid
in arbitrary dimension, so no further extension of its index
estimates is needed. Combined with the min-max construction
of Gaspar--Guaraco \cite{GG18}, Theorem~\ref{multone}
provides another proof of Yau's conjecture for bumpy metrics and
for metrics of positive Ricci curvature in dimensions three
through seven. The proof below is the argument of
Chodosh--Mantoulidis, with the multiplicity-one conclusion
now available in these dimensions.

For a closed minimal hypersurface, $\operatorname{nul}$ denotes
the dimension of the kernel of its normal Jacobi operator;
for an Allen--Cahn solution it denotes the dimension of the
kernel of its second variation. For a disconnected minimal
hypersurface, both index and nullity are the sums over its
connected components.

\begin{corollary}\label{convindex}
Under the hypotheses of Theorem~\ref{multone}, let
$V=|\Sigma|$ be a nonzero limit along a subsequence,
with $\Sigma=\bigcup_{\ell=1}^J\Sigma_\ell$.
All the conclusions below concern this subsequence, which we continue
to index by $j$.
For each $\beta,\theta\in(0,1)$, the level sets $\{u_j=t\}$
converge to $\Sigma$ as single normal graphs in $C^{2,\theta}$,
uniformly for $|t|\le1-\beta$. Moreover,
\begin{equation}\label{indexbound}
 \operatorname{ind}(\Sigma)
      \le\liminf_{j\to\infty}\operatorname{ind}(u_j),\qquad
 \operatorname{ind}(u_j)+\operatorname{nul}(u_j)
      \le\operatorname{ind}(\Sigma)+\operatorname{nul}(\Sigma)
\end{equation}
for all sufficiently large $j$ in the second inequality.
If $\mathbf g$ is bumpy, then eventually
\begin{equation*}
 \operatorname{ind}(u_j)=\operatorname{ind}(\Sigma),\qquad
 \operatorname{nul}(u_j)=0.
\end{equation*}
\end{corollary}

\begin{proof}
Theorem~\ref{multone} provides the smooth
two-sided limit with multiplicity one required in
\cite[Section~5]{CM20}. Its Lemma~5.4 implies the convergence of
the level sets, including near the points of $S$.
The first inequality in \eqref{indexbound} is the
index lower semicontinuity \cite[Theorem~A and Corollary~B]{Hie18}.
The second is \cite[Theorem~5.11]{CM20}.
For a bumpy metric, $\operatorname{nul}(\Sigma)=0$.
Since indices are integers, the first inequality eventually
implies $\operatorname{ind}(u_j)\ge\operatorname{ind}(\Sigma)$;
the second inequality proves the last assertion.
\end{proof}

\begin{corollary}\label{yau}
Let $(M^N,\mathbf g)$ be a smooth closed connected Riemannian
manifold, $3\le N\le7$, whose metric is bumpy or has positive
Ricci curvature. There is $C=C(M,\mathbf g,W)>1$ such that,
for every positive integer $p$, there is a smooth closed
embedded minimal hypersurface $\Sigma_p$, all of whose
components are two-sided, satisfying
\begin{equation}\label{areagrowth}
 C^{-1}p^{1/N}\le\mathcal H_{\mathbf g}^{N-1}(\Sigma_p)
                       \le Cp^{1/N}
\end{equation}
and
\begin{equation*}
 \operatorname{ind}(\Sigma_p)\le p
       \le\operatorname{ind}(\Sigma_p)+\operatorname{nul}(\Sigma_p).
\end{equation*}
If $\mathbf g$ is bumpy, then $\operatorname{ind}(\Sigma_p)=p$.
If $\operatorname{Ric}_{\mathbf g}>0$, then $\Sigma_p$ is connected.
In either case, $M$ contains infinitely many geometrically
distinct connected closed embedded two-sided minimal hypersurfaces.
\end{corollary}

\begin{proof}
We modify $W$ evenly outside $[-1,1]$ to meet the growth assumptions
in \cite{GG18}. The critical points used below take values in
$[-1,1]$, so the equation and their energies are unchanged.
For each $p$, the construction in \cite[Theorems~3.2 and~3.3]{GG18} yields
$\varepsilon_j\downarrow0$ and critical points $u_{p,j}$ such that
\begin{equation*}
 \begin{split}
 &C^{-1}p^{1/N}\le\mathcal E_{\varepsilon_j}(u_{p,j};M)
       \le Cp^{1/N},\\
 &\operatorname{ind}(u_{p,j})\le p
       \le\operatorname{ind}(u_{p,j})+\operatorname{nul}(u_{p,j}),
 \end{split}
\end{equation*}
where $C$ is independent of $p$.
The remaining steps are the dimension-independent part of
\cite[proof of Corollary~6.2]{CM20}. Theorem~\ref{multone}
and equipartition yield a smooth two-sided multiplicity-one limit
$|\Sigma_p|$ and \eqref{areagrowth}. Corollary~\ref{convindex}
implies the asserted index and nullity inequalities. For a bumpy
metric the nullity vanishes, so $\operatorname{ind}(\Sigma_p)=p$.
If $\operatorname{Ric}_{\mathbf g}>0$, Frankel's theorem
\cite{Frankel66} makes $\Sigma_p$ connected.
By \eqref{areagrowth}, the areas tend to infinity. Finitely many
connected hypersurfaces would produce only finitely many
multiplicity-one unions, all with bounded area. There are therefore
infinitely many geometrically distinct connected hypersurfaces
among the components of the $\Sigma_p$.
\end{proof}

For a bumpy metric we can also select connected components
with large index, as in \cite[Corollary~6.4]{CM20}.
By monotonicity, there exists a constant $a_0=a_0(M,\mathbf g)>0$ such
that every nonempty closed minimal hypersurface has area at
least $a_0$. Hence \eqref{areagrowth} implies
\begin{equation*}
 \#\pi_0(\Sigma_p)\le Cp^{1/N}.
\end{equation*}
Since the sum of the component indices is $p$, some connected
component $\Gamma_p\subset\Sigma_p$ satisfies
\begin{equation*}
 \operatorname{ind}(\Gamma_p)\ge c p^{(N-1)/N},\qquad
 \mathcal H_{\mathbf g}^{N-1}(\Gamma_p)\le Cp^{1/N}.
\end{equation*}

\appendix

\section{One-dimensional estimates}\label{aux}

This appendix contains the one-dimensional estimates used to recover
the full transition energy and to compute the interaction coefficients.
The first subsection treats truncated normal lines, including endpoint
terms and controlling nonmonotonicity so that Lemma~\ref{count}
accounts for the tails near both wells. The second proves
Proposition~\ref{quanti}, computing the leading interaction between
adjacent heteroclinics and estimating the remainders and their
derivatives with respect to the layer positions. The interaction
calculation accounts for the different decay rates required for asymmetric
potentials.

\subsection{Estimates on normal lines}
Our aim is to compare the energy in the tails with the number of
crossings of a nearby level. A normal line need not be monotone,
and truncating the line introduces endpoint terms. The next lemma
includes these terms and measures the effect of nonmonotonicity by
the one-dimensional discrepancy. This is the form needed in
Lemma~\ref{count}.

Fix a normal line and a closed interval $J=[z_-,z_+]$ on it, with
a fixed ambient neighborhood contained in the domain.

\begin{lemma}\label{odtail}
There are $\delta_*,r_*>0$, depending only on $W$ and the local
derivative constants, such that the following estimates hold for
$0<\delta\le\delta_*$. We take $\varepsilon$ sufficiently small that
$J^+=(z_--r_*\varepsilon,z_++r_*\varepsilon)$ stays in the fixed
neighborhood. Then
\begin{equation}\label{ultail}
\begin{split}
 &\int_{J\cap\{u_\varepsilon>1-\delta\}}
                         \sqrt{2W(u_\varepsilon)}|u_{\varepsilon,z}|\,dz\\
 \le& \left(\int_{1-\delta}^1\sqrt{2W(s)}\,ds\right)
       \#\{z\in\operatorname{int}J\mid 
                u_\varepsilon=1-\delta,\ u_{\varepsilon,z}\ne0\}\\
 &\quad+\frac C\varepsilon
        \int_{J^+\cap\{u_\varepsilon>1-2\delta\}}\left(\sqrt{2W(u_\varepsilon)}-\varepsilon|u_{\varepsilon,z}|\right)^2\,dz\\
 &\quad+\sum_{z\in\{z_-,z_+\}}
     \mathbf1_{\{u_\varepsilon(z)\ge1-\delta\}}
                  \int_{u_\varepsilon(z)}^1\sqrt{2W(s)}\,ds.
\end{split}
\end{equation}
For every fixed $0<\kappa<1$ one also has
\begin{equation}\label{lltail}
\begin{split}
 &\int_{J\cap\{u_\varepsilon>1-\delta\}}
                         \sqrt{2W(u_\varepsilon)}|u_{\varepsilon,z}|\,dz\\
 \ge& \left(\int_{1-\delta}^1\sqrt{2W(s)}\,ds\right)
       \#\{z\in\operatorname{int}J\mid 
         u_\varepsilon=1-\delta,\ 
         \varepsilon|u_{\varepsilon,z}|\ge\kappa\sqrt{2W(1-\delta)}\}\\
 &\quad-\frac {C_\kappa}\varepsilon
        \int_{J\cap\{u_\varepsilon>1-\delta\}}\left(\sqrt{2W(u_\varepsilon)}-\varepsilon|u_{\varepsilon,z}|\right)^2\,dz\\
 &\quad-\sum_{z\in\{z_-,z_+\}}
     \mathbf1_{\{u_\varepsilon(z)\ge1-\delta\}}
                  \int_{u_\varepsilon(z)}^1\sqrt{2W(s)}\,ds.
\end{split}
\end{equation}
\end{lemma}
\begin{proof}
The primitive change of variables in Hutchinson--Tonegawa
\cite[Section~2.1]{HT} and the one-dimensional chain rule
\cite[Sections~3.2 and~3.10]{AFP} express the integral in the lemma
as a total variation. Constant well solutions satisfy the estimates
immediately. For every other solution, \eqref{relip} implies
\begin{equation*}
 |D\log(1-u_\varepsilon)|\le C/\varepsilon,
 \qquad |u_{\varepsilon,zz}|
                    \le C\varepsilon^{-2}(1-u_\varepsilon).
\end{equation*}
Near the positive well, nondegeneracy also yields
\begin{equation*}
 c(1-s)\le\sqrt{2W(s)}\le C(1-s),\qquad
 c(1-s)^2\le\int_s^1\sqrt{2W(t)}\,dt\le C(1-s)^2.
\end{equation*}
We choose $r_*>0$ sufficiently small. If
$u_{\varepsilon,z}(z_0)=0$ and $u_\varepsilon(z_0)\ge1-\delta$,
the logarithmic bound and integration of $u_{\varepsilon,zz}$ imply,
on $|z-z_0|<r_*\varepsilon$,
\begin{equation*}
 \begin{gathered}
 u_\varepsilon>1-2\delta,\qquad
 \varepsilon|u_{\varepsilon,z}|
                       \le\tfrac12\sqrt{2W(u_\varepsilon)},\\
 c(1-u_\varepsilon(z_0))\le1-u_\varepsilon(z)
                       \le C(1-u_\varepsilon(z_0)).
 \end{gathered}
\end{equation*}
Consequently, for every $0<r\le r_*$,
\begin{equation}\label{clcost}
 \frac1\varepsilon
 \int_{z_0-r\varepsilon}^{z_0+r\varepsilon}
 \left(\sqrt{2W(u_\varepsilon)}
                -\varepsilon|u_{\varepsilon,z}|\right)^2\,dz
 \ge cr(1-u_\varepsilon(z_0))^2.
\end{equation}
The constants are uniform for $0<\delta\le\delta_*$.

For the upper bound, we partition the line into intervals of length
$r_*\varepsilon/10$. In each interval containing a critical point
in $J$ with value at least $1-\delta$, we select one such point
and remove that interval and its two neighbors from $J$.
Only finitely many intervals are involved, even if critical points
accumulate. The relative derivative bound controls the integral
over each removed triple and the primitive values at its ends by
$C(1-u_\varepsilon(z_0))^2$. By \eqref{clcost}, this is bounded
by the error integral over $|z-z_0|<r_*\varepsilon$.
These larger intervals lie in $J^+\cap\{u_\varepsilon>1-2\delta\}$
and have bounded overlap, since each partition interval contributes
at most one selected point.

On each remaining component of $\{u_\varepsilon>1-\delta\}$,
the function is monotone, so its integral equals the absolute
difference of the primitive values at that component's two ends.
Each remaining interior crossing is regular and borders one such
component. It contributes at most
$\int_{1-\delta}^1\sqrt{2W(s)}\,ds$.
The ends created by the removal have already been estimated,
with at most two assigned to each removed triple.
The original ends of $J$ contribute the sum in \eqref{ultail}.
We sum these bounds to prove the upper estimate.

For the lower bound, a component meeting exactly one end of $J$ has at most
one interior crossing satisfying the derivative condition in
\eqref{lltail}. If such a crossing is present, the integral over
the component is at least
$\int_{1-\delta}^1\sqrt{2W(s)}\,ds$ minus the primitive at that
end. A component meeting both ends of $J$ contributes zero to the
interior crossing count. This accounts for the end terms.

We consider an interior component with a crossing satisfying
the derivative condition in \eqref{lltail} and orient the line
into it from that crossing. Initially
$u_{\varepsilon,z}\ge c\kappa\delta/\varepsilon$.
The relative first and second derivative bounds keep this derivative
at least half as large for a distance $d_\kappa\varepsilon$,
where $d_\kappa>0$ is independent of $\delta$.
This interval stays in the component. At a maximum point $z_*$,
we consequently have
\begin{equation*}
 1-u_\varepsilon(z_*)\le(1-\gamma_\kappa)\delta
 \qquad\text{for some }\gamma_\kappa>0.
\end{equation*}
The logarithmic derivative bound places $z_*$ at distance at least
$c_\kappa\varepsilon$ from both ends of the component.
We choose $0<c'_\kappa\le r_*$ so that the interval of radius
$c'_\kappa\varepsilon$ about $z_*$ lies inside it.
The quadratic bound for the primitive and \eqref{clcost} imply
\begin{equation*}
 \int_{u_\varepsilon(z_*)}^1\sqrt{2W(s)}\,ds
 \le\frac{C_\kappa}{\varepsilon}
 \int_{z_*-c'_\kappa\varepsilon}^{z_*+c'_\kappa\varepsilon}
 \left(\sqrt{2W(u_\varepsilon)}
                 -\varepsilon|u_{\varepsilon,z}|\right)^2\,dz.
\end{equation*}
Both ends have value $1-\delta$, so the variation through $z_*$
is at least
\begin{equation*}
 2\int_{1-\delta}^1\sqrt{2W(s)}\,ds
       -2\int_{u_\varepsilon(z_*)}^1\sqrt{2W(s)}\,ds.
\end{equation*}
This controls its one or two counted crossings with the required
error. The error intervals for distinct components are disjoint.
Each counted component has length at least $d_\kappa\varepsilon$,
so only finitely many occur in $J$. All other components contribute
nonnegative integrals. Summation, including the components meeting
the ends of $J$, proves \eqref{lltail}.

The negative-well estimates follow by applying this argument to
$-u_\varepsilon$ and $W(-s)$.
\end{proof}

\subsection{The interaction integrals}\label{inter}
We now prove Proposition~\ref{quanti}. The main objective is to
compute the coefficient and sign of the interaction between
adjacent layers, and to show that the remaining integrals are of
smaller order. The estimates must also remain valid after
differentiation with respect to the layer positions, since the
projected equations use these differentiated interactions.
We follow the calculations of Wang--Wei
\cite[Appendix~A, Lemmas~A.1--A.2, and Appendix~B, Lemma~B.1]{WW},
keeping the two rates $\sqrt{W''(-1)}$ and $\sqrt{W''(1)}$ separate.
We also explain how the estimates behave under differentiation with
respect to the layer positions. The calculation of the leading
coefficient follows the integration by parts in
\cite[Appendix~A, proof of Lemma~A.1, Step~4]{WWends}.

\begin{proof}
We set $d_j=f_{j+1}-f_j$ and $k_j=\sqrt{W''(s_j)}$ for
$1\le j<Q$. By \eqref{gtail}, every fixed positive-order derivative
of a heteroclinic decays at the same rate as its difference from the
limiting well. Differentiation with respect to a layer position
preserves these bounds. The constants below depend on $W,Q$ and
the fixed differentiation order, and are independent of the layer
positions and $D$.

We first consider an increasing layer at $f_i$ and its decreasing
right neighbor. With $t=z-f_i$, the neighboring heteroclinic
differs from $1$ by
\begin{equation*}
  -A_+e^{-k_i d_i}e^{k_i t}
        +O(e^{-2k_i(d_i-t)})\qquad\text{for }t\le d_i/2.
\end{equation*}
We expand the integrand in \eqref{suu} about $g(t)$ and subtract
the contribution $W'(g(d_i-t))$ from the neighboring layer.
The resulting linear term is
\begin{equation*}
  -A_+e^{-k_i d_i}e^{k_i t}
                    [W''(g(t))-k_i^2]g'(t).
\end{equation*}
We evaluate its integral using $g'''=W''(g)g'$.
At $+\infty$, we have $W''(g)-k_i^2=O(e^{-k_i t})$.
At $-\infty$, both $e^{k_i t}$ and $g'$ decay.
The integral therefore converges, and integration by parts yields
\begin{equation}\label{gfwron}
  \begin{split}
  \int_{\mathbb R}e^{k_i t}[W''(g(t))-k_i^2]g'(t)\,dt
 &=\left[e^{k_i t}(g''(t)-k_i g'(t))\right]_{-\infty}^{+\infty}\\
 &=-2k_i^2A_+.
  \end{split}
\end{equation}
Thus the right neighbor contributes $2k_i^2A_+^2e^{-k_i d_i}$.
For an interaction through the well $-1$, the same calculation uses
$A_-$ and $W''(-1)$. Reflection of the variable $z$ treats decreasing
layers. With the signed test $s_i g'(s_i(z-f_i))$, the right-neighbor
contribution is positive and the left-neighbor contribution is
negative. These are the two leading terms in \eqref{suu}.

We now estimate the errors in this calculation. We divide the line
at the midpoints of consecutive layer positions. On the interval
containing $f_i$, we expand
$W'$ about $g(s_i(z-f_i))$. The other heteroclinics contribute their
differences from the well values facing this interval. The constant
well values cancel in the definition of $G$. After subtracting the
individual terms $W'(g(s_j(z-f_j)))$, each linear coefficient is
the difference between $W''(g(s_i(z-f_i)))$ and $W''$ at the relevant
well. Taylor's formula bounds the remainder by the sum of the
squares of the other heteroclinic differences.

The estimates for adjacent layers follow from their common decay
rate $k_j$ in the gap $(f_j,f_{j+1})$. The product of their
differences from the common well is bounded by $Ce^{-k_jd_j}$.
After including the decaying parts outside the gap, the integral
of its square is bounded by
\begin{equation*}
  C(1+d_j)e^{-2k_jd_j}
        \le Ce^{-c_WD}e^{-k_jd_j}.
\end{equation*}
When we multiply a quadratic remainder by the derivative of an
adjacent layer, the half of the gap nearest that layer contributes
at most
\begin{equation*}
  \int_0^{d_j/2}e^{-2k_j(d_j-t)}e^{-k_jt}\,dt
        \le Ce^{-3k_jd_j/2}.
\end{equation*}
The side of that layer facing away from the neighbor contributes
$O(e^{-2k_jd_j})$. These estimates also hold for fixed
positive-order derivatives of the heteroclinics. We choose $c_W>0$
small enough to absorb the factors $1+d_j$.

For layers farther apart, we use the decay across the intervening
gaps. After a heteroclinic has crossed the first complete gap
$(f_j,f_{j+1})$, its difference from the limiting well is at most
$Ce^{-k_jd_j}$. Each nonadjacent midpoint interval lies at least
$D/2$ farther away. The integral there against a fixed
heteroclinic derivative is bounded by
$Ce^{-k_jd_j}e^{-k_jD/2}$, since that derivative has uniformly
bounded $L^1$ norm. The integral of the squared heteroclinic
difference over these farther intervals is bounded by
$Ce^{-2k_jd_j}e^{-k_jD}$.
This estimate is also needed when $W''(-1)\ne W''(1)$:
a linear coefficient may stay bounded away from zero near the
opposite well, while the heteroclinic difference multiplying it
has already crossed a complete gap.

We apply these bounds to the Taylor expansion on each midpoint
interval. An adjacent linear term is bounded by $Ce^{-k_jd_j}$.
On the side toward the common well, its coefficient decays at the
same rate as the difference between the central heteroclinic and
that well. On the other side, the neighboring heteroclinic has
crossed the gap. Extending the leading integral
to the full line in \eqref{gfwron} produces an error bounded by
$Ce^{-c_WD}e^{-k_jd_j}$ at each end. If the test comes from a layer
outside this interval, its derivative has integral
$O(e^{-c_WD})$ there. For the quadratic terms, a test from an
adjacent layer is controlled by the half-gap estimate above.
Every other test is $O(e^{-c_WD})$ on the interval, while the
squared heteroclinic difference has integral $O(e^{-k_jd_j})$.
Summing over the intervals proves \eqref{suu} for $\alpha=0$.
The squared linear terms and Taylor remainders have integrals
bounded by $C(1+d_j)e^{-2k_jd_j}$ for adjacent layers. Combining
these bounds with the estimates for the farther layers proves
\eqref{drtf}.

For derivatives with respect to the layer positions, we
differentiate the integrand on the full line and Taylor's integral
remainder before dividing the line at the midpoints. This order
avoids boundary terms from the moving midpoints. Every
differentiated heteroclinic difference has the same decay rate,
and each differentiated quadratic remainder contains at least two
such differences or their derivatives. The derivatives of $W$
appearing as coefficients are bounded. Thus the preceding
estimates justify differentiation under the integral and control
the differentiated error. The leading terms are the corresponding
derivatives of the two exponential terms in \eqref{suu}, which proves
that estimate for every fixed multi-index $\alpha$.

It remains to estimate the functions in \eqref{tizqr}. On the
midpoint interval containing $f_i$, the definition of $G$ and the
mean value theorem imply
\begin{equation*}
 \begin{split}
  |W''(G)-W''(g(s_i(z-f_i)))|
  \le{}&C\sum_{j<i}|g(s_j(z-f_j))-s_j|\\
       &+C\sum_{j>i}|g(s_j(z-f_j))+s_j|.
 \end{split}
\end{equation*}
The wells facing this interval are $s_j$ for $j<i$ and $-s_j$
for $j>i$.
After multiplication by $g'(s_i(z-f_i))$ or $g''(s_i(z-f_i))$,
the preceding estimates for adjacent and farther layers apply on
this interval. They also control the region between this interval
and each neighboring layer. Beyond a neighboring layer, the
coefficient is bounded and the multiplying derivative has crossed
a complete gap. The integral of the square of this derivative is
at most $Ce^{-2k_jd_j}$, even near the opposite well. Here $j=i-1$
for the left neighboring layer and $j=i$ for the right neighboring layer.
The same argument controls the difference of the third
derivatives of $W$ multiplied by $g'(s_i(z-f_i))^2$, and the
products of derivatives of distinct heteroclinics.

Differentiation in $f_i$ replaces a difference of derivatives of
$W$ by a higher-order difference multiplied by a heteroclinic
derivative, or differentiates the existing multiplying factor.
Differentiation in another position introduces the decaying
derivative of that heteroclinic. Each resulting term therefore
has the factors used in the preceding $L^2$ estimates.
This proves the asserted bounds for \eqref{tizqr} and its
fixed-order position derivatives.
Finally, exponential decay bounds the integrals of products of
fixed positive-order derivatives of distinct heteroclinics by
$C(1+|f_i-f_j|)e^{-c_W|f_i-f_j|}$.
Position differentiation preserves this bound, which tends to
zero uniformly as $D\to\infty$ because $|f_i-f_j|\ge D$.
\end{proof}

\section*{Acknowledgements} 
Yong Liu was supported by  National Natural Science Foundation of China No. 12471204. Juncheng Wei was supported by National R\&D Program of China (Grant No. 2022YFA1005602), and Hong Kong General Research Fund (No. 14303125) ``On Fujita equation in the critical or supercritical regime''. Kelei Wang was supported by the National Key R\&D Program of China (No. 2022YFA1005602) and the National Natural Science Foundation of China (No. 12425108 and No. 12221001). Yong Wei was supported by the National Key R\&D Program of China (2021YFA1001800) and National Natural Science Foundation of China (No.~12531002).  Ke Wu is supported by the National Natural Science Foundation of China (No. 12401264) and Yunnan Revitalization Talent Support Program. The authors used OpenAI models as assistive tools in preparing
this manuscript. Mathematical arguments, proofs, and computations in the manuscript were verified and finalized by the authors.

\end{document}